\documentclass{article}

\usepackage{fullpage, parskip}
\usepackage[square,numbers]{natbib}
\usepackage[utf8]{inputenc} %
\usepackage[T1]{fontenc}    %
\usepackage[hypertexnames=false]{hyperref}       %
\usepackage{url}            %
\usepackage{booktabs}       %
\usepackage{amsfonts}       %
\usepackage{nicefrac}       %
\usepackage{microtype}      %
\usepackage{logistics}

\newif\ifarxiv

\arxivtrue
\renewcommand\paragraph\textbf

\title{Big-Step-Little-Step: \\
Efficient Gradient Methods for Objectives with Multiple Scales\ifarxiv\thanks{Authors are listed in alphabetical order.}\fi
}

\author{%
  Jonathan Kelner \\
  MIT \\
  \texttt{kelner@mit.edu}
   \and
   Annie Marsden \\
   Stanford University \\
  \texttt{marsden@stanford.edu}
  \and
   Vatsal Sharan \\
   USC\thanks{Part of the work was done while the author was at MIT.} \\
  \texttt{vsharan@usc.edu}
  \and
   Aaron Sidford \\
   Stanford University \\
  \texttt{sidford@stanford.edu}
   \and
   Gregory Valiant \\
   Stanford University \\
  \texttt{valiant@stanford.edu}
   \and
   Honglin Yuan \\
   Stanford University \\
  \texttt{yuanhl@cs.stanford.edu}
}
\begin{document}
\date{}
\maketitle
\thispagestyle{empty} 
\begin{abstract}
We provide new gradient-based methods for efficiently solving a broad class of ill-conditioned optimization problems. We consider the problem of minimizing a function $f : \mathbb{R}^d \rightarrow \mathbb{R}$ which is implicitly decomposable as the sum of $m$ unknown non-interacting smooth, strongly convex functions and provide a method which solves this problem with a number of gradient evaluations that scales (up to logarithmic factors) as the product of the square-root of the condition numbers of the components. This complexity bound (which we prove is nearly optimal) can improve almost exponentially on that of accelerated gradient methods, which grow as the square root of the condition number of $f$. Additionally, we provide efficient methods for solving stochastic, quadratic variants of this multiscale optimization problem. Rather than learn the decomposition of $f$ (which would be prohibitively expensive), our methods apply a clean recursive ``Big-Step-Little-Step'' interleaving of standard methods. The resulting algorithms use  $\tilde{\mathcal{O}}(d m)$ space, are numerically stable, and open the door to a more fine-grained understanding of the complexity of convex optimization beyond condition number.

\end{abstract}

\clearpage
\setcounter{tocdepth}{2}
\begin{spacing}{0.90} %
\tableofcontents
\thispagestyle{empty} 
\end{spacing}
\clearpage
\setcounter{page}{1}

\vspace{-20pt}
\section{Introduction}
\vspace{-5pt}
Smooth, strongly-convex function minimization is a fundamental and canonical problem in optimization theory and machine learning. Given an $L$-smooth, $\mu$-strongly convex $f : \reals^d \rightarrow \reals$ it is well known that gradient descent and accelerated gradient descent minimize $f$ with $\otilde(\kappa)$ and $\otilde(\sqrt{\kappa})$ gradient queries respectively for $\kappa \defeq L/\mu$.\footnote{$\otilde(\cdot)$ hides factors poly-logarithmic in the function error of the initial point, desired accuracy, and condition numbers.}  Further, any first-order method, i.e. one restricted to accessing $f$ through an oracle which returns the value and gradient of $f$ at a queried point, must make $\Omega(\sqrt{\kappa})$ queries \cite{Nemirovski.Yudin-83} in the (dimension-independent) worst case (even if randomized \cite{woodworth2017lower}).
Consequently $\kappa$, the \emph{condition number}, nearly captures the worst-case complexity of the problem.

In this paper, we seek to move beyond this traditional measure of problem complexity and obtain a more fine-grained understanding of the complexity of smooth, strongly convex function minimization. In the special case of quadratic function minimization where $\nabla^2 f$ has a small number of distinct eigenvalue clusters, it has long been known that methods like conjugate gradient can efficiently solve the problem with far fewer gradient queries than would be indicated by the condition number of the problem \cite{Trefethen.Bau-97}.  This fact has been leveraged in a variety of contexts for improved methods \cite{polyak1969conjugate,nocedal1996conjugate,saad2003iterative,nocedal2006conjugate}. %

The central question we ask in this paper is \emph{whether there is an analog of this phenomenon for non-quadratic and stochastic optimization problems}. Although methods like non-linear conjugate gradient \cite{Fletcher.Reeves-64,Hager.Zhang-06} and limited-memory Quasi-Newton methods \cite{Nocedal-80,Liu.Nocedal-MP89} are prevalent and effective in practice, we currently lack a complete theoretical understanding of when they are (provably) effective.
In this work, we answer this question in the affirmative and give efficient first-order methods for solving natural classes of non-quadratic and stochastic multi-scale optimization problems. 

We focus on the problem of minimizing a function $f$ which is decomposable as the sum of $m$ non-interacting smooth, strongly convex functions $f_i$ each with condition number $\cond_i$. When each $\kappa_i$ is small and the smoothness of each component, $L_i$, is similar, the overall function is well-conditioned and can be solved efficiently. However, when the components are at different scales, i.e. the $L_i$ vary, the overall function can be ill-conditioned. We provide methods that depend only poly-logarithmically on the overall condition number and polynomially on the $\kappa_i$, improving almost exponentially on the complexity of AGD in certain cases.
We complement this result with novel and nearly matching lower bounds which show that our methods are close to optimal for the class of first-order methods.

The motivation for considering the specific setting of a sum of non-interacting functions that are at different scales is largely theoretical. This model serves as a natural starting place when considering what types of structure can be algorithmically leveraged to go beyond guarantees in terms of the (global) condition number. Still, the setting we focus on is not too removed from the types of structure one might encounter in practice. Indeed, many optimization problems possess structure at unknown and widely varying scales, and optimization approaches  designed to gracefully handle such scaling, such as \texttt{AdaGrad} \cite{duchi2011adaptive}, have proved to be extremely useful in practice. Our assumption that the components of the objective function are completely non-interacting may be unrealistic, but we are hopeful that our techniques might extend to more general classes of structured, poorly conditioned, optimization settings.

\vspace{-5pt}
\subsection{Setup and overview}
\label{sec:sub:results}
\begin{definition}%
\label{def:multiscale_problem}
The \textbf{multiscale (convex) optimization problem} asks to approximately solve the problem that can be implicitly decomposed as follows:
\vspace{-8pt}
\begin{equation}
	\min_{\x \in \reals^d} f(\x) \defeq \sum_{i \in [m]} f_i(\proj_i  \x) 
	\enspace \text{ where } \enspace
	\mat{P}_i \in \R^{d_i \times d} 
	\enspace \text{ with } \enspace
					\proj_i  \proj_j^\top = 
	\begin{cases}
					\id_{d_i \times d_i} & \text{if $i = j$,}
				  \\
					\mat{0}_{d_i \times d_j} & \text{otherwise, }
	\end{cases}
\end{equation}
the projections $\proj_i$ and functions $f_i$ are \textbf{unknown}, each $f_i : \R^{d_i} \rightarrow \R$ has condition number $\kappa_i \defeq L_i/\mu_i$ where $f_i$ is $L_i$-smooth and $\mu_i$-strongly convex with $L_i < \mu_{i + 1}$ for all $i < m$.\footnote{
This is without loss of generality (up to logarithmic factors in our claimed bounds) by re-defining any $f_i$, $f_j$ pair with $[\mu_i, L_i] \cap [\mu_j, L_j] \neq \emptyset$ as a single $f_i$, sorting the $L_i$, and noting that $\mu_i \leq L_i$ (by known properties of smoothness and convexity).
}%
\end{definition}
We focus on optimizing such objectives given \emph{only} a gradient oracle for $f$. We also note that though the above problem (and our yet to be introduced algorithms) are well-defined for any $m\in[d]$, we will treat $m$ as a constant when stating the asymptotic bounds of our algorithms for simplicity.

For intuition, one simple case of \cref{def:multiscale_problem} is the quadratic minimization problem $f(\x) = \frac{1}{2} \vec{x}^\top \mat{A} \vec{x} - \vec{b}^\top \x$ where $\mat{A}$ is positive-definite and its eigenvalues are all located in $\bigcup_{i \in [m]}[\mu_i, L_i]$. To see this, note that the spectral decomposition of $\mat{A}$ can be written as $\mat{A} = \sum_{i \in [m]} \mat{P}_i^\top \mat{\Lambda}_i \mat{P}_i$ such that $\mat{\Lambda}_i$ is a diagonal matrix with diagonal entries within $[\mu_i, L_i]$, and that the matrices $\{\mat{P}_i\}_{i=1}^m$ are pairwise orthogonal. By setting $f_i(\y) := \frac{1}{2} \y^\top \mat{\Lambda}_i  \y - \b^\top \mat{P}_i^\top \y$ one can cast the problem in the form of \cref{def:multiscale_problem}. 

Since any objective $f$ satisfying \cref{def:multiscale_problem} must be $\mu_1$-strongly-convex and $L_m$-smooth, one may directly apply gradient descent or accelerated gradient descent to minimize $f$ within $\tildeo(\globalcond)$ or $\tildeo(\sqrt{\globalcond})$ gradient queries, where $\globalcond$ is the global conditioner $L_m/\mu_1$.
However, as we show in \cref{sec:line_search}, gradient descent with constant step-size or line-search does not take full advantage of the additional structure beyond the $\mu_1$-strong-convexity and $L_m$-smoothness from \cref{def:multiscale_problem}. In this work, we aim to leverage the structure of \cref{def:multiscale_problem} to develop faster algorithms.

Our first contribution is to develop a new algorithm ``Big-Step-Little-Step'' or $\bsls$ (Algorithm \ref{alg:bsls}) which takes advantage of the structure of \cref{def:multiscale_problem} and as a result solves the problem within $\tildeo(\prod_{i \in [m]} \kappa_i)$ gradient queries of $f$ (see \cref{thm:bsls:recursive}).
Since $\mu_1 \leq L_1 < \mu_2 \leq L_2 < \cdots < \mu_m \leq L_m$ by \cref{def:multiscale_problem}, it is always the case that $\prod_{i \in [m]} \kappa_i \leq \globalcond$. 
Therefore, $\bsls$ asymptotically outperforms gradient descent (which has complexity $\tildeo(\globalcond)$).  Moreover, if the clusters $[\mu_i, L_i]$ are well-separated, i.e.\ $\prod_{i \in [m]} \kappa_i \ll \globalcond$, the complexity of $\bsls$ can \emph{significantly} outperform accelerated gradient descent. Indeed, in the case where $m$ and each $\cond_i$ are constant, the asymptotic performance (with respect to $\globalcond$) of $\bsls$ is almost an exponential improvement on the performance of accelerated gradient method; see \cref{fig:expr} for a small experimental comparison in the quadratic minimization setting\footnote{Code for all experiments available here: \url{https://github.com/hongliny/BSLS}.}. We also show in \cref{thm:bsls:inexact} that $\bsls$ is numerically stable under finite-precision arithmetic.

The $\bsls$ algorithm consists of a natural interleaving of steps at different sizes. Intuitively, $\bsls$ alternates between taking a bigger step-size to make progress on an objective $f_i$ which has a smaller scale (i.e. a small value of $L_i$), followed by a sequence of smaller steps to fix the errors caused due to this large step in the objectives $f_j$ which have a larger scale (i.e. all $j > i$). The entire framework is recursive -- the sequence of smaller steps for $j>i$ are themselves defined recursively, see \cref{sec:bsls_intuition_results} for further intuition. 

Next, we develop an accelerated version of $\bsls$, namely $\acbsls$ (\cref{alg:acbsls}), that solves the problem within $\tildeo(\prod_{i \in [m]} \sqrt{\kappa_i})$ gradient queries of $f$ (see \cref{thm:acbsls}). Again, as $\prod_{i \in [m]} \kappa_i \leq \globalcond$, $\acbsls$ complexity is never worse than the $\tilde{O}(\sqrt{\globalcond})$ complexity of AGD and can significantly improve when the clusters are well separated. 
We also show in \cref{thm:acbsls:inexact} that $\acbsls$ is numerically stable under finite-precision arithmetic. 

We conclude the study of the multiscale optimization problem (\cref{def:multiscale_problem}) by developing a lower bound of $\tilde{\Omega}(\prod_{i \in [m]} \sqrt{\kappa_i})$ across first-order deterministic methods (see \cref{thm:lb}). This shows that $\acbsls$ is asymptotically optimal up to poly-logarithmic factors. Our proof framework consists of 1) a novel reduction of a first-order lower bound to discrete $\ell_2$ polynomial approximations on multiple intervals and a further reduction to a uniform approximation, 2) a standard reduction to Green's function based on potential theory \cite{Driscoll.Toh.ea-SIREV98}, and 3) a novel estimate of Green's function associated with multiple intervals. 

We summarize the main results in the following theorem.
\begin{theorem}[Informal version of \cref{thm:bsls:recursive,thm:lb,thm:bsls:inexact,thm:acbsls,thm:acbsls:inexact}]
\label{thm:bsls:informal}$\bsls$ solves the multiscale optimization problem to $\epsilon$-optimality with $\tildeo(\prod_{i \in [m]} \kappa_i)$ gradient evaluations. The accelerated version $\acbsls$ solves to $\eps$-optimality with $\tildeo(\prod_{i \in [m]} \sqrt{\kappa_i})$ gradient evaluations. %
Both $\bsls$ and $\acbsls$ only require logarithmic bits of precision and use $\tildeo(d)$ and $\tildeo(md)$ space, respectively. 
Further, $\acbsls$ is worst-case optimal across first-order deterministic algorithms up to poly-logarithmic factors.
\end{theorem}

In the case where the objective in  \cref{def:multiscale_problem} is quadratic, we show that the conjugate gradient method (CG) \cite{Hestenes.Stiefel-52} also solves the problem with $\tildeo( \prod_{i \in [m]} \sqrt{\kappa_i})$ gradient queries and is numerically stable (see \cref{sec:cg:ub}).
$\acbsls$ matches this performance in the quadratic setting and further extends the guarantee to a much broader class of non-quadratic problems. We discuss this implication further in Section \ref{sec:future}.

\begin{remark}
We provide several remarks on the setup of the multiscale optimization problem and our results. 
\begin{enumerate}[leftmargin=*,label={(\alph*)}]
	\item \cref{def:multiscale_problem} does not assume knowledge of the decomposition. If in addition one assumes that the decompositions (i.e., $f_i$ and $\mat{P}_i$) are individually accessible, the problem can be solved much more efficiently (and more trivially), 
	using $\tildeo(\sum_{i \in [m]} \sqrt{\kappa_i})$ sub-objective gradient queries (see \cref{sec:known_projection}), in contrast to $\tildeo(\prod_{i \in [m]} \sqrt{\kappa_i})$.
	\item Given the existence of efficient algorithms if the decomposition is known, one may be tempted to first recover the decomposition ($f_i$ and $\proj_i$) before solving. 
	However, we show in \cref{sec:hard:to:recover:decomposition} that 
	recovering the $\proj_i$ with access to a gradient oracle is costly, in that it takes $\Omega(d)$ queries in the worst case.
	\item \cref{def:multiscale_problem} assumes orthogonality conditions on $\{\proj_i\}_{i \in [m]}$. 
	We remark that some amount of disjointness is critical to obtaining these upper bounds. In \cref{append:disjoint} we show a $\Omega(\sqrt{\globalcond})$ lower bound on the complexity of the problem without such an orthogonality assumption.
	\item Our algorithms do \emph{not} necessarily require the knowledge of all the $\mu_i$ and $L_i$'s. 
	In fact, our theorems only require that $\mu_1, L_m$ and $\prod_{i \in [m]} \cond_i$ are known to obtain the claimed asymptotic query complexity. %
\end{enumerate}
\end{remark}
\begin{SCfigure}[1.3][t]
	\centering
	\includegraphics[width=0.36\textwidth]{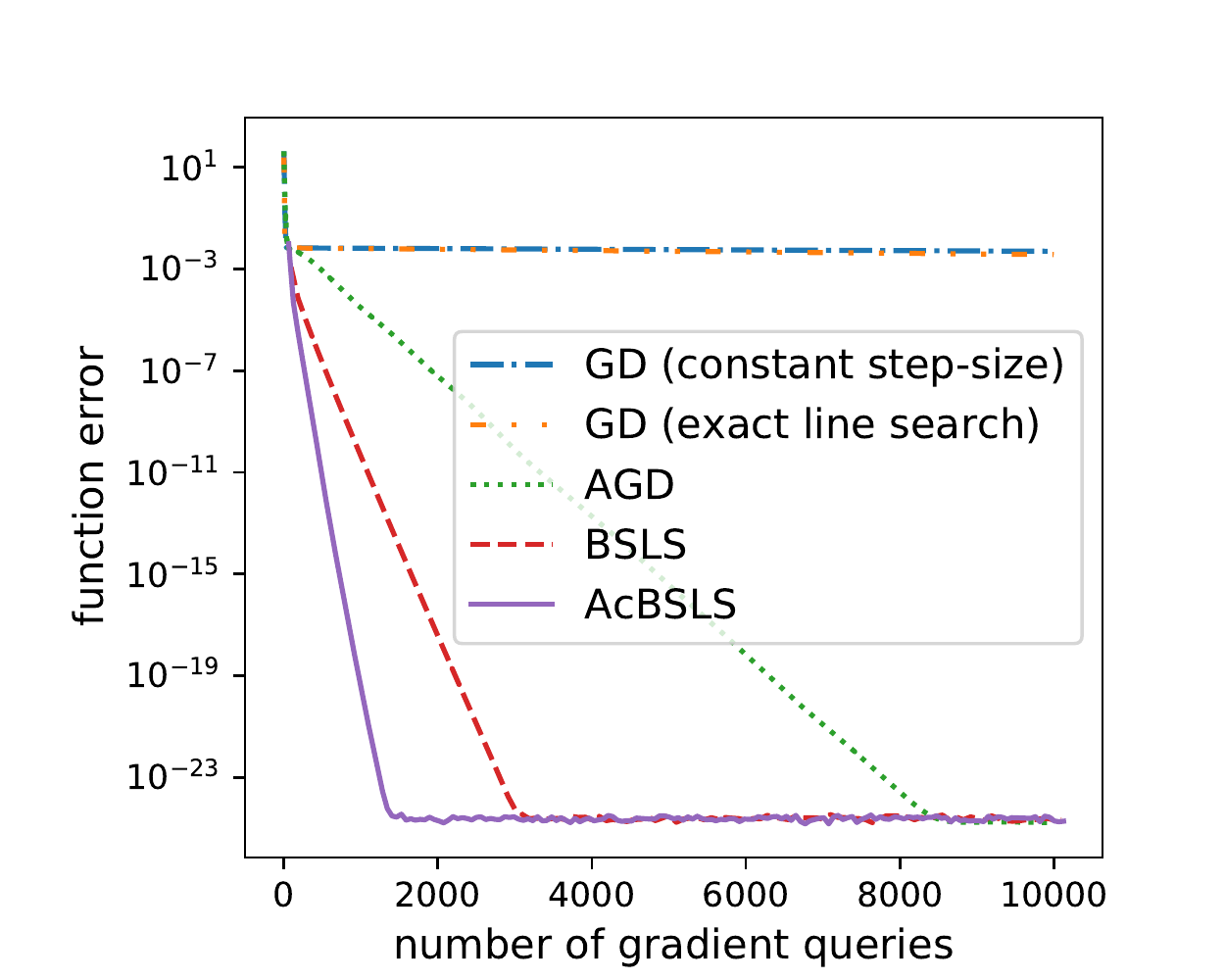}
	\caption{\textbf{A numerical example demonstrating the efficiency of $\bsls$ and $\acbsls$.} 
	Our objective is $f(\x) = \frac{1}{2} \x^\top \mat{A} \x - \b^\top \x$, where $\mat{A}$ has eigenvalues in $[0.0001, 0.0002] \cup [1, 10]$. 
    This objective satisfies \cref{def:multiscale_problem} with $\kappa_1 = 2$, $\kappa_2 = 10$ and $\globalcond = 10^5$. 
    We compare $\bsls$ and $\acbsls$ with Gradient Descent (GD) with constant step-size, Gradient Descent with exact line search, and Accelerated Gradient Descent (AGD). 
    Observe that $\acbsls$ and $\bsls$ clearly outperform the other algorithms.
	}
	\label{fig:expr}
\end{SCfigure}

Next, we consider the following stochastic, quadratic variant of the multiscale optimization problem.

\begin{definition} [Stochastic Quadratic Multiscale Optimization Problem]
\label{def:multiscale_stochastic}
The \textbf{stochastic quadratic multiscale optimization problem} asks to approximately solve the following problem
\begin{equation}
    \min_{\x \in \reals^d} \E_{(\vec{a},b) \sim \dist} \left[ \frac{1}{2} (\vec{a}^{\top} \x - b)^2 \right],
\end{equation}
where $b = \vec{a}^{\top} \x\opt $ for some fixed, unknown $\x\opt$ and the eigenvalues of the covariance matrix $\E_{\dist}[ \vec{a} \vec{a}^{\top} ]$ can be partitioned into $m$ ``bands'' such that for $i = 1, \dots, \nbands$ and $j = 1, \dots, \mult_i$, each eigenvalue $\lambda_{i_j}$ satisfies $\lambda_{i_j} \in [\mu_i, L_i]$ with $L_i < \mu_{i+1}$ for all $i < m$. 
\end{definition}

This objective is a special case of the multiscale optimization problem defined in Definition \ref{def:multiscale_problem} where each $f_i$ is a quadratic function of $\x$ (we make the connection explicit in Section \ref{sec:stochastic}). Therefore, in the non-stochastic case where we have access to noiseless gradients of $f(\x)$, our  guarantees for $\acbsls$ imply that the problem can be solved  with $ \tildeo(\prod_{i \in [m]} \sqrt{\kappa_i})$ gradient evaluations. 
In the theorem to follow we show that $\bslsstoch$ provides similar performance guarantees which are robust to stochasticity.  %

\begin{theorem}[Informal version of Theorem \ref{theorem:stochastic}] Under certain second-order independence and fourth moment assumptions on $\dist$, $\bslsstoch$ solves the stochastic quadratic multiscale optimization problem in expectation with $\epsilon$-optimality using $ \dim \cdot \prod_{i \in [m]} \tildeo(\kappa_i^2)$ stochastic gradient queries and $\tildeo(d)$ space. 
\end{theorem}

\subsection{Approach and results}

\label{sec:bsls_intuition_results}

In this section we give an overview of our algorithms and results.

\subsubsection*{Big-Step-Little-Step Algorithm ($\bsls$)}
We begin by introducing our main algorithm ``Big-Step-Little-Step'' ($\bsls$) for the multiscale optimization problem (\cref{def:multiscale_problem}).
As the name suggests, $\bsls$ adopts the idea of running a series of gradient descent steps with alternating step-sizes ranging from $L_m^{-1}$ to $L_1^{-1}$. To see the rationale behind alternating step-sizes consider the simple case of $m=2$ sub-objectives. If we were to run one step of GD on $f$, the smoother sub-objective $f_1$ would favor a ``big step'' of size $L_1^{-1}$, while the less smooth $f_2$ would favor a ``little step'' of size $L_2^{-1}$. 
In fact, the ``big-step'' will decrease $f_1$ considerably (by a factor of $1 - \kappa_1^{-1}$), but it may also increase $f_2$ (by no more than a factor of $\globalcond^2$). 
On the other hand, the ``little-step'' will decrease $f_2$ considerably (by a factor of $1 - \kappa_2^{-1}$), but will not decrease $f_1$ substantially (though it will not increase  $f_1$ either).
Thus, in order to make progress on both $f_1$ and $f_2$ efficiently, one could run one big step, followed by multiple little steps to fix the increase in $f_2$ from the previous large step-size.
The $\bsls$ algorithm is a careful interleaving of these big and little steps.

This intuition extends readily to the case of $m>2$ sub-objectives with a recursive framework (see \cref{alg:bsls}). We begin by executing $\bsls_1(\x^{(0)})$ on the initialization $\x^{(0)}$. We explain the recursive procedure via an illustrative example in \cref{fig:bsls}.

\begin{algorithm}[H]
	\caption{Big-Step-Little-Step Algorithm}
	 \label{alg:bsls}
	\begin{algorithmic}[1]
		\REQUIRE $\gd (\x; L)$ %
 			\RETURN $\x - \frac{1}{L} \cdot \nabla f(\x)$
	\end{algorithmic}
	\begin{algorithmic}[1] 
		\REQUIRE $\bsls_i$ $(\x^{(0)})$ %
		\STATE 
		$T_i \gets 
		    \begin{cases}
		         \left\lceil \kappa_1 \log \left( \frac{f(\x^{(0)}) - f^{\star}}{\epsilon} \right) \right\rceil & \text{if $i = 1$,} \\
		         \left\lceil \kappa_i \cdot (2 \log \globalcond + 1) \right\rceil & \text{otherwise.}
		    \end{cases}$
		\FOR{$t = 0, 1, \ldots, T_i-1$}
			\STATE $\tilde{\x}^{(t)} \gets 
			        \bsls_{i+1}({\x}^{(t)})$  if $i < m$, or $
			        {\x}^{(t)}$ otherwise.
			\STATE $\x^{(t+1)} \gets \gd (\tilde{\x}^{(t)}; L_i)$

		\ENDFOR
		\RETURN $\bsls_{i+1}(\x^{(T_i)})$ if $i < m$, or $\x^{(T_i)}$ otherwise.
	\end{algorithmic}
\end{algorithm}
\vspace{-10pt}
\begin{SCfigure}[2][!htbp]
	\centering
	\includegraphics[width=0.35\textwidth]{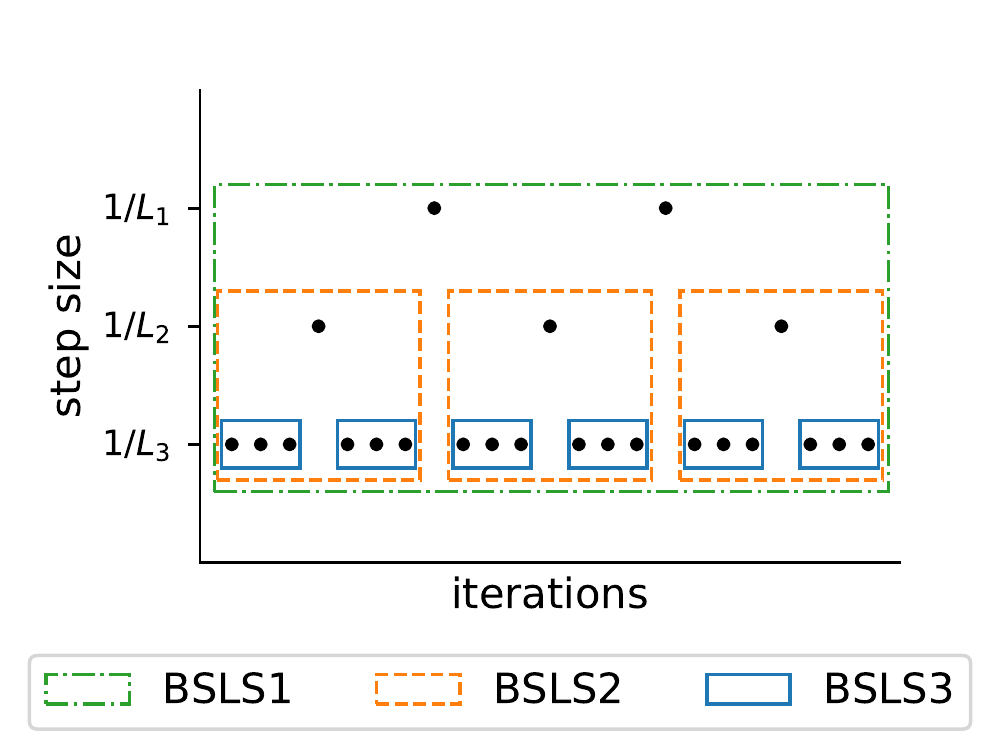}
	\caption{\textbf{An illustrative example of $\bsls$ (\cref{alg:bsls}) with $m = 3$, $T_1 = 2$, $T_2 = 1$, $T_3 = 3$.}
	The algorithm starts by executing $\bsls_1(\x^{(0)})$ at initialization $\x^{(0)}$. $\bsls_1$ will invoke $\bsls_2$ for $T_1 + 1 = 3$ times, with two ``big'' steps ($\gd$ of step-size $1/L_1$) in between. Within each invocation of $\bsls_2$, it will invoke $\bsls_3$ for $T_2+1=2$ times, with one ``medium'' step ($\gd$ of step-size $1/L_2$) in between. $\bsls_3$ only consists of $T_3=3$ little steps ($\gd$ of step-size $1/L_3$) since $m=3$ is the final layer. 
	Therefore, $\bsls_1$ effectively consists of three types of gradient steps structured in an interlacing order. 
}
	\label{fig:bsls}
\end{SCfigure}
The following theorem characterizes the convergence rate of $\bsls$. The proof appears in \cref{sec:proof:thm:bsls}.
\begin{theorem}%
  \label{thm:bsls:recursive}
	In the multiscale optimization (Def. \ref{def:multiscale_problem}), for any $\x^{(0)}$ and $\epsilon > 0$, $\bsls_1(\x^{(0)})$ returns an $\epsilon$-optimal solution with  
	$  \bigo \left( \left( \prod_{i \in [m]}\kappa_i  \right) \cdot \left(\log^{m-1}\globalcond \right) \cdot \log \left( \frac{ f(\x^{(0)}) - f^{\star}}{\epsilon} \right) \right) $ 
	gradient queries when $\{(\mu_i, L_i), i \in [m]\}$ are known. Moreover in the case where $\{(\mu_i, L_i), i \in [m]\}$ are unknown and only $m$, $\mu_{1}$, $L_m$ and $ \pi_{\kappa}=\prod_{i \in [m]} \kappa_i$ are known, we can achieve the same asymptotic sample complexity (up to constant factors suppressed in the $\bigo(\cdot)$). %
\end{theorem}
Given the rationale behind $\bsls$, it is natural to ask  whether such a careful step-size sequence is necessary to obtain fast convergence. Perhaps by simply performing line-searches in the direction of the gradient we can obtain a method which automatically finds the appropriate scale to make progress? 
As we also mentioned earlier in Section \ref{sec:sub:results}, \cref{sec:line_search} shows that this is not the case; indeed, we show instances where gradient descent with exact line search or constant step-sizes require $\Omega(\globalcond)$ gradient evaluations to solve the problem while $\bsls$ only requires $O(\log(\globalcond))$ gradient evaluations. This illustrates that it can be difficult to directly guess the right step-size and suggests the need for a step-size schedule such as that employed by $\bsls$.

\subsubsection*{Stability of $\bsls$ and why interlacing order matters}

For methods like conjugate gradient, there are known gaps between the best-known theoretical performance with infinite precision and finite-precision arithmetic \cite{paige1971computation,Greenbaum-89}, and there are related robustness issues in the face of statistical errors \cite{polyak1987introduction}. Consequently, when designing methods for the multiscale problems we consider, care is needed to ensure methods perform efficiently even without infinite precision arithmetic. Here, we discuss the stability properties of $\bsls$. We first note that \emph{under exact arithmetic} any reordering of the $\gd$ steps in $\bsls_1$ attains the same convergence rate:

\begin{proposition}
  \label{thm:bsls:ordering}
	In the multiscale optimization (\cref{def:multiscale_problem}), assume all operations are performed {under exact arithmetic}. 
	Then any reshuffling of $\gd$ steps in $\bsls_1$ (\cref{alg:bsls}) attains $\epsilon$-optimality.
\end{proposition}

In contrast, we show that \emph{under finite-precision}, the interlacing order defined by recursive $\bsls$ (\cref{alg:bsls}) is essential to guarantee the stability. %
Specifically, we show that our recursive $\bsls_1$ only requires roughly \emph{logarithmic bits of precision} (per floating-point number) to match the rate of convergence achieved under exact arithmetic, in contrast to potentially (at least) polynomial bits of precision for problematic orderings.

To understand why order matters in finite-precision, let us again consider  simply $m = 2$ sub-objectives. 
\Cref{thm:bsls:ordering} suggests a total of $\tildetheta(\kappa_1)$ big steps and $\tildetheta(\kappa_1 \kappa_2)$ little steps are needed to attain $\epsilon$-optimality under exact arithmetic. 
Consider a problematic ordering: begin with $\tildetheta(\kappa_1)$ big steps altogether and end with little steps altogether. 
With this ordering, the initial $\tildetheta(\kappa_1)$ big steps will amplify the error of $f_2$ by $\globalcond^{\tildetheta(\kappa_1)}$.
Under finite-precision, one needs $\tildetheta(\kappa_1)$ bits of precision to keep track of this growth, which is polynomial in the condition numbers. 
The same arguments apply if one runs all the $\tildetheta(\kappa_1\kappa_2)$ little steps first --- polynomial bits of precision are needed to secure the progress made by the little steps in $f_1$ before the big steps bring the error up. 
In contrast, our recursive $\bsls$ (\cref{alg:bsls}) overcomes this issue because the progress of all sub-objectives is balanced thanks to the interlacing step-sizes. 
We demonstrate this phenomenon in \cref{fig:stability} with a numerical example. 
We formally prove the stability of (recursive) $\bsls$ in \cref{sec:stability}.

\begin{SCfigure}[1.6][!htbp]
\includegraphics[width=0.33\textwidth]{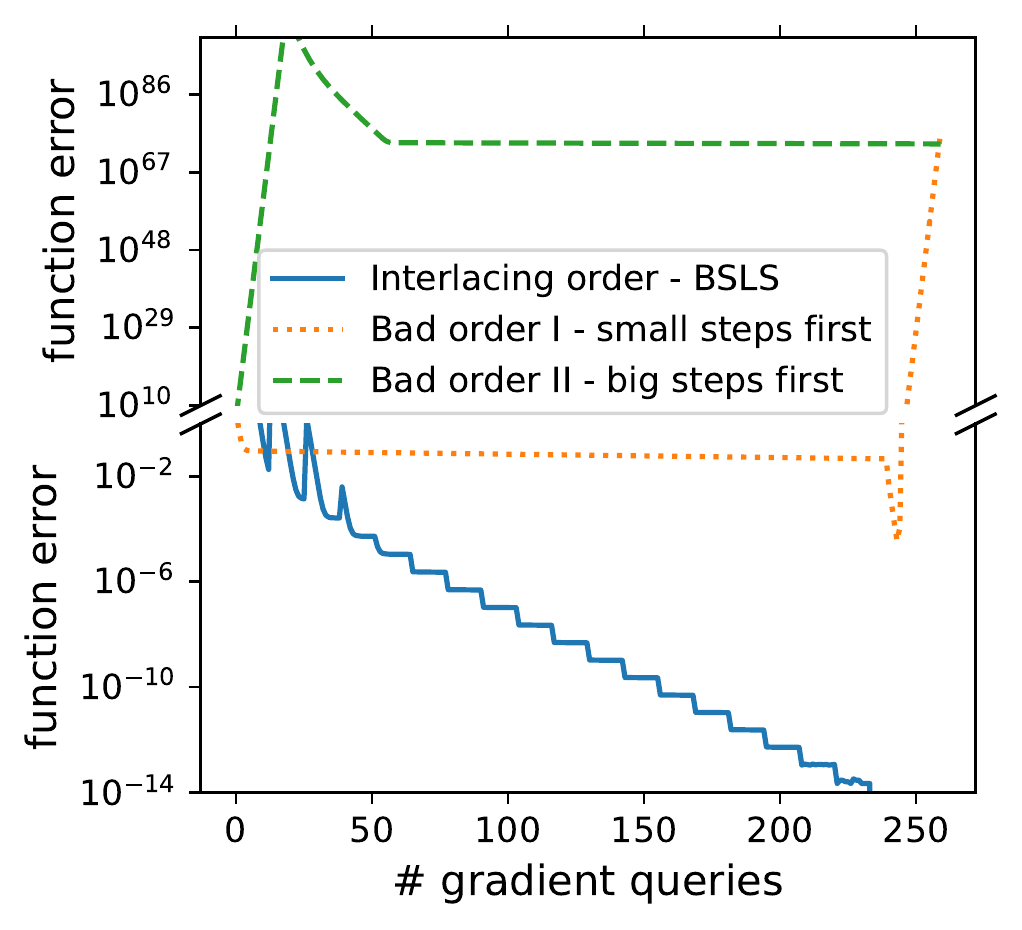}
\caption{\textbf{Importance of interlacing order in $\bsls$ under finite-precision arithmetic. } Our objective is $f(\x) = \frac{1}{2}\x^\top \mat{A} \x - \vec{b}^\top \x$, where $\mat{A}$ has eigenvalues in $[0.001, 0.002] \cup [0.5, 1]$. We consider three different step-size orderings, each running 20 big steps and 240 little steps in total. The first (solid) line runs $\bsls_1$ (\cref{alg:bsls}), namely every big step is followed by 12 little steps. The second (dotted) line runs all little steps before the big steps. The third (dashed) line runs all big steps before the little steps. Observe that only the principled $\bsls$ converges under finite arithmetic (double-precision floating point format).}
\vspace{-2mm}
\label{fig:stability}
\end{SCfigure}

\subsubsection*{Accelerated Big-Step-Little-Step algorithm ($\acbsls$)}
We provide an accelerated version of $\bsls$ algorithm, namely Accelerated $\bsls$ ($\acbsls$), which with 
$ \bigo \left( \left(\prod_{i \in [m]}\sqrt{\kappa_i}\right) \cdot \log^{m-1}\globalcond  \cdot \log \left( \frac{ f(\x^{(0)}) - f^{\star}}{\epsilon} \right) \right) $  gradient queries solves the multiscale optimization problem (\cref{def:multiscale_problem}) up-to $\epsilon$-optimality.
As we will see below, $\acbsls$ is optimal across first-order deterministic methods up-to poly-logarithmic factors.

$\acbsls$ shares the similar motivations of adopting alternating step-sizes as in $\bsls$.
Instead of running $\gd$, $\acbsls$ runs Accelerated Gradient Descent (AGD) with various ``step-sizes''. Formally, we use $\agd(\x, \v; L, \mu)$ to denote one-step of AGD with smooth estimate $L$ and convexity estimate $\mu$ (see the first block of \Cref{alg:acbsls} for definitions).  
The $\acbsls$ algorithm (the second block of \Cref{alg:acbsls}) then follows a similar recursive structure as in $\bsls$ defined in \Cref{alg:bsls}. 

The major difference between $\acbsls$ and the (un-accelerated) $\bsls$ lies in the difficulties of fixing the larger (less-smooth) sub-objectives after executing the big step-sizes. 
To understand this challenge, let us consider the simple case with only $m = 2$ sub-objectives. 
Recall that in $\bsls$, after executing one big GD step, we run $T_2$ little GD steps to fix the surge in $f_2$. 
This is backed by the fact that little GD steps will not increase the smaller (smoother) sub-objective $f_1$, as suggested by \cref{lem:BSLS:general}. 
Unfortunately, this relation does \emph{not} trivially extend to the accelerated setting, because $\agd(\x, \v; L_2, \mu_2)$ may \emph{not} keep the joint progress of $\x, \v$ in $f_1$.
Consequently, we adopt a more sophisticated branching strategy that fixes $\x$ and $\v$ separately, see Line 5 of $\acbsls_i$ in \Cref{alg:acbsls}. We refer readers to \Cref{sec:why:branching} for a numerical example on the non-convergence of naive $\acbsls$ without branching.

\begin{algorithm}
	\caption{Accelerated Big-Step-Little-Step Algorithm}
	\label{alg:acbsls}
	\begin{algorithmic}[1]
		\REQUIRE $\agd(\x, \v; L, \mu)$
		\STATE \(\kappa \gets L/\mu \); \(\alpha \gets \frac{\sqrt{\kappa}}{\sqrt{\kappa}+1}\); \(\beta \gets 1 - \frac{1}{\sqrt{\kappa}}\)
		\STATE $\y \gets \alpha \x + (1-\alpha) \v$; ~~
		$\v_{+} \gets \beta \v + (1-\beta) (\y - \frac{1}{\mu} \nabla f(\y))$; ~~
		$\x_{+} \gets \y - \frac{1}{L}\nabla f(\y)$.
		\RETURN $(\x_+, \v_+)$
	\end{algorithmic}
	\begin{algorithmic}[1]
	\REQUIRE $\acbsls_i$ $(\x^{(0)}, \v^{(0)})$
	\FOR{$t = 0, 1, \ldots, T_i-1$}
		 \STATE 
		 $T_i \gets 
		 \begin{cases}
		    \left \lceil \sqrt{\kappa_1} \log \left( \frac{2 (f(\x^{(0)})- f^{\star})} {\epsilon} \right) \right \rceil & \text{if $i = 1$,} \\
		    \left \lceil \sqrt{\kappa_i} (\log (4 \globalcond^4) + 1) \right \rceil & \text{otherwise.}
		 \end{cases}$
		\STATE $(\tilde{\x}^{(t)}, \tilde{\v}^{(t)}) \gets \agd(\x^{(t)}, \v^{(t)}; L_i, \mu_i)$
		\IF {$i < m$}
			\STATE $(\x^{(t+1)}, \_) \gets \acbsls_{i+1} (\tilde{\x}^{(t)}, \tilde{\x}^{(t)})$; ~~ $(\_, \v^{(t+1)}) \gets \acbsls_{i+1} (\tilde{\v}^{(t)}, \tilde{\v}^{(t)})$
		\ELSE
			\STATE $(\x^{(t+1)}, \v^{(t+1)}) \gets (\tilde{\x}^{(t)}, \tilde{\v}^{(t)})$
		\ENDIF
	\ENDFOR
	\RETURN $(\x^{(T_i)}, \v^{(T_i)})$
	\end{algorithmic}
\end{algorithm}

We specialize the initialization $\x^{(0)}$ and $\v^{(0)}$ to be the same to simplify the exposition of the theorem. We present and prove a general version of \cref{thm:acbsls:simplified} with arbitrary $\x^{(0)}, \v^{(0)}$ in \cref{sec:acbsls}.

\begin{theorem}[Simplified from \cref{thm:acbsls}]
    \label{thm:acbsls:simplified}
    In the multiscale optimization (\cref{def:multiscale_problem}), for any  $\x^{(0)}$ and $\epsilon > 0$, 
	$\acbsls(\x^{(0)}, \x^{(0)})$ returns an $\epsilon$-optimal solution with 
	$ \bigo \left( \left(\prod_{i \in [m]} \sqrt{\kappa_i} \right) \cdot \left(\log^{m-1}\globalcond \right) \cdot \log \left( \frac{ f(\x^{(0)}) - f^{\star}}{\epsilon} \right) \right) $  gradient queries when $\{(\mu_i, L_i), i \in [m]\}$ are known. Moreover in the case where $\{(\mu_i, L_i), i \in [m]\}$ are unknown and only $m$, $\mu_{1}$, $L_m$ and $ \pi_{\kappa}=\prod_{i \in [m]} \kappa_i$ are known, we can achieve the same asymptotic sample complexity (up to constant factors suppressed in the $\bigo(\cdot)$).
\end{theorem}
Similar to (un-accelerated) $\bsls$, under finite-precision arithmetic, $\acbsls$ can also attain the same rate of convergence with only logarithmic bits of precision. 
We defer the formal discussion to \cref{sec:acbsls:stability}.

\subsubsection*{Lower bound for the multi-scale optimization problem}
We demonstrate the optimality of $\acbsls$ (up to poly-logarithmic factors) by establishing the mini-max complexity lower bound of the multiscale optimization problem (\cref{def:multiscale_problem}) across first-order deterministic algorithm. 
We start by introducing the formal definition of a first-order deterministic algorithm from \cite{Carmon.Duchi.ea-MP21}.
\begin{definition}[Definition of first-order deterministic algorithms from \cite{Carmon.Duchi.ea-MP21}]
    \label{def:fo}
    An algorithm $\mathsf{A}$ operating on $f: \reals^d \to \reals$ is a \textbf{first-order deterministic algorithm} if it produces iterates $\{\x^{(t)}\}_{t=1}^{\infty}$ of the form 
    \begin{equation}
        \x^{(t)} = \mathsf{A}^{(t)} ( f(\x^{(1)}), \nabla f(\x^{(1)}), \ldots, f(\x^{(t-1)}), \nabla f(\x^{(t-1)}) ),
    \end{equation}
    where $\mathsf{A}^{(t)}: \reals^{(d+1)(t-1)} \to \reals^d$ is measurable (the dependency on $d$ is implicit).
\end{definition}
Note that the algorithm class considered in \cref{def:fo} is fairly general. 
For example, the definition does not require the algorithm to query points in the span of the previous gradients as in the some classic literature \cite{Nesterov-18}. 
(We refer readers to \cite{Carmon.Duchi.ea-MP20} for more detailed discussions on the generality of this function class.) The formal statement of our lower bound is as follows and the proof is relegated to \cref{sec:lb}.
\begin{theorem}[Lower bound of first-order deterministic algorithms for the multiscale optimization problem]
    \label{thm:lb}
    For any $\mu_j, L_j$ such that $\min_{j \in [m]} \kappa_j \geq 2$, for any deterministic first-order algorithm $\mathsf{A}$ defined in \cref{def:fo}, 
for any $t \in \naturals$, there exists an objective $f$ satisfying \cref{def:multiscale_problem} with $\|\nabla f(\vec{0}) \|_2 \leq \gradbound$ such that
    \begin{equation}
        \min_{\tau \in [t]}
        \| \nabla f(\x^{(\tau)})\|_2  \geq   \exp \left( - \frac{8t}{\sqrt{\prod_{i \in [m]} \cond_i} \cdot \prod_{i \in [m-1]} \left(0.03 \cdot \log (16 \frac{\mu_{i+1}}{L_i}) \right)} \right) \gradbound.
    \end{equation}
\end{theorem}
\cref{thm:lb} shows that our proposed $\acbsls$ is optimal up-to a poly-log factor due to the shared polynomial dependency $ \Theta(\prod_{i \in [m]} \sqrt{\cond_i})$. 
\cref{thm:lb} also reveals the necessity of the poly-logarithmic dependence on  $\globalcond$. 
For example, 
when the spectrum bands are evenly spaced such that $\frac{\mu_{i+1}}{L_i} \equiv \kappa_{\mathrm{gap}}$ and $\cond_i \ll \kappa_{\mathrm{gap}}$, then 
$\prod_{i=1}^{m-1} \log \frac{\mu_{i+1}}{L_i} \approx \log^{m-1} (\globalcond^{\frac{1}{m-1}}) = \frac{\log^{m-1} (\globalcond)}{(m-1)^{m-1}}$, 
which yields the same asymptotic dependency on $\globalcond$ as in the upper bound of $\acbsls$ (\cref{thm:acbsls}).

\subsubsection*{Stochastic BSLS}

Recall the stochastic version of a quadratic multiscale optimization problem from \cref{def:multiscale_stochastic}. We find that a variant of the $\bsls$ algorithm efficiently solves this problem. We define the stochastic analog of $\bsls$ in Algorithm~\ref{alg:bslsstoch}, which we call $\bslsstoch$. 
\begin{algorithm}
	\caption{Stochastic Variant of BSLS Algorithm}
	    \label{alg:bslsstoch}
		\begin{algorithmic}[1]
		\REQUIRE $\gdstoch (\x; L)$ 
		    \STATE $\vec{g} \gets \vec{0}$
		    \FOR{$i = 1, \dots, \navg$}
		    \STATE Receive $(\vec{a}^{(i)}, b^{(i)})$ and update $\vec{g} \gets \vec{g} + \frac{1}{\navg} \big( (\vec{a}^{(i) \top} \x) - b^{(i)} \big) \vec{a}^{(i)} $
		    \ENDFOR
			\RETURN $\x - \frac{1}{L} \cdot \vec{g}$
	\end{algorithmic}
	\begin{algorithmic}[1]
		\REQUIRE $\bslsstoch_i$ $(\x^{(0)})$
		\STATE 
		    $T_i \gets 
		        \begin{cases}
		            \left \lceil \cond_1 \log(9 \norm{\x^{(0)} - \x \opt}_2^2 /\epsilon) \right \rceil & \text{if $i = 1$} \\
		            8   \left \lceil \cond_i \log(\globalcond) \right \rceil & \text{otherwise}
		        \end{cases}
		    $
		\FOR{$t = 0, 1, \ldots, T_i-1$}
		    \STATE $\tilde{\x}^{(t)} \gets \bslsstoch_{i+1}(\x^{(t)})$ if $i < \nbands$, or $\x_t$ otherwise
		    \STATE $\x^{(t+1)} \gets \gdstoch (\tilde{\x}^{(t)}; L_i)$
		\ENDFOR
		\RETURN $\bslsstoch_{i+1}(\x^{(T_i)})$ if $i < m$, or $\x^{(T_i)}$ otherwise 
	\end{algorithmic}
\end{algorithm}

Our proofs require that the distribution $\dist$ generating samples $(\vec{a}, b)$ must satisfy a kind of ``second-order independence'' in the projected space $\mat{P} \vec{a}$; for any triple of distinct $i,j,k$ we must have that $\E[(\mat{P} \vec{a})_i ( \mat{P} \vec{a} )_k^2 (\mat{P} \vec{a})_j] = 0$. Note that this assumption is satisfied for natural distributions such as $\vec{a} \sim \mathsf{N}(\vec{0}, \mat{\Sigma})$ whenever $\mat{P}$ diagonalizes the covariance $\mat{\Sigma}$. For a few more non-trivial examples of distributions which satisfy this assumption see \cref{example:distassumption}.
We define $\distconst$ as the kurtosis of the distribution, which is the smallest constant such that for any $\vec{w} \in \reals^d$, $\E[ (\vec{w}^{\top} \vec{a})^4] \leq \distconst \E[(\vec{w}^{\top} \vec{a})^2]^2$. In the case where $\vec{a} \sim \mathsf{N}(\vec{0}, \mat{\Sigma})$ we have $\distconst = 3$. The kurtosis of the distribution will play a role in the necessary number of stochastic gradient queries taken by $\bslsstoch$. We establish the following theorem, the proof of which is relegated to \cref{sec:stochastic}. 

\begin{theorem}
\label{theorem:stochastic}
Consider the stochastic quadratic multiscale optimization problem from Definition~\ref{def:multiscale_stochastic}. Suppose $\dist$ is such that for any $i,j,k \in [\dim]$, $\E[(\mat{P} \vec{a})_i ( \mat{P} \vec{a} )_k^2 (\mat{P} \vec{a})_j] = 0$,  unless $i = j$ and for any $\vec{w} \in \reals^d$, $\E[ (\vec{w}^{\top} \vec{a})^4] \leq \distconst \E[(\vec{w}^{\top} \vec{a})^2]^2$.
If $m \leq \log(\globalcond)/3$ and $\left \{ \mu_i, L_i \right \}_{i \in m}$ are known, then given any $\x^{(0)}$ let 
\begin{equation}
    T_1 = \left \lceil \cond_1 \log(9 \norm{\x^{(0)} - \x \opt}_2^2 /\epsilon) \right \rceil,  \qquad  T_i = 
		            8   \left \lceil \cond_i \log(\globalcond) \right \rceil, \textrm{ for } i = 2, \dots m,
\end{equation}
and
\begin{equation}
    \navg \geq \distconst d m^2 \left( \prod \nolimits_{i \in [m]} T_i \right) \left( \max \nolimits_{i \in [m]} T_i \right).
\end{equation}
Then $\bslsstoch_1 \left( \x^{(0)} \right)$ returns an $\epsilon$-optimal solution in expectation using $\tildeo(d)$ space
, with a total of $\bigo \left( \navg \cdot 2^m \cdot \prod \nolimits_{i \in [m]} T_i \right) $
queries of $(\vec{a}, b) \sim \dist$. If only $\mu_1, L_m$, and $\prod_{i \in m} \cond_i$ are known and there exists some $K$ such that
    \begin{equation}
        \label{eqn:concentrationassumption}
        \norm{\mat{\Sigma}^{-1/2} \vec{a}}_2 \leq K  \left(  \E_{\vec{a} \sim \dist} \norm{\mat{\Sigma}^{-1/2} \vec{a}}_2^2 \right)^{1/2},
    \end{equation}
 then we can solve the stochastic quadratic multiscale optimization problem with an extra multiplicative factor of $\bigo \left( K^2 d \log \frac{4d}{\delta}  \left(1 + \sqrt{\frac{\varepsilon}{\delta}}  \right) \right) $ many more queries of $(\vec{a}, b) \sim \dist$.
\end{theorem}

\subsection{Prior work}

There is a vast literature on designing and analyzing first-order methods. Here, we survey several lines of research that are most closely related to our contributions. %

\paragraph{Complexity measures for first-order methods.} There are many results which consider notions other than smoothness and strong convexity for first-order methods. Some examples of this is work on star-convexity \cite{guminov2017accelerated,nesterov2018primal,HinderSS20}, quasi-strong convexity \cite{necoara2019linear}, semi-convexity \cite{van2007approximately}, the quadratic growth condition \cite{anitescu2000degenerate}, the error bound property \cite{luo1993error,fabian2010error}, restricted strong convexity \cite{zhang2013gradient,zhang2015restricted} and H\"{o}lder continuity \cite{zhang2013gradient,devolder2014first,yashtini2016global,grimmer2019convergence}. However, we are unaware of notions of fine-grained condition numbers for non-linear or stochastic problems %
 appearing previously in the literature. %

\paragraph{Structured linear systems.} As mentioned before, the conjugate gradient methods also solves the quadratic, noiseless version of the multiscale optimization problem. We refer the reader to some of the surveys for more discussion, including various preconditioning procedures
\cite{greenbaum1997iterative,saad2003iterative,nocedal2006conjugate}. There is also work on improving the condition number dependence of first-order methods to an \emph{average condition number} (ratio of the average of the eigenvalues of the Hessian and the smallest eigenvalue), which can be smaller than the condition number \cite{johnson2013accelerating,shalev2013accelerated,MuscoNSUW18}. There is also work on preconditioning the matrix by deflating large eigenvalues and hence reducing the average condition number in cases where there are a few very large eigenvalues \cite{gonen2016solving,MuscoNSUW18}.

\paragraph{Nonlinear CG.} Various nonlinear versions of CG have also been proposed such as Fletcher-Reeves (FR) method \cite{Fletcher.Reeves-64} and Polak-Ribière (PR) method \cite{Polak.Ribiere-69}. %
These methods are effective in practice and have been widely applied by the numerical optimization community \cite{Nocedal.Wright-06,Hager.Zhang-06,Dai-11a}.
However, for nonlinear CG, there is still a substantial gap between its practical performance and our theoretical understanding. 
On the negative side, it is known from Chap. 7 of \cite{Nemirovski.Yudin-83} that the FR and PR method do not match the accelerated GD rate $\tildeo(\sqrt{\globalcond})$.

\paragraph{Leveraging second-order structure via first-order methods.} %
There is a large body of work on methods to approximate second-order information including quasi-Newton methods such as DFP \cite{Davidon-SIOPT91}, BFGS \cite{Broyden-70}, L-BFGS \cite{Nocedal-80,Liu.Nocedal-MP89}, methods based on subsampling and sketching the Hessian \cite{pilanci2017newton,xu2020newton,roosta2019sub}, methods which learn diagonal preconditioners such as \texttt{AdaGrad} \cite{duchi2011adaptive} and \texttt{Adam} \cite{kingma2014adam}, stochastic second order methods \cite{agarwal2017second} and Newton-CG \cite{Royer.ONeill.ea-MP19,curtis2021trust}. \cite{CarmonD18} also provide accelerated methods that only use gradient and Hessian-vector queries and improve on the complexity of gradient descent for finding stationary points for certain non-convex problems. However, it is not known whether any of these algorithms achieves a worst-case complexity that does not depend polynomially on the overall condition number.

\paragraph{Stochastic methods.} Stochastic gradient methods are the workhorse for large scale optimization and machine learning problems \cite{NIPS2007_0d3180d6} and there is extensive work on stochastic gradient algorithms for solving linear systems, including randomized Kaczmarz \cite{strohmer2006randomized,NeedellSW16}, variance reduction techniques \cite{johnson2013accelerating,zhang2013linear,schmidt2017minimizing} and accelerated methods \cite{liu2016accelerated,allen2017katyusha,jain2018accelerating}. However, the complexity of all these methods depends polynomially on some measure of eigenvalue range or conditioning of the underlying matrix.

\paragraph{Lower bounds.} Starting with the seminal work of \cite{Nemirovski.Yudin-83}, there is a rich body of work on lower bounds for first-order methods. More recently, several works have extended these results to randomized algorithms \cite{woodworth2017lower, simchowitz2018randomized,braverman2020gradient,woodworth2021minimax}, and we use these results to show necessity of the orthogonality assumption in the multiscale optimization problem. To show query-complexity lower bounds for first-order methods for the multiscale optimization problem, we show a reduction from a first-order lower bound to a polynomial approximation problem on multiple intervals. There is extensive literature on polynomial approximations we leverage here, especially the work of \cite{Widom-69} (more references appear in Section \ref{sec:lb}). We also note that there is a long history of relating the convergence rates of optimization algorithms to polynomial approximation problems, including the work of \cite{Greenbaum-89, Musco.Musco.ea-SODA18} on convergence of Lanczos and CG.

\subsection{Implications and future directions}
\label{sec:future}
We view the multiscale optimization problem and our algorithmic results as promising first steps towards obtaining a more fine-grained complexity of convex optimization which goes beyond condition number. 
Though we give near-optimal rates for solving a class of smooth strongly-convex optimization problems, our work still leaves a number of open directions. Key among them are whether we can design methods with the full practical flexibility and applicability of methods like non-linear CG and limited-memory Quasi-Newton methods that have theoretical grounding as well, in the sense that they solve the types of problems that this work proposes. For instance, is there a variant of non-linear CG or limited-memory Quasi-Newton methods that provably solves our multiscale optimization problem, or a stochastic version of CG which solves the stochastic quadratic problem? More broadly, our work raises several intriguing questions regarding the role of memory in optimization, and when it is possible to achieve the convergence rates of second-order methods with only linear memory. %
Further, though we have established lower bounds on multiple modifications of our multiscale optimization problem, there are several natural related problems for which it remains open to develop fast methods---for example, problems for which the Hessian has some sort of consistent multi-scale structure and cases where the problems at different scales interact instead of being completely orthogonal.

We now further elaborate on some of these implications and directions.

\paragraph{Space limited optimization.} Recall that both $\bsls$ and $\bslsstoch$ work in $\tildeo(d)$ space, and $\acbsls$ uses $\tildeo(dm)$ space. Despite using linear memory, our algorithms only suffer a %
polylogarithmic dependence on the overall condition number $\globalcond$. In this context, they serve as a bridge between quadratic-memory second-order methods which achieve  a logarithmic dependence on the condition number, and previous linear-memory first-order methods which usually have a worse polynomial dependence on the condition number. For the stochastic case, we are unaware of any previous algorithm which only uses linear memory but still has a polylogarithmic dependence on $\globalcond$. In fact, some recent work \cite{sharan2019memory,woodworth2019open} %
conjectures that a polynomial dependence on $\globalcond$ is in general unavoidable for sub-quadratic memory algorithms. Our work shows that, at least for the structured problems we consider, it is possible to match the  polylogarithmic dependence on $\globalcond$ of second-order methods, while only using linear memory, and raises the question of whether this is possible for a larger class of problems. %

\paragraph{History and structure in accelerated methods.}
Our near-optimal accelerated method stores up to $2m$ points at a time; this is in contrast to CG, non-linear CG, and standard accelerated methods \cite{nesterov83} which store at most two points. It is an interesting open problem as to whether our space bound for accelerated methods could be improved. If not this raises several questions about the power of using additional history and memory in first-order methods.

\paragraph{Stochastic CG.} We note that the $\bslsstoch$ algorithm for the stochastic quadratic version of the multiscale optimization problem does not obtain an accelerated convergence rate. We suspect that the natural stochastic analog of CG where we approximate any matrix-vector products over a sufficiently large set of samples \emph{does} obtain an accelerated convergence rate for the stochastic quadratic problem, and showing this is an interesting direction for future work. This algorithm would additionally have the desirable property of not needing to guess the eigenvalues of the quadratic problem nor requiring a step size schedule.

\paragraph{Optimization problems with diagonal scaling.} Another interesting direction is to consider non-linear, convex optimization problems which are diagonally scaled ($\x\rightarrow \mat{D}\x$ for a diagonal matrix $\mat{D}$). This does not directly fit within our framework of \Cref{def:multiscale_problem} because the different scales could interact, but we believe the ideas in this paper may extend to this setting. We remark that our results do apply to the quadratic version of this problem and believe that methods like Newton-CG may be applicable in the non-quadratic case.
Further understanding and extending this setting
could pave the way for developing algorithms beyond \texttt{AdaGrad} for handling scaling in optimization problems.

\subsection{General notation} 
Let $[n]$ denote the set $\{1, 2, \ldots, n\}$. 
We use bold lower-case letter (e.g., $\x$) to denote vectors, bold upper-case letter (e.g., $\mat{A}$) to denote matrices. 
We use $\id$ to denote identity matrix, $\vec{1}$ to denote all-1 vector, $\vec{0}$ to denote all-zero vectors or matrices, 
$\unit_i$ to denote $i$-th unit vector ($i$-th column of $\id$).
When comparing two vectors or matrices, the ordinary inequality signs ($\leq, \geq$) denote element-wise inequality. 
For example, $\mat{A} \geq \mat{0}$ means $\mat{A}$ is a non-negative matrix. 
When comparing two matrices, ($\preceq, \succeq$) denote spectrum inequality.
For example, $\mat{A} \succeq \mat{0}$ means $\mat{A}$ is a positive semi-definite matrix.
We use $\|\cdot\|_1$ to denote vector $\ell_1$ or matrix $\ell_1$-operator (row-sum) norm, $\|\cdot\|_2$ to denote vector $\ell_2$ norm or matrix $\ell_2$-operator (spectrum) norm.
For any function $f$ we use $f^{\star}$ to denote the optimum (minimum) value of $f$.

\section{BSLS algorithm for multiscale optimization}
\label{section:nonlinear_problem}

In this section, we provide the formal proof of \Cref{thm:bsls:recursive} in \cref{sec:proof:thm:bsls}, and then formally establish the stability of $\bsls$ in \cref{sec:stability}.

\subsection{Proof of \Cref{thm:bsls:recursive} and \cref{thm:bsls:ordering}: $\bsls$ under exact arithmetic}
\label{sec:proof:thm:bsls}
Here we formalize the aforementioned intuition and prove \Cref{thm:bsls:recursive}. 
To begin, we first study the effect of $\gd(\x; L_i)$ on various subspaces $j \in [m]$.
\begin{lemma}%
  \label{lem:BSLS:general} 
    In the setting of \cref{thm:bsls:recursive},
	for any $\x$ and $i, j \in [m]$,
	\begin{equation}
		f_j(\proj_j \gd(\x; L_i))- f_j^{\star} \leq \left( f_j(\proj_j \x)- f_j^{\star} \right) \cdot 
		\begin{cases}
			1 & j < i
			\\
			1 - \kappa_i^{-1} & j = i 
			\\
			\globalcond^2 & j \geq i.
		\end{cases}
	\end{equation}
\end{lemma}

\begin{proof}[Proof of \Cref{lem:BSLS:general}]
	Let $\x_+$ denote the result of $\gd(\x; L_i)$. 
	By $L_j$-smoothness of $f_j$, we have
\begin{equation}
	f_j(\proj_j \x_+) - f_j^{\star} \leq f_j(\proj_j \x) - f_j^{\star} + \left(-\frac{1}{L_i} + \frac{L_j}{2L_i^2} \right) \left\|  \nabla f_j (\proj_j \x) \right\|_2^2.
	\label{eq:lem:GD1}
\end{equation}
Now we consider the three possible cases $j = i$, $j < i$ and $j > i$ separately.
\begin{enumerate}[(a), leftmargin=*]
	\item For $j=i$, the inequality \cref{eq:lem:GD1} becomes 
\begin{equation}
	f_j(\proj_j \x_+) - f_j^{\star} \leq f_j(\proj_j \x) - f_j^{\star} - \frac{1}{2 L_j} \left\|  \nabla f_j (\proj_j \x) \right\|_2^2.
\end{equation}
By $\mu_j$-strong-convexity of $f_j$ we have $	\left\|  \nabla f_j (\proj_j \x) \right\|_2^2 \geq 2\mu_j (f_j(\proj_j \x)  - f_j^{\star})$. 
Thus $	f_j(\proj_j \x_+) - f_j^{\star}  \leq (1 - \kappa_j^{-1}) (	f_j(\proj_j \x)- f_j^{\star})$.
	\item For $j < i$, the coefficient of the second term of \cref{eq:lem:GD1} is non-positive since $L_j \leq L_i$. Hence 
	$f_j(\proj_j \x_+) - f_j^{\star} \leq f_j(\proj_j \x) - f_j^{\star}$.
	\item For $j > i$, first observe that by $\mu_j$-strong convexity and $L_j$-smoothness of $f_j$, one has $2L_j(f_j(\proj_j \x)  - f_j^{\star}) \geq \left\|  \nabla f_j (\proj_j \x) \right\|_2^2 \geq 2\mu_j (f_j(\proj_j \x)  - f_j^{\star})$.
	Therefore by \cref{eq:lem:GD1}, we have
	\begin{align}
			f_j(\proj_j \x_+) - f_j^{\star} \leq  
			\left(1 - \frac{2\mu_j}{L_i} + \frac{L_j^2}{L_i^2} \right) (f_j(\proj_j \x) - f_j^{\star} )
			\leq 
			\left(-1 + \globalcond^2 \right)(f_j(\proj_j \x) - f_j^{\star} ),
	\end{align}
	where the last inequality is due to $\mu_j \geq L_i$ and   $\frac{L_j}{L_i} \leq \globalcond$ by definition of $\globalcond$. 
\end{enumerate}
Summarizing the above cases completes the proof of  \Cref{lem:BSLS:general}.
\end{proof}

With \Cref{lem:BSLS:general} at hand we are ready to prove \Cref{thm:bsls:recursive}.
\begin{proof}[Proof of \Cref{thm:bsls:recursive}]
	By expanding the $\bsls$ procedure, 
    we observe that $\bsls_1(\cdot)$ consists of $T_j \cdot \prod_{k=1}^{j-1} (T_k+1)$ steps of $\gd(\cdot; L_j)$ in total, for $j \in [m]$. 
	Therefore, by \Cref{lem:BSLS:general},	for any $i \in [m]$, the following inequality holds
	\begin{align}
		& f_i(\proj_i \bsls_1(\x^{(0)})) - f_i^{\star} 
		\leq \globalcond^{2 \sum_{j=1}^{i-1} T_j \prod_{k=1}^{j-1} (T_k+1)} \cdot \left( 1 - \kappa_i^{-1} \right)^{T_i \prod_{j=1}^{i-1} (T_j+1)}
		(	f_i(\proj_i \x^{(0)}) - f_i^{\star} )
		\\
		\leq &
		\exp \left( \underbrace{2 \log \globalcond \cdot \sum_{j=1}^{i-1}  T_j \prod_{k=1}^{j-1} (T_k+1)  - \kappa_i^{-1} T_i \prod_{j=1}^{i-1} (T_j+1) }_{\text{denoted as $\gamma_i$}} \right)
			(	f_i(\proj_i \x^{(0)}) - f_i^{\star} )
			\tag{since $1-x \leq e^{-x}$}.
	\end{align}
	It remains to upper bound $\gamma_i$.
	For $i = 1$, by definition, we have $\gamma_i := -\kappa_1^{-1} T_1 \leq \log \left( \frac{\epsilon}{f(\x^{(0)}) - f^{\star}} \right)$ due to the choice of $T_1$.
	For $i > 1$, we observe that
	\begin{align}
			\gamma_{i} - \gamma_{i-1} & = 2 \log \globalcond \cdot T_{i-1} \cdot \prod_{j=1}^{i-2} (T_j+ 1) - \kappa_{i}^{-1} T_i \prod_{j=1}^{i-1} (T_j+1) + \kappa_{i-1}^{-1} T_{i-1} \prod_{j=1}^{i-2} (T_j + 1)
			\\
	& \leq \left(T_{i-1} \prod_{j=1}^{i-2} (T_j+1)\right) \cdot \left( -\kappa_{i}^{-1} T_{i} + \kappa_{i-1}^{-1} + 2 \log \globalcond \right).
	\end{align}
	Since $T_{i} \geq \kappa_i (2 \log \globalcond +1)$ (due to the choice of $T_i$) we obtain $		\gamma_{i} - \gamma_{i-1} \leq \left(\prod_{j=1}^{i-1} T_j \right)\cdot \left( -1 +  \kappa_{i-1}^{-1} \right) \leq 0$.
	Consequently, $\gamma_m \leq \gamma_{m-1} \leq \cdots \leq \gamma_1 \leq \log \left( \frac{\epsilon}{f(\x^{(0)}) - f^{\star}} \right)$. Therefore for all $i \in [m]$, $
		f_i(\proj_i \bsls_1(\x^{(0)})) - f_i^{\star} 
		\leq
		\exp (\gamma_i) \left( f_i(\proj_i \x^{(0)}) - f_i^{\star}  \right)
		\leq
		\frac{\epsilon}{ f(\x^{(0)}) - f^{\star} } 	\left( f_i(\proj_i \x^{(0)}) - f_i^{\star}  \right)$.
	Summing over all $i \in [m]$ gives $f( \bsls_1(\x^{(0)})) - f^{\star} \leq \epsilon$.
	
    To show the last part of Theorem \ref{thm:bsls:recursive} regarding the setting where the parameters $\{(\mu_i,L_i), i \in [m]\}$ are unknown, we do a black-box reduction from the case where the parameters are known to when only $m, \mu_1, L_m$ and $\pi_{\kappa}$ are known. 
    
\begin{proposition}
	\label{prop:search}
    Let $\pi_{\kappa}=\prod_{i \in [m]} \kappa_i$. An algorithm $\mathsf{A}$ which solves the multiscale optimization problem in Definition \ref{def:multiscale_problem} to sub-optimality $\eps$ with $T(\pi_{\kappa}, \globalcond, m,\eps)$ gradient queries when the parameters ($\mu_i, L_i)$ are known, can be used to solve the multiscale optimization problem with $T(\pi_{\kappa}2^{5m}, \globalcond, m,\eps)\cdot O(\log^m(\globalcond))$ gradient queries when only $m$, $\mu_{1}$, $L_m$ and $ \pi_{\kappa}$ are known.
\end{proposition}
    
    The proof Proposition \ref{prop:search} works by simply doing a brute force search over all the parameters over a suitable grid and appears in \cref{subsec:search}. The last part of Theorem \ref{thm:bsls:recursive} now follows.
\end{proof}

The re-ordering \cref{thm:bsls:ordering} holds because \Cref{thm:bsls:recursive} only leverages the fact that $\bsls_1(\cdot)$ consists of $\Theta(\prod_{k=1}^j T_k)$ steps of $\gd(\cdot;L_j)$ in total, for $j \in [m]$.

\subsection{Theory on the stability of $\bsls$: why interlacing order matters}
\label{sec:stability}
Now we verify the intuition above and theoretically justify the stability of $\bsls$ (\cref{alg:bsls}). %
For clarity, let $\widehat{\gd}$ be the finite-precision implementation of $\gd$, and $\widehat{\bsls}_i$ be the finite-precision implementation of $\bsls_i$ by replacing $\gd$ with $\widehat{\gd}$. 
To understand finite-precision behavior without going into excessive details of low-level implementation, we impose the following \cref{req:gd} that $\widehat{\gd}$ returns a $\delta$-multiplicative approximation of the exact ${\gd}$.  \cref{req:gd} is reminiscent of the ``correct rounding'' requirement on basic operations in IEEE standard (c.f. Chap. 6 of \cite{Overton-01}). 
Technically, if $\gd$ operator is well-conditioned (and no overflow or underflow occurs), then Req. \ref{req:gd} can be satisfied by a floating-point system with $\bigo(\log(1/\delta))$ bits (c.f. Chap. 12 of \cite{Overton-01}).

\begin{requirement}
	\label{req:gd}
	There exists a $\delta<1$ such that for any $\x$ and $i$, for $\x_+ \gets \gd(\x; L_i)$ and $\widehat{\x_+} \gets \widehat{\gd} (\x; L_i)$, it is the case that $|\widehat{\x_+} - \x_+ | \leq \delta |\x_+|$, where $|\cdot|$ denotes element-wise absolute value.
\end{requirement}

In the following \Cref{thm:bsls:inexact}, we prove that finite-precision $\widehat{\bsls_1}$ can match the exact arithmetic rate under only logarithmic bits of precision in that $\delta$ only has to be polynomially small. 
As a conclusion, $\bsls$ can be implemented stably with $\tildeo(d)$ bits of memory.
We specialize the initialization $\x^{(0)}$ to $\vec{0}$ to simplify the exposition of the theorem. In Appendix \ref{apx:bsls:inexact}, we provide and prove the general version with arbitrary $\x^{(0)}$. 

\begin{theorem}[$\bsls$ under finite-precision arithmetic]
	\label{thm:bsls:inexact}
	Consider multiscale optimization problem (\Cref{def:multiscale_problem}), for any $\epsilon > 0$, assuming \cref{req:gd} with
	\begin{equation}
		\delta^{-1} \geq m \cdot 
			(10  \globalcond )^{2m-1} \cdot \frac{(f(\vec{0}) - f^{\star}) }{\epsilon} \cdot \left( \prod_{i \in [m]} T_i \right),
	\end{equation}
	then $\min\{f(\vec{0}), f(\widehat{\bsls_{1}}(\vec{0}))\} - f^{\star} \leq 3 \epsilon$ provided that $T_1, \ldots, T_m$ satisfy 
	\begin{equation}
		T_1 \geq \kappa_1 \log \left( \frac{f(\vec{0})- f^{\star}}{\epsilon} \right);
		\qquad
		T_i \geq \kappa_i (2 \log (\globalcond) + 1), \quad \text{for $i = 2, \ldots, m$},
	\end{equation}
	when $\{(\mu_i, L_i), i \in [m]\}$ are known. 
    We can also achieve the same asympototic sample complexity (up to constant factors suppressed in the $\bigo(\cdot)$) when $\{(\mu_i, L_i), i \in [m]\}$ are unknown and only $m$, $\mu_{1}$, $L_m$ and $ \pi_{\kappa}=\prod_{i=1}^m \kappa_i$ are known.
\end{theorem}

The proof of \cref{thm:bsls:inexact} is relegated to Appendix \ref{apx:bsls:inexact}.

\section{Accelerated BSLS algorithm for multiscale optimization}
\label{sec:acbsls}

In this section, we first state and prove the extended version of \cref{thm:acbsls:simplified} on the complexity of $\acbsls$ with general ($\x^{(0)}, \v^{(0)}$). 
Then we establish the stability result of $\acbsls$ in \cref{sec:acbsls:stability}. Finally, we discuss in \cref{sec:why:branching} on the necessity of branching procedure in $\acbsls$.

We will use standard potentials from accelerated GD to monitor the progress of $\acbsls$. For any $i \in [m]$ and $\x$, define
\begin{equation}
	\err_i(\x) \defeq f_i(\proj_i \x) - f_i^{\star}, 
	\qquad 
	\res_i(\x) \defeq \frac{\mu_i}{2}  \| \proj_i  (\x - \x^{\star})\|_2^2.
\end{equation}
For any $\x, \v$ pair, define
\begin{equation}
	\psi_i(\x, \v) \defeq \err_i(\x) + \res_i(\v), 
	\qquad
	\psi(\x, \v) \defeq \sum_{i \in [m]} \psi_i(\x, \v).
	\label{eq:acbsls:potential}
\end{equation}
We establish the following theorem. 
\begin{theorem}[$\acbsls$ with exact arithmetic]
	\label{thm:acbsls}
	Consider multiscale optimization problem defined in \Cref{def:multiscale_problem}, for any initialization $(\x^{(0)}, \v^{(0)})$ and $\epsilon > 0$, 	then $\psi(\acbsls_1(\x^{(0)}, \v^{(0)})) \leq \epsilon$ provided that $T_1, \ldots, T_m$ satisfy
	\begin{equation}
		T_1 \geq \sqrt{\kappa_1} \log \left( \frac{\psi(\x^{(0)}, \v^{(0)})}{\epsilon} \right),
		\qquad
		T_i \geq \sqrt{\kappa_i} (\log (4 \globalcond^4) + 1), \quad \text{ for $i = 2, \ldots, m$},
		\label{eq:acbsls:t:lb}
	\end{equation}
	when $\{(\mu_i, L_i), i \in [m]\}$ are known, and the total number of gradient queries which $\acbsls$ makes is $\bigo(\prod_{i \in [m]}T_i)$. 
	We can also achieve the same asymptotic query complexity for finding an $\epsilon$-optimal solution (up to constant factors suppressed in the $\bigo(\cdot)$) when $\{(\mu_i, L_i), i \in [m]\}$ are unknown and only $m$, $\mu_{1}$, $L_m$ and $ \pi_{\kappa}=\prod_{i=1}^m \kappa_i$ are known.
\end{theorem}

\subsection{Proof of \cref{thm:acbsls}: $\acbsls$ under exact arithmetic}
\label{sec:acbsls:proof}
The proof plan is as follows. We first study the effect of one $\agd$ step with various ``step-sizes'' on each sub-objective in \Cref{sec:thm:acbsls:1}.  
Then we inductively bound the progress of $\acbsls_i$ for all $i$ from $m$ down to $1$, with $i=1$ being the ultimate goal (see \Cref{sec:thm:acbsls:2}).
Then we finish the proof of \Cref{thm:acbsls} in \Cref{sec:thm:acbsls:3}. Note that the last part regarding the case where $\{(\mu_i,L_i), i \in [m]\}$ are unknown follows from our black-box reduction in Proposition \ref{prop:search} (in the same way as in the proof of Theorem \ref{thm:bsls:recursive}).
\subsubsection{Effect of one $\agd$ step with various ``step-sizes''}
\label{sec:thm:acbsls:1}
In this subsection, we study the effect of $\agd$ on all sub-objectives $f_i$'s. The main goal is to establish the following \cref{lem:agd1}.
\begin{lemma}[Effect of one $\agd$ step with various ``step-sizes'']
	\label{lem:agd1}
  Consider multiscale optimization (Def. \ref{def:multiscale_problem}), for any $\x, \v$ and $i \in [m]$, consider $(\x_+, \v_+) = \agd (\x, \v; L_i, \mu_i)$, then
	\begin{enumerate}[(a), leftmargin=*]
		\item (apply the right step-size) $			\psi_i \left( \x_+, \v_+ \right)  \leq \left( 1 - \frac{1}{\sqrt{\kappa_i}} \right) \psi_i \left( \x, \v \right)$.
		\item (apply small step-size) For any $j < i$, the following two inequality holds
		\begin{enumerate}[(i)]
			\item $	\max \left\{ \err_j(\x_+), \err_j(\v_+) \right\} \leq \max \left\{ \err_j(\x), \err_j(\v) \right\}$;
			\item $\max \left\{ \res_j(\x_+), \res_j(\v_+) \right\} \leq \max \left\{ \res_j(\x), \res_j(\v) \right\}$.
		\end{enumerate}
		\item (apply large step-size) For any $j > i$, the following three inequality holds
		\begin{enumerate}[(i)]
			\item $\max \left\{ \res_j(\v_+), \res_j(\x_+) \right\}  \leq 			\globalcond^2 \max \left\{ \res_j(\v), \res_j(\x) \right\}$;
			\item $\max \left\{ \err_j(\v_+), \err_j(\x_+) \right\} \leq \globalcond^2 \max \left\{ \err_j(\v), \err_j(\x) \right\}$;
			\item $\psi_j(\x_+, \v_+) \leq 2 \kappa_{j} \globalcond^2  (\x, \v)$.
		\end{enumerate}
	\end{enumerate}
\end{lemma}

\begin{remark}
	\cref{lem:agd1} is supposed to be the counterpart of \cref{lem:BSLS:general} (the progress of one-step $\gd$ in (un-accelerated) $\bsls$).
	 One may be tempted to establish the following (stronger) version of  \cref{lem:agd1}(b) 
	 \begin{equation}
		\psi_j(\x_+, \v_+) \leq \psi_j(\x, \v), 
		\quad
		\text{if $j < i$.}
		\label{eq:agd:false:claim}
	 \end{equation}
	 If this claim \cref{eq:agd:false:claim} were true, we would be able to guarantee the convergence of naive $\acbsls$ (akin to $\bsls$) without using the branching procedure.
	 Unfortunately, we can show that \cref{eq:agd:false:claim} is not always true, even for quadratic objective $f$. 
	 That is to say, the potential $\psi_j$ may not be conservative under $\agd(\cdot, \cdot; L_i, \mu_i)$ with $i > j$ (a.k.a. AGD with ``smaller step-sizes''). 
	 We provide more details on this topic in \cref{sec:why:branching}, including a numerical experiment against naive $\acbsls$.

	 In \cref{lem:agd1}, we instead show that  $\max\{\err(\x_+), \err(\v_+)\}$, $\max\{\res(\x_+), \res(\v_+)\}$ are non-increasing under AGD with smaller step-sizes. 
	 Since \cref{lem:agd1} (a) and (b) keep track of different quantities, we end up requiring the recursive branching procedure defined in $\acbsls$ (\cref{alg:acbsls}).
\end{remark}

We now prove \cref{lem:agd1}.
\begin{proof}[Proof of \Cref{lem:agd1}]
	\begin{enumerate}[(a), leftmargin=0pt, itemindent=15pt]
		\item The proof of (a) follows by standard accelerated gradient descent analysis \cite{Nesterov-18}, which we state here for completeness.
	  For clarity, let $\kappa_i = \frac{L_i}{\mu_i}$, $\alpha_i = \frac{\sqrt{\kappa_i}}{\sqrt{\kappa_i} + 1}$, $\beta_i = 1 - \frac{1}{\sqrt{\kappa_i}}$ be the corresponding $\kappa, \alpha, \beta$ in applying $\agd(\cdot,\cdot;L_i,\mu_i)$. Let us restate the recursion for clarity (we introduce an auxiliary variable $\z$ for ease of exposition).
		\begin{alignat}{2}
			& \y = \alpha_i \cdot \x + (1-\alpha_i) \cdot \v, \qquad &&  \z = \beta_i \cdot \v + (1-\beta_i) \cdot \y,
			\\
			& \v_+ = \z - \frac{1-\beta_i}{\mu_i} \cdot \nabla f(\y), \qquad && \x_+ = \y - \frac{1}{L_i} \cdot \nabla f(\y).
		\end{alignat}

		By definition of $\z$, one has
		\begin{align}
					 & \| \proj_i  (\z - \x^{\star}) \|_2^2 \leq \beta_i \| \proj_i  (\v - \x^{\star}) \|_2^2 + (1-\beta_i) \| \proj_i  (\y - \x^{\star} )\|_2^2
			\tag{by convexity}
			\\
			\leq & \beta_i \| \proj_i  (\v - \x^{\star}) \|_2^2 + \frac{2 (1-\beta_i)}{\mu_i} \left[ f_i^{\star} - f_i(\proj_i  \y) 
						+ \left\langle \proj_i  \nabla f(\y), \proj_i  (\y - \x^{\star}) \right\rangle \right].
						\tag{by $\mu_i$-strong-convexity of $f_i$}
        \label{eq:agd1:1}
		\end{align}
		By definition of $\v_+$, one has
		\begin{align}
				& \left\|  \proj_i  \left( \v_+ - \x^{\star}  \right)  \right\|_2^2 = 
					\left\|  \proj_i  \left( \z - \x^{\star} - \frac{1-\beta_i}{\mu_i} \nabla f(\y) \right)  \right\|_2^2
			\\
			= & \| \proj_i  (\z - \x^{\star}) \|_2^2 
			- \frac{2(1-\beta_i)}{\mu_i} \left\langle \proj_i  \nabla f(\y), \proj_i  (\z - \x^{\star})  \right\rangle
			+ \left( \frac{1-\beta_i}{\mu}  \right)^2 \left\| \proj_i  \nabla f(\y) \right\|_2^2
			\\
			\leq & \beta_i \| \proj_i  (\v - \x^{\star}) \|_2^2 + \frac{2 (1- \beta_i)}{\mu_i}  \left[ f_i^{\star} - f_i(\proj_i  \y) 
			+ \left\langle \proj_i  \nabla f(\y), \proj_i  (\y - \x^{\star}) \right\rangle \right]
			\\
			& -  \frac{2(1-\beta_i)}{\mu_i} \left\langle \proj_i  \nabla f(\y), \proj_i  (\z - \x^{\star})  \right\rangle
			+ \left( \frac{1-\beta_i}{\mu_i}  \right)^2 \left\| \proj_i  \nabla f(\y) \right\|_2^2
			\\
			= &  \beta_i \| \proj_i  (\v - \x^{\star}) \|_2^2 + \frac{2 (1- \beta_i)}{\mu_i}  \left[ f_i^{\star} - f_i(\proj_i  \y) 
			+ \underbrace{\left\langle \proj_i  \nabla f(\y), \proj_i  (\y - \z) \right\rangle}_{\text{\ding{172}}} \right]
			\\
			& \qquad + \left( \frac{1-\beta_i}{\mu_i}  \right)^2 \underbrace{\left\| \proj_i  \nabla f(\y) \right\|_2^2}_{\text{\ding{173}}}
      \label{eq:agd1:2}
		\end{align}
    Next we bound \ding{172} and \ding{173} in \eqref{eq:agd1:2}.
    First note that by definition $\y - \z = \beta_i (\y-\v) = \beta_i(\alpha_i (\x - \v))$, and $\x- \y = (1-\alpha_i) (\x-\v)$, we have $\y - \z = \frac{\beta_i \alpha_i}{1-\alpha_i} (\x - \y)$. Therefore \ding{172} is bounded as
		\begin{equation}
			\left\langle \proj_i  \nabla f(\y), \proj_i  (\y - \z) \right\rangle
			=
			\frac{\beta_i \alpha_i}{(1 - \alpha_i)} \left\langle \proj_i  \nabla f(\y), \proj_i  (\x - \y) \right\rangle
			\leq
			\frac{\beta_i \alpha_i}{(1 - \alpha_i)} \left( f_i(\proj_i \x) - f_i(\proj_i  \y) \right),
      \label{eq:agd1:3}
		\end{equation}
    where the last inequality is by convexity of $f_i$. 

    To bound \ding{173}, we note that $\x_+ = \y - \frac{1}{L_i} \nabla f(\y)$, which implies (by $L_i$-smoothness of $f_i$)
		\begin{align}
		    f_i(\proj_i  \x_+) \leq f_i(\proj_i   \y) - \left\langle \nabla f_i(\proj_i  \y) ,  \frac{1}{L_i} \proj_i  \nabla f(\y) \right\rangle 
			+ \frac{L_i}{2} \left\| \frac{1}{L_i} \proj_i  \nabla f(\y)  \right\|_2^2 =
			f_i(\proj_i   \y) - \frac{1}{2L_i}  \left\| \proj_i  \nabla f(\y)  \right\|_2^2.
		\end{align}
		Thus \ding{173} is upper bounded as
		\begin{equation}
			\left\| \proj_i  \nabla f(\y)  \right\|_2^2 \leq 2L_i \left( f_i(\proj_i   \y)  - f_i(\proj_i  \x_+)\right).
      \label{eq:agd1:4}
		\end{equation}

		Plugging the upper bound \eqref{eq:agd1:3}, \eqref{eq:agd1:4} down to \eqref{eq:agd1:2} yields 
		\begin{align}
			& \left\|  \proj_i  ( \v_+ - \x^{\star} )  \right\|_2^2
			\leq \beta_i \| \proj_i  (\v - \x^{\star}) \|_2^2 + \frac{2 (1- \beta_i)}{\mu_i}
			\left(  f_i^{\star} - f_i(\proj_i  y)  \right)
			\\
			& \qquad	+ \frac{2 (1- \beta_i)}{\mu_i} \frac{\beta_i \alpha_i}{(1 - \alpha_i)} \left( f_i(\proj_i  \x) - f_i(\proj_i  y) \right)
			+ \frac{2L_i (1-\beta_i)^2}{\mu_i^2} \left( f_i(\proj_i  y)  - f_i(\proj_i  \x_+)\right).
		\end{align}
		Substituting $\alpha_i = \frac{\sqrt{\kappa_i}}{\sqrt{\kappa_i} + 1}$ and $\beta_i = 1- \frac{1}{\sqrt{\kappa_i}}$ gives
		\begin{equation}
			\left\|  \proj_i  ( \v_+ - \x^{\star} )  \right\|_2^2 + \frac{2}{\mu_i}  \left( f_i(\proj_i  \x_+) - f_i^{\star} \right)
			\leq 
			\left( 1 - \frac{1}{\sqrt{\kappa_i}} \right)
			\left( \left\|  \proj_i  ( \v - \x^{\star} )  \right\|_2^2 + \frac{2}{\mu_i}  \left( f_i(\proj_i  \x) - f_i^{\star} \right) \right),
		\end{equation}
		which implies $\psi_i (\x_+, \v_+) \leq \left( 1 - \frac{1}{\sqrt{\kappa_i}} \right) \psi(\x, \v),$
    completing the proof of (a).

		\item  Let $\kappa_i = \frac{L_i}{\mu_i}$, $\alpha_i = \frac{\sqrt{\kappa_i}}{\sqrt{\kappa_i} + 1}$, $\beta_i = 1 - \frac{1}{\sqrt{\kappa_i}}$ be the corresponding $\kappa, \alpha, \beta$ in applying $\agd(\cdot,\cdot;L_i,\mu_i)$. 
    For clarity we restate the algorithm $\agd$ with an auxiliary state $\w$ 
		\begin{equation}
			\begin{aligned}
				&	\y  = \alpha_i \cdot \x + (1-\alpha_i) \cdot \v, \qquad
				&
				\w  = \y - \frac{1}{\mu_i}  \nabla f(\y), 
				\\
				& \v_+  = \beta_i \v + (1-\beta_i) \w, \qquad
				& \x_+  = \y - \frac{1}{L_i} \cdot \nabla f(\y).
			\end{aligned}
			\label{eq:acbsls:reparam}
		\end{equation}
    Since $\alpha \in [0,1]$ we have (by convexity)
    \begin{equation}
      \err_j(\y) = f_j(\proj_j\y) - f_j^{\star} \leq 
      \max \left\{ f_j(\proj_j \x), f_j(\proj_j \v) \right\} - f_j^{\star}
      = \max \left\{ \err_j (\x), \err_j(\v)\right\}.
    \end{equation}
    Since the step-size of the $\w$-step satisfies $\frac{1}{\mu_i} \leq \frac{1}{L_j}$ by assumption $j < i$, we obtain
    \begin{equation}
      f_j(\proj_j \w) \leq f_j(\proj_j \y) - \frac{1}{\mu_i} \left\langle \nabla f_j (\proj_j y),   \proj_j \nabla f_j (\y)  \right\rangle + \frac{L_j}{2} \left\|\frac{1}{\mu_i}\proj_j \nabla f_j (\y) \right\|_2^2 
      \leq f_j(\proj_j \y).
    \end{equation}
    For the same reason we have $f_j(\proj_j \x_+) \leq f_j(\proj_j \y)$ since the $\x_+$-step takes an even smaller step-size.
    These imply $\err_j(\w) \leq \err_j(\y) \leq \max \left\{ \err_j (\x), \err_j(\v)\right\}$ and $\err_j(\x_+) \leq \err_j(\y) \leq \max \left\{ \err_j (\x), \err_j(\v)\right\}$. By convexity we have $\err_j(\v_+) \leq \max \left\{ \err_j(\v), \err_j(\w) \right\} \leq \max \left\{ \err_j (\x), \err_j(\v)\right\}$,    which completes the proof of the first inequality. 
    The second inequality holds for the same reason. 
    
	\item Let $\kappa_i = \frac{L_i}{\mu_i}$, $\alpha_i = \frac{\sqrt{\kappa_i}}{\sqrt{\kappa_i} + 1}$, $\beta_i = 1 - \frac{1}{\sqrt{\kappa_i}}$ be the corresponding $\kappa, \alpha, \beta$ in applying $\agd(\cdot,\cdot;L_i,\mu_i)$. 
    For clarity we restate the algorithm $\agd$ with an auxiliary state $\w$, as in \eqref{eq:acbsls:reparam}.

	  First note that $r_j(\y) \leq \max \left\{ \res_j(\x), \res_j(\v)\right\}$ since $\y$ is a convex combination of $\x$ and $\v$.
    Now we analyze $r_j(\w)$
		\begin{align}
			& \frac{2}{\mu_j}  \res_j(\w) = \left\| \proj_j  \left( \y - \frac{1}{\mu_i} \nabla f(\y) - \x^{\star} \right) \right\|_2^2
			\\
			= & \left\| \proj_j  (\y - \x^{\star}) \right\|_2^2 - \frac{2}{\mu_i} \left\langle \proj_j \nabla f(\y), \proj_j (\y-\x^{\star})  \right\rangle \
				+ \frac{1}{\mu_i^2} \left\| \proj_j \nabla f(\y) \right\|_2^2
			\\
			\leq & \left( 1 - \frac{\mu_j}{\mu_i} \right) \left\| \proj_j  (\y - \x^{\star}) \right\|_2^2 
			+ 
			\left( -\frac{1}{\mu_i L_j} + \frac{1}{\mu_i^2}\right)\left\| \proj_j \nabla f(\y) \right\|_2^2
			\\
			\leq &  \left( 1 - \frac{\mu_j}{\mu_i} \right) \left\| \proj_j  (\y - \x^{\star}) \right\|_2^2 
			+ 
			\left( -\frac{L_j}{\mu_i} + \frac{L_j^2}{\mu_i^2}\right) \left\| \proj_j  (\y - \x^{\star}) \right\|_2^2 
			\\
			\leq & \frac{L_j^2}{\mu_i^2} \left\| \proj_j  (\y - \x^{\star}) \right\|_2^2  = \frac{L_j^2}{\mu_i^2} \res_j(\y),
		\end{align}
		Since $\v_+$ is a convex combination of $\w$ and $\v$, we obtain
		\begin{equation}
			r_j(\v_+) \leq \max \left\{ \res_j(\w), \res_j(\v) \right\} \leq \max \left\{ \frac{L_j^2}{\mu_i^2} \res_j(\y), \res_j(\v) \right\} \leq  \frac{L_j^2}{\mu_i^2} \max \left\{ \res_j(\x), \res_j(\v) \right\}.
		\end{equation}
		Similarly we have
		\begin{equation}
			r_j(\x_+) \leq \frac{L_j^2}{L_i^2} \res_j(\y) \leq  \frac{L_j^2}{L_i^2} \max \left\{ \res_j(\x), \res_j(\v) \right\},
		\end{equation}
		which yields the first inequality of (c). 
		The second inequality of (c) holds for the same reason.
    		The third inequality holds because 
		\begin{align}
			& \psi_j(\x_+, \v_+) = \err_j(\x_+) + \res_j(\v_+) \leq \max \{ \err_j(\x_+), \err_j(\v_+)\} +  \max \{ \res_j(\x_+), \res_j(\v_+)\}
			\\
			\leq & \globalcond^2 \left( \max \{ \err_j(\x), \err_j(\v)\}  + \max \{ \res_j(\x), \res_j(\v)\}  \right)
			\tag{by the first two inequalities}
			\\
			\leq & \globalcond^2 \left( \err_j (\x) + \err_j(\v) + \res_j(\x) + \res_j(\v) \right)
			\\
			\leq & \globalcond^2 (\kappa_j + 1) (\err_j(\x) + \res_j(\v)) \leq 2\globalcond^2 \kappa_j \psi_j (\x, \v).
		\end{align}
	\end{enumerate}
\end{proof}

\subsubsection{Estimating the progress of $\acbsls$}
\label{sec:thm:acbsls:2}
\begin{lemma}
  \label{lem:acbsls:1}
    Under the same settings of \cref{thm:acbsls}, for any $(\x, \v)$	and $i \in [m]$, let $(\x_+, \v_+) \gets \acbsls_i (\x, \v)$, then for any $j < i$, it is the case that
	\begin{enumerate}[(a), leftmargin=*]
		\item 	$\max\{ \err_j(\x_+), \err_j(\v_+) \} \leq \max \{ \err_j(\x), \err_j(\v)\}$.
		\item  	$\max\{ \res_j(\x_+), \res_j(\v_+) \} \leq \max \{ \res_j(\x), \res_j(\v)\}$.
	\end{enumerate}
\end{lemma}
\begin{proof}[Proof of \Cref{lem:acbsls:1}]
  We will fix $j$ and prove both statements by induction on $i$ in descent order (from $m$ to $j+1$). 
  Throughout the proof we denote $(\x^{(0)}, \v^{(0)}, \tilde{\x}^{(0)}, \tilde{\v}^{(0)}, \cdots, \x^{(T_{i})}, \v^{(T_{i})})$ the sequence generated by running $\acbsls_{i}(\x, \v)$.

  Induction base: for $i=m$, note that $\acbsls_m(\cdot,\cdot)$ is equivalent to $\agd^{T_m}(\cdot,\cdot;L_m,\mu_m)$. 
  Since $j<m$,  \Cref{lem:agd1}(b) suggests $\max\{ \err_j(\tilde{\x}^{(t)}), \err_j(\tilde{\v}^{(t)}) \} \leq \max \{ \err_j(\x^{(t)}), \err_j(\v^{(t)})\}$. Since $i=m$ we have $\x^{(t+1)} = \tilde{\x}^{(t)}$ and $\v^{(t+1)} =\tilde{\v}^{(t)}$, and consequently $\max \{ \err_j(\x^{(t+1)}), \err_j(\v^{(t+1)})\} \leq \max \{ \err_j(\x^{(t)}), \err_j(\v^{(t)})\}$. Telescoping $t$ from $0$ to $T_i$ yields $\max \{ \err_j(\x_{+}), \err_j(\v_{+})\} \leq \max \{ \err_j(\x), \err_j(\v)\}$. The same arguments hold for (b) as well.
  
  Now suppose the statements hold for $i+1 \leq m$ and we study $i$. Since $j < i$ we can apply \Cref{lem:agd1}(b) to show that $\max\{ \err_j(\tilde{\x}^{(t)}), \err_j(\tilde{\v}^{(t)}) \} \leq \max \{ \err_j(\x^{(t)}), \err_j(\v^{(t)})\}$. By induction hypothesis we have $\err_j(\x^{(t+1)}) \leq \err_j(\tilde{\x}^{(t)})$ and $\err_j(\v^{(t+1)}) \leq \err_j(\tilde{\v}^{(t)})$. Consequently $\max \{ \err_j(\x^{(t+1)}), \err_j(\v^{(t+1)})\} \leq \max \{ \err_j(\x^{(t)}), \err_j(\v^{(t)})\}$.   Telescoping $t$ from $0$ to $T_i$ yields $\max \{ \err_j(\x_{+}), \err_j(\v_{+})\} \leq \max \{ \err_j(\x), \err_j(\v)\}$. The same arguments hold for (b) as well.
\end{proof}

\begin{lemma}
  \label{lem:acbsls:2}
   Under the same settings of \cref{thm:acbsls}, for any $(\x, \v)$	and $i \in [m]$, let $(\x_+, \v_+) \gets \acbsls_i (\x, \v)$, then  $\psi_i(\x_+, \v_+) \leq \left( 1 - \frac{1}{\sqrt{\kappa_i}} \right)^{T_i} \psi_i (\x, \v)$.
\end{lemma}
\begin{proof}[Proof of \Cref{lem:acbsls:2}]
  Let $(\x^{(0)}, \v^{(0)}, \tilde{\x}^{(0)}, \tilde{\v}^{(0)}, \cdots, \x^{(T_{i})}, \v^{(T_{i})})$ be the trajectory generated by running $\acbsls_{i}(\x, \v)$.  
  For $i = m$, $\acbsls_m(\cdot,\cdot)$ is equivalent to $\agd^{T_m}(\cdot,\cdot;L_m,\mu_m)$. \Cref{lem:agd1}(a) suggests that $\psi_i(\x^{(t+1)}, \v^{(t+1)}) \leq (1 - \frac{1}{\sqrt{\kappa_m}}) \psi_i(\x^{(t)}, \v^{(t)})$. Telescoping $t$ from $0$ to $T_m$ shows $\psi_i(\x_+, \v_+) \leq (1 - \frac{1}{\sqrt{\kappa_m}})^{T_m} \psi_i(\x, \v)$.

  For $i < m$, we first note that  \Cref{lem:agd1}(a) suggests  $\psi_i(\tilde{\x}^{(t)}, \tilde{\v}^{(t)}) \leq (1 - \frac{1}{\sqrt{\kappa_m}}) \psi_i(\x^{(t)}, \v^{(t)})$. Since $(\x^{(t+1)}, \_) = \acbsls_{i+1}(\tilde{\x}^{(t)}, \tilde{\x}^{(t)})$, \Cref{lem:acbsls:1} suggests that $\err_i(\x^{(t+1)}) \leq \err_i(\tilde{\x}^{(t)})$. For the same reason we have $r_i(\v^{(t+1)}) \leq \res_i(\tilde{\v}^{(t)})$. Consequently $\psi_i(\x^{(t+1)}, \v^{(t+1)}) \leq \psi_i(\tilde{\x}^{(t)}, \tilde{\v}^{(t)}) \leq (1 - \frac{1}{\sqrt{\kappa_m}}) \psi_i(\x^{(t)}, \v^{(t)})$. 
	Telescoping $t$ from $0$ to $T_i$ completes the proof.

\end{proof}

\begin{lemma}
  \label{lem:acbsls:3}
   Under the same settings of \cref{thm:acbsls}, for any $(\x, \v)$ and $i \in [m]$, let $(\x_+, \v_+) \gets \acbsls_i (\x, \v)$, then for any $j \geq i$, the following inequality holds
	\begin{equation}
		\psi_j( \x_+, \v_+ ) \leq
		\exp \left( - \frac{1}{\sqrt{\kappa_j}} \prod_{k=i}^j T_{k} + \left( \sum_{k=i}^{j-1} \prod_{l=i}^k T_l \right) \log \left( 4\kappa_j^2 \globalcond^2 \right)   \right)
		\psi_j(\x, \v).
	\end{equation}
\end{lemma}
\begin{proof}[Proof of \Cref{lem:acbsls:3}]
  We will fix $j$ and prove by induction on $i$ in descent order (from $j$ to $1$). 

  Induction base: for $i=j$, the statement (b) follows by \Cref{lem:acbsls:2} 
  \begin{equation}
    \psi_j(\x_+, \v_+) \leq \left( 1 - \frac{1}{\sqrt{\kappa_j}} \right)^{T_j} \leq \exp \left( - \frac{T_j}{\sqrt{\kappa_j}} \right) \psi_j(\x, \v).
  \end{equation}

	Now assume the claim holds for $i+1 \leq j$, and we study the case of $i$. Denote $(\x^{(0)}, \v^{(0)}, \tilde{\x}^{(0)}, \tilde{\v}^{(0)}, \cdots, \x^{(T_{i})}, \v^{(T_{i})})$ the sequence generated by running $\acbsls_{i}(\x, \v)$. 
  Since $(\tilde{\x}^{(t)}, \tilde{\v}^{(t)}) \gets \agd(\x^{(t)}, \v^{(t)}; L_i, \mu_i)$ and $i < j$, \Cref{lem:agd1}(c) suggests that 
  \begin{equation}
    \max \left\{ \res_j(\tilde{\x}^{(t)}), \res_j(\tilde{\v}^{(t)})\right\}
		 \leq
    \globalcond^2  \max  \left\{ \res_j(\x^{(t)}), \res_j(\v^{(t)})\right\}.
  \end{equation}
  Since $f_j$ is $\kappa_j$-conditioned we have $r_j \leq \err_j \leq \kappa_j \res_j$, which implies
  \begin{equation}
    \psi_j(\tilde{\x}^{(t)},\tilde{\v}^{(t)}) = \err_j(\tilde{\x}^{(t)}) + \res_j(\tilde{\v}^{(t)}) \leq \kappa_j \res_j(\tilde{\x}^{(t)}) +  \res_j(\tilde{\v}^{(t)}) \leq 2 \kappa_j \max  \left\{ \res_j(\x^{(t)}), \res_j(\v^{(t)})\right\},
  \end{equation}
  and
  \begin{equation}
    \max  \left\{ \res_j(\x^{(t)}), \res_j(\v^{(t)})\right\} 
    \leq
    \res_j(\x^{(t)}) + \res_j(\v^{(t)}) 
    \leq
    \err_j(\x^{(t)}) + \res_j(\v^{(t)})
    =
    \psi_j(\x^{(t)}, \v^{(t)}).
  \end{equation}
  In summary we have
  \begin{equation}
    \psi_j(\tilde{\x}^{(t)},\tilde{\v}^{(t)}) \leq 2 \kappa_j \globalcond^2 \psi_j(\x^{(t)}, \v^{(t)})
    \label{eq:acbsls:3:1}
  \end{equation} 

  Since $(\x^{(t+1)}, \_) \gets \acbsls_{i+1} (\tilde{\x}^{(t)}, \tilde{\x}^{(t)})$, by induction hypothesis of (a) we have
  \begin{equation}
    \err_j(\x^{(t+1)}) \leq \psi_j(\x^{(t+1)}, \_) \leq \exp \left( - \frac{1}{\sqrt{\kappa_j}} \prod_{k=i+1}^j T_{k} + \left( \sum_{k=i+1}^{j-1} \prod_{l=i+1}^k T_l \right) \log ( 4 \kappa_j^2 \globalcond^2 )\right)
      \psi_j(\tilde{\x}^{(t)}, \tilde{\x}^{(t)}).
      \label{eq:acbsls:3:2}
  \end{equation}
  For the same reason we have
  \begin{equation}
    \res_j(\v^{(t+1)}) \leq \psi_j(\_, \v^{(t+1)}) \leq  \exp \left( - \frac{1}{\sqrt{\kappa_j}} \prod_{k=i+1}^j T_{k} + \left( \sum_{k=i+1}^{j-1} \prod_{l=i+1}^k T_l \right) \log  ( 4 \kappa_j^2 \globalcond^2 ) \right)
    \psi_j(\tilde{\v}^{(t)}, \tilde{\v}^{(t)}).
    \label{eq:acbsls:3:3}
  \end{equation}
  Since $\psi_j (\tilde{\x}^{(t)}, \tilde{\x}^{(t)}) = \err_j(\tilde{\x}^{(t)}) + \res_j(\tilde{\x}^{(t)}) \leq 2 \err_j (\tilde{\x}^{(t)})$ and 
  $\psi_j (\tilde{\v}^{(t)}, \tilde{\v}^{(t)}) = \err_j(\tilde{\v}^{(t)}) + \res_j(\tilde{\v}^{(t)}) \leq (\kappa_j+1) \res_j (\tilde{\v}^{(t)})$
  we have
  \begin{equation}
    \psi_j (\tilde{\x}^{(t)}, \tilde{\x}^{(t)}) + \psi_j (\tilde{\v}^{(t)}, \tilde{\v}^{(t)}) \leq 2 \kappa_j \psi_j(\tilde{\x}^{(t)}, \tilde{\v}^{(t)})
    \label{eq:acbsls:3:4}
  \end{equation}
  Combining \eqref{eq:acbsls:3:2} \eqref{eq:acbsls:3:3} and \eqref{eq:acbsls:3:4} gives
  \begin{align}
			& \psi_j(\x^{(t+1)},\v^{(t+1)})
    =
    \err_j(\x^{(t+1)}) + \res_j(\v^{(t+1)})
    \\
		 \leq & 
    2 \kappa_j \cdot 
    \exp \left( - \frac{1}{\sqrt{\kappa_j}} \prod_{k=i+1}^j T_{k} + \left( \sum_{k=i+1}^{j-1} \prod_{l=i+1}^k T_l \right) \log  ( 4 \kappa_j^2 \globalcond^2 ) \right)
    \psi_j(\tilde{\x}^{(t)}, \tilde{\v}^{(t)}).
    \label{eq:acbsls:3:5}
	\end{align}
  By \eqref{eq:acbsls:3:1} and \eqref{eq:acbsls:3:5} we arrive at 
  \begin{equation}
    \psi_j(\x^{(t+1)},\v^{(t+1)})
    \leq 
    4 \kappa_j^2 \globalcond^2 \cdot 
    \exp \left( - \frac{1}{\sqrt{\kappa_j}} \prod_{k=i+1}^j T_{k} + \left( \sum_{k=i+1}^{j-1} \prod_{l=i+1}^k T_l \right) \log  ( 4 \kappa_j^2 \globalcond^2 ) \right)
    \psi_j({\x}^{(t)}, {\v}^{(t)}).
  \end{equation}
  Telescoping $t$ from $0$ to $T_i$ yields
  \begin{equation}
    \psi_j(\x^{(T_i)}, \v^{(T_i)})
    \leq
    \exp \left( - \frac{1}{\sqrt{\kappa_j}} \prod_{k=i}^j T_{k} + \left( \sum_{k=i}^{j-1} \prod_{l=i}^k T_l \right) \log  ( 4 \kappa_j^2 \globalcond^2 ) \right)
    \psi_j({\x}^{(0)}, {\v}^{(0)}),
  \end{equation}
  completing the induction proof of \Cref{lem:acbsls:3}.
\end{proof}

\subsubsection{Finishing the proof of \Cref{thm:acbsls}}
\label{sec:thm:acbsls:3}
With \Cref{lem:acbsls:3} at hands we are ready to finish the proof of \cref{thm:acbsls}. 
This part of proof is almost identical to the proof of \cref{thm:bsls:recursive} presented in \cref{sec:proof:thm:bsls}.
\begin{proof}[Proof of \Cref{thm:acbsls}]
	Applying \Cref{lem:acbsls:3} yields (for any $i \in [m]$)
	\begin{equation}
		\psi_i( \acbsls_1(\x^{(0)}, \v^{(0)})) \leq
		\exp \left( 
			\underbrace{- \frac{1}{\sqrt{\kappa_i}} \prod_{k=1}^i T_{k} + \left( \sum_{k=1}^{i-1} \prod_{l=1}^k T_l \right) \log \left( 4 \globalcond^4\right)}_{\text{denoted as $\gamma_i$}}    \right)
		\psi_i(\x^{(0)}, \v^{(0)}),
	\end{equation}
	Observe that for any $i= 2,\dots,m$, 
	\begin{align}
			\gamma_{i} - \gamma_{i-1} = & \log (4  \globalcond^4) \cdot \prod_{j=1}^{i-1} T_j - \kappa_{i}^{-\frac{1}{2}} \prod_{j=1}^{i} T_j + \kappa_{i-1}^{-\frac{1}{2}}  \prod_{j=1}^{i-1} T_j
			\\
		= & \prod_{j=1}^{i-1} T_j \cdot \left( -\kappa_{i}^{-\frac{1}{2}} T_{i} + \kappa_{i-1}^{-\frac{1}{2}} + \log (4 \globalcond^4) \right).
	\end{align}
	Since $T_{i} \geq \sqrt{\kappa_i}(\log (4 \globalcond^4) +1)$ we have 
	\begin{equation}
		\gamma_{i} - \gamma_{i-1} \leq \prod_{j=1}^{i-1} T_j \cdot \left( -1 +  \kappa_{i-1}^{-\frac{1}{2}} \right) \leq 0.
	\end{equation}
	For $\gamma_1$ we observe that $		\gamma_1 = - \frac{1}{\sqrt{\kappa_i}} T_1 \leq \log \frac{\epsilon}{\psi(\x^{(0)}, \v^{(0)})}$. 
	Hence $\gamma_m \leq \gamma_{m-1} \leq \cdots \leq \gamma_1 \leq \log \frac{\epsilon}{\psi(\x^{(0)}, \v^{(0)})}$. Therefore for all $i \in [m]$ it is the case that 
	\begin{equation}
		\psi_i(\x, \v)
		\leq
		\exp (\gamma_i) \cdot	\psi_i(\x^{(0)}, \v^{(0)})
		\leq
		\frac{\epsilon}{	 \psi(\x^{(0)}, \v^{(0)}) } 	\psi_i(\x^{(0)}, \v^{(0)})
	\end{equation}
	Taking summation over $i$ gives
	\begin{equation}
		\psi(\acbsls_1(\x^{(0)}, \v^{(0)})) \leq \sum_{i \in [m]} \frac{\epsilon}{\psi(\x^{(0)}, \v^{(0)})} \psi_i(\x^{(0)}, \v^{(0)}) = \epsilon,
	\end{equation}
	completing the proof of \Cref{thm:acbsls}.
\end{proof}

\subsection{Stability of $\acbsls$}
\label{sec:acbsls:stability}
Similar to (un-accelerated) $\bsls$, under finite-precision arithmetic, $\acbsls$ can also attain the same rate of convergence with only logarithmic bits of precision. 

Formally, let $\widehat{\agd}$ be the finite-precision implementation of $\agd$, and $\widehat{\acbsls}_i$ be the finite-precision implementation of $\acbsls$ by replacing $\agd$ with $\widehat{\agd}$. 
We impose the following requirement such that $\widehat{\agd}$ can return a $\delta$-multiplicative approximation of ${\agd}$ in both $\x$ and $\v$:
\begin{requirement}
	\label{req:agd}
	There exists a $\delta<1$ such that for any $\x$, $\v$, and $i$,  considering $(\x_+, \v_+) \gets \agd(\x, \v; L_i, \mu_i)$ and $(\widehat{\x_+}, \widehat{\v_+}) \gets \widehat{\agd} (\x, \v; L_i,\mu_i)$, it is the case that $		\left|\widehat{\x_+} - {\x_+} \right| \leq \delta \left|\x_+ \right|$ and $\left|\widehat{\v_+} - {\v_+} \right| \leq \delta \left|\v_+ \right|$.
	(We use $|\cdot|$ to denote element-wise absolute values).
\end{requirement}

We specialized the initializations to $\vec{0}$ to simplify the exposition of the theorem. In Appendix \ref{apx:acbsls:finite}, we provide and prove the general version with arbitrary $\x^{(0)}, \v^{(0)}$. 
\begin{theorem}[$\acbsls$ under finite-precision arithmetic]
	\label{thm:acbsls:inexact}
	Consider multiscale optimization problem defined in \Cref{def:multiscale_problem}, for any $\epsilon > 0$, assuming \cref{req:agd} with
	\begin{equation}
		\delta^{-1} \geq 4 m  \left( \prod_{i \in [m]} T_i \right) 
		\cdot (10 \globalcond^2 )^{2m-1} \cdot \frac{\psi(\vec{0}, \vec{0}) }{\epsilon}, 
	\end{equation}
	then $\min\{\psi(\vec{0}, \vec{0}), \psi(\widehat{\acbsls_{1}}(\vec{0}, \vec{0}))\} \leq 3 \epsilon$ provided that $T_1, \ldots, T_m$ satisfy \cref{eq:acbsls:t:lb} (with $\x^{(0)} = \v^{(0)} = \vec{0}$), when $\{(\mu_i, L_i), i \in [m]\}$ are known.
	We can also achieve the same asymptotic sample complexity (up to constant factors suppressed in the $\bigo(\cdot)$) when $\{(\mu_i, L_i), i \in [m]\}$ are unknown and only $m$, $\mu_{1}$, $L_m$ and $ \pi_{\kappa}=\prod_{i=1}^m \kappa_i$ are known.
	
\end{theorem}

The proof of \cref{thm:acbsls:inexact} is deferred to Appendix \ref{apx:acbsls:finite}.

\subsection{Why do we need branching for $\acbsls$}
\label{sec:why:branching}
In this subsection we demonstrate why naive $\acbsls$ may not converge. 
Specifically, we consider the following \cref{alg:naive:acbsls}. 
The only difference compared with the principled $\acbsls$ (defined in \cref{alg:acbsls}) is the replacement of the branching procedure with a naive recursion.
\begin{algorithm}
	\caption{Naive Accelerated Big-Step Little-Step Algorithm (may not converge)}
	\label{alg:naive:acbsls}
	\begin{algorithmic}[1]
		\REQUIRE $\texttt{NaiveAcBSLS}_i$ $(\x^{(0)}, \v^{(0)})$
		\FOR{$t = 0, 1, \ldots, T_i-1$}
			\STATE $(\tilde{\x}^{(t)}, \tilde{\v}^{(t)}) \gets \agd(\x_t, \v_t; L_i, \mu_i)$
			\COMMENT{$\agd$ is the same as the original one (\cref{alg:acbsls})}
			\IF {$i < m$}
				\STATE $(\x^{(t+1)}, \v^{(t+1)}) \gets \acbsls_{i+1} (\tilde{\x}^{(t)}, \tilde{\v}^{(t)})$
				\COMMENT{Naively recurse instead of branching}
			\ELSE
				\STATE $(\x^{(t+1)}, \v^{(t+1)}) \gets (\tilde{\x}^{(t)}, \tilde{\v}^{(t)})$
			\ENDIF
		\ENDFOR
		\RETURN $(\x^{(T_i)}, \v^{(T_i)})$
	\end{algorithmic}
\end{algorithm}

\subsubsection{Theoretical evidence}
Following the discussion after \cref{lem:agd1}, we provide a simple result suggesting the potential for AGD may not be ``backward compatible'' (specifically, the potential governing small $[\mu_i, L_i]$ may not be conservative under $\agd$ with larger $[\mu, L]$, although the latter takes smaller step.)
Therefore one cannot replace \cref{lem:agd1} with 		\cref{eq:agd:false:claim}. 
Formally, we prove the following proposition. 
\begin{proposition}
    \label{lem:naive:acbsls}
	There exists a function $f: \reals^d \to \reals$ that is $\mu_1$-strongly-convex and $L_1$-smooth, but for certain $(\x^{(0)}, \v^{(0)})$ it is the case that
	\begin{equation}
	    \psi_1(\agd(\x^{(0)}, \v^{(0)}; L_2, \mu_2)) > \psi_1(\x^{(0)}, \v^{(0)}).
	\end{equation}
	for some $\mu_2, L_2$ such that $L_2 > \mu_2 > L_1$. 
	Here $\psi_1(\x, \v)$ is the potential associated with $f$, namely
	\begin{equation}
	    \psi_1(\x, \v) \defeq f(\x) - f^{\star} + \frac{\mu_1}{2} \| \v - \x^{\star} \|_2^2
	\end{equation}
\end{proposition}
\begin{remark}
    Although \cref{lem:naive:acbsls} does not rule out the possibility of other conservative potentials, we conjecture that such a potential may not exist given the inherent instability of accelerated GD (c.f., Section F of \cite{Yuan.Ma-NeurIPS20}).
\end{remark}
\begin{proof}[Proof of \cref{lem:naive:acbsls}]
    Consider the following objective $f: \reals^2 \to \reals$: $f(\x) = 0.5 x_1^2 + 5 x_2^2$. Apparently $f$ is $1$-strongly-convex, $10$-smooth.
    Consider initialization $\x^{(0)} = (0, 0)^\top, \v^{(0)} = (1, 1)^\top.$
    Then one can verify that $        \psi_1(\x^{(0)}, \v^{(0)}) = 1$ but $\psi_1 (\agd(\x^{(0)}, \v^{(0)}; L_2, \mu_2)) > 1.18$ for $L_2 = 200$ and $\mu_2 = 100$. 
\end{proof}

\subsubsection{Numerical evidence}
Next, we provide numerical evidence against the convergence of naive $\acbsls$, see \cref{fig:naive:acbsls}.
We synthesize a quadratic objective with eigenvalues belonging to $[10^{-4}, 10^{-3}] \cup [0.5, 1]$. 
We implement both the principled $\acbsls$ (with branching, see \cref{alg:acbsls}) and naive $\acbsls$ (\cref{alg:naive:acbsls}) with the corresponding $\mu_1, \mu_2, L_1, L_2$. 
We observe that the principled $\acbsls$ (with branching) converge with $T_2 = 8$, as expected. 
On the other hand, the naive $\acbsls$ fails to converge for any $T_2 \in \{8, 16, 32, 64\}$. 
The implementation details can be found in the accompanying notebook in supplementary materials.

\begin{figure}[t]
    \centering
    \includegraphics[width=8cm]{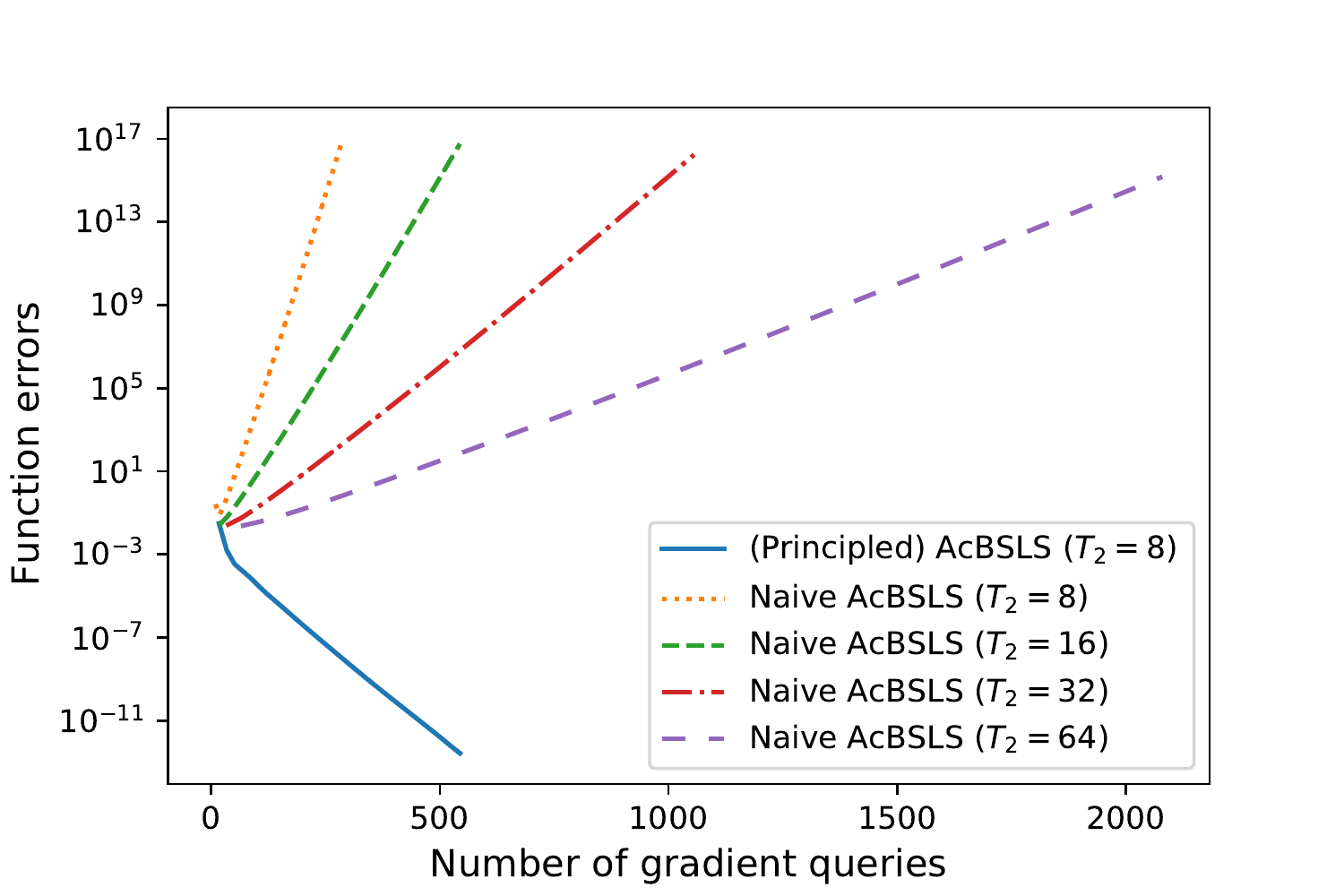}
    \caption{
    \textbf{Numerical evidence that Naive $\acbsls$ (\cref{alg:naive:acbsls}) may not converge}. 
    Observe that  the principled $\acbsls$ (with branching) converge with $T_2 = 8$, but the naive $\acbsls$ fails to converge for any $T_2 \in \{8, 16, 32, 64\}$.}
    \label{fig:naive:acbsls}
\end{figure}

\section{Lower bound for multiscale optimization}
\label{sec:lb}
In this section, we prove our lower bound results (\cref{thm:lb}) of the multi-scale optimization problem.

\subsection{Proof structure of \cref{thm:lb}}
We will separate the proof of \cref{thm:lb} into three parts.

\paragraph{Part I: Reduction to uniform polynomial approximation.}
In the first part, we reduce the problem of a lower bound over arbitrary first-order deterministic algorithms to a constrained polynomial uniform approximation problem on $S = \bigcup_{i \in [m]} [\mu_i, L_i]$ across $\polys_k^0$, where (throughout this section) 
\begin{equation}
    \polys_k^0 = \{p :  \text{$p$ is a polynomial of degree at most $k$ and } p(0) = 1\}.
    \label{eq:def:pk}
\end{equation}
The result is as follows.
\begin{lemma}[Reduction to a uniform polynomial approximation problem]
    \label{lem:lb:reduce:to:apx}
    For any first-order deterministic algorithm $\mathsf{A}$,  for any $k \in \naturals$ and $\gradbound > 0$, there exists an objective $f$ satisfying \cref{def:multiscale_problem} with $\|\nabla f(\vec{0})\|_2 \leq \gradbound$ such that
    \begin{equation}
        \min_{\tau \in [k]}\| \nabla f(\x^{(\tau)})\|_2  \geq 
        \left( \min_{p \in \polys_k^0} \max_{\lambda \in S} |p(\lambda)| \right) \cdot \gradbound.
    \end{equation}
\end{lemma}
The rough proof idea of \cref{lem:lb:reduce:to:apx} is to 1) first reduce the general first-order deterministic algorithm class to the construction of a tri-diagonal objective for which zero-respecting algorithm (see \cite{Carmon.Duchi.ea-MP21} for definition) is hard , then 2) reduce to the problem of discrete weighted $\ell_2$ polynomial approximation over $S$, and finally 3) reduce to uniform polynomial approximation over $S$. 
The detailed proof of \cref{lem:lb:reduce:to:apx} is relegated to  \cref{sec:lb:reduction:to:apx}. 

\paragraph{Part II: Reduction to Green's function.}
In the second part,  we cite classic results from potential theory literature to reduce the uniform polynomial approximation problem raised in \cref{lem:lb:reduce:to:apx} to the estimation of Green's function. The results are as follows.
\begin{lemma}[Reduction to Green's function]
    \label{lem:lb:reduce:to:Green}
    Let $S = \bigcup_{i \in [m]} [\mu_i, L_i]$, then for any $k \in \naturals$, the following inequality holds 
    \begin{equation}
        \min_{p \in \polys_k^0} \max_{\lambda \in S} |p(\lambda)| \geq \exp (-k g_S(0))
    \end{equation}
    where $g_S(0)$ is the Green's function associated with $S$ (with pole at $\infty$), see \cref{def:green} in \cref{sec:reduce:to:green} for formal definition.
\end{lemma}
The detailed reference of \cref{lem:lb:reduce:to:Green} is relegated to \cref{sec:reduce:to:green}.

\paragraph{Part III: Estimating (upper bound) the Green's function.}
In the last part, we provide a bound of $g_S(0)$ as follows.  
We identify that this estimate may be of independent interest.
\begin{lemma}[Estimating the Green's function]
    \label{lem:lb:green:ub}
    Let $S = \bigcup_{j=1}^m [\mu_j, L_j]$, and assume $\frac{L_j}{\mu_j} \geq 2$ for $j \in [m]$. 
    Then the Green's function associated with $S$ satisfies
    \begin{equation}
        g_S(0) \leq \frac{8}{\sqrt{\prod_{i \in [m]} \frac{L_i}{\mu_i}} \cdot \prod_{i \in [m-1]} \left(0.03 \cdot \log (16 \frac{\mu_{i+1}}{L_i}) \right)}.
    \end{equation}
\end{lemma}
The proof of \cref{lem:lb:green:ub} is relegated to \cref{sec:lb:green:ub}.

The \cref{thm:lb} then follows immediately from \cref{lem:lb:reduce:to:apx,lem:lb:reduce:to:Green,lem:lb:green:ub}.

\subsection{Proof of \cref{lem:lb:reduce:to:apx}: Reduction to uniform polynomial approximation}
\label{sec:lb:reduction:to:apx}
In this subsection we will prove \cref{lem:lb:reduce:to:apx} on the reduction from the lower bound of arbitrary first-order deterministic algorithms to the uniform polynomial approximation problem. 

We will prove \cref{lem:lb:reduce:to:apx} in three steps.

\paragraph{Step 1: Reduction to a first-order zero chain (or hard tri-diagonal quadratic objective).}
Following the techniques of \cite{Carmon.Duchi.ea-MP21}, we first reduce the lower bound across all first-order deterministic algorithms to the construction of a ``first-order zero chain'' \cite{Carmon.Duchi.ea-MP21}. 
Specifically, we reduce to the existence of tri-diagonal quadratic objectives with ``large'' gradients under limited supports.

\begin{lemma}[Reduction from arbitrary first-order deterministic algorithms to first-order zero-chains]
    \label{lem:lb:reduction:to:chain}
    Let $S = \bigcup_{i \in [m]} [\mu_i, L_i]$, suppose for some $\gradbound > 0$, $\epsilon > 0$ and $k \in \naturals$, there exists a symmetric tri-diagonal matrix $\mat{T} \in \reals^{(k + m) \times (k+m)}$ with eigenvalues all belonging to $S$, and suppose the objective $f(\x) := \frac{1}{2} \x^\top \mat{T} \x + \gradbound \cdot \unit_1^\top \x$ satisfies
    \begin{equation}
        \min_{\x: x_{k+1} = x_{k+2} = \cdots = x_{k+m} = 0} \|\nabla f(\x) \|_2 \geq \epsilon.
    \end{equation}
    Then for any first-order deterministic algorithm $\mathsf{A}$, there exists a function $\tilde{f}$ satisfying \cref{def:multiscale_problem} with $\|\nabla f(\vec{0})\|_2 \leq \gradbound$ 
    such that the trajectory $\{\x^{(t)}\}_{t=1}^{\infty}$ generated by $\mathsf{A}$ on $\tilde{f}$ satisfies 
    \begin{equation}
        \min_{\tau \in [k+1]}\| \nabla \tilde{f}(\x^{(\tau)})\|_2 \geq \epsilon.
    \end{equation}
\end{lemma}
The proof of \cref{lem:lb:reduction:to:chain} is similar to the original proof of lower bounds in \cite{Carmon.Duchi.ea-MP21}. 
We first reduce the range of arbitrary deterministic first-order algorithms to zero-respecting algorithms via the equivalency result in \cite{Carmon.Duchi.ea-MP21}, and then show that any zero-respecting algorithm can only reveal one dimension per step for the tri-diagonal quadratic objective. 
The detailed proof of  \cref{lem:lb:reduction:to:chain} is relegated to \cref{sec:pf:lem:lb:reduction:to:chain}.

\paragraph{Step 2: Reduction to discrete weighted $\ell_2$ polynomial approximation.}
Next, we reduce the problem raised in \cref{lem:lb:reduction:to:chain} to the following constrained weighted discrete $\ell_2$ polynomial approximation problem. 
\begin{lemma}[Reduction to discrete weighted $\ell_2$ polynomial approximation]
    \label{lem:lb:reduction:to:l2:poly:apx}
    Let $S = \bigcup_{i \in [m]} [\mu_i, L_i]$, then for any $k \in \naturals$, for any $\gradbound > 0$, there exists a symmetric tri-diagonal matrix $\mat{T} \in \reals^{(k + m) \times (k + m)}$ with eigenvalues all belonging to $S$ such that the objective $f(\x) = \frac{1}{2} \x^\top \mat{T} \x + \gradbound \cdot \unit_1^\top \x$ satisfies
    \begin{equation}
        \min_{\x: x_{k+1} = \cdots = x_{k+m} = 0} \|\nabla f(\x)\|_2 
        \geq 
        \gradbound \cdot 
        \max_{\lambda_1, \ldots, \lambda_{k+m} \in S}
        \max_{\sum_{i \in [k+m]} v_i^2 = 1}
        \min_{p \in \polys_{k}^0} \sqrt{\sum_{i \in [k+m]} (p(\lambda_i) v_i)^2}.
        \label{eq:lem:lb:reduction:to:l2:poly:apx}
    \end{equation} 
    Recall $\polys_k^0$ is defined in \cref{eq:def:pk} as the set of polynomials $p$ of degree at most $k$ with $p(0) = 1$.
\end{lemma}
The proof of \cref{lem:lb:reduction:to:l2:poly:apx} is constructive, where we explicitly construct a symmetric tri-diagonal matrix $\mat{T}$ with large $\min_{\x: x_{k+1} = \cdots = x_{k+m} = 0} \|\nabla f(\x)\|_2$. 
The detailed proof is relegated to \cref{sec:pf:lem:lb:reduction:to:l2:poly:apx}.

\paragraph{Step 3: Reduction to uniform polynomial approximation on $S$.}
Finally, we reduce the discrete weighted $\ell_2$-approximation problem raised in \cref{lem:lb:reduction:to:l2:poly:apx} to a uniform polynomial approximation over $S$ across $\polys_k^0$.
\begin{lemma}[Reduction to uniform polynomial approximation on $S$]
    \label{lem:lb:reduction:to:unif}
    Let $S = \bigcup_{i \in [m]} [\mu_i, L_i]$, then for any $k \in \naturals$, the following inequality holds
    \begin{equation}
        \max_{\lambda_1, \ldots, \lambda_{k+m} \in S}
        \max_{\sum_{i \in [k+m]} v_i^2 = 1}
        \min_{p \in \polys_{k}^0} \sqrt{\sum_{i \in [k+m]} (p(\lambda_i) v_i)^2}
        \geq
        \min_{p \in \polys_{k}^0} \max_{\lambda \in S} |p(\lambda)|.
    \end{equation}
    Recall $\polys_k^0$ is defined in \cref{eq:def:pk} as the set of polynomials $p$ of degree at most $k$ with $p(0) = 1$.
\end{lemma}
The proof of \cref{lem:lb:reduction:to:unif} is based on the fact that the best uniform approximation over $S$, denoted as $p_k^{\star}$, is also the best discrete weighted $\ell_2$ approximation over the extreme points on $p_k^{\star}$ with appropriate weights. 
To this end, we will show that $p_k^{\star}$ is orthogonal to low-degree polynomials under these weights.
The detailed proof of \cref{lem:lb:reduction:to:unif} is relegated to \cref{sec:pf:lem:lb:reduction:to:unif}.

The proof of \cref{lem:lb:reduce:to:apx} then follows immediately from \cref{lem:lb:reduction:to:chain,lem:lb:reduction:to:l2:poly:apx,lem:lb:reduction:to:unif}.

\subsubsection{Deferred proof of \cref{lem:lb:reduction:to:chain}}
\label{sec:pf:lem:lb:reduction:to:chain}
\begin{proof}[Proof of \cref{lem:lb:reduction:to:chain}]
Since $\mat{T}$ has all its eigenvalues within $S = \bigcup_{i \in [m]} [\mu_i, L_i]$, the objective $f(\x)$ satisfies \cref{def:multiscale_problem}.
Since $\mat{T}$ is tri-diagonal, for any zero-respecting first-order algorithm $\mathsf{A}_{\mathrm{zr}}$ (see \cite{Carmon.Duchi.ea-MP21} for definition) initialization at $\vec{0}$, the first $k+1$ iterates $\x^{(1)}, \ldots, \x^{(k+1)}$ are all supported in the first $k$ coordinates. Hence 
\begin{equation}
    \min_{\tau \in [k+1]} \|\nabla f(\x^{(\tau)}) \|_2
    \geq
    \min_{\x: x_{k+1} = x_{k+2} = \cdots = x_{k+m} = 0} \|\nabla f(\x) \|_2 \geq \epsilon.
\end{equation}
Let $\mathcal{F}(\gradbound)$ denote the union, over $d \in \naturals$, of the collections of $\mathcal{C}^{\infty}$ convex functions $f: \reals^d \to \reals$ satisfying \cref{def:multiscale_problem} and $\|\nabla f(\vec{0})\|_2 \leq \gradbound$. 
Since $\mathcal{F}(\gradbound)$ is orthogonally invariant, by Proposition 1 of \cite{Carmon.Duchi.ea-MP21}, the time complexities over all first-order deterministic algorithms are lower bounded by the zero respecting first-order algorithms, completing the proof.

\end{proof}

\subsubsection{Deferred proof of \cref{lem:lb:reduction:to:l2:poly:apx}}
\label{sec:pf:lem:lb:reduction:to:l2:poly:apx}
We introduce the following definition for ease of exposition.
\begin{definition}
    A symmetric tri-diagonal matrix is \textbf{non-degenerate} if none of its sub-diagonal entries are zero.
\end{definition}
We first show that for any distinct $\{\lambda_i\}_{i \in [k+m]}$ and positive $\{v_i\}_{i \in [k+m]}$, one can construct a desired tri-diagonal matrix.
\begin{lemma}
    \label{lem:tri-diagonalization}
    Let $\lambda_1, \lambda_2, \ldots, \lambda_{k+m}$ be a set of distinct positive numbers, and $v_1, v_2, \ldots, v_{k+m}$ be another set of positive numbers with $\sum_{i \in [k+m]} v_i^2 = 1$. 
    Then there exists an orthogonal matrix $\mat{Q} \in \reals^{(k+m) \times (k+m)}$ such that 
    \begin{enumerate}
        \item $\mat{Q} \unit_1 = \vec{v}$, where $\vec{v} := [v_1, v_2, \ldots, v_{k+m}]^\top$.
        \item $\mat{Q}^\top \mat{\Lambda} \mat{Q}$ is non-degenerate tri-diagonal, where $\mat{\Lambda} = \diag \left( \lambda_1, \lambda_2, \ldots, \lambda_{k+m} \right)$.
    \end{enumerate}
\end{lemma}
\begin{proof}[Proof of \cref{lem:tri-diagonalization}]
    We construct $\mat{Q} = [\vec{q}_1, \vec{q}_2, \ldots, \vec{q}_{k+m}]$ column by column as follows:
    \begin{enumerate}[(a)]
        \item $\vec{q}_1 = \vec{v}$.
        \item For any $j = 2, \ldots, k+m$, let $\tilde{\vec{q}}_j = (\id - \sum_{i \in [j-1]} \vec{q}_i \vec{q}_i^\top)\mat{\Lambda} \vec{q}_{j-1}$, $\vec{q}_j = \frac{\tilde{\vec{q}}_j}{\|\tilde{\vec{q}}_j\|_2}$.
    \end{enumerate}
    One can verify that 
    \begin{enumerate}[(i)]
        \item For any $i \in [k+m]$, $ \| \vec{q}_i \|_2 = 1$.
        \item For any $i < j$, $\vec{q}_i^\top \tilde{\vec{q}}_j = 0$ and thus $\vec{q}_i^\top \vec{q}_j = 0$.
        \item For any $i < j - 1$, $\vec{q}_i^\top \mat{\Lambda} \vec{q}_j = 0$ (To see this, first observe that $\spa \left \langle \vec{q}_1, \vec{q}_2, \ldots, \vec{q}_j \right \rangle = \spa \left \langle \v, \mat{\Lambda} \v, \ldots, \mat{\Lambda}^{j-1} \v\right \rangle$ for any $j$. Thus $\mat{\Lambda} \vec{q}_i \in \spa \left \langle \vec{q}_1, \ldots, \vec{q}_{i+1} \right \rangle$. Consequently we have $\vec{q}_i^\top \mat{\Lambda} \vec{q}_j = 0$ by point (ii) above for any $i < j -1$).
    \end{enumerate}
    By (i) and (ii) we know $\mat{Q}$ is orthogonal. By (iii) we know $\mat{Q}^\top \mat{\Lambda} \mat{Q}$ is tri-diagonal.
    The non-degeneracy of $\mat{T}$ follows by the linear independence of $\{\v, \mat{\Lambda} \v, \ldots, \mat{\Lambda}^{k+m-1}\}$ since $\vec{v} > \vec{0}$, $\mat{\Lambda} > \mat{0}$ and the distinctness of $\{\lambda_i\}_{i \in [k+m]}$.
\end{proof}
Next, following \cref{lem:tri-diagonalization}, we show that the tri-diagonal objective has large $\|\nabla f(\x)\|_2$ when the last $m$ coordinates of $\x$ are zero.
\begin{lemma}
    \label{lem:tri-diagonalization:2}
    Under the same settings and notation of \cref{lem:tri-diagonalization}, let $\mat{T} = \mat{Q}^\top \mat{\Lambda} \mat{Q}$, and consider objective $f(\x) = \frac{1}{2} \x^\top \mat{T} \x + \gradbound \cdot \unit_1^\top \x$. Then
    \begin{equation}
        \min_{\x: x_{k+1} = \cdots = x_{k+m} =0} \|\nabla f(\x)\|_2 = 
        \gradbound \cdot \min_{p \in \polys_{k}^0} \| p(\mat{\Lambda}) \mat{Q} \unit_1 \|_2 
        = 
        \gradbound \cdot \min_{p \in \polys_{k}^0} \sqrt{\sum_{i \in [k+m]} (p(\lambda_i) v_i)^2}.
    \end{equation}
\end{lemma}
\begin{proof}[Proof of \cref{lem:tri-diagonalization:2}]
    By non-degeneracy of $\mat{T}$ we have
    \begin{equation}
        \{ \x : x_{k+1} = \cdots = x_{k+m} = 0 \} = \{p (\mat{T}) \unit_1: \deg p \leq k-1  \}.
    \end{equation}
    Thus the following sets are identical
    \begin{align}
        &   \{ \nabla f(\x):  x_{k+1} = \cdots = x_{k+m} = 0  \} 
        = 
        \{   \mat{T} \x + \gradbound \cdot \unit_1:   x_{k+1} = \cdots = x_{k+m} = 0  \} 
        \\ 
        = & \{ \gradbound \cdot p(\mat{T}) \unit_1 : p \in \polys_{k}^0 \}.
    \end{align}
    It follows that
    \begin{equation}
        \min_{\x:  x_{k+1} = \cdots = x_{k+m} = 0} \|\nabla f(\x)\|_2
        = 
        \gradbound \cdot \min_{p \in \polys_{k}^0} \| p(\mat{T}) \unit_1\|_2
        =
        \gradbound \cdot \min_{p \in \polys_{k}^0} \| p(\mat{\Lambda}) \mat{Q} \unit_1 \|_2.
    \end{equation}
    The last equality is due to $\mat{Q} \unit_1 = \vec{v}$.
\end{proof}

Now we finish the proof of  \cref{lem:lb:reduction:to:l2:poly:apx}.
\begin{proof}[Proof of \cref{lem:lb:reduction:to:l2:poly:apx}]
    By \cref{lem:tri-diagonalization,lem:tri-diagonalization:2}, for any distinct $\lambda_1, \ldots, \lambda_{k+m} \in S$ and $v_1, \ldots, v_{k+m} > 0$ such that $\sum_{i \in [k]} v_i^2 = 1$, we have
    \begin{equation}
        \min_{\x:  x_{k+1} = \cdots = x_{k+m} = 0} \|\nabla f(\x)\|_2 
        \geq 
        \gradbound \cdot 
        \min_{p \in \polys_{k}^0} \sqrt{\sum_{i \in [k+m]} (p(\lambda_i) v_i)^2}.
   \end{equation} 
    If $\lambda_1, \ldots, \lambda_{k+m}$ are not distinct, then one can find another set of distinct $\lambda_i$'s such that the RHS is not smaller. 
    The same arguments hold if one of the $v_i$ is zero.
    Hence
    \begin{equation}
        \min_{\x:  x_{k+1} = \cdots = x_{k+m} = 0} \|\nabla f(\x)\|_2 
        \geq 
        \gradbound \cdot 
        \max_{\lambda_1, \ldots, \lambda_{k+m} \in S}
        \max_{\sum_{i \in [k+m]} v_i^2 = 1}
        \min_{p \in \polys_{k}^0} \sqrt{\sum_{i \in [k+m]} (p(\lambda_i) v_i)^2}.
        \label{eq:lem:lb:reduce:to:apx}
    \end{equation} 
    
\end{proof}

\subsubsection{Deferred proof of \cref{lem:lb:reduction:to:unif}}
We first cite the following \cref{lem:lb:cheb} that characterizes the uniform approximation on $S$ within $\polys_k^0$.     Recall $\polys_k^0$ is defined as the set of polynomials $p$ of degree at most $k$ with $p(0) = 1$.
\label{sec:pf:lem:lb:reduction:to:unif}
\begin{lemma}[Characterization of uniform approximation on $S$, adapted from \cite{Schiefermayr.Peherstorfer-99,Schiefermayr-11}]
    \label{lem:lb:cheb}
    Let $S = \bigcup_{i \in [m]} [\mu_i, L_i]$, denote $I_i = [\mu_i, L_i]$ then for any $k \in \naturals$, 
    \begin{enumerate}[(a), leftmargin=*]
        \item  The best uniform approximation $\min_{p \in \polys_{k}^0} \max_{x \in S} |p(x)|$ is attained, denoted as $p_k^{\star}$. Denote $\|p_k^{\star}\|_S :=  \max_{\lambda \in S} |p_k^{\star}(\lambda)|$ hereinafter.
        \item $|p_k^{\star}|$ attains $\| p_k^{\star} \|_S$  in $S$ for $s \in \{k+1, \ldots, k+m\}$ times, denoted as $\lambda_1 < \lambda_2 < \ldots < \lambda_{s}$. (That is to say $|p_k^{\star} (\lambda_1)| = |p_k^{\star} (\lambda_2)| = \cdots = |p_k^{\star} (\lambda_s)|$ and $\lambda_i \in S$ for $i \in [s]$). The $\lambda_i$'s are called ``e-points'' in the literature.
        \item $p_k^{\star}(\lambda_1), \ldots, p_k^{\star}(\lambda_s)$ change signs for exactly $k$ times, namely
        \begin{equation}
            |\{j: \sgn(p_k^{\star}(\lambda_j)) \neq \sgn( p_k^{\star} (\lambda_{j+1}))\}| = k
        \end{equation}
        \item If $\lambda_j$ and $\lambda_{j+1}$ belong to the same interval $I_i$, then $\sgn(p_k^{\star} (\lambda_j)) \cdot \sgn(p_k^{\star} (\lambda_{j+1})) < 0$. 
        \item If $\lambda_j$ and $\lambda_{j+1}$ belong to two different intervals $I_{i_1}, I_{i_2}$, and $\sgn(p_k^{\star} (\lambda_j)) \cdot \sgn(p_k^{\star} (\lambda_{j+1})) > 0$, then $\lambda_j = L_{i_1}$, $\lambda_{j+1} = \mu_{i_2}$.
        \item If $\lambda_j$ and $\lambda_{j+1}$ belong to two different intervals $I_{i_1}, I_{i_2}$, and $\sgn(p_k^{\star} (\lambda_j)) \cdot \sgn(p_k^{\star} (\lambda_{j+1})) < 0$, then $\max_{\lambda_j \leq \lambda \leq \lambda_{j+1}} |p_k^{\star}(\lambda)| = \| p_k^{\star} \|_S$. Therefore $p_k^{\star}$ is also the best unifrom approximation within $\polys_k^0$ for $I_1 \cup \cdots \cup I_{i_1 - 1} \cup [\mu_{i_1}, L_{i_2}] \cup I_{i_2 + 1} \cdots \cup I_m$. 
    \end{enumerate}
\end{lemma}
(a-c) extends the well-known Chebyshev equioscillation theorem to the union of multiple intervals. 
These results were originally developed by Achieser \cite{Achieser-28,Achieser-32,Achieser-33,Achieser-33a,Achieser-34} for the union of two intervals and later generalized by \cite{Grcar-81}. We adapt the statements from \cite{Schiefermayr-11}. 
(d-f) is adapted from \cite{Schiefermayr.Peherstorfer-99}. 
Similar characterizations can also be found in \cite{Widom-69,Nevai-86,Lubinsky-87,Fischer-92,Peherstorfer-93,Fischer-11}.

Next, we show that the best uniform approximation $p_k^{\star} \in \polys_k^0$ is also the best discrete $\ell_2$ approximation on the e-points $\lambda_1, \ldots, \lambda_s$ with a specific set of weights $v_i$'s.
\begin{lemma}
    \label{lem:lb:Tpoly}
    Under the same setting and notation of \cref{lem:lb:cheb}, define 
    \begin{equation}
        c_j = \begin{cases}
            \frac{1}{2} & \text{if $\sgn(p_k^{\star}(\lambda_j)) \cdot \sgn(p_k^{\star}(\lambda_{j+1})) > 0$ or $j = 1$ or $j = s$}
            \\
            1 & \text{otherwise}
        \end{cases}, \quad
        v_j = \sqrt{ \frac{c_j / \lambda_j}{\sum_{i \in [s]} {c_i}/{\lambda_i}}}
    \end{equation}
    Then for any $p \in \polys_k^0$, the following inequality holds
    \begin{equation}
        \sum_{j \in [s]} (v_j p(\lambda_j))^2 \geq \|p_k^{\star}\|_S^2.
    \end{equation}
\end{lemma}

\begin{proof}[Proof of \cref{lem:lb:Tpoly}]
    Repeat the interval-merging procedure in \cref{lem:lb:cheb}(f) until there isn't any consecutive pair $\lambda_j, \lambda_{j+1}$ that belongs to two different intervals but $\sgn(p_k^{\star} (\lambda_j)) \cdot \sgn(p_k^{\star} (\lambda_{j+1})) < 0$. 
    After merging, there are exactly $s-k$ intervals left,  denoted as $S' = \bigcup_{i \in [s-k]} J_i$. 
    
    By definition, $p_k^{\star}$  is a (un-normalized) T-polynomial on $S'$ (see \cite{Schiefermayr.Peherstorfer-99} for definition).
    By Theorem 2.3 of  \cite{Schiefermayr.Peherstorfer-99},  for any polynomial $q$ with $\deg q < k$, it is the case that $\sum_{j\in[s]} c_j p_k^{\star} (\lambda_j) q (\lambda_j) = 0$.
    Since both $p, p_k^{\star} \in \polys_k^0$, we know that $\frac{p(\lambda)- p_k^{\star}(\lambda)}{\lambda}$ is a polynomial with degree $<k$ (since $p(0) = p_k^{\star}(0) = 1$). 
    Hence
    \begin{equation}
        \sum_{j \in [s]} v_j^2 p_k^{\star}(\lambda_j) (p(\lambda_j)- p_k^{\star}(\lambda_j)) =
        \frac{1}{\sum_{j \in [s]} \frac{c_j}{\lambda_j}} \sum_{j \in [s]} c_j p_k^{\star}(\lambda_j) \frac{p(\lambda_j)- p_k^{\star}(\lambda_j)}{\lambda_j} = 0.
    \end{equation}
    Therefore (by orthogonality)
    \begin{equation}
        \sum_{j \in [s]} (v_j p(\lambda_j))^2 \geq 2 \sum_{j \in [s]} v_j^2 p_k^{\star}(\lambda_j) (p(\lambda_j) - p_k^{\star}(\lambda_j)) + \sum_{j \in [s]} (v_j p_k^{\star}(\lambda_j))^2 
        = 
        \sum_{j \in [s]} (v_j p_k^{\star}(\lambda_j))^2 
        =
        \| p_k^{\star} \|^2_S.
    \end{equation}
\end{proof}

The proof of \cref{lem:lb:reduction:to:unif} is immediate once we have \cref{lem:lb:cheb} and \cref{lem:lb:Tpoly}.
\begin{proof}[Proof of \cref{lem:lb:reduction:to:unif}]
    Apply \cref{lem:lb:cheb} and \cref{lem:lb:Tpoly}, one has for some $s \in \{k + 1, \ldots, k+m\}$.
    \begin{equation}
        \max_{v_1, \ldots, v_{s}, \sum_{i \in [s]} v_i^2 = 1} \max_{\lambda_1, \ldots, \lambda_s \in S} \min_{p \in \polys_k^0} 
        \sqrt{ \sum_{j \in [s]} (v_j p(\lambda_j))^2}
        \geq \min_{p \in \polys_{k}^0} \max_{\lambda \in S} |p(\lambda)|.
    \end{equation}
    Since $s \leq k+m$, we have therefore
    \begin{equation}
        \max_{v_1, \ldots, v_{k+m}, \sum_{i \in [k+m]} v_i^2 = 1} \max_{\lambda_1, \ldots, \lambda_{k+m} \in S} \min_{p \in \polys_k^0} 
        \sqrt{ \sum_{j \in [k+m]} (v_j p(\lambda_j))^2}
        \geq \min_{p \in \polys_{k}^0} \max_{\lambda \in S} |p(\lambda)|.
    \end{equation}
\end{proof}

\subsection{Reference of \cref{lem:lb:reduce:to:Green}: Reduction to the estimation of Green's function}
\label{sec:reduce:to:green}
In this section, we will cite literature from potential theory to reduce the uniform approximation problem raised in \cref{lem:lb:reduce:to:apx} to estimating Green's function, as stated in \cref{lem:lb:reduce:to:Green}. 
Most of the results in this subsection are classic (c.f., \cite{Widom-69,Grcar-81,Aptekarev-86,Nevai-86,Lubinsky-87,Driscoll.Toh.ea-SIREV98,Embree.Trefethen-SIREV99,Shen.Strang.ea-01,Andrievskii-04,Kuijlaars-SIREV06,Saff-10,Fischer-11}).
We follow the statements from \cite{Driscoll.Toh.ea-SIREV98}.

The following lemma gives the lower bound of uniform approximation by asymptotic convergence factor $\rho_S$.
\begin{lemma}[Asymptotic convergence factor as non-asymptotic lower bound,  slightly adapted from \cite{Driscoll.Toh.ea-SIREV98}]
    Let $S$ be a compact (possibly not connected) subset of complex planes $\complexes$. 
    Then the following limit exists
    \begin{equation}
        \lim_{k \to \infty} \left( \min_{p \in \polys_k^0} \max_{\lambda \in S} |p(\lambda)| \right)^{\frac{1}{k}} = \rho_S \leq 1, 
    \end{equation}
    where the limiting value $\rho_S$ is called the \textbf{asymptotic convergence factor} of $S$. 
    Moreover, for any $k \in \naturals$, the following inequality holds
    \begin{equation}
        \min_{p \in \polys_k^0} \max_{\lambda \in S} |p(\lambda)| \geq \rho_S^k.
        \label{eq:asym:factor:lb}
    \end{equation}
    Recall $\polys_k^0$ is defined in \cref{eq:def:pk} as the set of polynomials $p$ with degree at most $k$ and $p(0) = 1$.
\end{lemma}
\begin{remark}
    \cite{Schiefermayr-11} shows that the RHS of inequality \eqref{eq:asym:factor:lb} can be improved to $\frac{2 \rho_S^k}{1 + \rho_S^{2k}}$ in the case that $S$ is the union of a finite number of real intervals. We will still use the loose bound \eqref{eq:asym:factor:lb} for simplicity since they gave the same order of bound asymptotically.
\end{remark}

The asymptotic convergence factor of $S$ can be analytically represented by the Green's function of $S$. We formally define the Green's function as follows.
\begin{definition}[Definition of Green's function, borrowed from \cite{Driscoll.Toh.ea-SIREV98}]
    \label{def:green}
    Let $S$ be a compact (possibly not-connected) subset of $\complexes$ with no isolated points. Then the \textbf{Green's function} associated with $S$ (with pole at $\infty$) is the unique $\reals$-valued function defined on $\complexes \backslash S$ such that 
    \begin{enumerate}[(a), leftmargin=*]
        \item $g_{S}$ is harmonic at $\complexes \backslash S$.
        \item $g_{S}(z) \to 0$ as $z \to \partial S$.
        \item $g_{S}(z) - \log|z| \to C$ as $|z| \to \infty$ for some constant $C$.
    \end{enumerate}
\end{definition}
The following result establishes the fundamental connection between Green's function of $S$ and the asymptotic convergence factor of $S$. 
This result is classic and we cite the statement from  \cite{Driscoll.Toh.ea-SIREV98}.
\begin{lemma}[Representation of asymptotic convergence factor via Green's function, slightly adapted from \cite{Driscoll.Toh.ea-SIREV98}]
    Let $S$ be a compact (possibly not-connected) subset of $\complexes$ with no isolated points. 
    Let $g_{S}(z)$ be the Green's function associated with $S$.
    Then the asymptotic convergence factor of $S$ is given by $\rho_S = \exp (-g_{S}(0))$.
\end{lemma}

The proof of \cref{lem:lb:reduce:to:Green} then follows immediately from the above two lemmas.

\subsection{Proof of \cref{lem:lb:green:ub}: Estimating the Green's function}
\label{sec:lb:green:ub}
In this subsection, we will establish \cref{lem:lb:green:ub} on the upper bound of $g_S(0)$ for $S = \bigcup_{i \in [m]} [\mu_i, L_i]$.
Our startpoint is the following classic results due to \citep{Widom-69} on the explicit formula of $g_S$.
\begin{lemma}[Green's function with respect to the union of real intervals, adapted from Section 14 of \cite{Widom-69}]
    \label{lem:lb:widom}
    Let $S = \bigcup_{j=1}^m [\mu_j, L_j] \subset \reals$ for some $0 < \mu_1 < L_1 < \cdots < \mu_m < L_m$. Let $q(z)$ be the polynomial 
    \begin{equation}
        q(z) := \prod_{j \in [m]} (z - \mu_j) (z - L_j).
        \label{eq:lb:q}
    \end{equation}
    Let $h(z)$ be the unique $(m-1)$-degree monic polynomials satisfying
    \begin{align}
        \int_{L_{k}}^{\mu_{k+1}} \frac{h(\zeta)  \diff \zeta }{\sqrt{q(\zeta)}} = 0, \quad k = 1, \ldots, m-1.
    \end{align}
    Then the Green's function for $S$ (with pole at $\infty$) at 0 is given by
    \begin{equation}
        g_{S}(0) = (-1)^{m+1} \int_{0}^{\mu_1} \frac{h(\zeta) \diff \zeta}{\sqrt{q(\zeta)}}.
    \end{equation}
\end{lemma}

\begin{remark}
Although \cref{lem:lb:widom} by \cite{Widom-69} gives an exact formula to compute $g_S(0)$ (up to integration), it is hard to read off the dependency of $g_S(0)$ with respect to the condition numbers of the problem (local condition number $\kappa_i$ and global condition number $\globalcond$). 
Numerous follow-up works have attempted to establish more concrete estimates of the Green's function when $S$ has two or more intervals \citep{Grcar-81,Lubinsky-87,Fischer-92,Peherstorfer-93,Shen.Strang.ea-01,Andrievskii-04,Schiefermayr-08,Schiefermayr-11,Schiefermayr-JCAM11,Alpan.Goncharov.ea-16,Schiefermayr-17}. 
Unfortunately, to the best of our knowledge, the existing estimate is either not sharp or not explicit for our purpose. 
\end{remark}

We will give an explicit upper bound of $g_S(0)$. This estimate is novel to the best of our knowledge.
Starting from \cref{lem:lb:widom}, the proof of \cref{lem:lb:green:ub} relies on the following three technical lemmas. 
The first lemma upper bounds $g_S(0)$ with the product of the roots of $h$ determined in \cref{lem:lb:widom}.
\begin{lemma}
    \label{lem:reduction:to:roots}
    Let $S = \bigcup_{j=1}^m [\mu_j, L_j] \subset \reals$, and assume $\frac{L_j}{\mu_j} \geq 2$ for $j \in [m]$.
    Let $h(z)$ be the unique polynomial determined in \cref{lem:lb:widom}, then $h(z)$ has $m-1$ real roots $r_1, r_2, \ldots, r_{m-1}$ such that $r_k \in [L_k, \mu_{k+1}]$, and the following inequality holds
    \begin{equation}
        g_S(0) \leq \frac{ 7 \prod_{k \in [m-1]} {\frac{r_k}{\mu_{k+1}}} }{\sqrt{\prod_{k \in [m]} \frac{L_k}{\mu_k}}}.
    \end{equation}
\end{lemma}
\begin{remark}
Note that 
\cref{lem:reduction:to:roots} immediately implies a coarse bound of $g_S(0) \leq \frac{7}{\sqrt{\prod_{k \in [m]} \frac{L_k}{\mu_k}}}$ since $r_{k} \leq \mu_{k+1}$.
\end{remark}

The second lemma establishes the following upper bound of $r_k$ by the ratio of two integrals.
\begin{lemma}
    \label{lem:lb:roots:ub}
    Under the same settings of \cref{lem:reduction:to:roots}, the $k$-th root of polynomial $h$ satisfies the following inequality.
    \begin{equation}
        r_k \leq 4 \cdot \frac{\int_{L_{k}}^{\mu_{k+1}} \frac{ \zeta \diff \zeta} { \sqrt{ (\zeta - \mu_{k})(\zeta - L_{k}) (\mu_{k+1} - \zeta) ( L_{k+1}-\zeta)}} }
        {\int_{L_{k}}^{\mu_{k+1}} \frac{ \diff \zeta} { \sqrt{ (\zeta - \mu_{k})(\zeta - L_{k}) (\mu_{k+1} - \zeta) ( L_{k+1}-\zeta)}} } .
    \end{equation}
\end{lemma}

The third lemma upper bounds the ratio encountered in \cref{lem:lb:roots:ub}.
\begin{lemma}
    \label{lem:lb:ratio}
    Assume $\frac{L_j}{\mu_j} \geq 2$ (for any $j \in [m]$), then the following inequality holds for any $k \in [m-1]$,
    \begin{align}
        \frac{
        \int_{L_k}^{\mu_{k+1}} \frac{\zeta \diff \zeta}{\sqrt{(\zeta - \mu_k) (\zeta - L_k) (\mu_{k+1} - \zeta) (L_{k+1} - \zeta)}}
        }
        {
        \int_{L_k}^{\mu_{k+1}} \frac{\diff \zeta}{\sqrt{(\zeta - \mu_k) (\zeta - L_k) (\mu_{k+1} - \zeta) (L_{k+1} - \zeta)}}
        }
        \leq 
        \frac{7 \mu_{k+1}}{  \log \left(16 \frac{\mu_{k+1}}{L_k} \right)}.
    \end{align}
\end{lemma}

The proof of \cref{lem:reduction:to:roots,lem:lb:roots:ub,lem:lb:ratio} are standard yet tedious estimation of definite integrals, which we defer to Appendix \ref{apx:lb:deferred}.

The proof of \cref{lem:lb:green:ub} then follows immediately from \cref{lem:reduction:to:roots,lem:lb:roots:ub,lem:lb:ratio}.
\begin{proof}[Proof of \cref{lem:lb:green:ub}]
    By \cref{lem:reduction:to:roots,lem:lb:roots:ub,lem:lb:ratio},
    \begin{align}
        g_S(0) \leq \frac{ 7 \prod_{k \in [m-1]} {\frac{r_k}{\mu_{k+1}}} }{\sqrt{\prod_{k \in [m]} \frac{L_k}{\mu_k}}}.
    \leq \frac{ 7 \prod_{k \in [m-1]} \frac{28}{\log (16 \frac{\mu_{k+1}}{L_k})} }{\sqrt{\prod_{k \in [m]} \frac{L_k}{\mu_k}}}
    \leq \frac{7}{\sqrt{\prod_{k \in [m]} \frac{L_k}{\mu_k}} \cdot \prod_{k \in [m-1]} \left(0.03 \log (16 \frac{\mu_{k+1}}{L_k}) \right)}.
    \end{align}
\end{proof}

\section{Stochastic BSLS algorithm for quadratic multiscale optimization}\label{sec:stochastic}

In this section we prove \cref{theorem:stochastic}, showing that a variant of $\bsls$, which we call $\bslsstoch$, efficiently solves the stochastic version of a quadratic multiscale optimization problem from \cref{def:multiscale_stochastic}, restated below for convenience.
\begin{definition*}[Restated \cref{def:multiscale_stochastic}]
The \textbf{stochastic quadratic multiscale optimization problem} asks to approximately solve the following problem
\begin{equation}
    \min_{\x \in \reals^d} \E_{(\vec{a},b) \sim \dist} \left[ \frac{1}{2} (\vec{a}^{\top} \x - b)^2 \right],
\end{equation}
where $b = \vec{a}^{\top} \x\opt $ for some fixed, unknown $\x\opt$ and the eigenvalues of the covariance matrix $\E_{\dist}[ \vec{a} \vec{a}^{\top} ]$ can be partitioned into $m$ ``bands'' such that for $i = 1, \dots, \nbands$ and $j = 1, \dots, \mult_i$, each eigenvalue $\lambda_{i_j}$ satisfies $\lambda_{i_j} \in [\mu_i, L_i]$ with $L_i < \mu_{i+1}$ for all $i < m$. 
\end{definition*}
We introduce some additional notation. Let $\maxmult = \max_{i \in [m]} \mult_i$. We let $\mat{H}_i \defeq \diag(\eig_{i_1}, \eig_{i_2}, \dots, \eig_{i_{\mult_{i}}})$ be the diagonal matrix with eigenvalues that lie in the band $\left[\mu_i, L_i \right]$ and $\mat{P}_i$ be an orthonormal matrix such that $\mat{P}_i \mat{\Sigma} \mat{P}_i^{\top} = \mat{H}_i$. Let $\mat{P} = \begin{pmatrix} \mat{P}_1^{\top}, \dots, \mat{P}_m^{\top} \end{pmatrix}^{\top}$ and $\mat{H} = \diag \left( \mat{H}_1, \dots, \mat{H}_m \right)$. We will use the notation that for matrices and vectors, $\mat{M}^{(t)}$ or $\x^{(t)}$ refers to the $\nth{t}$ element of a sequence, while $\mat{M}_{ij}$ or $\x_i$ refers to the index of that matrix or vector. 
We note that this problem can be translated to the multiscale optimization problem formulation as per Def.~\ref{def:multiscale_problem}. Indeed,
\begin{align}
& \E \left[ \frac{1}{2} (\vec{a}^{\top} \x - b)^2 \right] = \frac{1}{2} \x^{\top} \mat{\Sigma} \x -  \left( \mat{\Sigma} \x \opt\right)^{\top} \x + \norm{\mat{\Sigma} \x \opt}_2^2 \\
= & \sum_{i \in [\nbands]} \left( \frac{1}{2} \left( \mat{P}_i \x \right)^{\top} \mat{H}_i \left( \mat{P}_i \x \right) - \left( \mat{H}_i \mat{P}_i \x \opt \right)^{\top} \left( \mat{P}_i\x \right)  + \frac{1}{m} \norm{\mat{\Sigma} \x \opt}_2^2 \right) \\
= & \sum_{i \in [\nbands]}  f_i(\mat{P}_i \x), 
\enspace
\text{ for }
\enspace
f_i(\vec{v}) \defeq \frac{1}{2} \v^{\top} \mat{H}_i \v  - \left( \mat{H}_i \mat{P}_i \x \opt \right)^{\top} \v + \frac{1}{\nbands} \norm{\mat{\Sigma} \x \opt}_2^2,
\end{align}
where each $f_i : \reals^{\mult_i} \to \reals$ satisfies the constraints of Def.~\ref{def:multiscale_problem}.
Therefore the problem from Def.~\ref{def:multiscale_stochastic} can be thought of as a stochastic version of the general problem from Section~\ref{section:nonlinear_problem}.

\subsection{Proof overview of Theorem~\ref{theorem:stochastic}} 
\label{subsec:proofofstoch}
In what follows we prove \cref{theorem:stochastic}, guaranteeing the convergence rate of $\bslsstoch$ in expectation for the stochastic quadratic multiscale optimization problem (\cref{def:multiscale_stochastic}). First, Section~\ref{section:stochastic_inbody} uses our distributional assumptions to establish that if $\bslsstoch_1$ takes $\nsteps_1$ steps then
\begin{equation}
    \E \left[ \norm{\bslsstoch(\x^{(0)}) - \x \opt }_2^2 \right] = (\x^{(0)} - \x \opt)^{\top} \mat{D}^{(\nsteps_1)}(\x^{(0)} - \x\opt), 
\end{equation}
where $\left \{ \mat{D}^{(t)} \right \}_{t = 0}^{\nsteps_1}$ is a sequence of matrices with a clean recurrence relation. Next, Section~\ref{section:inbody_normbound} uses this recurrence relation to bound the spectral norm of each $\mat{D}^{(t)}$. This is where the band structure of the eigenvalues plays a role and the stochasticity poses an obstacle. Finally in Theorem~\ref{thm:helperstochastic} we use the previous work to prove Theorem~\ref{theorem:stochastic} without too much effort since Section~\ref{section:inbody_normbound} guarantees that $\mat{D}^{(\nsteps_1)}$ has sufficiently small spectral norm. Finally in Section~\ref{section:paramsunknown} we extend our analysis to the setting where only $m$, $\mu_1$, $L_m$, and $\prod_{i \in m} \cond_i$ are known. 

To this end, we introduce some notation in addition to the notation from the beginning of \cref{sec:stochastic}. We let $\mat{A} \defeq \frac{1}{\navg} \sum_{i \in \navg} \vec{a}^{(i)} \vec{a}^{(i)\top}$ denote the empirical covariance matrix and $\vec{b} \defeq \frac{1}{\navg} \sum_{i \in \navg} b^{(i)} \vec{a}^{(i)}$ be the empirical approximation to $\mat{\Sigma} \x \opt$. For convenience we introduce $\delta \defeq \distconst \maxmult/\navg$, which (very roughly) corresponds to the noise induced by stochasticity.
\begin{assumption}[Distribution assumptions]
\label{assumption:dist}
For $\vec{a} \sim \dist$ we assume 
\begin{enumerate}[(a),leftmargin=*]
 		\item For any $i,j,k \in [\dim]$ with $i \neq j$ it is the case that $\E[(\mat{P} \vec{a})_i ( \mat{P} \vec{a} )_k^2 (\mat{P} \vec{a})_j] = 0$. (Recall that in this section, $\x_i$ refers to the $\nth{i}$ index of vector $\x$.)
		\item There exists a constant $\distconst$ such that for any $\vec{w} \in \reals^d$, $ \E[ (\vec{w}^{\top} \vec{a})^4] \leq \distconst \E[(\vec{w}^{\top} \vec{a})^2]^2$.
	\end{enumerate}
\end{assumption}

\subsection{Simplifying the stochasticity}
\label{section:stochastic_inbody}
With the notation and assumptions in place, we begin with Lemma~\ref{lemma:secondmoments} which bounds the degree to which stochasticity poses an obstacle. For motivation, first suppose that we had no stochasticity in that instead of approximating $\mat{\Sigma}$ by $\mat{A} = \frac{1}{n} \sum_{i \in n} \vec{a}_i \vec{a}_i^{\top} $, we had access to $\mat{\Sigma}$ itself. Then in $\gdstoch(\x; L)$ we would have instead
\begin{equation}
    \vec{g} =  \mat{\Sigma} ( \x - \x \opt), 
\end{equation}
and
\begin{equation}
     \gdstoch(\x;L) - \x\opt  = \x - \frac{1}{L} \mat{\Sigma} (\x - \x \opt) - \x\opt = \left( \id - \frac{1}{L} \mat{\Sigma} \right) (\x - \x \opt).
\end{equation}
Therefore if $t$ denotes the total number of calls to $\gdstoch$ and $\stepsize^{(t)}$ is the stepsize taken at step $t$ we have
\begin{align}
    \x^{(t)} - \x \opt & = \left( \id - \stepsize^{(t)} \mat{\Sigma} \right) \left( \x^{(t-1)} - \x\opt \right) 
    = \left( \id - \stepsize^{(t)} \mat{\Sigma} \right) \cdots  \left( \id - \stepsize^{(1)} \mat{\Sigma} \right)  \left( \x^{(0)} - \x\opt \right).
\end{align}
Critically, since $\mat{\Sigma}$ commutes with itself we can simplify the above to
\begin{equation}
       \x^{(t)} - \x \opt   = \left( \prod_{i = 1}^{m} \left( \id - \frac{1}{L_i}  \mat{\Sigma} \right)^{\nsteps_i } \right) (\x^{(0)} - \x \opt). 
\end{equation}
Therefore
\begin{align}
     \norm{\x^{(t)} - \x \opt }_2^2  & = \left( \mat{P} (\x^{(0)} - \x \opt)\right)^{\top} \left( \prod_{i = 1}^{m} \left( \id - \frac{1}{L_i}  \mat{H} \right)^{\nsteps_i } \right)^{2}   \mat{P} (\x^{(0)} - \x \opt). 
\end{align}
Then using the fact that for each eigenvalue of $\mat{\Sigma}$ we have $\lambda_{i_j} \in [ \mu_i, L_i ]$ we have,
\begin{equation}
    \norm{\x^{(\nsteps_1)} - \x \opt }_2^2  \leq \sum_{j = 1}^m  \norm{\mat{P}_j^{\top} \left( \x^{(0)} - \x \opt \right) }_2^2  \cdot \prod_{i = 1}^m \left( 1 - \frac{\mu_j}{L_i} \right)^{2 \nsteps_i}  \leq \sum_{j = 1}^m  \norm{\mat{P}_j^{\top} \left( \x^{(0)} - \x \opt \right) }_2^2  \cdot \left( 1 - \frac{1}{\cond_j} \right)^{\nsteps_j}  \globalcond^{2 \sum_{i = j+1}^m \nsteps_i}. 
\end{equation}
Written this way we see that for some constant $C$ we can bound by $\norm{\x^{(\nsteps_1)} - \x \opt}_2^2$ by $\epsilon$ if $\nsteps_j \geq C \cond_j \log( \globalcond) \sum_{i = j+1}^m \nsteps_i $ and $\nsteps_1 \geq C \log(\norm{\x^{(0)} - \x \opt}_2 /\epsilon) \cond_1 \log( \globalcond) \sum_{i = 2}^m \nsteps_i $. This would give an overall query complexity of $\bigo \left( \left( \prod_{i \in m}  \cond_i \right)  \log^m( \globalcond)  \log\left( \norm{\x^{(0)} - \x \opt}_2/ \epsilon \right) \right)$. Instead, in the stochastic case we have
\begin{align}
    \label{eqn:noisy2norm}
    & \norm{ \x^{(t)} - \x \opt }_2^2 
     =  \left( \x^{(0)} - \x\opt \right)^{\top}  \left( \prod_{s = 1}^t \left( \id - \stepsize^{(s)} \mat{A}^{(s)}  \right) \right) \left( \prod_{s = 1}^t \left( \id - \stepsize^{(t-s)} \mat{A}^{(t-s)}  \right) \right)  \left( \x^{(0)} - \x\opt \right)  .
\end{align}
The random instances $\mat{A}^{(t)}$ do not necessarily commute with each other and so simplifying their product is not as simple as the non-stochastic case. The following lemma roughly shows we can replace the above $\mat{A}^{(t)}$ with a perturbation of $\mat{\Sigma}$. 
\begin{lemma}[Understanding Second Moments]
\label{lemma:secondmoments}
Recall $\vec{a} \overset{\textrm{iid}}{\sim} \dist$ and $\mat{A}$. Suppose some matrix $\mat{D}$ commutes with the covariance matrix $\mat{\Sigma}$. Then $\E[\mat{A} \mat{D} \mat{A} ]$ also commutes with $\mat{\Sigma}$ and 
\begin{equation}
    \E[\mat{A} \mat{D} \mat{A} ] \preceq \frac{\navg-1}{\navg} \mat{\Sigma} \mat{D} \mat{\Sigma} + \frac{\distconst}{\navg} \tr \left( \mat{D} \mat{\Sigma} \right) \mat{\Sigma}.
\end{equation}
\end{lemma}
Next we can simplify Eq.~\ref{eqn:noisy2norm} by sequentially conditioning on $\mat{A}^{(1)}, \dots, \mat{A}^{(t-1)}$ and then invoking Lemma~\ref{lemma:secondmoments} for $\mat{A}^{(t)}$. Lemma~\ref{lemma:Ddefinition} does this explicitly and in doing so constructs the aforementioned sequence $\left \{ \mat{D}^{(t)} \right \}_{t=0}^{\nsteps_1}$. After Lemma~\ref{lemma:Ddefinition} the purpose of the remainder of the proof is only to bound the spectral norm of $\mat{D}^{(\nsteps_1)}$.
\begin{lemma}
\label{lemma:Ddefinition}
Recall $\vec{a} \sim \dist$ and the definitions of $\mat{A}$ and $\vec{b}$. Recall we use $\vec{g} = \mat{A} \x - \vec{b}$ in the subroutine $\gdstoch(\x; L)$. Define
\begin{equation}
    \iter{D}{t} \defeq \expect[(\id - \stepsize^{({\nsteps - t + 1})} \iter{A}{\nsteps - t + 1}) \iter{D}{t-1} (\id - \stepsize^{({\nsteps -t +1})} \iter{A}{\nsteps - t + 1})], \qquad \iter{D}{0} \defeq \id.
\end{equation} 
Then if $\nsteps$ denotes the total number of calls to $\gdstoch(\x;L)$ we have
\begin{equation}
    \expect\left[\norm{ \bslsstoch_1(\x^{(0)}) - \x\opt}_2^2\right] = (\x^{(0)} - \x\opt)^{\top} \mat{D}^{(\nsteps)} (\x^{(0)} - \x \opt). 
\end{equation}
\end{lemma}

\begin{proof}[Proof of \cref{lemma:Ddefinition}]
We begin our proof by noting that for $\vec{g}$ as defined in $\gdstoch(\x;L)$ we have 
\begin{equation}
    \vec{g} = \frac{1}{\navg} \sum_{i \in \navg} \vec{a}^{(i)} \vec{a}^{(i)\top} (\x -  \x \opt).
\end{equation}
Thus we have 
\begin{align}
    \gdstoch(\x;L) - \x\opt & = \x - \frac{1}{L} \mat{A} (\x - \x \opt) - \x\opt = \left( \id - \frac{1}{L} \mat{A} \right) (\x - \x \opt).
\end{align}
Therefore if $t$ denotes the total number of calls to $\gdstoch$, $\mat{A}^{(t)}$ denotes the random matrix generated in the $\nth{t}$ call to $\gdstoch$, and $\stepsize^{(t)}$ is the stepsize taken at step $t$ we have
\begin{align}
    \x^{(t)} - \x \opt & = \left( \id - \stepsize^{(t)} \mat{A}^{(t)}  \right) \left( \x^{(t-1)} - \x\opt \right) 
    = \left( \id - \stepsize^{(t)} \mat{A}^{(t)}  \right) \cdots  \left( \id - \stepsize^{(1)} \mat{A}^{(1)} \right)  \left( \x^{(0)} - \x\opt \right).
\end{align}
Therefore, 
\begin{align}
\label{eqn:unravel}
     & \norm{\x^{(\nsteps)} - \x \opt}_2^2 = (\x^{(\nsteps)} - \x \opt)^{\top} (\x^{(\nsteps)} - \x \opt) \\
    = & (\x^{(0)} - \x \opt)^{\top} \left( \id - \stepsize^{(1)} \mat{A}^{(1)} \right)^{\top} \dots \left( \id - \stepsize^{({\nsteps})} \mat{A}^{(\nsteps)} \right)^{\top} \left( \id - \stepsize^{({\nsteps})} \mat{A}^{(\nsteps)} \right) \dots \left( \id - \stepsize^{(1)} \mat{A}^{(1)} \right) (\x^{(0)} - \x \opt).
\end{align}
Let $\mat{D}^{(0)} \defeq \id$ and $\mat{D}^{(t)} \defeq \E \left[ (\id - \stepsize^{(t)} \mat{A}^{(\nsteps- t +1)} ) \mat{D}^{(t-1)} (\id - \stepsize^{(t)} \mat{A}^{(\nsteps - t + 1)} ) \right]$. For short let \\ $\mat{M}^{(t)} \defeq \left( \id - \stepsize^{(1)} \mat{A}^{(1)} \right)  \dots  \left( \id - \stepsize^{(t)} \mat{A}^{(t)} \right)$ and $\mat{M}^{(t)}_{\textrm{rev}} \defeq \left( \id - \stepsize^{(t)} \mat{A}^{(t)} \right)  \dots \left( \id - \stepsize^{(1)} \mat{A}^{(1)} \right)$. Using this notation and using the independence of $\mat{A}^{(1)}, \dots, \mat{A}^{(\nsteps)}$ we have for any $k \geq 1$,
\begin{small}
\begin{align}
    & \E \left[ (\x^{(0)} - \x \opt)^{\top} \mat{M}^{(\nsteps-k + 1)} \mat{D}^{(k-1)} \mat{M}^{(\nsteps-k + 1)}_{\textrm{rev}}  (\x^{(0)} - \x \opt) \right] \\
    = & \E \left[ \E \left[ (\x^{(0)} - \x \opt)^{\top} \mat{M}^{(\nsteps-k + 1)} \mat{D}^{(k-1)} \mat{M}^{(\nsteps-k + 1)}_{\textrm{rev}}  (\x^{(0)} - \x \opt)  \mid \mat{A}^{(1)}, \dots, \mat{A}^{(\nsteps-k)} \right] \right] \\
    = &  \E \left[  (\x^{(0)} - \x \opt)^{\top} \mat{M}^{\nsteps-k}   \E \left[\left( \id - \stepsize^{(t)} \mat{A}^{(\nsteps-k + 1)} \right) \mat{D}^{(k-1)}  \left( \id - \stepsize^{(t)} \mat{A}^{(\nsteps-k + 1)} \right)  \right]  \mat{M}_{\textrm{rev}}^{\nsteps-k} (\x^{(0)} - \x \opt)   \right] \\
    = & \E \left[  (\x^{(0)} - \x \opt)^{\top} \mat{M}^{\nsteps-k}   \mat{D}^{(k)} \mat{M}_{\textrm{rev}}^{\nsteps-k} (\x^{(0)} - \x \opt)   \right]
\end{align}
\end{small}
Therefore using Eq.~\ref{eqn:unravel} in the first equality and the above recursion in the second equality we have
\begin{small}
\begin{equation}
    \E \left[ \norm{\x^{\nsteps} - \x \opt}_2^2 \right] = \E \left[  (\x^{(0)} - \x \opt)^{\top} \mat{M}^{\nsteps} \mat{D}^{(0)} \mat{M}_{\textrm{rev}}^{\nsteps} (\x^{(0)} - \x \opt)^{\top} \right]  = (\x^{(0)} - \x \opt)^{\top} \mat{D}^{(\nsteps)} (\x^{(0)} - \x\opt).
\end{equation}
\end{small}
\end{proof}

\subsection{Bounding the spectral norm of $\mat{D}^{(t)}$}
\label{section:inbody_normbound}
This section is where we address the difficulty posed by stochasticity. From \cref{lemma:Ddefinition} we see that it suffices to bound the spectral norm of $\mat{D}^{\nsteps_1}$. To that end, in the following lemma we construct a clean recursive form to analyze the sequence $\left \{ \mat{D}^{(t)} \right \}$.

\begin{lemma}
\label{lemma:Drecursiongeneral}
For $\iter{D}{t}$ as defined in Lemma~\ref{lemma:Ddefinition} we have that $\iter{D}{t}$ commutes with $\mat{\Sigma}$. Moreover we have the following spectral upperbound, 
\begin{equation}
    \iter{D}{t}  \preceq (\id - \stepsize^{(t)} \mat{\Sigma})^2 \iter{D}{t-1} + \frac{\distconst \stepsize^{(t)2}}{\navg}  \tr(\iter{D}{t-1} \mat{\Sigma}) \mat{\Sigma}.
\end{equation}
\end{lemma}

\begin{proof}[Proof of Lemma~\ref{lemma:Drecursiongeneral}]
Recalling the definition of $\iter{D}{t}$ from Lemma~\ref{lemma:Ddefinition},
\begin{equation}
    \iter{D}{t} \defeq \expect[(\id - \stepsize^{(t)} \iter{A}{t}) \iter{D}{t-1} (\id - \stepsize^{(t)} \iter{A}{t})] \qquad \iter{D}{0} \defeq \id.
\end{equation}    
we have
\begin{equation}
    \iter{D}{t} = \left( (\id - \stepsize^{(t)} \mat{\Sigma}) \iter{D}{t-1} (\id - \stepsize^{(t)} \mat{\Sigma})   \right) + \stepsize^{(t)2} \left(  \E[\iter{A}{t} \iter{D}{t-1} \iter{A}{t} ]  - \mat{\Sigma} \iter{D}{t-1} \mat{\Sigma} \right).
\end{equation}

Lemma~\ref{lemma:secondmoments} allows us to bound $\E[\iter{A}{t} \mat{D}^{(t-1)} \iter{A}{t} ]$ from above. Using this we have,
\begin{equation}
    \iter{D}{t} \preceq \left( (\id - \stepsize^{(t)} \mat{\Sigma}) \iter{D}{t-1} (\id - \stepsize^{(t)} \mat{\Sigma})   \right) + \frac{\stepsize^{(t) 2}}{\navg} \left( \distconst \tr(\iter{D}{t-1} \mat{\Sigma}) \mat{\Sigma} - \mat{\Sigma} \iter{D}{t-1} \mat{\Sigma} \right).
\end{equation}
By Lemma~\ref{lemma:secondmoments} $\mat{\Sigma}$ and $\iter{D}{t-1}$ commute and thus we have more simply, 
\begin{align}
    \iter{D}{t} & \preceq  (\id - \stepsize^{(t)} \mat{\Sigma})^2 \iter{D}{t-1} + \frac{\stepsize^{(t)2}}{\navg} \left( \distconst \tr(\iter{D}{t-1} \mat{\Sigma}) \mat{\Sigma} - \iter{D}{t-1} \mat{\Sigma}^2 \right) \\
    & \preceq (\id - \stepsize^{(t)} \mat{\Sigma})^2 \iter{D}{t-1} + \frac{\distconst \stepsize^{(t)2}}{\navg}  \tr(\iter{D}{t-1} \mat{\Sigma}) \mat{\Sigma}. \tag{ $\iter{D}{t-1}\mat{\Sigma}^2 $ is PSD}
\end{align}

\end{proof}

\begin{remark}
\label{remark:Dtor}
Recall that $\mat{P} \mat{\Sigma} \mat{P}^{\top} = \mat{H}$ and $\mat{H} = \diag(\mat{H}_1, \dots, \mat{H}_{\nbands})$, 
where $\mat{H}_i \defeq \diag( \eig_{i,1}, \dots, \eig_{i,\mult_i} )$ represents the $\nth{i}$ eigenvalue band. 
By Lemma~\ref{lemma:Drecursiongeneral} each $\iter{D}{t}$ commutes with $\mat{\Sigma}$. Therefore if $\iter{\tilde{D}}{t} = \mat{P} \mat{D} \mat{P}^{\top}$ then $\iter{\tilde{D}}{t}$ is diagonal and 
\begin{equation}
    \label{eqn:Drecursiongeneraldiag}
    \iter{\tilde{D}}{t}  \leqentry (\id - \stepsize^{(t)} \mat{H})^2 \iter{\tilde{D}}{t-1} + \frac{\distconst \stepsize^{(t)2}}{\navg}  \tr(\iter{\tilde{D}}{t-1} \mat{H}) \mat{H}.
\end{equation}
Thus the structure on $\mat{H}$ induces structure on matrix $\iter{\tilde{D}}{t}$, 
\begin{equation}
    \iter{\tilde{D}}{t} = \diag( \iter{\tilde{D}}{t}_1, \dots, \iter{\tilde{D}}{t}_{\nbands}).
\end{equation}
\end{remark}

\begin{lemma}
\label{lemma:residuals}
Let 
\(
    g_i(\stepsize) \defeq \max \left \{ (1 - \stepsize \mu_i )^2, (1 - \stepsize L_i)^2 \right \}. 
\)
Define the ``update'' matrix from stepsize $\stepsize^{(t)}$ as 
 \begin{equation}
    \label{eqn:Udef}
    \left( \mat{U}_{\stepsize^{(t)}}\right)_{ij} \defeq \begin{cases}
    g_i(\stepsize) & \textrm{ if } i = j, \\
    \delta \stepsize^2 L_i L_j & \textrm{ else.}
    \end{cases}
\end{equation}
Define the following vector to represent the maximum entry of each $\iter{\tilde{D}}{t}_i$,
\begin{equation}
    \label{def:res}
    \resvo^{(t)}_i \defeq \linf{\iter{\tilde{D}}{t}_i}, 
\end{equation}
(and note that since $\iter{\tilde{D}}{0} = \id$ then $\resvo^{(0)} = \vec{1}$). Then for any $i = 1, \dots, \nbands$ we have for $t \geq 1$,
\begin{equation}
    \resvo^{(t)} \leq \mat{U}_{\stepsize^{(t)}} \resvo^{(t-1)}.
\end{equation}
\end{lemma}

\begin{proof}[Proof of \cref{lemma:residuals}]
From Lemma~\ref{lemma:Drecursiongeneral} we can bound the growth of $\resvo^{(t)}$. 
Let $\sigma(k, \ell) \in \mathbb{Z}$ denote the index corresponding to the $\nth{\ell}$ smallest eigenvalue of the $\nth{k}$ band so that $\mat{H}_{\sigma(k,\ell)} = \eig_{k, \ell}$. Letting 
\begin{equation}
    g_i(\stepsize) \defeq \max \left \{ (1 - \stepsize \mu_i )^2, (1 - \stepsize L_i)^2 \right \},
\end{equation}
we have (using Eq.~\ref{eqn:Drecursiongeneraldiag}), 
\begin{align}
    \resvo^{(t+1)}_i & = \max \limits_{j = 1, \dots, \mult_i} \left \{ ( 1 - \stepsize^{(t)} \eig_{i,j})^2 \resvo_i^{(t)} +  \frac{\distconst \stepsize^{(t)2}}{\navg} \sum_{k = 1}^{\nbands} \sum_{\ell = 1}^{\mult_k} \iter{D}{t}_{\sigma(k,\ell)} \mat{H}_{\sigma(k,\ell)} \eig_{i,j}\right \}\\
    & \leq \max \limits_{j = 1, \dots, \mult_i} \left \{  (1 - \stepsize^{(t)} \eig_{i,j} )^2 \right \} \resvo^{(t)}_i + \frac{\distconst \stepsize^{(t)2}}{\navg} \left( \sum_{k= 1}^{\nbands} \mult_k \resvo^{(t)}_k L_k \right) L_i \\
    & \leq \max \left \{ (1 - \stepsize^{(t)} \mu_i )^2, (1 - \stepsize^{(t)} L_i)^2 \right \}  \resvo^{(t)}_i + \frac{\distconst \stepsize^{(t)2}}{\navg} \left( \sum_{k= 1}^{\nbands} \mult_k \resvo^{(t)}_k L_k \right) L_i \\
    & \leq  g_i(\stepsize^{(t)}) \resvo^{(t)}_i + \frac{\distconst \maxmult \stepsize^{(t)2}}{\navg} \left( \sum_{k= 1}^{\nbands} \resvo^{(t)}_k L_k \right) L_i. 
\end{align}
Inspecting the definition of $\mat{U}_{\stepsize^{(t)}}$ finishes the proof.
\end{proof}

\begin{remark}
\label{rmk:bslsresexplain}
For simplicity, let $\mat{U}_{i_t}$ denote $\mat{U}_{\stepsize^{(t)}}$ where $i_t$ is the index belonging to $\left \{ 1, \dots, m \right \}$ such that the stepsize in the $\nth{t}$ step corresponds to the $\nth{i}$ eigenvalue band; that is: $\stepsize^{(t)} = 1/L_{i_t} $. Recall that Lemma~\ref{lemma:residuals} guarantees that for $\resvo_i^{(t)} \defeq \linf{\mat{\tilde{D}}_i^{(t)}}$ 
\begin{equation}
    \resvo^{(t)} \leq \mat{U}_{i_t} \resvo^{(t-1)}.
\end{equation}
We ultimately want to bound $\linf{\resvo^{\nsteps_1}}$, however the evolution of $\left \{ \resvo^{(t)} \right \}_{t=0}^{\nsteps_1}$ is difficult to track exactly. Instead we can analyze the evolution of $\left \{ \resv^{(t)}  \right \}_{t=0}^{\nsteps}$ where
\begin{equation}
    \resv^{(t)} \defeq \mat{U}_{i_t} \resv^{(t-1)} \qquad \resv^{(0)} = \resvo^{(0)}.
\end{equation}
Taking this another step further, for convenience we define
\begin{equation}
    \left(\mat{V}_i\right)_{jk} \defeq \begin{cases}
    \smallgrowth[\nsteps_{i+1}]{\resv^{(0)}}, \textrm{ if } j = k \textrm{ and } j < i, \\
    \gamma_i, \textrm{ if } j = k \textrm{ and } j = i, \\
    \frac{L_j}{L_i} \smallgrowth{\resv^{(0)}}^{\nsteps_{i+1}+1}, \textrm{ if } j > i \textrm{ and } $k = i$, \\
    0, \textrm{ else.}
    \end{cases}
\end{equation}
and
\begin{equation}
    \left(\mat{W}_i\right)_{jk} \defeq \begin{cases}
    1, \textrm{ if } j = k=i, \\
    \left(\mat{V}_i\right)_{jk}, \textrm{ else. }
    \end{cases}
\end{equation}
Suppose we now re-define $\left \{ \resv^{(t)} \right \}_{t=0}^{\nsteps}$ where either
\begin{equation}
    \label{eqn:Vdef}
    \resv^{(t)} \defeq \max \left \{ \mat{U}_{i_t} \resv^{(t-1)} , \mat{V}_{i_t} \resv^{(t-1)} \right \} \qquad \resv^{(0)} = \resvo^{(0)},
\end{equation}
or
\begin{equation}
    \label{eqn:Wdef}
    \resv^{(t)} \defeq \max \left \{ \mat{U}_{i_t} \resv^{(t-1)} , \mat{W}_{i_t} \resv^{(t-1)} \right \} \qquad \resv^{(0)} = \resvo^{(0)}.
\end{equation}
Then since 
\begin{equation}
    \resvo^{(t)} \leq \resv^{(t)}
\end{equation}
we can analyze $\left \{ \resv^{(t)} \right \}_{t=0}^{\nsteps_1}$ and bound $\linf{\resv^{(\nsteps_1)}}$ to get a bound on $\linf{\resvo^{\nsteps_1}}$. This is convenient because the evolution of $\resv^{(t)}$ is easier to track while capturing the critical behavior of the evolution of $\resvo^{(t)}$. Towards this end, we introduce Algorithm~\ref{alg:bslsres} which we call as $\bslsres$ ($\normalfont{\texttt{Res}}$ for ``residuals'') and which mirrors the structure of $\bslsstoch$. Lemma~\ref{lemma:maininduction}, which bounds $\linf{\resv^{(\nsteps_1)}}$ from Algorithm~\ref{alg:bslsres}, is the heart of the proof of Theorem~\ref{theorem:stochastic}.
\end{remark}

\begin{algorithm}
	\caption{BSLS Residuals [For analysis of the stochastic variant]}
	\begin{algorithmic}[1]
	    \label{alg:bslsres}
		\REQUIRE $\bslsres_i$ $(\resv)$
		\FOR{$t = 0, 1, \ldots, T_i-1$}
		    \IF {$i < \nbands$}
		    	\STATE $\tilde{\resv}^{(t)} \gets \bslsres_{i+1}(\resv^{(t)})$
		    \ELSE
		        \STATE $\tilde{\resv}^{(t)} \gets \resv^{(t)}$
		    \ENDIF
		    \IF {$t  > \left \lceil \threshold{i} \right \rceil $ and $i \geq 2$}
		    	\STATE$\resv^{(t+1)} \gets \max \left \{ \mat{U}_i \tilde{\resv}^{(t)}, \mat{W}_i \tilde{\resv}^{(t)} \right \}$ (for $\mat{W}_i$ as defined in Eq.~\ref{eqn:Wdef}) 
		    \ELSE 
		    	\STATE $\resv^{(t+1)} \gets \max \left \{ \mat{U}_i \tilde{\resv}^{(t)}, \mat{V}_i \tilde{\resv}^{(t)} \right \}$ (for $\mat{U}_{i}$ and $\mat{V}_i$ as defined in Eqs.~\ref{eqn:Udef}, \ref{eqn:Vdef} respectively ) 
		    \ENDIF
		\ENDFOR
		\RETURN $\bslsres_{i+1}(\resv^{(T_i)})$ %
	\end{algorithmic}
\end{algorithm}

\begin{lemma}
\label{lemma:maininduction}
For any $i = 1, \dots, \nbands$ define $ \gamma_i \defeq 1 - (2 \cond_i)^{-1}.$ Fix some $i \in \{2, \dots, m \}$ and let $T_i = \left \lceil 8 \cond_i \log( \globalcond) \right \rceil$. Let $\nsteps_i = \prod_{j = i}^m (2 T_i + 1)$ and $T_{\textrm{max}} = \max_i T_i$. Define
\begin{align}
    & \basegrowth{\resv} \defeq \inf \left \{t \,\middle|\,  \resv_k \leq t \max \left \{ \frac{L_j}{L_k}, \frac{L_k}{L_j}\right \} \resv_j \textrm{ for any $j,k \in \left[ m \right]$ } \right \} \\
    & \smallgrowth{\resv} \defeq 1 + 3 \delta m  \cdot \basegrowth{\resv}\\
    & \totalgrowth{\resv} \defeq \smallgrowth{\resv}^{T_{\textrm{max}}(\nsteps_1 + 2)^2 + 1} \cdot \max_{\ell \in \left[ m \right]} \left \{ \frac{1}{\gamma_{\ell}^2} \right \}
\end{align} 
Suppose that 
\begin{equation}
    \label{eqn:betacond}
    \basegrowth{\resv} \totalgrowth[m-i+1]{\resv} \leq  \min \left \{ \frac{\globalcond}{\smallgrowth{\resv}^{\nsteps_1}}, \frac{1}{144 \nsteps_1 T_{\textrm{max}} \delta m}, \frac{1}{2 \max_{\ell \in \left[ m \right] } \cond_{\ell}} \frac{1}{6(\nsteps_1 + 1) \delta m}  \right \}.
\end{equation}
Further suppose that
\begin{equation}
    \navg \geq \distconst m^2 \maxmult \left( \prod_{i \in [m]} \cond_i \right) \left( \max_{i \in [m] } \cond_i \right) \log \left( \frac{9 \norm{\x^{(0)} - \x \opt}_2^2}{\varepsilon^2} \right) \log^m(\globalcond).
\end{equation}
 Then if $\tilde{\resv} = \bslsres_i(\resv)$ we have that for all $j \geq i$,
\begin{equation}
    \tilde{\resv}_j \leq  \frac{L_{i-1}}{L_j} \resv_{i-1},
\end{equation}
and for all $j < i$,
\begin{equation}
    \tilde{\resv}_j \leq \smallgrowth[{\nsteps_i}]{\resv} \cdot \resv_j.
\end{equation}
\end{lemma}

The proof of Lemma~\ref{lemma:maininduction} requires careful and somewhat tedious analysis of the evolution of $\resv^{(t)}$. The difficulty lies in controlling the error induced by stochasticity. For a full proof see Appendix~\ref{appendarxiv:maininductionproof}. With Lemma~\ref{lemma:maininduction}, we can now easily bound the convergence of $\resv^{(t)}$ which then allows us to bound the spectral norm of $\tilde{\mat{D}}^{(t)}$. 

\begin{lemma}
\label{lemma:bslsres}
Suppose that $m \leq \log(\globalcond)/3$ and 
\begin{equation}
    \navg \geq \distconst m^2 \maxmult \left( \prod_{i \in [m]} \cond_i \right) \left( \max_{i \in [m] } \cond_i \right) \log \left(9 \norm{\x^{(0)} - \x \opt}_2^2/\varepsilon \right) \log^m(\globalcond). 
\end{equation}
For $i = 2, \dots, m$ let $T_i$ be as in Lemma~\ref{lemma:maininduction} and let  $T_1 = \left \lceil 2  \cond_1 \log \left( \frac{9 \norm{\x^{(0)} - \x \opt}_2^2 }{\epsilon} \right) \right \rceil$. Then if $\tilde{\resv} = \bslsres_1(\vec{1})$ we have for all $i$
\begin{equation}
    \tilde{\resv}_i \leq \frac{\epsilon}{ \norm{\x^{(0)} - \x\opt}_2^2}.
\end{equation}
\end{lemma}

\begin{proof}[Proof of \cref{lemma:bslsres}]
First we show that Eq.~\ref{eqn:betacond} holds for $\resv^{(0)} = \vec{1}$. We bound $\basegrowth{\vec{1}}, \smallgrowth{\vec{1}},$ and $\totalgrowth{\vec{1}}$. Using that $ \max_{\ell \in [ m ]}  \left \{ \frac{1}{\gamma_{\ell}^2} \right \} \leq 4$ to bound $\totalgrowth{\vec{1}}$ we have
\begin{align}
    & \basegrowth{\vec{1}} \leq 1 \\
    & \smallgrowth{\vec{1}} = 1 + 3 \delta m \left( \max_{i \leq m-1} \frac{L_i}{L_{i+1}} \right) \leq 1 + 3 \delta m \\
    & \totalgrowth{\vec{1}} \leq 4 \left( 1 + 3 \delta m \right)^{T_{\textrm{max}}( \nsteps_1 + 2)^2 + 1}.
\end{align}
To show Eq.~\ref{eqn:betacond} holds we must show 
\begin{equation}
    \label{eqn:messybetacond}
     4^m (1 + 3 \delta m)^{m \left( T_{\textrm{max}}( \nsteps_1 + 2)^2 + 1 \right) }  \leq \min \left \{ \frac{\globalcond}{(1 + 3 \delta m)^{\nsteps_1}}, \frac{1}{144 \nsteps_1 T_{\textrm{max}} \delta m}, \frac{1}{2 \max_{\ell \in [m]} \left \{  \cond_{\ell} \right \} } \frac{1}{6(\nsteps_1 + 1) \delta m} \right \}. 
\end{equation} 
Since 
\begin{equation}
    \navg \geq \distconst m^2 \maxmult \left( \prod_{i \in [m]} \cond_i \right) \left( \max_{i \in [m] } \cond_i \right) \log \left( \frac{9 \norm{\x^{(0)} - \x \opt}_2^2}{\varepsilon^2} \right) \log^m(\globalcond), 
\end{equation}
then recalling that $\delta = \distconst \maxmult/\navg$ and noting that $\nsteps_1 \leq 2 \cond_1 \log \left( 9 \norm{\x^{(0)} - \x \opt}_2^2/\varepsilon^2 \right) \prod_{i = 2}^m 8 \cond_i \log(\globalcond) $ we have
\begin{align}
    \delta \leq \min \bigg \{ & \frac{1}{3m} \frac{1}{T_{\textrm{max}}(\nsteps_1 + 2)^2 + 1}, \\
    & \frac{1}{6m(4^m)} \left( 48 m T_{\textrm{max}}^2 (\nsteps_1 + 2)^3\right)^{-1/2},\\
    & \frac{1}{6m(4^m)} \left( \max_{\ell \in [m]} \left \{ \cond_{\ell} \right \} \cdot m T_{\textrm{max}}^2 (\nsteps_1 + 2)^3 \right)^{-1/2}  \bigg \}. 
\end{align}
This guarantees that Eq.~\ref{eqn:messybetacond} holds; to see the details please refer to Appendix~\ref{append:messydetailsstochastic}.
Next we show
\begin{equation}
    T_1 \geq \log_{(1/\gamma_1)}\left( \frac{9 \norm{\x^{(0)} - \x \opt}_2^2 }{\epsilon} \right).
\end{equation}
Indeed, using that $\log(1/(1-x)) \geq x$ and $\gamma_1 = 1 - \frac{1}{2 C_1 \cond_1}$, 
\begin{align}
    \log_{(1/\gamma_1)}\left(  \frac{9 \norm{\x^{(0)} - \x \opt}_2^2 }{\epsilon} \right) & = \frac{\log \left( \frac{ 9 \norm{\x^{(0)} - \x \opt}_2^2 }{\epsilon} \right)}{\log \left( 1/\gamma_1 \right)}  \leq 2 C_1 \cond_1  \log \left(  \frac{ 9 \norm{\x^{(0)} - \x \opt}_2^2 }{\epsilon} \right).
\end{align}
Next recall Claim~\ref{claim:induct} from the proof of Lemma~\ref{lemma:maininduction} which shows that
\begin{equation}
    \resv_1^{(T_1)} \leq \gamma_1^{T_1} \resv_1^{(0)}.
\end{equation}
The proof holds in this case as well and so we have that 
\begin{equation}
    \resv_1^{(T_1)} \leq \gamma_1^{T_1} \resv_1^{(0)} \leq \frac{\epsilon}{ 9  \norm{\x^{(0)} - \x \opt}_2^2}. 
\end{equation}
Note that $\bslsres_1(\vec{1}) = \bslsres_2(\resv^{(T_1)})$. So we have that if $\tilde{\resv} =\bslsres_1(\vec{1}) $ then
\begin{equation}
    \tilde{\resv}_1 \leq \frac{\epsilon}{ \norm{\x^{(0)} - \x \opt}_2^2}. 
\end{equation}
Finally we use that for any $j \geq 2$,
\begin{equation}
    \tilde{\resv}_j \leq \frac{L_1}{L_j} \resv_1^{(T_1)} \leq \frac{\epsilon}{\norm{\x^{(0)} - \x \opt}_2^2}. 
\end{equation}
\end{proof}

Finally we can combine the previous results to give the proof of Theorem~\ref{theorem:stochastic}.
\begin{theorem}
\label{thm:helperstochastic}
Suppose Assumption \ref{assumption:dist} holds. For $i = 2, \dots, m$ let $T_i = \left \lceil 8 \cond_i \log( \globalcond) \right \rceil$ and let $ T_1 = \left \lceil 2 \cond_1 \log \left( 9 \norm{\x^{(0)} - \x \opt}_2^2 /\epsilon \right) \right \rceil$. Let $\maxmult$ denote the maximum number of eigenvalues lying in any single band region $[\mu_i, L_i]$ and further suppose
\begin{equation}
    \navg \geq \distconst m^2 \maxmult \left( \prod_{i \in [m]} T_i \right) \left( \max_{i \in [m] } T_i \right). 
\end{equation}
Then 
\begin{equation}
    \E \left[ \norm{(\bslsstoch_1(\x^{(0)}) - \x \opt)}_2^2 \right] \leq \epsilon.
\end{equation}
Therefore since $\bslsstoch_1$ requires $\bigo \left( \navg 2^m \prod_{i \in [m]} T_i \right)$ queries of $(\vec{a}^{(i)}, b^{(i)})$ we conclude that $\bslsstoch_1$ can return in expectation a $\epsilon$-optimal solution with 
\begin{equation}
    \bigo \left( \maxmult \distconst \left( \prod_{i \in [m]} \kappa_i^2 \right) \left( \max_{i \in [m]}\left \{  \cond_i \right \}\right)   \log^{2 \nbands }(\globalcond) \log^2 \left(  L_m \norm{\x^{(0)} - \x \opt}_2 / \epsilon \right)  \right)
\end{equation}
first order queries.
\end{theorem}

\begin{proof} [Proof of Theorem~\ref{thm:helperstochastic}]
By Lemma~\ref{lemma:secondmoments} and Lemma~\ref{lemma:Ddefinition} we have that for $\nsteps_1 = \prod_{i = 1}^{\nbands} (2 T_i + 1)$ and for
\begin{equation}
    \iter{D}{t} \defeq \expect[(\id - \stepsize^{({\nsteps_1 - t + 1})} \iter{A}{\nsteps_1 - t + 1}) \iter{D}{t-1} (\id - \stepsize^{({\nsteps_1 -t +1})} \iter{A}{\nsteps_1 - t + 1})] \qquad \iter{D}{0} \defeq \id,
\end{equation} 
then
\begin{equation}
    \label{eqn:connectlater}
    \E \left[\norm{\bslsstoch_1(\x^{(0)}) - \x \opt}_2^2 \right] = (\x^{(0)} - \x \opt)^{\top} \mat{D}^{(\nsteps_1)} (\x^{(0)} - \x \opt).
\end{equation}
Recall that $\tilde{\mat{D}} \defeq \mat{P} \mat{D} \mat{P}^{\top}$. As in Lemma~\ref{lemma:residuals} we define the following vector to represent the maximum entry of each $\iter{\tilde{D}}{t}_i$,
\begin{equation}
    \resvo^{(t)}_i \defeq \linf{\iter{\tilde{D}}{t}_i} \qquad \resvo^{(0)} = \vec{1}.
\end{equation}
By Lemma~\ref{lemma:residuals} we have $\resvo^{(0)} = \vec{1}$ and 
\begin{equation}
    \resvo^{(t)} \leq \mat{U}_{\stepsize^{(t)}} \resvo^{(t-1)}. 
\end{equation}
By Remark~\ref{rmk:bslsresexplain} it suffices to argue about the convergence of $\bslsres(\vec{1})$. Applying Lemma~\ref{lemma:bslsres} with error $\epsilon/L_m$ we have that
\begin{equation}
    \linf{\bslsres_1(\vec{1})} \leq \frac{\epsilon}{ L_m \norm{\x^{(0)} - \x \opt}_2^2}.
\end{equation}
From Remark~\ref{rmk:bslsresexplain} this implies that if $\resvo^{(\nsteps_1)}$ is our residuals vector at the end of $\bslsstoch_1$ we have, 
\begin{equation}
    \linf{\resvo^{(\nsteps_1)}} \leq \frac{ \epsilon}{L_m \norm{\x^{(0)} - \x \opt}_2^2}.
\end{equation}
Therefore using that $\mat{P}$ is an orthonormal matrix,
\begin{equation}
    \linf{\mat{D}^{(\nsteps_1)}}  = \linf{\mat{P}^{\top} \tilde{\mat{D}}^{(\nsteps_1)} \mat{P} }\leq \linf{\tilde{\mat{D}}^{(\nsteps_1)} }  \leq \linf{\resvo^{(\nsteps_1)}} \leq \frac{ \epsilon}{L_m \norm{\x^{(0)} - \x \opt}_2^2}.
\end{equation}
Thus by Eq.~\ref{eqn:connectlater},
\begin{align}
    \E \left[\norm{\bslsstoch_1(\x^{(0)}) - \x \opt}_2^2 \right] & = (\x^{(0)} - \x \opt)^{\top} \mat{D}^{(\nsteps_1)} (\x^{(0)} - \x \opt) \leq  \epsilon/L_m.
\end{align}
Note that
\begin{equation}
    \E_{(\vec{a},b) \sim \dist} \left[ \frac{1}{2} \left( \vec{a}^{\top} \x - b \right)^2 \right] = \frac{1}{2} (\x - \x \opt)^{\top} \mat{\Sigma} (\x - \x \opt) \leq L_m \E \left[ \norm{\x - \x \opt}_2^2 \right] \leq \epsilon.
\end{equation}
This concludes the proof.
\end{proof}

\subsection{Setting where only $m$, $\mu_1$, $L_m$, and $\prod_{i \in m} \cond_i$ are known}
\label{section:paramsunknown}
To extend to this setting we have the following proposition, similar to \cref{prop:search}, 

 \begin{proposition}
 \label{prop:searchstoch}
      Let $\pi_{\kappa} = \prod_{i \in m} \cond_i$. Suppose with failure probability at most $p$, we can evaluate $f$ up to a multiplicative constant factor $C$ with $\tilde{T}(p, C)$ oracle queries; that is we can construct some $\hat{f}$ such that for any $\x$,
      \begin{equation}
          f(\x)/C \leq \hat{f}(\x) \leq C f(\x).
      \end{equation}
      A randomized algorithm $A$ which, in expectation, solves the stochastic multiscale optimization problem in Definition \ref{def:multiscale_stochastic} to sub-optimality $\eps$ with $T(\pi_{\kappa}, \globalcond, m,\eps)$ gradient queries when the parameters ($\mu_i, L_i)$ are known, can be used along with the approximate function evaluation to solve the stochastic multiscale optimization problem with failure probability at most $C^2 \epsilon + p$ with $\tilde{T}(p, C) \cdot T(\pi_{\kappa}2^{5m}, \globalcond, m,\eps^2)\cdot O(\log^m(\globalcond))$ oracle queries when only $m$, $\mu_{1}$, $L_m$ and $ \pi_{\kappa}$ are known.
\end{proposition}

To apply this proposition to $\bslsstoch$ requires guaranteeing first that $f(\bslsstoch(\x^{(0)})) < \epsilon$ with good probability (for this we use Markov's inequality since we have bounded $\E f( \bslsstoch (\x^{(0)}))$) and second that we can estimate $f(\x)$ up to a multiplicative constant factor (this is why we must include the assumption from \cref{eqn:concentrationassumption}). This results in the following corollary,

\begin{corollary}
\label{cor:bslssearch}
   Assume the setting from \cref{theorem:stochastic} except that only $m$, $\mu_1$, $L_m$ and $\pi_k$ are known. Suppose that $\dist$ is such that for $\vec{a} \sim \dist$ there exists some $K$ where
    \begin{equation}
        \norm{\mat{\Sigma}^{-1/2} \vec{a}}_2 \leq K  \left(  \E \norm{\mat{\Sigma}^{-1/2} \vec{a}}_2^2 \right)^{1/2}.
   \end{equation}
    Then with failure probability at most $\delta$, $\bslsstoch$ can be used to solve the stochastic quadratic multiscale optimization problem from Definition~\ref{def:multiscale_stochastic} with $\tildeo(d)$ space and an extra multiplicative factor of 
$\bigo \left( K^2 d \log \frac{4d}{\delta}  \left(1 + \sqrt{\varepsilon/\delta}  \right) \right) $ queries of $(\vec{a}, b) \sim \dist$. 
\end{corollary}
The proofs of \cref{prop:searchstoch} and \cref{cor:bslssearch} are in \cref{subsec:search}.

\section{Extended results regarding the multiscale optimization problem}
\label{apx:misc}

\subsection{Black-box reduction from unknown $(\mu_i, L_i)$ to known $(\mu_i, L_i)$}\label{subsec:search}

In this subsection, we show that the assumption in $\bsls$ that the  $\mu_i, L_i$ parameters are known is essentially without loss of generality, since we can reduce from the case where these are unknown to the case where they are known without changing the asymptotic complexity. The reduction is black-box and does not utilize any special properties of our algorithm.

\begin{proposition}[Restated \cref{prop:search}]
    Let $\pi_{\kappa} = \prod_{i \in m} \cond_i$. Suppose with failure probability at most $p$, we can evaluate $f$ up to a multiplicative constant factor $C$ with $\tilde{T}(p, C)$ oracle queries; that is we can construct some $\hat{f}$ such that for any $\x$,
    \begin{equation}
        f(\x)/C \leq \hat{f}(\x) \leq C f(\x).
    \end{equation}
    A randomized algorithm $A$ which, in expectation, solves the stochastic multiscale optimization problem in Definition \ref{def:multiscale_stochastic} to sub-optimality $\eps$ with $T(\pi_{\kappa}, \globalcond, m,\eps)$ gradient queries when the parameters ($\mu_i, L_i)$ are known, can be used along with the approximate function evaluation to solve the stochastic multiscale optimization problem with failure probability at most $C^2 \epsilon + p$ with $\tilde{T}(p, C) \cdot T(\pi_{\kappa}2^{5m}, \globalcond, m,\eps^2)\cdot O(\log^m(\globalcond))$ oracle queries when only $m$, $\mu_{1}$, $L_m$ and $ \pi_{\kappa}$ are known.
\end{proposition}

\begin{proof}

Let $\{(\mu_i,L_i), i \in [m]\}$ be the original parameters of the multiscale optimization problem. The proof relies on a simple brute force search over these parameters over a suitable grid. In the first step, we will do a brute force search for the parameters $\kappa_i \; \forall  \; i \in [m]$. Then, we do a brute force search over the the parameters $(\mu_i,L_i) \; \forall \; i\in [m]$ and run the algorithm with every instance of these parameters. One of these choices will be guaranteed to work because of the guarantees of the algorithm. The full procedure is given in Algorithm \ref{alg:search}.

\begin{algorithm}[H]
	\caption{Brute force search over $(\mu_i,L_i)$ parameters}
	 \label{alg:search}
	\begin{algorithmic}[1]
		\REQUIRE Search($m$, $\mu_1$, $L_m$, $\pi_{\kappa}, \eps$)
		\STATE $\pi_{\kappa,\log} \gets \lceil \log_2(\pi_{\kappa}) \rceil + 4m$
		\STATE $\mu_{1,\log} \gets \lfloor \log_2(\mu_1) \rfloor$, $L_{m,\log} \gets \lceil \log_2(L_m) \rceil$
		\FOR{all $\{\kappa_{i,\log}: i \in [m],\kappa_{i,\log} \in [\pi_{\kappa,\log}] \}$ such that $\sum_{i=1}^m \kappa_{i,\log}= \pi_{\kappa,\log} $}
		    \FOR{all $\{\mu_{i,\log}: i \in \{2,\dots,m\}, \mu_{i,\log} \in \{\mu_{1,\log}, \dots, L_{m,\log}\}, \mu_{i,\log}\le \mu_{i+1,\log} \; \forall \; i \le m-1\}$}
		        \STATE $\forall \;i \in [m]$, $L_{i, \log} \gets \mu_{i,\log} + \kappa_{i,\log}$
		        \STATE $(m', \{\mu_{i, \log},  L_{i, \log}, i \in [m']\}) \gets \text{MergeOverlapping}(m, \{\mu_{i, \log},  L_{i, \log}, i \in [m]\}$)
		        \STATE $\forall\; i \in [m'], \mu_{i'} \gets 2^{\mu_{i,\log}}, L_{i'} \gets 2^{L_{i',\log}}$.
		        \STATE $\pi'_{\kappa} \gets 2^{\pi_{\kappa,\log}}$
		        \STATE Let $\x'$ be result of running Algorithm $A$ with parameters $\{(\mu_{i'},L_{i'}), i \in [m'] \}$ for $T(\pi'_{\kappa}, \globalcond, m',\eps)$ gradient steps, and $\eps'$ be the function error of $\x'$.
		        \IF {$\eps' < \eps$}
                    \RETURN $\x'$.
     			\ENDIF
		    \ENDFOR
		\ENDFOR
	\RETURN $\emptyset$.
	\end{algorithmic}
	\begin{algorithmic}[2]
	\REQUIRE MergeOverlapping($m, \{\mu_{i, \log},  L_{i, \log}, i \in [m]\}$)
	    \STATE $m'\gets m$
		        \FOR{all $i\in [m'-1]$}
		            \IF {$L_{i,\log} \ge \mu_{i+1,\log}$}
		                \STATE $L_{i,\log}\gets L_{i+1,\log}$
		                \FOR{all $i+1\le i'\le m-1$}
		                    \STATE $\mu_{i',\log}\gets \mu_{i'+1,\log}$
		                    \STATE $L_{i',\log}\gets L_{i'+1,\log}$
		                \ENDFOR
		                \STATE $m'\gets m'-1$
		            \ENDIF
		        \ENDFOR
		 \RETURN $(m', \{\mu_{i, \log},  L_{i, \log}, i \in [m']\})$
	\end{algorithmic}
\end{algorithm}

We first remark that at least one of the runs of Algorithm $A$ has the property that for all $i\in [m]$ there exists some $i'\in [m']$ such that $\mu_{i'} \le \mu_{i}$ and $L_{i'}\ge L_{i}$, i.e. the original function $f(\x)$ is a multiscale optimization problem with parameters $\{(\mu_{i'},L_{i'}), i \in [m'] \}$. Note that it is sufficient to show that this is true for the choice of parameters before the MergeOverlapping function is called, since the MergeOverlapping function will preserve this property. To verify that the property is true before the MergeOverlapping function is called, note that one of the choices in the brute force search satisfies (a) $\forall \;  i\in [m], \lceil \log_2(\kappa_i)\rceil +1 \le \kappa_{i,\log}$, (b) $\forall \;  i\in [m], \mu_{i,\log} = \lfloor\log_2(\mu_i)\rfloor$. (a) and (b) together ensure that $L_{i,\log} \ge \log_2(L_i)$, which verifies that $\forall \; i\in [m], \mu_{i,\log} \le \log_2(\mu_i)$ and $L_{i,\log}\ge \log_2(L_i)$.

Finally, we claim that Algorithm \ref{alg:search} runs with at most $T(\pi_{\kappa}2^{5m}, \globalcond, m,\eps)\cdot O(\log^m(\globalcond))$  gradient evaluations. This follows because (a) each run of Algorithm $A$ runs for $T(\pi'_{\kappa}, \globalcond, m',\eps)$ steps where  $\pi'_{\kappa} \le 2^{5m}\pi_{\kappa}$ and $m'\le m$, and (b) there are at most $O(\log^m(\globalcond))$ choices for the brute force search over the parameters.

\end{proof}

Next we extend \cref{prop:search} to the stochastic setting. Recall \cref{prop:searchstoch}, restated here for convenience: 
\begin{proposition}[Restated \cref{prop:searchstoch}]
    Let $\pi_{\kappa} = \prod_{i \in m} \cond_i$. Suppose with failure probability at most $p$, we can evaluate $f$ up to a multiplicative constant factor $C$ with $\tilde{T}(p, C)$ oracle queries; that is we can construct some $\hat{f}$ such that for any $\x$,
      \begin{equation}
          f(\x)/C \leq \hat{f}(\x) \leq C f(\x).
      \end{equation}
      A randomized algorithm $A$ which, in expectation, solves the stochastic multiscale optimization problem in Definition \ref{def:multiscale_stochastic} to sub-optimality $\eps$ with $T(\pi_{\kappa}, \globalcond, m,\eps)$ gradient queries when the parameters ($\mu_i, L_i)$ are known, can be used along with the approximate function evaluation to solve the stochastic multiscale optimization problem with failure probability at most $C^2 \epsilon + p$ with $\tilde{T}(p, C) \cdot T(\pi_{\kappa}2^{5m}, \globalcond, m,\eps^2)\cdot O(\log^m(\globalcond))$ oracle queries when only $m$, $\mu_{1}$, $L_m$ and $ \pi_{\kappa}$ are known.
\end{proposition}
\begin{proof}[Proof of \cref{prop:searchstoch}]
Consider \cref{alg:search} with the change in line $9$ of $Search$ that the function error is estimated using $\hat{f}$. First note that if Search($m$, $\mu_1$, $L_m$, $\pi_{\kappa}, \eps^2$) returns any $\x'$, it satisfies that $\hat{f}(\x) < \epsilon/C$ and so with failure probability at most $p$ $f(\x') < \epsilon$. Next we want to show it will return an $\x'$ with probability at least $1 - C^2 \epsilon$. By the proof of ~\cref{prop:search} we know there is at least one run of algorithm A with parameters $\left \{ (\mu_{i}', L_{i}'), i \in [m'] \right \}$ such that the original function $f(\x)$ is a multiscale optimization problem with respect to these parameters and $\pi'_{\kappa} \leq 2^{5m} \pi_{\kappa}$. Thus with $T(\pi'_{\kappa}, \globalcond, m',\eps^2)$ many oracle queries, algorithm A returns some $\x'$ such that in expectation (over the randomness of the algorithm's output $\x$) $f(\x') < \eps^2$. Then by Markov's Inequality, 
\begin{equation}
    P( f(\x') \geq \eps/C^2) \leq \frac{\E[f(\x')]}{\eps/C^2} \leq C^2 \eps. 
\end{equation}
Therefore with failure probability at most $C^2 \eps$, $f(\x') < \eps/C^2$. Then since for any $\x$,
\begin{equation}
    f(\x)/C \leq \hat{f}(\x) \leq C f(\x), 
\end{equation}
we have that with $\hat{f}(\x') < \eps/C$ and so Search($m$, $\mu_1$, $L_m$, $\pi_{\kappa}, \eps^2$) will return this $\x'$ if it hasn't already returned another $\x'$. 
\end{proof}
    
Next recall \cref{cor:bslssearch}, restated here for convenience:
\begin{corollary}[Restated \cref{cor:bslssearch}]
   Assume the setting from \cref{theorem:stochastic} except that only $m$, $\mu_1$, $L_m$ and $\pi_k$ are known. Suppose that $\dist$ is such that for $\vec{a} \sim \dist$ there exists some $K$ where
    \begin{equation}
        \norm{\mat{\Sigma}^{-1/2} \vec{a}}_2 \leq K  \left(  \E \norm{\mat{\Sigma}^{-1/2} \vec{a}}_2^2 \right)^{1/2}.
   \end{equation}
    Then with failure probability at most $\delta$, $\bslsstoch$ can be used to solve the stochastic quadratic multiscale optimization problem from Definition~\ref{def:multiscale_stochastic} with $\tildeo(d)$ space and an extra multiplicative factor of 
$\bigo \left( K^2 d \log \frac{4d}{\delta}  \left(1 + \sqrt{\varepsilon/\delta}  \right) \right) $ queries of $(\vec{a}, b) \sim \dist$. 
\end{corollary}

In the proof of \cref{cor:bslssearch} we will make use of the following Theorem 5.6.1 from \cite{Vershynin19} which we state for the reader's convenience.

\begin{theorem}
\label{thm:matconcentration}
Let $\x$ be a random vector in $\R^d$, $d \geq 2$. Let $\mat{\Sigma} = \E[ \x \x^{\top} ] $ and $\hat{\mat{\Sigma}}_n = \frac{1}{n} \sum_{i \in n} \x_i \x_i^{\top}$ for i.i.d. $\x_i$. 
Assume that for some $K \geq 1$,
\begin{equation}
    \norm{\x}_2 \leq K(\E [\norm{\x}_2^2] )^{1/2} \textrm{ almost surely.}
\end{equation}
Then, for every positive integer $n$ and any $t \geq 0$,
\begin{equation}
     \norm{\hat{\mat{\Sigma}}_n - \mat{\Sigma}}  \leq \left( \sqrt{\frac{K^2 d (\log d + t)}{n}} + \frac{2 K^2 d (\log d + t)}{n} \right) \norm{\mat{\Sigma}}, 
\end{equation}
with probability at least $1 - 2e^{-t}$.
\end{theorem}

\begin{proof}[Proof of \cref{cor:bslssearch}]
We will show that for $\tilde{T}(p,C) = \frac{8 K^2 d \log (2d/p) }{\left(1 - \frac{1}{C} \right)^2}$ oracle queries of $(\vec{a}, b) \sim \dist$ we can construct $\hat{f}$ to estimate $f$ up to multiplicative error $C$. 
Recall 
\begin{equation}
    f(\x) = \frac{1}{2}\E_{(\vec{a},b) \sim \dist} \left[ (\vec{a}^{\top} \x - b \right] = \frac{1}{2} ( \x - \x \opt)^{\top} \mat{\Sigma} (\x - \x \opt).
\end{equation} 
We construct $\hat{f}_n(\x)$ as
\begin{equation}
    \hat{f}_n(\x) \defeq \frac{1}{2n} \sum_{i \in n} \frac{1}{2} (\vec{a}_i^{\top} \x - b_i)^2 = \frac{1}{2}  (\x - \x \opt)^{\top} \hat{\mat{\Sigma}}_n (\x - \x \opt).
\end{equation}
For $\vec{a} \sim \dist$, consider the random vector $\tilde{\vec{a}} = \mat{\Sigma}^{-1/2} \vec{a}$. Note that by assumption
\begin{equation}
    \norm{\mat{\Sigma}^{-1/2} \vec{a}}_2 \leq K \left( \E \norm{\mat{\Sigma}^{-1/2} \vec{a}}_2^2 \right)^{1/2},
\end{equation}
almost surely. Suppose for $C \geq 1$
\begin{equation}
    n = \frac{8 K^2 d \log (2d/p) }{\left(1 - \frac{1}{C} \right)^2}.
\end{equation} 
Then by \cref{thm:matconcentration}, with failure probability at most $p$,
\begin{equation}
    \label{eqn:additive}
    \norm{\mat{\Sigma}^{-1/2} \hat{\mat{\Sigma}}_n \mat{\Sigma}^{-1/2} - \id } \leq \frac{1 - \frac{1}{C}}{2} + \frac{\left( 1 - \frac{1}{C}\right)^2}{4} \leq 1 - \frac{1}{C}.
\end{equation}
Note that \cref{eqn:additive} holds if and only if
\begin{equation}
    \frac{1}{C} \mat{\Sigma} \preceq  \hat{\mat{\Sigma}}_n \preceq \left( 2 + \frac{1}{C}\right).
\end{equation}
Therefore for $C \geq 3$ since $2 + \frac{1}{C} \leq C$, 
\begin{equation}
    \frac{1}{C} f(\x) \leq \hat{f}(\x) \leq C f(\x). 
\end{equation}
To conclude the proof of \cref{cor:bslssearch} we simply recall \cref{theorem:stochastic} and apply \cref{prop:searchstoch}. 

\end{proof}

\subsection{If decomposition is known, then can solve with $\tildeo(\sum_{i \in [m]} \sqrt{\kappa_i})$ gradient queries}\label{sec:known_projection}
In this subsection we show that when the gradient of sub-objectives are known, then the multiscale optimization problem (\cref{def:multiscale_problem}) can be solved in $\tildeo (\sum_{i \in [m]} \sqrt{\kappa_i} )$ queries.
To prove this claim, 
consider the algorithm that run accelerated gradient descent on each sub-objective $f_j$ independently. 
Since each sub-objective $f_j$ takes $\bigo(\sqrt{\kappa_j} \log(m/\eps))$ to converge to $\frac{\epsilon}{m}$-optimality, we only need a total of $\sum_{j \in [m]} \bigo(\sqrt{\kappa_j} \log(m/\epsilon))$ gradients for $f$ to converge to $\epsilon$-optimality.

However, as we will see in the next subsection (\cref{sec:hard:to:recover:decomposition}), recovering the projection $\proj_i$ is costly.
    		
\subsection{Recovering the projections $\proj_i$ is costly}
\label{sec:hard:to:recover:decomposition}

In this subsection we show that recovering the projections $\proj_i$ in the multiscale optimization problem (\cref{def:multiscale_problem}) requires $\Omega(d)$ gradient evaluations in the worst-case. 

\begin{proposition}
\label{prop:hard:to:recover}
Consider the  multiscale optimization problem in \cref{def:multiscale_problem} for $m=2$. There exist functions $f_1,f_2$ such that recovering $\proj_1, \proj_2$ requires $\Omega(d)$ gradient evaluations in the worst-case, even if $f_1, f_2$ are known.
\end{proposition}
\begin{proof}[Proof of \cref{prop:hard:to:recover}]
Let $\proj_1, \proj_2 \in \R^{d/4 \times d}$. Consider the following multiscale optimization problem:
\begin{equation}
    f_1(\x) = \norm{\proj_1 \x}^2, \quad 
    f_2(\x) = (1/\globalcond) \norm{\proj_2 \x}^2, \quad
    f(\x) = f_1(\x)+ f_2(\x).
\end{equation}
Note that by our choice of $f_1$ and $f_2$, $\kappa_1=\kappa_2 = 1$ and $\globalcond$ is the overall condition number. Now we can write,
\begin{gather}
    \nabla f(\x) = 2\proj_1^\top \proj_1  \x + (2/\globalcond)\proj_2^\top \proj_2  \x.
\end{gather}
Let $S_1, S_2$ be the $d/4$ dimensional subspaces spanned by $\proj_1^\top$ and $\proj_2^\top$ respectively.
Note that $\proj_1^\top \proj_1  \x$ is the projection of $\x$ onto $S_1$, and similarly for $S_2$. Therefore $\nabla f(\x)$ is a linear combination of the projections onto $S_1$ and $S_2$, and with every gradient evaluation we get a single vector in the span of $S_1$ and $S_2$. Since the union of $S_1$, $S_2$ is a $d/2$ dimensional subspace, any algorithm needs $d/2$ vectors from the subspace to learn it. Hence any algorithm needs at least $d/2$ gradient evaluations to learn $S_1$, $S_2$, and hence also to learn $\proj_1, \proj_2$. We note that it could be possible to extend this argument to approximately learning the subspaces $S_1$ and $S_2$  using bounds on quantization on the Grassmann manifold \cite{dai2007volume}. 
\end{proof}

\subsection{GD with exact line search or constant step-sizes cannot match the guarantee of $\bsls$}
\label{sec:line_search}
In this subsection we prove that gradient descent with exact line search or any constant step-size cannot match the guarantee of $\bsls$ (\cref{alg:bsls}) given by \cref{thm:bsls:recursive}. 

Formally, gradient descent with exact line search has the form:
\begin{equation}
    \x^{(t+1)} \gets \argmin_{\x} \left\{ f(\x) \middle| \x = \x^{(t)} - \stepsize \nabla f(\x^{(t)}) \text{ for some $\stepsize \in \reals$} \right\}.
\end{equation}

\begin{proposition}
    \label{thm:exactlinesearch}
    Consider the objective $f(\x) = \frac{1}{2} \x^{\top} \mat{A} \x - \vec{b}^{\top} \x$ with
    \begin{equation}
    \mat{A} = 
        \begin{bmatrix}
            \lambda_1 & 0 \\
            0 & \lambda_2
        \end{bmatrix}, 
        \qquad \vec{b} = \vec{1}.
    \end{equation}
    Let $\lambda_1 < \lambda_2$ and note that $\globalcond = \lambda_2/\lambda_1$. Then
    \begin{enumerate}[(a),leftmargin=*]
        \item  Gradient descent with 
    exact line search initialized at $\x^{(0)} = \vec{0}$ requires at least $\left \lfloor \frac{\globalcond}{8} \log \left( \frac{ (1/\lambda_1) + (1/\lambda_2) }{2\epsilon}\right)  \right \rfloor$ gradient queries to attain $\epsilon$-optimality.
    \item    If $\globalcond \geq 2$, gradient descent with a constant step size requires at least $\left \lfloor \frac{\globalcond}{2} \log \left( \frac{1}{\lambda_1 \epsilon}\right)  \right \rfloor$ gradient queries to attain $\epsilon$-optimality.
    \end{enumerate}
\end{proposition}
\begin{remark}
Theorem~\ref{thm:bsls:recursive} states that $\bsls$ only requires $\left(2 \log(\globalcond + 1\right) \log\left( \frac{(1/\lambda_1) + (1/\lambda_2)}{2 \epsilon} \right)$ gradient queries which compares favorable to gradient descent with exact line search and constant stepsize. 
\end{remark}

\begin{proof}[Proof of \cref{thm:exactlinesearch} - Exact Line Search] 
Assume we initialize $\x^{(0)} = \vec{0}$. For exact line search we update $
    \x^{(t+1)} = \x^{(t)} - s_t \fdiv{t}{}$
with $
    s_t = \frac{\fdiv{t}{\top}\fdiv{t}{}}{\fdiv{t}{\top} \mat{A}\fdiv{t}{}}.
$
It is a fact that
    \begin{equation}
        \label{eqn:linesearcheq}
        f(\x^{(t+1)}) - f(\x \opt) = \left(1 - \frac{1}{\cond_t} \right) \left( f(\x^{(t)}) - f(\x \opt) \right),
    \end{equation}
    where if $\fdiv{t}{}\defeq \nabla f(\x^{(t)})$,
    \begin{equation}
        \cond_t \defeq \frac{\fdiv{t}{\top} \mat{A} \fdiv{t}{}}{\fdiv{t}{\top} \fdiv{t}{}} \frac{\fdiv{t}{\top} \mat{A}^{-1} \fdiv{t}{}}{\fdiv{t}{\top} \fdiv{t}{}}.
    \end{equation}
    Our goal is to show that a large portion of the gradients $\fdiv{1}{}, \dots, \fdiv{T}{}$ are such that $\cond_t$ is close to $\cond$. 
    Note $\fdiv{t}{}= \mat{A}(\x^{(t)} - \x \opt) $. For simplicity define
    \begin{equation}
        \label{def:uv}
        u_t \defeq\fdiv{t}{}_1 = \lambda_1 \left( \x^{(t-1)}_1 - \x_1 \opt\right) \qquad v_t \defeq \fdiv{t}{}_2 = \lambda_2 \left(  \x^{(t-1)}_2 - \x_2 \opt \right).
    \end{equation}
     We can rewrite $\cond_t$ as,
    \begin{small}
    \begin{align}
        \label{eqn:condition}
        \cond_t & = \frac{\left( \lambda_1  u_t^2 + \lambda_2 v_t^2\right) }{u_t^2 + v_t^2} \frac{\left( \frac{1}{\lambda_1}  u_t^2 + \frac{1}{\lambda_2} v_t^2\right) }{u_t^2 + v_t^2} 
        = \frac{ u_t^4 + v_t^4 + \globalcond u_t^2 v_t^2 + \frac{1}{\globalcond} u_t^2 v_t^2 }{(u_t^2 + v_t^2)^2}  
        = \frac{ \left( \frac{u_t^2}{v_t^2} + \frac{v_t^2}{u_t^2} \right) + \globalcond  + \frac{1}{\globalcond} }{ \left( \frac{u_t^2}{v_t^2} + \frac{v_t^2}{u_t^2}\right) + 2}.
    \end{align}
    \end{small}
     This motivates us to understand the ratio $u_t^2/v_t^2$ or rather $\left( \x^{(t)}_1 - \x_1 \opt \right)^2/ \left(  \x^{(t)}_2 - \x_2 \opt \right)^2$. Since
     \begin{equation}
         \x^{(t+1)} - \x \opt = (\id - s_t \mat{A}) \x^{(t)},
     \end{equation}
     we have
    \begin{align}
        \label{eqn:xtgrowth}
        \frac{ \left( \x^{(t)}_1 - \x_1 \opt \right)^2 }{\left( \x^{(t)}_2 - \x_2 \opt \right)^2} & = \frac{\prod_{\ell = 1}^t (1 - s_{\ell} \lambda_1)^2  \left( \x^{(0)}_1 - \x_1 \opt \right)^2 }{\prod_{\ell = 1}^t (1 - s_{\ell} \lambda_2)^2  \left( \x^{(0)}_2 - \x_2 \opt \right)^2}.
    \end{align}
    Define
    \begin{equation}
        p_t \defeq \prod_{\ell = 1}^t \frac{(1 - s_{\ell} \lambda_1)^2}{(1 - s_{\ell} \lambda_2)^2}. 
    \end{equation}
    Notice that 
    \begin{equation}
        \fdiv{t+1}{} = (\id - s_t \mat{A})\fdiv{t}{}.
    \end{equation}
    Therefore since $\fdiv{0} = - \vec{b} = - \vec{1}$,
    \begin{align}
        s_{t+1} & = \frac{\fdiv{t+1}{2}_1 + \fdiv{t+1}{2}_2 }{ \lambda_1 \fdiv{t+1}{2}_1 + \lambda_2 \fdiv{t+1}{2}_2} \\
        & = \frac{\left( \prod_{\ell = 1}^t (1 - s_{\ell} \lambda_1)^2 \right)  + \left( \prod_{\ell = 1}^t (1 - s_{\ell} \lambda_2)^2 \right) }{ \lambda_1 \left( \prod_{\ell = 1}^t (1 - s_{\ell} \lambda_1)^2 \right)  + \lambda_2 \left( \prod_{\ell = 1}^t (1 - s_{\ell} \lambda_2)^2 \right)} \\
        & = \frac{(p_t + 1) \left( \prod_{\ell = 1}^t (1 - s_{\ell} \lambda_2)^2 \right) }{ (\lambda_1 p_t + \lambda_2) \left( \prod_{\ell = 1}^t (1 - s_{\ell} \lambda_2)^2 \right)} \\
        & = \frac{(p_t + 1)  }{ (\lambda_1 p_t + \lambda_2) }.
    \end{align}
    Therefore,
    \begin{align}
        \frac{1 - s_t \lambda_1}{1 - s_t \lambda_2} & = \frac{1 -  \lambda_1 \frac{(p_t + 1)  }{ (\lambda_1 p_t + \lambda_2) }}{1 -  \lambda_2 \frac{(p_t + 1)  }{ (\lambda_1 p_t + \lambda_2) }} 
        = \frac{(\lambda_1 p_t + \lambda_2) -  \lambda_1(p_t + 1)}{ (\lambda_1 p_t + \lambda_2) -  \lambda_2 (p_t + 1)  }  
        = \frac{\lambda_2 - \lambda_1}{p_t(\lambda_1 - \lambda_2)} 
        = -\frac{1}{p_t}.
    \end{align}
    Then since
    \begin{equation}
        p_{t+1} = p_t \left( \frac{1 - s_t \lambda_1}{1 - s_t \lambda_2} \right)^2,
    \end{equation}
    we have
    \begin{equation}
        p_{t+1} = \frac{1}{p_t}.
    \end{equation}
    Finally since $s_1 = 2/(\lambda_1 + \lambda_2)$ we have $p_1 = 1$ and therefore for any $t$, $p_t = 1$.
    Thus, recalling Eq.~\ref{eqn:xtgrowth} we have
    \begin{equation}
        \frac{ \left( \x^{(t)}_1 - \x_1 \opt \right)^2 }{\left( \x^{(t)}_2 - \x_2 \opt \right)^2} =  \frac{ \left( \x^{(0)}_1 - \x_1 \opt \right)^2 }{\left( \x^{(0)}_2 - \x_2 \opt \right)^2} = \frac{\x_1^{*2}}{ \x_2^{*2} } = \globalcond^2.
    \end{equation}
    Recalling the definitions of $u_t$ and $v_t$ in Eq.~\ref{def:uv} we have
    \begin{equation}
        \frac{u_t}{v_t} =  \frac{ \lambda_1^2 \left( \x^{(t-1)}_1 - \x_1 \opt \right)^2 }{ \lambda_2^2 \left( \x^{(t-1)}_2 - \x_2 \opt \right)^2} = 1.
    \end{equation}
    Finally, recalling Eq.~\ref{eqn:condition} we have
    \begin{equation}
        \cond_t = \frac{2 + \globalcond + \frac{1}{\globalcond}}{4} \geq \globalcond/4.
    \end{equation}
    Therefore we can lower bound the progress made by exact line search. Using Eq.~\ref{eqn:linesearcheq} we find
    \begin{equation}
        f(\x^{(t+1)}) - f(\x \opt) \geq \left( 1 - \frac{4}{\globalcond} \right) \left( f(\x^{(t)}) - f(\x \opt) \right).
    \end{equation}
    Thus exact line search requires at least $\left \lfloor \frac{\globalcond}{8} \log \left( \frac{f(\vec{0}) - f (\x \opt)}{\epsilon}\right)  \right \rfloor$ gradient queries. Since $f(\vec{0}) - f (\x \opt) = \frac{1}{2}((1/\lambda_1) + (1/\lambda_2))$ we conclude exact line search requires at least $\left \lfloor \frac{\globalcond}{8} \log \left( \frac{ (1/\lambda_1) + (1/\lambda_2) }{2\epsilon}\right)  \right \rfloor$ gradient queries.
\end{proof}
\begin{proof}[Proof of \cref{thm:exactlinesearch} - Constant step-sizes]
We make use of the equality
\begin{equation}
    \label{eqn:usefulequality}
   f(\x) - f(\x \opt)= \frac{1}{2} \norm{ \x - \x \opt}_{\mat{A}}^2,
\end{equation}
where $ \norm{ \x }_{\mat{A}}^2 \defeq \x^{\top} \mat{A} \x$. Since 
\begin{equation}
    \nabla f(\x) = \mat{A} (\x - \x \opt),
\end{equation}
we have that the gradient descent algorithm with constant stepsize $\alpha$ produces the recursion,
\begin{align}
    \x^{(t+1)} - \x \opt & = \x^{(t)} - \alpha \nabla f(\x) 
    = (\id - \alpha \mat{A}) (\x^{(t)} - \x \opt) 
    = (\id - \alpha \mat{A})^{t+1} (\x^{(0)} - \x \opt).
\end{align}
Therefore using Eq.~\ref{eqn:usefulequality} we find,
\begin{equation}
    f(\x^{(t)}) - f(\x \opt) = (\x^{(0)} - \x \opt)^{\top} (\id - \alpha \mat{A})^{t} \mat{A} (\id - \alpha \mat{A})^{t} (\x^{(0)}  - \x \opt).
\end{equation}
Using the definition of $\mat{A}$, the fact that $\x^{(0)} = \vec{0}$, and finally the fact that $\x \opt = \begin{bmatrix}
    1/ \lambda_1 & 1/\lambda_2
\end{bmatrix}^{\top}$,  we have
\begin{equation}
    f(\x^{(t)}) - f(\x \opt) = \frac{1}{\lambda_1}\left(1 - \alpha \lambda_1 \right)^{2t} + \frac{1}{\lambda_2}\left( 1 - \alpha \lambda_2 \right)^{2t}.
\end{equation}
In order to have the function error decrease we must choose $\alpha \in [0, 2/\lambda_2]$. For $\alpha$ in this range we have,
\begin{align}
    f(\x^{(t)}) - f(\x \opt) \geq \frac{1}{\lambda_1}\left(1 - \alpha \lambda_1 \right)^{2t} 
    \geq \frac{1}{\lambda_1}\left(1 - \frac{2 \lambda_1 }{\lambda_2} \right)^{2t}.
\end{align}
Suppose $t < \frac{\globalcond}{4} \log \left( \frac{1}{2 \lambda_1 \epsilon}\right)$. Using that $\left(1 - \frac{1}{x} \right)^x \geq \frac{1}{2}$ for all $x \geq 2$ and our assumption that $\globalcond \geq 2$ we have
\begin{align}
    f(\x^{(t)}) - f(\x \opt) & \geq \frac{1}{\lambda_1}\left(1 - \frac{2 \lambda_1 }{\lambda_2} \right)^{2t} \geq \frac{1}{\lambda_1}\left(1 - \frac{2 \lambda_1 }{\lambda_2} \right)^{\frac{\globalcond}{2} \log \left( \frac{1}{2 \lambda_1 \epsilon}\right)} \geq 2 \epsilon.
\end{align}
This concludes the proof.
\end{proof}

\subsection{The orthogonality assumption in \cref{def:multiscale_problem} is necessary}
\label{append:disjoint}
In this section we show that without the assumption in \cref{def:multiscale_problem} that
\begin{equation}
    	\proj_i  \proj_j^\top = 
	\begin{cases}
					\id_{d_i \times d_i} & \text{if $i = j$,}
				  \\
					\mat{0}_{d_i \times d_j} & \text{otherwise, }
	\end{cases}
\end{equation}
\cref{thm:bsls:informal} does not necessarily hold. Indeed, in \cref{prop:randomized_counterexample} we show that any first-order method (even if randomized) must have query complexity at least  $\Omega(\sqrt{\globalcond}/\polylog(d))$. 

\begin{proposition}
    \label{prop:randomized_counterexample}
    There exists a distribution over instances
    \begin{equation}
        f(\x) = \sum_{i=1}^4 f_i(\mat{P}_i \x) ,
    \end{equation}
    where $\mat{P}_i\in \R^{d/2\times d}$ is such that
    \begin{equation}
        \mat{P}_i\mat{P}_j^{\top}  \neq \mat{0}
    \end{equation}
    for $i\ne j$ and each $f_i$ is well conditioned with $\cond_i \leq 10$ and $\globalcond=\Theta(d^2)$ such that any first-order method (even if randomized) which returns a $\frac{1}{10\globalcond}$-optimal solution with probability at least $0.9$
needs at least $\Omega(\sqrt{\globalcond}/\polylog(d))$ first-order queries.
\end{proposition}

\begin{proof}

We use  the following result from \cite{braverman2020gradient}  which establishes a hardness result for solving linear systems. Let $\kappa(\mat{M})$ denote the condition number of any matrix $\mat{M}.$

\begin{theorem}\label{thm:hard_linear}
(Theorem 6 of \cite{braverman2020gradient}) Let $d_0$ be a universal constant. Let $d \geq d_0$ be any ambient dimension and let $\mathcal{A}$ be any linear system algorithm. Suppose $\mathcal{A}$ is such that for all positive semi-definite matrices $\mat{A}$ with condition number $\kappa(\mat{A})\le d^2$ and for all initial vectors $\x_0\in \R^d$ and $\vec{b} \in \R^d$, 
\begin{align}
        \Pr\left[\norm{\hat{\x}-\mat{A}^{-1}\vec{b}}_\mat{A}^2 \le \frac{1}{10d^2} \right]\ge 1-\frac{1}{e}.
    \end{align}
Then $\mathcal{A}$ must have query complexity at least $\Omega(\kappa(\mat{A})/\polylog(d) )$.
\end{theorem}

We prove our lower bound by showing that the hard instance in \cite{braverman2020gradient} admits a decomposition as a multi-scale optimization problem.

The hard distribution over matrices $\mat{A}$ is $\mat{A}=(\gamma-1) \mat{I}+ (1/5)\mat{W}$, where $\mat{W}$ is sampled from the Wishart distribution, i.e.  $\mat{W}=\mat{X}\mat{X}^{\top}$ where $\mat{X}\in \R^{d\times d}$ and $\mat{X}_{i,j}$ is distribution as i.i.d. $N(0,1/d)$, and $\gamma=1+\Theta(1/d^2)$. We show that with high probability $\mat{A}$ admits a simple decomposition into the sum of four well-conditioned matrices, hence proving our bound.

Let $\mat{X}_1\in \R^{d\times d/2}$ be the submatrix of $\mat{X}$ corresponding to its first $d/2$ columns and $\mat{W}_1=\mat{X}_1\mat{X}_1^{\top}$. Similarly, let $\mat{X}_2\in \R^{d\times d/2}$ be the submatrix of $\mat{X}$ corresponding to its last $d/2$ columns and $\mat{W}_2=\mat{X}_2\mat{X}_2^{\top}$. Note that $\mat{W}=\mat{W}_1+\mat{W}_2$, therefore,
\begin{align}
    \mat{A}=\underbrace{\frac{\gamma-1}{2}\mat{I}+\frac{1}{5}\mat{W}_1}_{\mat{U}_1}+\underbrace{\frac{\gamma-1}{2}\mat{I}+\frac{1}{5}\mat{W}_2}_{\mat{U}_2}.
\end{align}

Let $\mathcal{E}_1$ be the event that all the non-zero eigenvalues of $\mat{W}_1$ and $\mat{W}_2$ lie in the interval $[0.2,2] $. By concentration bounds for the spectrum of Wishart matrices (see for example Corollary 5.35 of \cite{vershynin2010introduction}), for sufficiently large $d$, $\mathcal{E}_1$ happens with probability at least $0.99$. Let$\mathcal{E}_2$ be the event that $\kappa(\mat{A}=\Theta(d^2)$. By similar concentration bounds for Wishart matrices (for example Corollary 12 of \cite{braverman2020gradient}), $\mathcal{E}_2$ happens with probability at least $0.99$. %
We condition on the events $\mathcal{E}_1$ and  $\mathcal{E}_2$ for the rest of the proof.

Let $\mat{W}_1=\mat{P}_1^{\top}\mat{\Sigma}_1\mat{P}_1$ denote the singular-value decomposition of $\mat{W}_1$. Let $\mat{P}_2\in \R^{d/2\times d}$ be the matrix whose rows form an orthonormal basis for the orthogonal space to the column space of $\mat{W}_1$.  Then we can decompose $\mat{U}_1$ as $\mat{U}_1=\mat{A}_1+\mat{A}_2$, where $\mat{A}_1=\mat{P}_1^{\top}(\frac{\gamma-1}{2}\mat{I}+\frac{1}{5}\mat{\Sigma}_1)\mat{P}_1$ and $\mat{A}_2=\frac{\gamma-1}{2}\mat{P}_2^{\top}\mat{P}_2$. Let $\tilde{\mat{A}}_1=\frac{\gamma-1}{2}\mat{I}+\frac{1}{5}\mat{\Sigma}_1$ and $\tilde{\mat{A}}_1= \frac{\gamma-1}{2} \mat{I}$. Note that the eigenvalues of $\tilde{\mat{A}}_1$ lie in the interval $[0.2+\frac{\gamma-1}{2},2+\frac{\gamma-1}{2}]$ and all eigenvalues of $\tilde{\mat{A}}_2$ are $\frac{\gamma-1}{2}$. Therefore, $\kappa(\tilde{\mat{A}}_1)\le  10$, and $\kappa(\tilde{\mat{A}}_2)=1$.  

Similarly, let $\mat{W}_2=\mat{P}_3^{\top}\mat{\Sigma}_3\mat{P}_3$ denote the singular-value decomposition of $\mat{W}_2$. Let $\mat{P}_4\in \R^{d/2\times d}$ be the matrix whose rows form an orthonormal basis for the orthogonal space to the column space of $\mat{W}_2$.  Then we can decompose $\mat{U}_2$ as $\mat{U}_2=\mat{A}_3+\mat{A}_4$, where $\mat{A}_3=\mat{P}_3^{\top}\tilde{\mat{A}}_3\mat{P}_3$ and $\mat{A}_4=\mat{P}_4^{\top}\tilde{\mat{A}}_4 \mat{P}_4$, where $\tilde{\mat{A}}_3=\frac{\gamma-1}{2}\mat{I}+\frac{1}{5}\mat{\Sigma}_3$ and $\tilde{\mat{A}}_4= \frac{\gamma-1}{2} \mat{I}$. As before $\kappa(\tilde{\mat{A}}_3)\le  10$, and $\kappa(\tilde{\mat{A}}_4)=1$.  

For any matrix $\mat{A}$ and vectors $\x,\vec{b}$, let $g(\mat{A},\vec{b}, \x)=\x^{\top}\mat{A}\x-2\vec{b}^{\top}\x$. Using the decomposition of $\mat{A}=\sum_{i=1}^4\mat{P}_i^{\top} \tilde{\mat{A}}_i\mat{P}_i$ and the fact that $\x=\mat{P}_1\x +\mat{P}_2\x = \mat{P}_3\x +\mat{P}_4\x $ we can write,
\begin{align}
    f(\x) = g(\mat{A},\vec{b}, \x)=g(\tilde{\mat{A}}_1,\vec{b}/2, \mat{P}_1\x)+g(\tilde{\mat{A}}_2,\vec{b}/2, \mat{P}_2\x)+g(\tilde{\mat{A}}_3,\vec{b}/2, \mat{P}_3\x)+g(\tilde{\mat{A}}_4,\vec{b}/2, \mat{P}_4\x).
\end{align}
Note that for any $1\le i \le 4$ the condition number of $g(\tilde{\mat{A}}_i,\vec{b}/2, \mat{P}_i\x)$ is $\kappa(\tilde{\mat{A}}_i)\le 10$. The condition number of $g(\mat{A},\vec{b}, \x)$ is $\kappa(\mat{A})=\Theta(d^2)$.

Since $\mathcal{E}_1 \cap \mathcal{E}_2$ happens with probability at least $0.98$, we have the above decomposition with probability at least $0.98$. Note that any $\hat{\x}$ which is a  $\frac{1}{10d^2}$-optimal solution to the above problem satisfies,
\begin{align}
    \norm{\hat{\x}-\mat{A}^{-1}\vec{b}}_\mat{A}^2 \le \frac{1}{10d^2}.
\end{align}
Therefore, any algorithm which solves the multi-scale optimization problem
probability at least $0.9$, also solves the hard instance of $\mat{A}$ in \cite{braverman2020gradient} with probability at least $0.8$. By Theorem \ref{thm:hard_linear}, this requires at least $\Omega(\kappa(\mat{A})/\polylog(d))$ gradient queries.
\end{proof}

\subsection{Complexity of conjugate gradient for quadratic multiscale optimization}
\label{sec:cg:ub}
\newcommand{\cheb}[1]{\ensuremath{T_{#1}}}
\newcommand{\rescale}[2]{\ensuremath{\ell_{[#1,#2]}}}
\newcommand{\damp}[3]{\ensuremath{p_{[#1,#2]}^{(#3)}}}
\newcommand{\globalpoly}[1]{P_{#1}} 
\newcommand{\epsdeg}[2]{d_{#1}\mathopen{}\left(#2\right)\mathclose{}}
\newcommand{\mv}{\mathrm{mv}}

In this section, we give a simple proof that the conjugate gradient algorithm can  stably solve the multiscale optimization problem in the special case when \emph{$f$ is quadratic} in a number of iterations that is comparable to what our accelerated BSLS algorithm requires.  
More precisely, we show:

\begin{theorem}[Complexity of conjugate gradient for multiscale quadratic optimization]
    \label{thm:cg_ub_main}
	Consider an instance of the multiscale optimization problem (Def.  \ref{def:multiscale_problem}) in which each 
	$f_i$ is quadratic. 
	 
	 For any  $\x^{(0)}$ and $\epsilon > 0$, 
	 the conjugate gradient method started at $\x^{(0)}$, can return an $\epsilon$-optimal solution with 
	$\left( \prod_{i \in [m]} \bigo (\sqrt{\kappa_i}) \right) \cdot \bigo \left( \left(\log^{m-1}\globalcond \right) \cdot \log \left( \frac{ f(\x^{(0)}) - f^{\star}}{\epsilon} \right) \right) $  gradient queries.
	This remains true if all operations are performed using a number of bits of precision that is logarithmic in the problem parameters.
\end{theorem}

The conjugate gradient method is usually discussed as an algorithm for solving linear systems, so we begin by rephrasing our quadratic optimization problem in this form.

We can write our quadratic objective function as  $f(\x)=\x^{\top} \mat{A} \x - 2\vec{b}^{\top} \x$ for some $\mat{A}\in\R^{d\times d}$ and $\vec{b}\in \R^d$. 
The assumption that $f$ is strongly convex corresponds to the requirement that the matrix $\mat{A}$ be positive definite, and the assumption about the existence of a decomposition of $f$ in terms of $f_i$ with the given smoothness and convexity properties corresponds to the assumption that the eigenvalues of $\mat{A}$ all lie in the set $S = \bigcup_{i \in [m]} [\mu_i, L_i]$.

Since $\nabla f(x)= 2(\mat{A}\x-\vec{b})$, $f$ is minimized at $\x\opt=\mat{A}^{-1}\vec{b}$, and the function error at some other point $\x$ is given by
\begin{align}
    f(\x)-f(\x\opt)&= (\x^{\top} \mat{A} \x - 2\vec{b}^{\top} \x) - (\x^{\star\top} \mat{A} \x\opt - 2\vec{b}^{\top} \x\opt)\\
    &=(\x^{\top} \mat{A} \x - 2\x^{\star\top} \mat{A} \x) - (\x^{\star\top} \mat{A} \x\opt - 2\x^{\star\top}\mat{A} \x\opt)\\
    &=\x^{\top} \mat{A} \x - 2\x^{\star\top} \mat{A} \x +\x^{\star\top}\mat{A} \x\opt\\
    &=(\x-\x\opt)^{\top} \mat{A} (\x-\x\opt) \defeq \|\x- \x\opt\|_{\mat{A}}^2.
\end{align}
Minimizing $f$ is thus equivalent to solving the linear system $\mat{A}\x=\vec{b}$, and the function error at a point equals its distance from the optimal solution in the $\mat{A}$-norm. 
To prove Theorem~\ref{thm:cg_ub_main}, it thus suffices to bound the convergence rate in the $\mat{A}$-norm of the conjugate gradient method applied to matrices with eigenvalues in $S$.

Our proof relies on the connection between the performance of the conjugate gradient algorithm and polynomial approximation.
If the algorithm uses exact arithmetic, the classical analysis of the conjugate gradient algorithm asserts that, after $k$ iterations, the algorithm returns a vector $\x^{(k)}$ such that 
\begin{equation}
    \|\x^{(k)}-\x\opt\|_{\mat{A}} \leq \|\x^{(0)}-\x\opt\|_{\mat{A}} \cdot 
    \min_{p\in \polys_k^0} \max_{i\in [d]} |p(\lambda_i(\mat{A})|,
\end{equation} 
where $\polys_k^0$ denotes the set of polynomials of degree at most $k$ with $p(0) = 1$.  

To prove Theorem~\ref{thm:cg_ub_main} under exact arithmetic, it thus suffices to construct a polynomial $p\in\polys_k^0$ for $k$ less than or equal to the given bound on the number of gradient queries with $|p(x)|\leq \epsilon$ for all $x\in S$.
 
For finite-precision arithmetic, we apply the following theorem of Greenbaum, which says that  the convergence rate of the conjugate gradient method applied to a matrix $\mat{A}$ with eigenvalues in $S$ using precision that is logarithmic in the problem parameters can be bounded in terms of its convergence rate under \emph{exact}
arithmetic on a matrix with eigenvalues in a slightly enlarged set $S'\supseteq S$.

\newcommand{\Apert}{\mat{\tilde{A}}}
\newcommand{\bpert}{\vec{\tilde{b}}}
\newcommand{\xpert}{\vec{\tilde{x}}}
\begin{theorem}[\cite{greenbaum1989behavior}, as simplified in \cite{Musco.Musco.ea-SODA18}]
Given a positive definite matrix $\mat{A}\in \R^{n\times n}$ and a vector $\vec{b} \in \R^n$, let $\x$ be the result of running the conjugate gradient method for $k$ iterations on the linear system $A\x=\vec{b}$ with all operations  performed using  $\Omega\left(\log \frac{nk(\|\mat{A}\|+1)}{\min(\eta,\lambda_{\min}(\mat{A})} \right)$ 
bits of precision.

Let $\Delta=\min\left(\eta,\frac{\lambda_{\min}(\mat{A})}{5}\right)$.
There exists a matrix $\Apert$ with eigenvalues in $S':=\bigcup_{i=1}^n[\lambda_i(\mat{A})-\Delta, \lambda_i(\mat{A})+\Delta]$ and a vector $\bpert$ with $\|\Apert^{-1}\bpert\|_{\Apert} = \|\mat{A}^{-1}\vec{b}\|_{\mat{A}}$ such that, if $\xpert$ is the result of running the conjugate gradient method for $k$ iterations on the linear system $\Apert \xpert = \bpert$ in \emph{exact} arithmetic, then
\[
\|\mat{A}^{-1}\vec{b}-\x\|_{\mat{A}}\leq 1.2 \|\Apert^{-1}\bpert-\xpert\|_{\Apert}.
\]
\end{theorem}

Replacing $S$ with $S'$ does not change the asymptotic behavior of the bound asserted in Theorem~\ref{thm:cg_ub_main}, so it suffices to show the existence of polynomials with the properties required in the exact case.
The following Theorem asserts the existence of such polynomials, from which Theorem~\ref{thm:cg_ub_main} follows.

\begin{theorem}[Existence of good polynomials for unions of intervals]
\label{thm:poly_construction}
Let $S = \bigcup_{i \in [m]} [\mu_i, L_i]$, where $\mu_1<L_1< \mu_{2}<L_2<\dots$.
For any $\epsilon > 0$, there exists a polynomial $P$ such that $P(0)=1$, $|P(x)|\leq \epsilon$ for all $x\in S$, and 
$\deg(P)\leq \left( \prod_{i \in [m]} \bigo (\sqrt{\kappa_i}) \right) \cdot\bigo \left( \log^{m-1}\left(\globalcond \right) \cdot \log \left( { 1}/{\epsilon} \right) \right)$.
\end{theorem}

We devote the remainder of this section to constructing the polynomials required by this theorem.
Note that our goal here is to present a simple construction that has the desired asymptotic behavior rather than to optimize the constants, and the polynomials given are not the exactly optimal polynomials for $S$.

\subsubsection{Constructing good polynomials for unions of intervals}
The basic building blocks of our construction are Chebyshev polynomials.
Simple transformations of Chebyshev polynomials give optimal polynomials for individual intervals.
We construct good polynomials for unions of intervals by multiplying such polynomials together.
The main technical difficulty is that the polynomial for one interval can be quite large on another interval, so naively multiplying together the polynomials for the individual intervals will not produce something that is small on all of $S$. 
Instead we will carefully choose the degrees of the polynomials on the different intervals to manage the error caused by these interactions.

We begin by reviewing the definition of Chebyshev polynomials and providing some standard bounds on their magnitude. 
Let $\cheb{d}(x) = \frac{1}{2}\left(x+\sqrt{x^2-1}\right)^d+\frac{1}{2}\left(x-\sqrt{x^2-1}\right)^d$ be the degree-$d$ Chebyshev polynomial (of the first kind).  This defines a degree-$d$ polynomial with the following well-known properties:\footnote{For an introduction to Chebyshev polynomials and their basic properties, see \cite{Mason.Handscomb-03}}

\begin{enumerate}
	\item If $|x|\leq 1$, $|\cheb{d}(x)|\leq 1$.\label{eq:chebprop1}
	\item If $|x|\geq 1$, \label{eq:chebprop2}
	\[
		\frac{1}{2}\left(1+\sqrt{2(|x|-1)}\right)^d \leq|\cheb{d}(x)|\leq |2x|^d,
	\]
	and $|\cheb{d}(x)|$ for such $x$ is a monotonically increasing function of $|x|$.
\end{enumerate} 

To construct a polynomial $p$ with $p(0)=1$ that is small on a single interval, we can simply compose Chebyshev polynomials with a linear function that maps our interval onto $[-1,1]$ and then normalize to get $p(0)=1$.  

To this end, let 
\[
	\rescale{a}{b}(x)=\frac{b+a-2x}{b-a}
\]
be the linear function that maps $[a,b]$ to $[-1,1]$ with $\rescale{a}{b}(a)=1$ and $\rescale{a}{b}(b)=-1$, and define
\[
	\damp{a}{b}{d}(x)=\frac{\cheb{d}(\rescale{a}{b}(x))}{\cheb{d}(\rescale{a}{b}(0))}. 
\] 

\begin{lemma}\label{lem:damping_bound}
For any $b\geq 2a>0$, if $\damp{a}{b}{d}(x)$ is the polynomial defined above, and $\kappa=b/a$,
then 
\[
	\damp{a}{b}{d}(0)=1
\]
 and 
\[
\left|\damp{a}{b}{d}(x)\right|\leq 
\begin{cases}
1 & \text{if $x\in [0,a]$}\\
2\left(1+\frac{2}{\sqrt{\kappa}}\right)^{-d} & \text{if $x\in [a,b]$}\\
\left(\frac{ 8x}{ b}\right)^d  & \text{if $x\geq b$}.
\end{cases}
\] 
\end{lemma}

\begin{proof}
	The fact that $\damp{a}{b}{d}(0)=1$ follows immediately from the definition.

	For the bound on $|\damp{a}{b}{d}(x)$, lets first suppose that $x\in[0,a]$.
	In this case, $\rescale{a}{b}(x)\in \left[1,\rescale{a}{b}(0)\right]$ with $\rescale{a}{b}(0)\geq 1$.  
	By the monotonicity assertion in property~\ref{eq:chebprop2} of Chebyshev polynomials,
	$\cheb{d}(\rescale{a}{b}(x))\leq \cheb{d}(\rescale{a}{b}(0))$, so 
	$\left|\damp{a}{b}{d}(x)\right|\leq \left|\damp{a}{b}{d}(0)\right|=1$, as claimed.

	Now, suppose $x\in[a,b]$.
	In this case, note that $\rescale{a}{b}(x)\in [-1,1]$, so $|\cheb{d}(\rescale{a}{b}(x))|\leq 1$ by property~\ref{eq:chebprop1} of Chebyshev polynomials.
	For the denominator, we have 
	\[
		\rescale{a}{b}(0)= \frac{b+a}{b-a} 
		=1 + \frac{2a}{b-a}
		=1 + \frac{2}{\kappa-1},
	\]
	so, by property~\ref{eq:chebprop2} of Chebyshev polynomials and the fact that $\cheb{d}(x)\geq 0$ for $x\geq 1$,
	\[
		\cheb{d}(\rescale{a}{b}(0))
		=\cheb{d}\left(1 + \frac{2}{\kappa-1} \right)
	 \geq \frac{1}{2}\left(1+\sqrt{2\left(\frac{2}{\kappa-1}\right)}\right)^d 	
	 \geq \frac{1}{2}\left(1+\frac{2}{\sqrt{\kappa}}\right)^d. 	
	\]
	Combining our bounds on the numerator and denominator gives the asserted bound on $\left|\damp{a}{b}{d}(x)\right|$ for $x\in [a,b]$.
	
	Finally, suppose $x\geq b$, and let $\gamma =x/b \geq 1$.  We have $|\rescale{a}{b}(x)|\geq 1$, so, by property~\ref{eq:chebprop2} of Chebyshev polynomials,
\begin{align}
		& |\cheb{d}(\rescale{a}{b}(x))| 
		\leq \left|2\rescale{a}{b}(x))\right|^d
		 =\left|2\left(\frac{b+a-2x}{b-a}\right)\right|^d
		 =\left|2\left(\frac{\kappa+1-2x/a}{\kappa-1}\right)\right|^d	
		 \\
		=& \left|2\left(\frac{\kappa+1-2\kappa \gamma}{\kappa-1}\right)\right|^d	
		=\left|-2\left(1+2(\gamma-1)\left(1+\frac{1}{\kappa-1}\right)\right)\right|^d		
		=\left|2+4(\gamma-1)\left(1+\frac{1}{\kappa-1}\right)\right|^d.	
\end{align}
Our assumption that $b\geq 2a$ implies that $\kappa\geq 2$, so we have
\[
	|\cheb{d}(\rescale{a}{b}(x))| \leq \left| 2+8(\gamma-1)\right|^d
	= \left| 8\gamma-6\right|^d
		= \left| \frac{8x}{b}-6\right|^d
		\leq \left( \frac{8x}{b}\right)^d. 
\]
Combining this with the fact that $\cheb{d}(\rescale{a}{b}(0))\geq 1$ by the monotonicity in property~\ref{eq:chebprop2} gives the desired bound on $\left|\damp{a}{b}{d}(x)\right|$ for $x\geq b$.
\end{proof}

\begin{proof}[Proof of Theorem~\ref{thm:poly_construction}]
To simplify the calculations, we assume that $\kappa_i\geq 2$ for all $i$.
By enlarging and combining our intervals as necessary, we can easily reduce the general theorem to this case.

Let $\epsdeg{\kappa}{\eps}=\sqrt{\kappa}\lceil\log(2/\epsilon)\rceil$, and note that, for $\kappa\geq 2$ and $\epsilon>0$,
 \[
 	2\left(1+\frac{2}{\sqrt{\kappa}}\right)^{-\epsdeg{\kappa}{\eps}}<\epsilon.
 \]

Let $S = \bigcup_{i \in [m]} [\mu_i, L_i]$, where $\mu_1<L_1< \mu_{2}<L_2<\dots$, and assume that $\kappa_i:=L_i/\mu_i \geq 2$ for all $i$.

We will obtain a good polynomial for $S$ by multiplying together the polynomials for the different intervals with carefully-chosen degrees.
To this end, let

\[
	\globalpoly{d_1,\dots d_{m}}(x)
	  =\prod_{i\in [m]} \damp{\mu_i}{L_i}{d_i}(x),	
\]
and note that $\globalpoly{d_1,\dots d_{m}}(0)=1$.

By Lemma~\ref{lem:damping_bound}, for $x\in [\mu_j,L_j]$, we have 
\begin{align}
	\left|\globalpoly{d_1,\dots d_{m}}(x)\right|
	&=\prod_{i\in [m]}\left| \damp{\mu_i}{L_i}{d_i}(x)\right|
	\leq 
	\prod_{i\in [m]}
		\begin{cases}
			1 & \text{if $x\in [0,\mu_i]$}\\
			2\left(1+\frac{2}{\sqrt{\kappa_i}}\right)^{-d_i} & \text{if $x\in [\mu_i,L_i]$}\\
			\left(\frac{ 8x}{L_i}\right)^{d_i}   & \text{if $x\geq L_i$}
		\end{cases}\\	
	&=
	\prod_{i<j}\left(\frac{ 8x}{L_i}\right)^{d_i}
	  \cdot 2\left(1+\frac{2}{\sqrt{\kappa_j}}\right)^{-d_j}
	  \cdot \prod_{i>j} 1
	=2\left(1+\frac{2}{\sqrt{\kappa_j}}\right)^{-d_j}\prod_{i<j}\left(\frac{ 8x}{L_i}\right)^{d_i}.
\end{align}
If we want $\left|\globalpoly{d_1,\dots d_{m}}(x)\right|\leq \epsilon$ for all $x\in S$, it thus suffices to choose the $d_j$ so that, for all $j$,
\[
	2\left(1+\frac{2}{\sqrt{\kappa_j}}\right)^{-d_j}\leq \epsilon \cdot \prod_{i<j}\left(\frac{ 8x}{L_i}\right)^{-d_i}.
\]
We can achieve this by setting $d_1=\epsdeg{\kappa_1}{\epsilon}=\sqrt{\kappa_1}\lceil\log(2/\epsilon)\rceil$ and then recursively setting
\begin{align}
	d_j &= \epsdeg{\kappa_j}{ \epsilon \cdot \prod_{i<j}\left(\frac{ 8x}{L_i}\right)^{-d_i}}\\
&=\sqrt{\kappa_j}\left\lceil\log\left( (2/\epsilon) \cdot \prod_{i<j}\left(\frac{ 8x}{L_i}\right)^{d_i}\right)\right\rceil
=\sqrt{\kappa_j}
	\left\lceil
		\log\left({2}/{\epsilon}\right) +\sum_{i<j}d_i \log\left(\frac{ 8x}{L_i}\right)\right\rceil
\end{align}
Since $x/L_i \leq \globalcond/\kappa_1$ and $d_1\geq \sqrt{\kappa_1} \log(2/\epsilon)$, and using the fact that $\sqrt{\kappa_1}\log(9\kappa_1/8)\geq 1$ for $\kappa_1\geq 2$, 
 we have the bound
\begin{align}
	d_j
	&\leq \sqrt{\kappa_j}
	\left\lceil
		\log\frac{2}{\epsilon} +\log\frac{8 \globalcond}{\kappa_1}\sum_{i<j}d_i \right\rceil
	= \sqrt{\kappa_j}
	\left\lceil
		\log\frac{2}{\epsilon}
		-\log \frac{9\kappa_1}{8}\sum_{i<j}d_i
		+\log(9 \globalcond)\sum_{i<j}d_i \right\rceil\\
	&\leq \sqrt{\kappa_j}
	\left\lceil
		\log\frac{2}{\epsilon}
		-\log \frac{9\kappa_1}{8}\sqrt{\kappa_1} \log\frac{2}{\epsilon}
		+\log(9 \globalcond)\sum_{i<j}d_i \right\rceil	
	\leq \sqrt{\kappa_j}
	\left\lceil
		\log(9 \globalcond)\right\rceil\sum_{i<j}d_i .
\end{align}

Since our recurrence guarantees that $d_{i+1}\geq 2d_{i}$, we have $\sum_{i<j}d_i \leq 2d_{j-1}$, so our bound becomes
\[
d_j\leq\sqrt{\kappa_j}\cdot	2\left\lceil	\log(9 \globalcond) \right\rceil d_{j-1}.
\]
Applying this recursively gives
\begin{align}
d_j
&\leq \left(\prod_{i=2}^j\sqrt{\kappa_j}\right)
      \left(2\left\lceil\log(9 \globalcond) \right\rceil\right)^{j-1} d_1
= \left(\prod_{i=2}^j\sqrt{\kappa_j}\right)
      \left(2\left\lceil\log(9 \globalcond) \right\rceil\right)^{j-1} 
	\left(\sqrt{\kappa_1}\lceil\log(2/\epsilon)\rceil\right)\\
&=\left(\prod_{i=1}^j\sqrt{\kappa_j}\right)
      \left(2\left\lceil\log(9 \globalcond) \right\rceil\right)^{j-1} 
	\lceil\log(2/\epsilon)\rceil,
\end{align}
and the total degree is then bounded by $2d_{m}$, which obeys the desired asymptotic bound.
\end{proof}

\ifarxiv
\section*{Acknowledgements}
We thank the anonymous reviewers for their suggestions and comments on earlier versions of this paper.  
JK was supported in part by NSF awards CCF-1955217,  CCF-1565235, and DMS-2022448.
AS was supported in part by a Microsoft Research Faculty Fellowship, NSF CAREER Award CCF-1844855, NSF Grant CCF-1955039, a PayPal research award, and a Sloan Research Fellowship. GV and AM were supported by NSF awards 1704417 and 1813049, ONR Young Investigator Award N00014-18-1-2295 and DOE award DE-SC0019205. HY was supported in part by the TOTAL Innovation Scholars program. 
\fi

\bibliographystyle{alpha}

\begin{small}
\bibliography{refs}
\end{small}

\clearpage
\appendices
\section{Proof of $\bsls$ under finite-precision arithmetic}
\label{apx:bsls:inexact}
In this section we prove $\bsls$ with finite-precision arithmetic (subject to \cref{req:gd}). 

In \cref{thm:bsls:inexact}, we specialized our initialization of $\x^{(0)}$ to $\vec{0}$ to simplify the exposition of the theorem. 
In fact, we can (and will) prove the following general (but less clean) version with arbitrary $\x^{(0)}$.
\begin{theorem}[$\bsls$ under finite-precision arithmetic, general initialization]
	\label{thm:bsls:inexact:general}
	Consider multiscale optimization problem defined in \Cref{def:multiscale_problem}, for any initialization $\x^{(0)}$ and $\epsilon > 0$, assuming \cref{req:gd} with 
	\begin{equation}
		\delta^{-1} \geq \left(\prod_{i \in [m]} T_i \right) \max \left\{
			(10  \globalcond )^{2m-1} m, 
		  (10  \globalcond )^{2m-1} m  \cdot \frac{f(\x^{(0)}) - f^{\star} }{\epsilon} , 
		  4^{m+1} m L_m \frac{\|\x^{\star}\|_2^2 }{\epsilon}
		\right\}
	\end{equation}
	then 
	$f(\widehat{\bsls_{1}}(\x^{(0)})) - f^{\star}\leq 3 \epsilon$ provided that $T_1, \ldots, T_m$ satisfy 
	\begin{equation}
		T_1 \geq \kappa_1 \log \left( \frac{f(\x^{(0)})- f^{\star}}{\epsilon} \right);
		\qquad
		T_i \geq \kappa_i (2 \log (\globalcond) + 1), \quad \text{ for $i = 2, \ldots, m$}.
		\label{eq:bsls:t:lb}
	\end{equation}
	We can also achieve the same asymptotic sample complexity (up to constant factors suppressed in the $\bigo(\cdot)$) when $\{(\mu_i, L_i), i \in [m]\}$ are unknown and only $m$, $\mu_{1}$, $L_m$ and $ \pi_{\kappa}=\prod_{i=1}^m \kappa_i$ are known.
\end{theorem}
\cref{thm:bsls:inexact} is clearly a corollary of \cref{thm:bsls:inexact:general}.
\begin{proof}[Proof of \cref{thm:bsls:inexact} based on \cref{thm:bsls:inexact:general}]
    Follows by the fact that
    \begin{equation}
	4^{m+1} m L_m \frac{\|\x^{\star}\|_2^2}{\epsilon} \leq 4^{m+2} m \globalcond \cdot \frac{f(\vec{0}) - f^{\star}}{\epsilon} 
	\leq 
	m (10 \globalcond)^{2m-1} \cdot \frac{f(\vec{0}) - f^{\star}}{\epsilon}.
\end{equation}
\end{proof}

From now on we focus on the proof of  \cref{thm:bsls:inexact:general}. The proof of \cref{thm:bsls:inexact:general} is structured as follows. 
We first study the progress of one (inexact) $\widehat{\gd}$ step in \Cref{sec:bsls:inexact:1}, 
and inductively estimate the progress of $\widehat{\bsls_i}$ by matrix inequalities for all $i \in [m]$ in descent order (see \Cref{sec:bsls:inexact:2}).
The proof of \Cref{thm:bsls:inexact:general} is then finished in \Cref{sec:proof:thm:bsls:inexact}. Note that the last part regarding the case where $\{(\mu_i,L_i), i \in [m]\}$ are unknown follows from our black-box reduction in Proposition \ref{prop:search} (in the same way as in the proof of Theorem \ref{thm:bsls:recursive}).

\paragraph{Additional notation.}
We introduce notation to simplify the exposition. 
For any $\x$ and $i \in [m]$, define function $\err_i (\x) \defeq f_i(\proj_i \x) - f_i^{\star}$. 
Define vector potentials 
\begin{equation}
	\vec{\err}(\x) \defeq \begin{bmatrix}
		 \err_1(\x),
		 \err_2(\x),
		\cdots
		,
		 \err_m(\x)
	\end{bmatrix}^\top 
	\in \reals^m.
	\label{eq:bsls:r:def} 
\end{equation}
We will monitor the progress via this vector potential $\vec{\err}$. 
Recall that when comparing two vectors or matrices, we use plain inequalities $(\leq, \geq)$ to denote element-wise inequality. 

\subsection{Progress of one $\gd$ step under finite arithmetic}
\label{sec:bsls:inexact:1}
In this subsection, we study the effect of one (inexact) gradient step $\widehat{\gd}$ on the vector potential $\vec{\err}$ under finite arithmetic (subject to \cref{req:gd}). 
The goal is to establish the following \cref{lem:gd:finite}. 
\begin{lemma}[Progress of one $\widehat{\gd}$ step under finite arithmetic]
	\label{lem:gd:finite}
	Consider multiscale optimization (Def. \ref{def:multiscale_problem}), and assuming \cref{req:gd}, then for any $\x$ and $i \in [m]$, the following inequality holds
	\begin{equation}
		\vec{\err}(\widehat{\gd} (\x;L_i)) \leq \left(\id + 5 \delta \globalcond \ones \ones^\top \right) \mat{D}_i \vec{\err}(\x) + 2 \delta \|\x^{\star}\|_2^2 L_m \cdot  \ones,
		\label{eq:lem:gd:finite}
	\end{equation}
	where $\mat{D}_i$ is an $m \times m$ diagonal matrix defined by
	\begin{equation}
		(\mat{D}_i)_{jj} = \begin{cases}
			1	 & \text{ if } j < i,
			\\
			1 - \kappa_i^{-1} & \text{ if } j = i,
			\\
			 \globalcond^2 & \text{ if } j > i.
		\end{cases}
		\label{eq:bsls:D:def}
	\end{equation}
	To simplify the notation
	we will define (throughout this section) that 
	\begin{equation}
		\widehat{\mat{D}_i} \defeq \left(\id + 5 \delta \globalcond \ones \ones^\top \right) \mat{D}_i.
		\label{eq:bsls:Dhat:def}
	\end{equation}
	Then \cref{eq:lem:gd:finite} becomes 
	\begin{equation}
		\vec{\err}(\widehat{\gd} (\x;L_i)) \leq \widehat{\mat{D}_i} \vec{\err}(\x) + 2 \delta \|\x^{\star}\|_2^2 L_m \cdot  \ones.
	\end{equation}
\end{lemma}
\begin{remark}
	The key observation from \Cref{lem:gd:finite} is that under finite-precision arithmetic, the function error in $j$-th subspace (i.e., $\err_j$) also depends on the the errors from other subspace, as well as an constant additive term. 
	If the error in one of the subspaces is too large, it could flow into the other subspaces and ruins the progress elsewhere. 
  Consequently, the order of step-size schedule is crucial in finite-precision arithmetic.
\end{remark}

To prove \cref{lem:gd:finite}, we first study the sensitivity of potential $\vec{\err}$ under multiplicative perturbtaion. 
\begin{lemma}[Sensitivity of vector potential $\vec{\err}$ under multiplicative perturbation]
	\label{lem:implication:req:gd:1}
	Assuming $\widehat{\x}, \x$ satisfies 
	\begin{equation}
		|\widehat{\x} - {\x}| \leq \delta|\x|
	\label{eq:req:gd}
	\end{equation}
	for some $\delta < 1$,	then for any $j \in [m]$, 
	\begin{equation}
		 \err_j  (\widehat{\x}) \leq \err_j (\x) +  5 \delta \globalcond \sum_{k=1}^m \err_k(\x) + 2 \delta L_j \|\x^{\star}\|_2^2.
	\end{equation}
	In vector form we have (in a looser form)
	\begin{equation}
		\vec{\err}(\widehat{\x}) \leq (I + 5 \delta \globalcond \ones \ones^\top) \vec{\err}(\x) + 2 \delta L_m \|\x^{\star}\|_2^2  \ones.
	\end{equation}
\end{lemma}

\begin{proof}[Proof of \Cref{lem:implication:req:gd:1}]
		For any $j \in [m]$,
		\begin{align}
			& \err_j(\widehat{\x}) = f_j(\proj_j \widehat{\x}) - f_j^{\star}
    			\tag{by definition of $\err_j$}
			\\
			\leq & f_j (\proj_j \x) - f_j^{\star} + \left\langle \nabla f_j(\proj_j \x), \proj_j( \widehat{\x} -  \x ) \right\rangle + \frac{L_j}{2} \|\proj_j (\widehat{\x} -  \x )\|_2^2
			\tag{by $L_j$-smoothness of $f_j$}
			\\
			\leq & f_j (\proj_j \x) - f_j^{\star} +
			\frac{\delta}{2 L_j}
			\left\|\nabla f_j(\proj_j \x)  \right\|_2^2 
			+
			\frac{L_j}{2 \delta}
			\left\| \proj_j( \widehat{\x} -  \x )  \right\|_2^2
			 + \frac{L_j}{2} \|\proj_j (\widehat{\x} -  \x )\|_2^2
			 \tag{by Cauchy-Schwartz inequality}
			 \\
			\leq & f_j (\proj_j \x) - f_j^{\star} +
			\frac{\delta}{2 L_j}
			\left\|\nabla f_j(\proj_j \x)  \right\|_2^2 
			+
			\frac{L_j}{\delta}
			\left\| \proj_j( \widehat{\x} -  \x )  \right\|_2^2
			\tag{since $\delta \leq 1$}
			\\
			\leq & (1 + \delta) (f_j (\proj_j \x) - f_j^{\star}) +
			\frac{L_j}{\delta}
			\left\| \proj_j( \widehat{\x} -  \x )  \right\|_2^2.
			\tag{by $L_j$-smoothness of $f_j$}
			\\
			\leq & (1 + \delta) \err_j(\x) + \frac{L_j}{\delta} 			\| \widehat{\x} -  \x \|_2^2
			\tag{by definition of $\err_j$}
		\end{align}
		By assumption \cref{eq:req:gd}
		\begin{equation}
			\| \widehat{\x} -  \x \|_2^2
			\leq 
			\delta^2 \|\x\|_2^2
			\leq
			2 \delta^2 \| \x - \x^{\star} \|_2^2 + 2 \delta^2 \| \x^{\star} \|_2^2,
			\tag{by Cauchy-Schwartz inequality}
		\end{equation}
		and strong convexity of $f_i$'s
		\begin{equation}
			\| \x - \x^{\star} \|_2^2 
			=
			\sum_{i \in [m]} \| \proj_i  (\x - \x^{\star} )\|_2^2 
			\leq
			\sum_{i \in [m]} \frac{2}{\mu_i} \err_i (\x),
		\end{equation}
		we arrive at
		\begin{equation}
			\frac{L_j}{\delta} \| (\widehat{\x} -  \x) \|_2^2 
			\leq 
			2 \delta L_j \|\x^{\star}\|_2^2 +	\sum_{i \in [m]} \frac{4 \delta L_j}{\mu_i} \err_i(\x)
			\leq 
			2 \delta L_j \|\x^{\star}\|_2^2 + 4 \delta \globalcond	\sum_{i \in [m]} \err_i(\x),
		\end{equation}
		where the last inequality is due to $\frac{L_j}{\mu_i} \leq \globalcond$ by definition of $\globalcond$. 
		In summary		
		\begin{align}
			\err_j (\widehat{\x}) & \leq (1 + \delta) \err_j (\x) + 4 \delta \globalcond \sum_{i \in [m]} \err_i(\x) + 2 \delta L_j \|\x^{\star}\|_2^2
			\\
			& \leq
			\err_j (\x) + 5 \delta \globalcond \sum_{i \in [m]} \err_i(\x) + 2 \delta L_j \|\x^{\star}\|_2^2.
		\end{align}
		In vector form we have (since $L_1 \leq L_2 \leq \cdots \leq L_m$)
		\begin{equation}
			\vec{\err}(\widehat{\x}) \leq (I + 5 \delta \globalcond \ones \ones^\top) \vec{\err}(\x) + 2 \delta L_m \|\x^{\star}\|_2^2 \ones,
		\end{equation}
		completing the proof.
\end{proof}
With \Cref{lem:implication:req:gd:1} at hands we are ready to prove \Cref{lem:gd:finite}:
\begin{proof}[Proof of \Cref{lem:gd:finite}]
	Apply \cref{lem:implication:req:gd:1}, we have
	\begin{equation}
		\vec{\err}(\widehat{\gd} (\x;L_i)) \leq (\id + 5 \delta \globalcond \ones \ones^\top) \vec{\err}(\gd (\x;L_i)) + 2 \delta L_m \|\x^{\star}\|_2^2 \cdot \ones.
	\end{equation}
	By \Cref{lem:BSLS:general} from exact \bsls\ analysis we have
	\begin{equation}
		\vec{\err}({\gd} (\x;L_i)) \leq \mat{D}_i \vec{\err}(\x).
	\end{equation}
	Combining the two inequalities above yields \cref{lem:implication:req:gd:1}.
\end{proof}

\subsection{Inductively bound the progress of inexact $\bsls$ by matrix inequalities}
\label{sec:bsls:inexact:2}
In the following \cref{lem:bsls:stab:abs}, we iteratively construct the bound of $\vec{\err}$ after executing $\widehat{\bsls_i}$.
\begin{lemma}[Estimate the progress of $\widehat{\bsls_i}$ by matrix inequalities]
	\label{lem:bsls:stab:abs}
	Considering multiscale optimization problem (\Cref{def:multiscale_problem}), and assuming \cref{req:gd}, define the following three sequences of $m \times m$  matrices $\{\mat{F}_i\}_{i=1}^{m+1}$, $\{\mat{E}_i\}_{i=1}^{m+1}$, $\{\mat{Z}_i\}_{i=1}^{m+1}$ as follows:
	\begin{equation}
		\mat{F}_{m+1} \defeq \id, \qquad \mat{E}_{m+1} \defeq \mat{0}, \qquad \mat{Z}_{m+1} \defeq \mat{0}
	\end{equation}
	and for $i = m, m-1$ down to $1$, define
	\begin{equation}
		\mat{F}_i \defeq \left( \mat{F}_{i+1} \mat{D}_i \right)^{T_i} \mat{F}_{i+1}, \quad 
		\mat{E}_i \defeq 
		\left( (\mat{F}_{i+1} + \mat{E}_{i+1})\widehat{\mat{D}_{i}} \right)^{T_i} 	(\mat{F}_{i+1} + \mat{E}_{i+1}) 
		- \left( \mat{F}_{i+1} \mat{D}_i \right)^{T_i} \mat{F}_{i+1}.
	\end{equation}
	and
	\begin{equation}
		\mat{Z}_i \defeq 	((\mat{F}_{i+1} + \mat{E}_{i+1}) \widehat{\mat{D}_i})^{T_i} \mat{Z}_{i+1} + 
			\sum_{t_i=0}^{T_i-1}
			\left( (\mat{F}_{i+1} + \mat{E}_{i+1})\widehat{\mat{D}_{i}} \right)^{t_i} \left( \mat{Z}_{i+1} + \mat{F}_{i+1} + \mat{E}_{i+1} \right),
	\end{equation}
	where $\mat{D}_i$ and $\widehat{\mat{D}_i}$ were defined in \cref{eq:bsls:D:def,eq:bsls:Dhat:def}.
	Then,
	\begin{enumerate}[(a),leftmargin=*]
		\item For any $i \in [m+1]$, $\mat{F}_i$, $\mat{E}_i$ and $\mat{Z}_i$ are non-negative matrices.
		\item For any $i \in [m]$, the following bound holds
		\begin{equation}
			\vec{\err}(\widehat{\bsls_{i}}(\x)) \leq (\mat{F}_i + \mat{E}_i) \vec{\err}(\x) + 2 \delta  L_m \|\x^{\star}\|_2^2 \mat{Z}_i  \ones.
		\end{equation} 
	\end{enumerate}
\end{lemma}
\begin{proof}[Proof of \Cref{lem:bsls:stab:abs}]
	\begin{enumerate}[(a), leftmargin=0pt, itemindent=15pt]
		\item 	We first prove (a) by induction in reverse order (from $m+1$ down to $1$). 

		For $i = m+1$ the statement apparently holds.
		Now assume (a) holds for $i+1$, then we study the case of $i$.
		For $\mat{F}_i$ we have $\mat{F}_i = (\mat{F}_{i+1}  \mat{D}_i)^{T_i} \mat{F}_{i+1} \geq \mat{0}$ since both $\mat{F}_{i+1}$ and $\mat{D}_i$ are non-negative.
		For $\mat{E}_i$ we have
		\begin{equation}
			\mat{E}_i = ((\mat{F}_{i+1} + \mat{E}_{i+1}) (\id + 5 \delta \globalcond \ones \ones^\top) \mat{D}_i)^{T_i} (\mat{F}_{i+1} + \mat{E}_{i+1}) - (\mat{F}_{i+1} \mat{D}_i)^{T_i} \mat{F}_{i+1} \geq \mat{0}.
		\end{equation}
	
		The non-negativity of $\mat{Z}_i$ is obvious from the non-negativity of $\mat{F}_i$ and $\mat{E}_i$.
		
		\item 
		Next, we prove (b) by induction in reverse order. 
		To simplify the induction let us define $\widehat{\bsls_{m+1}} \defeq \mathrm{Id}$ and prove (b) for all $i \in [m+1]$. 
		The inequality holds trivially for $i = m+1$. 
		Now assume the inequality holds for the case of $i+1$, then we study the case of $i$.

		By definition of $\widehat{\bsls_{i}}$ (including $i = m$),
		\begin{small}
		\begin{equation}
			\widehat{\bsls_i} = \left( \widehat{\bsls_{i+1}} \circ \widehat{\gd}(\cdot;L_i) \right)^{T_i} \circ  \widehat{\bsls_{i+1}}
			=
			\underbrace{\widehat{\bsls_{i+1}} \circ \widehat{\gd}(\cdot;L_i) \circ \cdots \circ \widehat{\bsls_{i+1}} \circ \widehat{\gd}(\cdot;L_i)}_{\text{$T_i$ iterations of $\widehat{\bsls_{i+1}} \circ \widehat{\gd}(\cdot;L_i)$}}
			\circ \widehat{\bsls_{i+1}}
		\end{equation}
		\end{small}
		
		By induction hypothesis, for any $\x$,
		\begin{equation}
			\vec{\err}(\widehat{\bsls_{i+1}} (\x)) \leq (\mat{F}_{i+1} + \mat{E}_{i+1}) \vec{\err}(\x) + 2 \delta L_m \|\x^{\star}\|_2^2 \cdot \mat{Z}_{i+1} \ones.
		\end{equation}
		For $\widehat{\gd}(\cdot;L_i)$ step we have by \Cref{lem:gd:finite} (for any $\x$)
		\begin{equation}
			\vec{\err}(\gd(\x; L_i)) \leq \widehat{\mat{D}_{i}} \vec{\err}(\x) + 2 \delta L_m \|\x^{\star}\|_2^2 \ones.
		\end{equation}
		Combining the above two inequalities, we obtain 
		\begin{equation}
			\vec{\err}(\widehat{\bsls_{i+1}} (\widehat{\gd}(\x; L_i))) \leq (\mat{F}_{i+1} + \mat{E}_{i+1})\widehat{\mat{D}_{i}} \vec{\err}(\x) + 2 \delta L_m \|\x^{\star}\|_2^2 \left( \mat{Z}_{i+1} + \mat{F}_{i+1} + \mat{E}_{i+1} \right) \ones.
		\end{equation} 
		Telescoping
		\begin{small}
		\begin{align}
				& \vec{\err}(\widehat{\bsls_{i}}(\x))
			\\ 
			\leq & \left( (\mat{F}_{i+1} + \mat{E}_{i+1})\widehat{\mat{D}_{i}} \right)^{T_i} 
			(\mat{F}_{i+1} + \mat{E}_{i+1}) \vec{\err}(\x) 
			\\
			& + 2 \delta L_m \|\x^{\star}\|_2^2 
		     \left[ 
			((\mat{F}_{i+1} + \mat{E}_{i+1}) \widehat{\mat{D}_i})^{T_i} \mat{Z}_{i+1} + 
			\sum_{t_i=0}^{T_i-1}
			\left( (\mat{F}_{i+1} + \mat{E}_{i+1})\widehat{\mat{D}_{i}} \right)^{t_i} \left( \mat{Z}_{i+1} + \mat{F}_{i+1} + \mat{E}_{i+1} \right) \right] \ones 
			\\
			= & (\mat{F}_i + \mat{E}_i) \vec{\err}(\x) + 2 \delta  L_m \|\x^{\star}\|_2^2 \mat{Z}_i  \ones.
			\tag{by definition of $\mat{F}_i$, $\mat{E}_i$ and $\mat{Z}_i$}
		\end{align}
		\end{small}
	\end{enumerate}
\end{proof}

Next, we estimate the upper bounds of $\|\mat{F}_i\|_1$ (in \Cref{sec:lem:bsls:stab:F:bound}), $\|\mat{E}_i\|_1$ (in \Cref{sec:lem:bsls:stab:E:bound}), and $\|\mat{Z}_i\|_1$ (in \Cref{sec:lem:bsls:stab:Z:bound}).

\subsubsection{Upper bound of $\mat{F}$}
\label{sec:lem:bsls:stab:F:bound}
We first bound $\|\mat{F}_i\|_1$ with the following lemma.
\begin{lemma}[Upper bound of $\|\mat{F}_i\|_1$]
	\label{lem:bsls:stab:F:bound}
	Using the same notation as in \Cref{lem:bsls:stab:abs}, and in addition assuming $T_1, \ldots, T_m$ satisfies \cref{eq:bsls:t:lb}, then the following statements hold
	\begin{enumerate}[(a), leftmargin=*]
		\item For any $i \in [m+1]$, $\mat{F}_i$ is a diagonal matrix of the form $\prod_{j=i}^m \left(  \mat{D}_j^{T_j \cdot \prod_{k=i}^{j-1} (T_k+1)} \right)$.
		\item For any $i \in [m]$, $\|\mat{F}_{i+1} \mat{D}_i\|_1 \leq 1$.
		\item For any $i \in [m]$, $\|\mat{F}_i\|_1 \leq 1$.
		\item $\|\mat{F}_1\|_1 \leq \frac{\epsilon}{f(\x^{(0)}) - f^{\star}}$.
	\end{enumerate}
\end{lemma}
\begin{proof}[Proof of \Cref{lem:bsls:stab:F:bound}]
	\begin{enumerate}[(a), leftmargin=0pt, itemindent=15pt]
		\item  The first statement (a) follows immediately by definition of $\mat{F}_i$'s. 
		We prove by induction in reverse order from $m+1$ down to $1$. 
		For $i=m+1$ we have $\mat{F}_{m+1} = \id$ which is consistent.
		Now assume the statement holds for the case of $i+1$, then the case of $i$ also holds in that
		\begin{small}
		\begin{equation}
			\mat{F}_i = (\mat{F}_{i+1} \mat{D}_i)^{T_i} \mat{F}_{i+1} = \left(\prod_{j=i+1}^m \left( \mat{D}_j^{T_j \prod_{k=i+1}^{j-1} (T_k+1)} \right) \cdot \mat{D}_i \right)^{T_i} 
			\prod_{j=i+1}^m \left( \mat{D}_j^{T_j \prod_{k=i+1}^{j-1} (T_k+1)} \right) 
			= \prod_{j=i}^m \left(  \mat{D}_j^{T_j \cdot \prod_{k=i}^{j-1} (T_k+1)} \right),
		\end{equation}
		\end{small}
		where the last equality is due to the commutability among diagonal matrices.
		\item Follows by the same analysis as in the exact arithmetic  proof in \Cref{thm:bsls:recursive}.
		\item By (b), $\|\mat{F}_{i}\|_1 = \|(\mat{F}_{i+1} \mat{D}_i)^{T_i} \mat{F}_{i+1}\|_1 \leq \|(\mat{F}_{i+1} \mat{D}_i)\|_1^{T_i} \|\mat{F}_{i+1}\|_1 \leq 1$.
		\item Follows by the same analysis as in the exact arithmetic  proof in \Cref{thm:bsls:recursive}.
	\end{enumerate}
\end{proof}

\subsubsection{Upper bound of $\mat{E}$}
\label{sec:lem:bsls:stab:E:bound}
Before we state the upper bound of $\|\mat{E}_i\|_1$, we first establish the following \Cref{lem:bsls:stab:diff:bound},  which is essential towards the bound for $\mat{E}_i$'s and $\mat{Z}_i$'s.
\begin{lemma}
	\label{lem:bsls:stab:diff:bound}
	Using the same notation of \Cref{lem:bsls:stab:abs} and assuming the same assumptions of \Cref{lem:bsls:stab:F:bound}, and in addition assume $\delta \leq \frac{1}{10 m\globalcond}$, then
	for any $t \geq 0$, the following inequality holds
	\begin{align}
		\left\| ((\mat{F}_{i+1} + \mat{E}_{i+1}) \widehat{\mat{D}_i})^{t} - (\mat{F}_{i+1} \mat{D}_i)^{t} \right\|_1 
		\leq
		\begin{cases}
			\varphi(5  t \delta m \globalcond  ) & i = m
			\\
			\varphi(2 t \globalcond^2  \|\mat{E}_{i+1}\|_1 + 5 t \delta m  \globalcond^3  )   & i < m,
		\end{cases}
	\end{align}	
	where $\varphi(x) \defeq xe^x$.
\end{lemma}
\begin{proof}[Proof of \Cref{lem:bsls:stab:diff:bound}]
	Denote $ \bXi_i \defeq (\mat{F}_{i+1} + \mat{E}_{i+1}) \widehat{\mat{D}_i} - \mat{F}_{i+1} \mat{D}_i$, then
	\begin{align}
		&  \left\| ((\mat{F}_{i+1} + \mat{E}_{i+1}) \widehat{\mat{D}_i})^{t} - (\mat{F}_{i+1} \mat{D}_i)^{t}  \right\|
		=  \left\| \left( (\mat{F}_{i+1} \mat{D}_i + \bXi_i \right)^t 
		- (\mat{F}_{i+1} \mat{D}_i)^{t}  \right\|_1
		\tag{by definition of $\bXi_i$}
		\\
		\leq & \sum_{s=1}^{t} {t \choose s} \|\bXi_i\|_1^s		\left\| \mat{F}_{i+1} \mat{D}_i \right\|_1^{t-s} 
		\leq \sum_{s=1}^{t} {t \choose s} \|\bXi_i\|_1^s
		\tag{since $\left\| \mat{F}_{i+1} \mat{D}_i \right\|_1 \leq 1$ by \Cref{lem:bsls:stab:F:bound}}
		\\
		\leq &  \|\bXi_i\|_1 t \sum_{s=0}^{t-1} {t-1\choose s} \|\bXi_i\|_1^s
		\tag{by helper \Cref{helper:choose}}
		\\
		= & \|\bXi_i\|_1 t \left( 1 + \|\bXi_i\|_1 \right)^{t-1} 
		\leq 
		\|\bXi_i\|_1 t \exp (\|\bXi_i\|_1 t).
		\label{eq:bsls:stab:diff:bound:1}
	\end{align}
	It remains to bound $\|\bXi_i\|_1$. 
	For $i = m$ we have $\mat{E}_{m+1} = 0$, $\mat{F}_{m+1} = \id$, $\|\mat{D}_m\|_1 \leq 1$, which suggests
	\begin{equation}
		\|\bXi_m\|_1 = \|\widehat{\mat{D}_m} - \mat{D}_m\|_1  \leq 5 \delta \globalcond \|\ones \ones^\top\|_1 = 5 \delta m \globalcond .
	\end{equation}

	For other $i < m$, note that $\|\mat{D}_i\|_1 \leq  \globalcond^2$, $\| \mat{F}_{i+1}\|_1 \leq 1$ (by \Cref{lem:bsls:stab:F:bound}), $\|\ones \ones^\top\|_1 = m$, we have
	\begin{align}
		& \|\bXi_i\|_1 = \left\| (\mat{F}_{i+1} + \mat{E}_{i+1}) (\id + 5 \delta \globalcond \ones \ones^\top) \mat{D}_i - \mat{F}_{i+1} \mat{D}_i \right\|_1
		\\
		\leq & \| 5 \delta \globalcond \ones \ones^\top \mat{D}_i \|_1 + \|\mat{E}_{i+1} (\id + 5 \delta \globalcond \ones \ones^\top) \mat{D}_i \|_1
		\leq 5 \delta m  \globalcond^3 + 1.5  \globalcond^2  \| \mat{E}_{i+1} \|_1  \tag{since $\delta \leq \frac{1}{10m\globalcond}$}.
	\end{align}
	Substituting back to \cref{eq:bsls:stab:diff:bound:1} completes the proof.
\end{proof}

\begin{lemma}[Upper bound of $\|\mat{E}_i\|_1$]
	\label{lem:bsls:stab:E:bound}
	Using the same notation of \Cref{lem:bsls:stab:abs} and assuming the same assumptions of \Cref{lem:bsls:stab:F:bound}, and in addition assume
	\begin{equation}
		\delta \leq \frac{1}{m (10 \globalcond )^{2m-1} \prod_{i \in [m]} T_i},
		\label{eq:bsls:stab:delta:bound:1}
	\end{equation}
	then the following inequality holds for any $i \in [m]$,
	\begin{equation}
		\|\mat{E}_i\|_1 \leq  \delta m \cdot (10  \globalcond )^{2(m-i)+1} \cdot \prod_{j=i}^{m} T_j.
	\end{equation}
\end{lemma}
\begin{proof}[Proof of \cref{lem:bsls:stab:E:bound}]
    First observe that 
    \begin{align}
        & \|\mat{E}_i\|_1 = \left\|	\left( (\mat{F}_{i+1} + \mat{E}_{i+1})\widehat{\mat{D}_{i}} \right)^{T_i} 	(\mat{F}_{i+1} + \mat{E}_{i+1}) 
		- \left( \mat{F}_{i+1} \mat{D}_i \right)^{T_i} \mat{F}_{i+1} \right\|_1
		\\
	\leq & \left\|	\left( (\mat{F}_{i+1} + \mat{E}_{i+1})\widehat{\mat{D}_{i}} \right)^{T_i} 	(\mat{F}_{i+1} + \mat{E}_{i+1}) 
		- \left( \mat{F}_{i+1} \mat{D}_i \right)^{T_i} (\mat{F}_{i+1} + \mat{E}_{i+1}) \right\|_1 + \left\| \left( \mat{F}_{i+1} \mat{D}_i \right)^{T_i}  \mat{E}_{i+1} \right\|_1
		\\
		\leq & 	\left\|	\left( (\mat{F}_{i+1} + \mat{E}_{i+1})\widehat{\mat{D}_{i}} \right)^{T_i} - \left( \mat{F}_{i+1} \mat{D}_i \right)^{T_i}  \right\|_1 
	    (1 +  \|\mat{E}_{i+1}\|_1) + \|\mat{E}_{i+1}\|_1,
    \end{align}
    where in the last inequality we applied the fact that $\|\mat{F}_{i+1} \mat{D}_i\|_1 \leq 1$.
    
    Next, we prove by induction in reverse order from $m$ down to $1$.

	For $i = m$, by definition of $\mat{E}_m$,
	\begin{align}
			  & \|\mat{E}_m\|_1 \leq \left\| ((\mat{F}_{m+1} + \mat{E}_{m+1}) \widehat{\mat{D}_m})^{T_m} - (\mat{F}_{m+1} \mat{D}_m)^{T_m} \right\|_1
				\\
		\leq &
		5 \delta m \globalcond  T_m \exp (5 \delta m \globalcond T_m) \tag{by \Cref{lem:bsls:stab:diff:bound}}
		\\
		\leq &
		5 \sqrt{e}  \delta m \globalcond  T_m 
		\leq 
		10 \delta m \globalcond T_m .
		\tag{since $\delta \leq \frac{1}{10 m \globalcond T_m}$by assumption \cref{eq:bsls:stab:delta:bound:1}}
	\end{align}
	
	Now suppose the statement holds for the case of $i+1$, we then study the case of $i$. 
	By definition of $\mat{E}_i$ we have 
	\begin{align}
				 & 	\|\mat{E}_i\|_1 \leq	\left\| ((\mat{F}_{i+1} + \mat{E}_{i+1}) \widehat{\mat{D}_i})^{T_i} - (\mat{F}_{i+1} \mat{D}_i)^{T_i} \right\|_1 (1 + \|\mat{E}_{i+1}\|_1) + \|\mat{E}_{i+1}\|_1
			\\
		\leq &
		\left \|\mat{E}_{i+1}\|_1 + (1 + \|\mat{E}_{i+1}\|_1)( 2 \|\mat{E}_{i+1}\|_1 + 5 \delta m \globalcond \right)  \globalcond^2 T_i \cdot
		\exp \left(\left( 2 \|\mat{E}_{i+1}\|_1 + 5 \delta m \globalcond  \right)  \globalcond^2 T_i \right)
		\tag{by \Cref{lem:bsls:stab:diff:bound}}
	\end{align}
	By induction hypothesis $\|\mat{E}_{i+1}\|_1 \leq  (10  \globalcond )^{2(m-i)-1} \cdot \delta m \prod_{j=i+1}^{m} T_j$, we obtain
	\begin{align}
	  &	\|\mat{E}_{i+1}\|_1 + (1 + \|\mat{E}_{i+1}\|_1) \left( 2  \|\mat{E}_{i+1}\|_1 + 5 \delta m \globalcond  \right)  \globalcond^2 T_i
	  \\
    \leq & 4  \globalcond^2 \left( (10  \globalcond )^{2(m-i)-1} \cdot \delta m \prod_{j=i+1}^{m} T_j \right) T_i  + 5 \delta m  \globalcond^3  T_i
		\\
		\leq & 10  \globalcond^2 \cdot (10  \globalcond )^{2(m-i)-1} \delta m \cdot \prod_{j=i}^{m} T_j.
	\end{align}
	Also by $\delta$ bound \cref{eq:bsls:stab:delta:bound:1}
	\begin{align}
		&	\exp \left( \left( 2 \|\mat{E}_{i+1}\|_1 + 5 \delta m \globalcond \right)  \globalcond^2 T_i \right)
		\leq
		\exp \left( 9  \globalcond^2 \cdot (10  \globalcond )^{2(m-i)-1} \delta m \cdot \prod_{j=i}^{m} T_j \right)
		\leq e^{0.07}.
	\end{align}
	Since $9e^{0.07} < 10$ we have
	\begin{equation}
	    \|\mat{E}_i\|_1 \leq
	    \left\| ((\mat{F}_{i+1} + \mat{E}_{i+1}) \widehat{\mat{D}_i})^{T_i} - (\mat{F}_{i+1} \mat{D}_i)^{T_i} \right\|_1 
	    \leq
	    0.1 (10  \globalcond )^{2(m-i)+1} \delta m \cdot \prod_{j=i}^{m} T_j.
	    \label{eq:bsls:stab:diff:bound:8}
	\end{equation}
	completing the induction proof.
\end{proof}

\subsubsection{Upper bound of $\mat{Z}$}
\label{sec:lem:bsls:stab:Z:bound}
Finally we bound $\|\mat{Z}_i\|_1$ with the following \Cref{lem:bsls:stab:Z:bound}.
\begin{lemma}[Upper bound of $\|\mat{Z}_i\|_1$]
	\label{lem:bsls:stab:Z:bound}
	Using the same notation of \Cref{lem:bsls:stab:abs} and assuming the same assumptions of \Cref{lem:bsls:stab:E:bound}, then
	the following inequality holds for any $i \in [m]$ 
	\begin{equation}
		\|\mat{Z}_i\|_1 \leq 2^{m-i+1} \prod_{j=i}^m (T_j + 1) \leq 4^{m-i+1} \prod_{j=1}^m T_j
	\end{equation}
\end{lemma}
\begin{proof}[Proof of \Cref{lem:bsls:stab:Z:bound}]
	Recall the definition of $\mat{Z}_i$
	\begin{align}
	    		\mat{Z}_i \defeq 	((\mat{F}_{i+1} + \mat{E}_{i+1}) \widehat{\mat{D}_i})^{T_i} \mat{Z}_{i+1} + 
			\sum_{t_i=0}^{T_i-1}
			\left( (\mat{F}_{i+1} + \mat{E}_{i+1})\widehat{\mat{D}_{i}} \right)^{t_i} \left( \mat{Z}_{i+1} + \mat{F}_{i+1} + \mat{E}_{i+1} \right),
	\end{align}
	Therefore
	\begin{align}
	   	\|\mat{Z}_i\|_1 \leq 
			\left\| \sum_{t_i=0}^{T_i}
			\left( (\mat{F}_{i+1} + \mat{E}_{i+1})\widehat{\mat{D}_{i}} \right)^{t_i} \right\|_1 \left( \|\mat{Z}_{i+1}\|_1 + \|\mat{F}_{i+1}\|_1 + \|\mat{E}_{i+1}\|_1 \right),
	\end{align}
	We will bound $\left\|  \sum_{t_{i}=0}^{T_i} ( (\mat{F}_{i+1} + \mat{E}_{i+1}) \widehat{\mat{D}_i})^{t_i} \right\|_1$ 
	and $ \left( \|\mat{Z}_{i+1}\|_1 + \|\mat{F}_{i+1}\|_1 + \|\mat{E}_{i+1}\|_1 \right)$ separately. 
	The former is bounded as
	\begin{align}
				 & \left\| \sum_{t_{i}=0}^{T_i} ( (\mat{F}_{i+1} + \mat{E}_{i+1}) \widehat{\mat{D}_i})^{t_i} \right\|_1
		\leq
		\sum_{t_{i}=0}^{T_i} \left\|  ( (\mat{F}_{i+1} + \mat{E}_{i+1}) \widehat{\mat{D}_i})^{t_i} \right\|_1
		\tag{by triangle inequality}
		\\
		\leq & \sum_{t_{i}=0}^{T_i}	\left( \left\| (\mat{F}_{i+1} \mat{D}_i)^{t_i} \right\|_1 +  \left\|  ((\mat{F}_{i+1} + \mat{E}_{i+1}) \widehat{\mat{D}_i})^{t_i} - (\mat{F}_{i+1} \mat{D}_i)^{t_i}   \right\|_1 \right)
		\tag{by triangle inequality}
		\\
		\leq & (T_i + 1) + \sum_{t_{i}=0}^{T_i}\left\|  ((\mat{F}_{i+1} + \mat{E}_{i+1}) \widehat{\mat{D}_i})^{t_i} - (\mat{F}_{i+1} \mat{D}_i)^{t_i}   \right\|_1 
		\tag{since $\|\mat{F}_{i+1} \mat{D}_i\|_1 \leq 1$ by \Cref{lem:bsls:stab:F:bound}}
		\\
		\leq & (T_i+1) \left( 1 +   0.1 (10  \globalcond )^{2(m-i)+1} \delta m \cdot \prod_{j=i}^{m} T_j  \right)
		\tag{by \cref{eq:bsls:stab:diff:bound:8}}
		\\
		\leq & 1.1 (T_i + 1) 
		\tag{by $\delta \leq \frac{1}{m (10 \globalcond )^{2m-1} \prod_{i \in [m]} T_i}$, see \cref{eq:bsls:stab:delta:bound:1}}
	\end{align}
	For $\left\|  \mat{F}_{i+1} + \mat{E}_{i+1} + \mat{Z}_{i+1} \right\|_1$ we have
	\begin{align}
		& \|\mat{F}_{i+1}+ \mat{E}_{i+1} + \mat{Z}_{i+1} \|_1 \leq 
		\|\mat{F}_{i+1} \|_1 + \|\mat{E}_{i+1}\|_1 + \|\mat{Z}_{i+1} \|_1 
		\\
		\leq & 
		1 +  (10  \globalcond )^{2(m-i)-1} \cdot \delta m \prod_{j=i+1}^{m} T_j
		+ 
		\|\mat{Z}_{i+1}\|_1 
		\leq 1.01 +   \|\mat{Z}_{i+1} \|_1
		\tag{by $\delta$ bound \cref{eq:bsls:stab:delta:bound:1}}
	\end{align}
	Consequently $\| \mat{Z}_i \|_1 \leq 1.1 (T_i  + 1)\left( 1.01 + \| \mat{Z}_{i+1} \|_1  \right)$.
	By induction we have $\| \mat{Z}_i \|_1 \leq  2^{m-i+1} \prod_{j=i}^m (T_j+1).$
\end{proof}

\subsection{Finishing the proof of \Cref{thm:bsls:inexact:general}}
\label{sec:proof:thm:bsls:inexact}
We are ready to finish the proof of \Cref{thm:bsls:inexact:general} (general initialization).
\begin{proof}[Proof of \Cref{thm:bsls:inexact:general}]
	By \Cref{lem:bsls:stab:abs} we have
	\begin{equation}
		\vec{\err}(\widehat{\bsls_{1}}(\x^{(0)})) \leq (\mat{F}_1 + \mat{E}_1) \vec{\err}(\x^{(0)}) + 2 \delta L_m \|\x^{\star}\|_2^2 \mat{Z}_1 \ones.
	\end{equation} 
	Since $f(\x) - f^{\star} = \| \vec{\err}(\x)\|_1$, we obtain
	\begin{equation}
		f(\widehat{\bsls_{1}}(\x^{(0)})) - f^{\star} \leq (\|\mat{F}_1\|_1 + \| \mat{E}_1 \|_1)  (f(\x^{(0)}) - f^{\star})  + 2 \delta m L_m \|\x^{\star}\|_2^2 \|\mat{Z}_1\|_1
	\end{equation}
	Plugging in the bound of $\|\mat{F}_1\|_1$ (from \Cref{lem:bsls:stab:F:bound}), $\|\mat{E}_1\|_1$ (from \Cref{lem:bsls:stab:E:bound}), and $\|\mat{Z}_1\|_1$ (from \Cref{lem:bsls:stab:Z:bound}), we obtain
	\begin{align}
		& f(\widehat{\bsls_{1}}(\x^{(0)})) - f^{\star} 
		\\
		\leq & \left( \frac{\epsilon}{ f(\x^{(0)}) - f^{\star}} + (10  \globalcond )^{2m-1} \cdot \delta m \prod_{i=1}^{m} T_i \right)  (f(\x^{(0)}) - f^{\star}) 
		+ 2 \cdot 4^{m} \delta m L_m \|\x^{\star}\|_2^2 \cdot \prod_{i=1}^m T_i.
	\end{align}
	By $\delta$ bound 
	\begin{equation}
		\delta \leq 
		\prod_{i \in [m]} T_i^{-1} \min \left\{
			\frac{1}{m  \cdot (10  \globalcond )^{2m-1} },
		 \frac{\epsilon}{ m (10  \globalcond )^{2m-1}   \cdot (f(\x^{(0)}) - f^{\star}) } , 
		\frac{\epsilon}{4^{m+1} m L_m \|\x^{\star}\|_2^2 }
		\right\},
	\end{equation}
	we immediately obtain $f(\widehat{\bsls_{1}}(\x^{(0)})) - f^{\star} \leq 3\epsilon$, completing the proof of \Cref{thm:bsls:inexact:general}.
\end{proof}

\section{Proof of $\acbsls$ under finite-precision arithmetic}
\label{apx:acbsls:finite}
In this section, we will prove $\acbsls$ under finite-precision arithmetic. 

In \cref{thm:acbsls:inexact}, we specialized our initialization of $\x^{(0)}, \v^{(0)}$ both to $\vec{0}$ to simplify the exposition of the theorem. 
In fact, we can (and will) prove the following general (but less clean) version with arbitrary $\x^{(0)}, \v^{(0)}$.
\begin{theorem}[$\acbsls$ under finite-precision arithmetic, general initialization]
	\label{thm:acbsls:inexact:general}
	Consider multiscale optimization problem defined in \Cref{def:multiscale_problem}, for any initialization $(\x^{(0)}, \v^{(0)})$ and $\epsilon > 0$, assuming \cref{req:agd} with
	\begin{small}
	\begin{equation}
		\delta^{-1} \geq \left( \prod_{i \in [m]} T_i \right) 
		\cdot \max
		\left\{
		{2 \cdot (10  \globalcond^2)^{2m-1}}, 
		2 \cdot (10 \globalcond^2 )^{2m-1} m \cdot \frac{\psi(\x^{(0)}, \v^{(0)}) }{\epsilon}, 
		4 \cdot 3^{m+1} \cdot m L_m \globalcond \frac{\|\x^{\star}\|_2^2 }{\epsilon}
		\right\}
	\end{equation}
	\end{small}
	then $\psi(\widehat{\acbsls_{1}}(\x^{(0)}, \v^{(0)})) \leq 3 \epsilon$ provided that $T_1, \ldots, T_m$ satisfy \eqref{eq:acbsls:t:lb}, which we restate here for ease of reference
	\begin{equation}
		T_1 \geq \sqrt{\kappa_1} \log \left( \frac{\psi(\x^{(0)}, \v^{(0)})}{\epsilon} \right),
		\qquad
		T_i \geq \sqrt{\kappa_i} (\log (4 \globalcond^4) + 1), \quad \text{ for $i = 2, \ldots, m$}.
	\end{equation}
	We can also achieve the same asymptotic sample complexity (up to constant factors suppressed in the $\bigo(\cdot)$) when $\{(\mu_i, L_i), i \in [m]\}$ are unknown and only $m$, $\mu_{1}$, $L_m$ and $ \pi_{\kappa}=\prod_{i=1}^m \kappa_i$ are known.
\end{theorem}
\cref{thm:acbsls:inexact} is clearly a corollary of \cref{thm:acbsls:inexact:general} since
\begin{equation}
	4 \cdot 3^{m+1} \cdot m L_m \globalcond \frac{\|\x^{\star}\|_2^2 }{\epsilon}
	\leq 
	4 m (10 \globalcond^2 )^{2m-1} \frac{\psi(\vec{0}, \vec{0}) }{\epsilon}.
\end{equation}

The proof of \cref{thm:acbsls:inexact:general} is structured as follows. 
We first define two vector potentials and establish their relations in \Cref{sec:thm:acbsls:inexact:1}. 
We then study the progress of one (inexact) $\widehat{\agd}$ step on these vector potentials in \Cref{sec:thm:acbsls:inexact:2}, and inductively estimate the progress of $\widehat{\acbsls_i}$ by matrix inequalities for all $i \in [m]$ in descent order (see \Cref{sec:thm:acbsls:inexact:3}). 
The proof of \Cref{thm:acbsls:inexact:general} is then finished in \Cref{sec:thm:acbsls:inexact:5}. As before, the last part regarding the case where $\{(\mu_i,L_i), i \in [m]\}$ are unknown follows from our black-box reduction in Proposition \ref{prop:search} (in the same way as in the proof of Theorem \ref{thm:bsls:recursive}).

\subsection{Introduction of vector potentials and their relations}
\label{sec:thm:acbsls:inexact:1}
We introduce a few more notation to simplify the presentation. 
For any $(\x, \v)$ and $i \in [m]$, define
\begin{equation}
	\err_i^{\max}(\x, \v) \defeq \max \{\err_i (\x), \err_i(\v)\}, \qquad
	\res_i^{\max}(\x, \v) \defeq \max \{\res_i (\x), \res_i(\v)\}.
	\label{eq:def:rmax}
\end{equation}

Define a series of vector-valued potential functions $\bphi_1^{\err}, \ldots, \bphi_m^{\err}$ and $\bphi_1^{\res}, \ldots, \bphi_m^{\res}$:
\begin{equation}
	\bphi^{\err}_{i}(\x, \v) 	
	\defeq 
	\begin{bmatrix}
		\err^{\max}_1(\x, \v)\\
		\vdots \\
		\err^{\max}_{i-1}(\x, \v) \\
		\frac{1}{2} {\psi}_{i}(\x, \v) \\
		\vdots \\
		\frac{1}{2} \psi_m(\x, \v)
	\end{bmatrix},
	\quad
	\bphi^{r}_{i}(\x, \v) 	
	\defeq 
	\begin{bmatrix}
		\res^{\max}_1(\x, \v) \\
		\vdots \\
		\res^{\max}_{i-1}(\x, \v)  \\
		\frac{1}{2} {\psi}_{i}(\x, \v) \\
		\vdots \\
		\frac{1}{2} {\psi}_{m}(\x, \v) \\
	\end{bmatrix}.
	\label{eq:def:phi}
\end{equation}

We establish two lemmas on the relations of vector potentials $\bphi_{i}^{\err}$ and $\bphi_{i}^r$ for varying $i$. 
\Cref{lem:agd:max:of:potentials} bounds the maximum of two vector potentials; \Cref{lem:agd:sum:of:potentials:er} bounds the sum of two vector potentials.

\subsubsection{Bounding the maximum of two vector potentials}
\begin{lemma}
	\label{lem:agd:max:of:potentials}
	Consider multiscale optimization problem defined in \Cref{def:multiscale_problem},	for any $\x, \v$, for any $i \in [m-1]$, the following two matrix inequalities hold (recall $\leq$ denotes entry-wise inequality)
	\begin{equation}
		\max \left\{ \bphi_{i+1}^{\err} (\x, \x), \bphi_{i+1}^{\err}(\v, \v) \right\} \leq
		\begin{bmatrix}
			\id_{i-1} & 0 \\
			0 & 2 \globalcond \cdot \id_{m-i+1}
		\end{bmatrix}
		\bphi_i^{\err}(\x, \v),
	\end{equation}
	\begin{equation}
		\max \left\{ \bphi_{i+1}^r (\x, \x), \bphi_{i+1}^r(\v, \v) \right\} \leq
		\begin{bmatrix}
			\id_{i-1} & 0 \\
			0 & 2 \globalcond \cdot \id_{m-i+1}
		\end{bmatrix}
		\bphi_i^{r}(\x, \v).
	\end{equation}
\end{lemma}
\begin{proof}[Proof of \Cref{lem:agd:max:of:potentials}]
	We study $\unit_j^\top \max \left\{  \bphi_{i+1}^{\err}(\x, \x) , \bphi_{i+1}^{\err}(\v, \v)  \right\}$ for three possible cases: $j < i$, $j = i$, or $j > i$. 

	\paragraph{Case of $j < i$.} By definition of $\bphi_i^{\err}$ we have $\unit_j^\top \bphi_{i+1}^{\err}(\x, \x) = \err_j(\x)$ for any $\x$. Thus
	\begin{align}
		\unit_j^\top \max \left\{  \bphi_{i+1}^{\err}(\x, \x) , \bphi_{i+1}^{\err}(\v, \v)  \right\} = \max \left\{ \err_j(\x), \err_j(\v) \right\}
		= 	\unit_j^\top \bphi_i^{\err}(\x,\v),
	\end{align}
	where the last equality is by definition of $	\unit_j^\top \bphi_i^{\err}(\x, \v)$.

	\paragraph{Case of $j = i$.} Again by definition of $\bphi_i^{\err}$
	\begin{align}
		& \unit_i^\top \max \left\{  \bphi_{i+1}^{\err}(\x, \x) , \bphi_{i+1}^{\err}(\v, \v)  \right\} = \max \left\{ \err_i(\x), \err_i(\v) \right\} 
		\\
		\leq &  \max \{\err_i(\x) ,\kappa_i \res_i(\v) \}
		\leq  \kappa_i \psi_i(\x, \v)
		\tag{by definition of $\psi_i$}
		\\
		= & 2 \kappa_i \cdot \unit_i^\top \bphi_i^{\err}(\x,\v),
	\end{align}
	where the last equality is by definition of $\bphi_i^{\err}(\x, \v)$ since $\unit_i^\top \bphi_i^{\err}(\x, \v) = \frac{1}{2} \psi_i(\x, \v)$.
	
	\paragraph{Case of $j > i$.} By definition of $\bphi_{i+1}^{\err}$:
	\begin{align}
		& \unit_j^\top \max \left\{  \bphi_{i+1}^{\err}(\x, \x) , \bphi_{i+1}^{\err}(\v, \v)  \right\} = \frac{1}{2} \max \left\{ \psi_{j}(\x, \x), \psi_j(\v, \v)  \right\} \tag{by definition}
		\\
		= & \frac{1}{2} \max \left\{  \err_j(\x) + \res_j(\x), \err_j(\v) + \res_j(\v) \right\}
		\leq  \frac{1}{2} (1 + \kappa_j ) \psi_j(\x, \v)
		\leq \kappa_j \psi_j (\x, \v) 
		\\
		= & 2 \kappa_j \unit_j^\top \bphi_i^{\err}(\x,\v) \tag{by definition}
	\end{align}
	Concatenating the above three inequalities yields the first statement of the lemma. The second statement holds for the same reason.
\end{proof}

\subsubsection{Bounding the sum of two vector potentials}
\begin{lemma}
	\label{lem:agd:sum:of:potentials:er}
	Consider multiscale optimization problem defined in \Cref{def:multiscale_problem},	for any $\x, \v$, for any $i \in [m-1]$, the following two inequalities hold
	\begin{equation}
		\bphi_{i+1}^{\err}(\x, \x) + \bphi_{i+1}^r (\v, \v) \leq 
		\begin{bmatrix}
			2 \globalcond  \id_{i} &  & \\
			 & 2  & \\
			& & 2 \globalcond \id_{m-i-1}
		\end{bmatrix}
		\bphi_i^{\err}(\x, \v) ,
	\end{equation}
	\begin{equation}
		\bphi_{i+1}^{\err}(\x, \x) + \bphi_{i+1}^r (\v, \v) \leq 
		\begin{bmatrix}
			2 \globalcond  \id_{i} &  & \\
			 & 2  & \\
			& & 2 \globalcond  \id_{m-i-1}
		\end{bmatrix}
		\bphi_i^{r}(\x, \v).
	\end{equation}
\end{lemma}
\begin{proof}[Proof of \Cref{lem:agd:sum:of:potentials:er}]
	We study $\unit_j^\top (\bphi_{i+1}^{\err}(\x, \x) + \bphi_{i+1}^r(\v, \v))$ for three possible cases: $j < i$, $j = i$, or $j > i$. 

	\paragraph{Case of $j < i$.} By definition of $\bphi_i^{\err}$ and $\bphi_i^r$ we have $\unit_j^\top \bphi_{i+1}^{\err}(\x, \x) = \err_j(\x)$ and $\unit_j^\top \bphi_{i+1}^r(\v, \v) = \res_j(\v)$. Thus
	\begin{align}
		& \unit_j^\top (\bphi_{i+1}^{\err}(\x, \x) + \bphi_{i+1}^r(\v, \v)) = \err_j(\x) + \res_j(\v)  
		\\
		\leq & \err_j (\x) + \kappa_j \err_j(\v) \leq 2 \kappa_j \max \{ \err_j(\x), \err_j(\v) \} 
		\\
		= & 2 \kappa_j \unit_j^\top \bphi_i^{\err}(\x,\v). \tag{by definition of $\bphi_i^{\err}$}
	\end{align}
	Similarly $ \unit_j^\top (\bphi_{i+1}^{\err}(\x, \x) + \bphi_{i+1}^r(\v, \v))   \leq 2 \kappa_j \unit_j^\top \bphi_i^{r}(\x,\v)$.

	\paragraph{Case of $j = i$.} Similarly
	\begin{align}
		& \unit_i^\top (\bphi_{i+1}^{\err}(\x, \x) + \bphi_{i+1}^r(\v, \v)) = \err_j(\x) + \res_j(\v)  \tag{by definition}
		\\
		= & \psi_i(\x, \v) %
		= 2 \unit_i^\top \bphi_i^{\err}(\x, \v) \tag{by definition} = 2 \unit_i^\top \bphi_i^r(\x, \v).
	\end{align}

	\paragraph{Case of $j \geq i$.} By definition we have $\unit_j^\top \bphi_{i+1}^{\err}(\x, \x) = \frac{1}{2} \psi_j(\x, \x)$ and $\unit_j^\top \bphi_{i+1}^r(\v, \v) = \frac{1}{2} \psi_j(\v, \v)$. Thus
	\begin{align}
		& \unit_j^\top (\bphi_{i+1}^{\err}(\x, \x) + \bphi_{i+1}^r(\v, \v)) = \frac{1}{2} \psi_j(\x, \x) + \frac{1}{2} \psi_j(\v, \v)  \tag{by definition}
		\\
		= & \frac{1}{2} (\err_j (\x) + \res_j (\v) + \err_j (\v) + \res_j(\v))
		\\
		\leq & \frac{1}{2} (1 + \kappa_{j}) \psi_j(\x,\v) \leq \kappa_{j} \psi_j(\x, \v)
		= 2 \kappa_{j} \cdot \unit_j^\top \bphi_i^{\err}(\x,\v) =  2 \kappa_{j} \cdot \unit_j^\top \bphi_i^r(\x,\v).
	\end{align}
	Concatenating the above inequalities completes the proof.
\end{proof}

\subsection{Progress of $\agd$ step under finite arithmetic}
\label{sec:thm:acbsls:inexact:2}
In this subsection, we study the effect of one inexact $\agd$ (also denoted as $\widehat{\agd}$ step) on the vector potentials 
$\bphi^{\err}_i$ and $\bphi^{\res}_i$.
The main goal of this subsection is to prove the following \Cref{lem:implication:req:agd:4}.
\begin{lemma}[Progress of one $\widehat{\agd}$ step under finite arithmetic]
	\label{lem:implication:req:agd:4}
	Consider multiscale optimization problem defined in \Cref{def:multiscale_problem}, assuming \cref{req:agd}, then for any $\x, \v$ and $i \in [m]$, the following two inequalities hold
	\begin{enumerate}[(a), leftmargin=*]
		\item $	{\bphi}^\err_{i}(\widehat{\agd}(\x, \v; L_i, \mu_i)) \leq (\id + 10 \delta \globalcond^2 \ones \ones^\top) \mat{D}_i 
		{\bphi}^\err_{i}(\x, \v) + 4 \delta L_m \|\x^{\star}\|_2^2 \ones$.
		\item ${\bphi}^{\res}_{i}(\widehat{\agd}(\x, \v; L_i, \mu_i)) \leq (\id + 10 \delta \globalcond^2 \ones \ones^\top) \mat{D}_i 
		{\bphi}^{\res}_{i}(\x, \v) + 4 \delta L_m \|\x^{\star}\|_2^2 \ones$.
	\end{enumerate}
	where $\mat{D}_1$, $\mat{D}_2$, \ldots, $\mat{D}_m$ are $m \times m$ diagonal matrices defined by
	\begin{equation}
		\mat{D}_i = \begin{bmatrix}
			\id_{i-1} & & \\
							& 1 - \kappa_i^{-\frac{1}{2}} & \\
							& & 2 \globalcond^2 \id_{m-i}
		\end{bmatrix}
		\label{eq:agd:D:def}
	\end{equation}
	To simplify the notation we will define (throughout this section)
	\begin{equation}
		\widehat{\mat{D}_i} \defeq (\id + 10 \delta \globalcond^2 \ones \ones^\top)  \mat{D}_i.
		\label{eq:agd:Dhat:def}
	\end{equation}
\end{lemma}
We will prove \Cref{lem:implication:req:agd:4} in three steps.
First, we first bound the perturbation of residual $r_j$ and functional error $\err_j$ under multiplicative error in \Cref{lem:implication:req:agd}. 
Then we bound the potential $r_j^{\max}, \err_j^{\max}$, and $\psi_j$, and the vector potential $\bphi_j^r$ and $\bphi_j^{\err}$ in \Cref{lem:implication:req:agd:2}. 
The proof of \Cref{lem:implication:req:agd:4} is finished in \Cref{sec:proof:lem:implication:req:agd:4}.

\subsubsection{Sensitivity of residual and functional error under multiplicative error}
In this subsubsection, we establish the first supporting lemma for \Cref{lem:implication:req:agd:4}.
\begin{lemma}[Sensitivity of residual and functional error under multiplicative error]
	\label{lem:implication:req:agd}
	Assuming $\widehat{\x}, \x$ satisfies 
	\begin{equation}
		|\widehat{\x} - {\x}| \leq \delta|\x|
	\label{eq:req}
	\end{equation}
	for some $\delta < 1$,	then for any $j \in [m]$,
	\begin{enumerate}[(a), leftmargin=*]
		\item $\err_j (\widehat{\x}) \leq \err_j (\x) +  5 \delta \globalcond \sum_{k=1}^m \err_k(\x) + 2 \delta L_j \|\x^{\star}\|_2^2$
		\item $r_j (\widehat{\x}) \leq \res_j (\x) + 5 \delta \globalcond \sum_{k=1}^m \res_k(\x) + 2 \delta \mu_j \|\x^{\star}\|_2^2$
	\end{enumerate}
\end{lemma}
\begin{proof}[Proof of \Cref{lem:implication:req:agd}]
	\begin{enumerate}[(a), leftmargin=0pt, itemindent=15pt]
		\item   Same as \cref{lem:implication:req:gd:1}. 
		\item 	Let $\beps \defeq \widehat{\x} - \x$, then by Cauchy-Schwartz inequality,
		\begin{align}
			\| \proj_j ( \widehat{\x} - \x^{\star})\|_2^2
			=
			\| \proj_j ( {\x} - \x^{\star} + \beps)\|_2^2
			\leq 
			(1 + \delta) \| \proj_j ( {\x} - \x^{\star} )\|_2^2 + 2 \delta^{-1} 	\| \proj_j \beps\|_2^2
			\label{eq:implication:req:agd:1}
		\end{align}

		By assumption \eqref{eq:req} we have 
		\begin{equation}
			\| \proj_j \beps\|_2^2 
			\leq
			\| \beps \|_2^2
			\leq
			\delta^2 \|\x\|_2^2
			\leq
			2\delta^2 \|\x - \x^{\star}\|_2^2 + 2 \delta^2 \|\x^{\star}\|_2^2
			\leq
			2\delta^2 \sum_{k=1}^m\|\proj_k  (\x - \x^{\star})\|_2^2 + 2 \delta^2 \|\x^{\star}\|_2^2
			\label{eq:implication:req:agd:2}
		\end{equation}
		Combining \eqref{eq:implication:req:agd:1} and \eqref{eq:implication:req:agd:2} yields 
		\begin{equation}
			\| \proj_j ( \widehat{\x} - \x^{\star})\|_2^2
			\leq 
			(1 + \delta) \| \proj_j ( \x - \x^{\star})\|_2^2
			+ 4 \delta \sum_{k=1}^m\|\proj_k  (\x - \x^{\star})\|_2^2
			+ 4 \delta \|\x^{\star}\|_2^2
			\label{eq:implication:req:agd:3}
		\end{equation}
		It follows that 
		\begin{align}
			& \res_j(\widehat{\x}) \defeq \frac{1}{2} \mu_j 	\| \proj_j ( \widehat{\x} - \x^{\star})\|_2^2
			\tag{by definition of $r_j$}
			\\
			\leq & (1+\delta) \cdot \frac{1}{2} \mu_j \| \proj_j ( \x - \x^{\star})\|_2^2 + 2  \mu_j \delta  \sum_{k=1}^m\|\proj_k  (\x - \x^{\star})\|_2^2 + 2 \delta \mu_j \|\x^{\star}\|_2^2
			\tag{by \eqref{eq:implication:req:agd:3}}
			 \\
			= &  (1+\delta) \res_j(\x) + 4   \delta  \sum_{k=1}^m \frac{\mu_j}{\mu_k} \res_k(\x) + 2 \delta \mu_j \|\x^{\star}\|_2^2
			\tag{by definition of $r_j$'s}
			\\
			\leq & (1+\delta) \res_j(\x) + 4   \delta \globalcond  \sum_{k=1}^m  \res_k(\x) + 2 \delta \mu_j \|\x^{\star}\|_2^2
			\tag{since $\frac{\mu_j}{\mu_k} \leq \globalcond$ for any $j, k \in [m]$} 
			\\
			\leq & \res_j(\x) + 5 \delta \globalcond \sum_{k=1}^m \res_k(\x) +  2 \delta \mu_j \|\x^{\star}\|_2^2.
		\end{align}
	\end{enumerate}
\end{proof}

\subsubsection{Sensitivity of potentials under multiplicative error}
In this subsubsection, we establish the second supporting lemma for \Cref{lem:implication:req:agd:4}.
\begin{lemma}[Sensitivity of potentials under multiplicative error]
	\label{lem:implication:req:agd:2}
	Assuming $\widehat{\x}, \x, \widehat{\v}, \v$ satisfies 
	\begin{equation}
		|\widehat{\x} - {\x}| \leq \delta|\x|, \quad 
		|\widehat{\v} - {\v}| \leq \delta|\v|
		\label{eq:req:2}
	\end{equation}
	for some $\delta < 1$. Then for any $j \in [m]$,
	\begin{enumerate}[(a), leftmargin=*]
		\item $	r_j^{\max} (\widehat{\x}, \widehat{\v})
		\leq 
		r_j^{\max} (\x, \v) + 5 \delta \globalcond \sum_{k=1}^m \res_k^{\max} (\x, \v) + 2 \delta \mu_j  \|\x^{\star}\|_2^2$.
		\item $\err_j^{\max} (\widehat{\x}, \widehat{\v})
		\leq 
		\err_j^{\max} (\x, \v) + 5 \delta \globalcond \sum_{k=1}^m \err_k^{\max} (\x, \v) + 2 \delta L_j  \|\x^{\star}\|_2^2$.
		\item $	\psi_j(\widehat{\x}, \widehat{\v}) \leq \psi_j(\x, \v) + 5 \delta \globalcond \sum_{k=1}^m \psi_k(\x, \v) + 4\delta L_j \|\x^{\star}\|_2^2$.
		\item ${\bphi}^{\err}_i(\widehat{\x}, \widehat{\v})
		\leq	(\id + 10 \delta \globalcond^2 \ones \ones^\top)
		{\bphi}^{\err}_i ({\x}, {\v}) 
		+ 4 \delta L_m \|\x^{\star}\|_2^2 \ones.$
		\item ${\bphi}^{r}_i(\widehat{\x}, \widehat{\v})
		\leq	(\id + 10 \delta \globalcond^2  \ones \ones^\top)
		{\bphi}^{r}_i ({\x}, {\v}) 
		+ 4 \delta L_m \|\x^{\star}\|_2^2 \ones.$
	\end{enumerate}
\end{lemma}
\begin{proof}[Proof of \Cref{lem:implication:req:agd:2}]
	\begin{enumerate}[(a), leftmargin=0pt, itemindent=15pt]
		\item 	By \Cref{lem:implication:req:agd}, for any $j \in [m]$,
		\begin{align}
				& \res_j^{\max} (\widehat{\x}, \widehat{\v}) = \max \{ \res_j(\widehat{\x}), \res_j (\widehat{\x}) \} \tag{by definition of $r_j^{\max}$ \eqref{eq:def:rmax}}
				\\
				\leq &
				\max \left\{ \res_j (\x) + 5 \delta \globalcond \sum_{k=1}^m \res_k(\x) + 2 \delta \mu_j \|\x^{\star}\|_2^2
				, \res_j(\v)  + 5 \delta \globalcond \sum_{k=1}^m \res_k(\v) + 2 \delta \mu_j  \|\x^{\star}\|_2^2  \right\} 
				\\
				\leq &  \res_j^{\max} (\x, \v) + 5 \delta \globalcond \sum_{k=1}^m \res_k^{\max} (\x, \v) + 2 \delta \mu_j \|\x^{\star}\|_2^2 
				\tag{by definition of $r_j^{\max}$}
			\end{align}
		\item Holds for the same reason as (a).
		\item For any $j \in [m]$, by \Cref{lem:implication:req:agd},
		\begin{align}
				& \psi_j (\widehat{\x}, \widehat{\v}) = \err_j(\widehat{\x}) + \res_j(\widehat{\v}) 
				\tag{by definition of $\psi_j$}
			\\
			\leq & 			\err_j(\x)  + 5 \delta \globalcond \sum_{k=1}^m \err_k(\x) + 2 \delta L_j  \|\x^{\star}\|_2^2 	+ 
			r_j(\v)  + 5 \delta \globalcond \sum_{k=1}^m \res_k(\v) + 2 \delta \mu_j  \|\x^{\star}\|_2^2 \tag{by \Cref{lem:implication:req:agd}}
			\\
			\leq & \psi_j(\x, \v) + 5 \delta \globalcond \sum_{k=1}^m \psi_k(\x, \v) + 4 \delta L_j  \|\x^{\star}\|_2^2.
			\tag{by definition of $\psi_j$ and $\mu_j \leq L_j$}
		\end{align}
		\item 
		We will prove (d) by considering $\unit_j^\top {\bphi}^\err_i$ for two different cases: $j < i$ or $j \geq i$ (recall $\unit_j$ is defined as the $j$-th unit vector).
	For $j < i$, by definition of $\bphi_i^{\err}$ \eqref{eq:def:phi}, we have
	\begin{align}
			& \unit_j^\top 	{\bphi}^{\err}_i (\widehat{\x}, \widehat{\v}) = \err_j^{\max} (\widehat{\x}, \widehat{\v})
		\tag{definition of $\bphi_i^{\err}$}
		\\
		\leq & \err_j^{\max} (\x, \v) + 5 \delta \globalcond \sum_{k=1}^m \err_k^{\max} (\x, \v) + 2 \delta L_j \|\x^{\star}\|_2^2  \tag{by (b)}
		\\
		\leq & \err_j^{\max} (\x, \v) + 5 \delta \globalcond \left( \sum_{k < i} \err_k^{\max} (\x, \v) + \globalcond \sum_{k \geq i} \psi_k (\x, \v) \right) + 2 \delta L_j  \|\x^{\star}\|_2^2
		\tag{since $\err_k^{\max}(\x,\v) = \max\{\err_k(\x), \err_k(\v)\} \leq \err_k(\x) + \err_k(\v) \leq \err_k(\x) + \kappa_k \res_k(\v) \leq \kappa_k\psi_k (\x, \v)$ }
		\\
		\leq & \err_j^{\max} (\x, \v) + 5 \delta \globalcond^2 \left( \sum_{k < i} \err_k^{\max} (\x, \v) + \sum_{k \geq i} \psi_k (\x, \v) \right) + 2 \delta L_j  \|\x^{\star}\|_2^2
		\\
		= & \unit_j^\top {\bphi}^{\err}_i(\x, \v) + 5 \delta \globalcond^2 \sum_{k=1}^m \unit_k^\top {\bphi}^{\err}_i(\x, \v)  +  2 \delta L_j  \|\x^{\star}\|_2^2
		\tag{definition of $\bphi_i^{\err}$}
	\end{align}
	For $j \geq i$, by definition,
	\begin{align}
		& \unit_j^\top 	{\bphi}^{\err}_i (\widehat{\x}, \widehat{\v}) = \psi_j (\widehat{\x}, \widehat{\v})
		\tag{definition of $\bphi_i^{\err}$}
	\\
	\leq & \psi_j (\x, \v) + 5 \delta \globalcond \sum_{k=1}^m \psi_k (\x, \v) + 4 \delta L_j  \|\x^{\star}\|_2^2  \tag{by (c)}
	\\
	\leq & \psi_j (\x, \v) + 5 \delta \globalcond \left( 2  \sum_{k < i} \err_k^{\max} (\x, \v) + \sum_{k \geq i} \psi_k (\x, \v) \right) + 4 \delta L_j \|\x^{\star}\|_2^2
	\tag{since $\psi_k(\x, \v) = \err_k(\x) + \res_k(\v) \leq \err_k(\x) + \err_k(\v) \leq 2 \err_k^{\max} (\x, \v)$}
	\\
	\leq & \psi_j (\x, \v) + 10 \delta \globalcond \left( \sum_{k < i} \err_k^{\max} (\x, \v) + \sum_{k \geq i} \psi_k (\x, \v) \right) + 4 \delta L_j  \|\x^{\star}\|_2^2
	\\
	= & \unit_j^\top {\bphi}^{\err}_i(\x, \v) + 10 \delta \globalcond \sum_{k=1}^m \unit_k^\top {\bphi}^{\err}_i(\x, \v)  +  4 \delta L_j  \|\x^{\star}\|_2^2
	\tag{definition of $\bphi_i^{\err}$}
	\end{align}
	In matrix form we arrive at 
	\begin{equation}
		{\bphi}^{\err}_i(\widehat{\x}, \widehat{\v})
	 \leq	(\id + 10 \delta \globalcond^2 \ones \ones^\top)
	 {\bphi}^{\err}_i ({\x}, {\v}) 
	 + 4 \delta L_m \|\x^{\star}\|_2^2 \ones.
\end{equation}
	\item Holds for the same reason as (d).
	\end{enumerate}
\end{proof}

\subsubsection{Finishing the proof of \Cref{lem:implication:req:agd:4}}
\label{sec:proof:lem:implication:req:agd:4}
We are ready to finish the proof of \Cref{lem:implication:req:agd:4}.
\begin{proof}[Proof of \Cref{lem:implication:req:agd:4}]
	By \Cref{lem:agd1} from exact AGD analysis we have
	\begin{equation}
		\bphi_i^{\err}(\agd(\x, \v; L_i, \mu_i)) \leq \mat{D}_i {\bphi}^\err_{i}(\x, \v),
	\end{equation}
	then applying \Cref{lem:implication:req:agd:2} shows (a). (b) holds for the same reason.
\end{proof}

\subsection{Inductively bound the progress of inexact $\acbsls$ by matrix inequalities}
\label{sec:thm:acbsls:inexact:3}
In the following lemma, we iteratively construct the bound of vector potentials ${\bphi}^{\err}$ and ${\bphi}^{\res}$ after executing $\widehat{\acbsls}_i$.
\begin{lemma}[Estimate the progress of $\widehat{\acbsls}_i$ by matrix inequalities]
	\label{lem:acbsls:stab:abs}
	Consider multiscale optimization problem defined in \Cref{def:multiscale_problem}, and assuming \cref{req:agd}, define the following three sequences of $m \times m$ matrices $\{\mat{F}_i\}_{i=1}^{m+1}$, $\{\mat{E}_i\}_{i=1}^{m+1}$, $\{\mat{Z}_i\}_{i=1}^{m+1}$:
	\begin{equation}
		\mat{F}_{m+1} \defeq \id, \qquad \mat{E}_{m+1} \defeq \mat{0}, \qquad \mat{Z}_{m+1} \defeq \mat{0}
	\end{equation}
	and for $i = m, m-1$ down to $1$, define
	\begin{equation}
		\widehat{\mat{F}_{i+1}} \defeq 
		\begin{bmatrix}
			[\id_{i-1} | \mat{0}_{(i-1) \times (m-i+1)}] (\mat{F}_{i+1} + \mat{E}_{i+1}) 
				\begin{bmatrix}
				\id_{i-1} & \mat{0} \\
				\mat{0} & 2 \globalcond \cdot \id_{m-i+1}
			\end{bmatrix} 
			\\
			\unit_i^\top (\mat{F}_{i+1} + \mat{E}_{i+1}) 			
				\begin{bmatrix}
					\globalcond \id_{i} &  & \\
				 & 1  & \\
				& & \globalcond \id_{m-i-1}
			\end{bmatrix}
			\\
			[\mat{0}_{(m-i) \times i} | \id_{m-i}]  (\mat{F}_{i+1} + \mat{E}_{i+1})
		\end{bmatrix}
	\end{equation}
	\begin{equation}
		\mat{F}_i \defeq (\mat{K}_i	\mat{F}_{i+1} \mat{D}_i )^{T_i}, \quad 
		\mat{E}_i \defeq (\mat{K}_i \widehat{\mat{F}_{i+1}} \widehat{\mat{D}_i})^{T_i} - (\mat{K}_i	\mat{F}_{i+1} \mat{D}_i )^{T_i}, 
	\end{equation}
	\begin{equation}
		\mat{Z}_i \defeq \sum_{t_{i}=0}^{T_i-1} ( \mat{K}_i \widehat{\mat{F}_{i+1}} \widehat{\mat{D}_i})^{t_i} \left( \mat{K}_i \widehat{\mat{F}_{i+1}} + 2 \mat{Z}_{i+1} \right).
	\end{equation}
	where $\mat{D}_i, \widehat{\mat{D}_i}$ were defined in \cref{eq:agd:D:def,eq:agd:Dhat:def}, and $\mat{K}_i$ is defined by 
	\begin{equation}
		\mat{K}_i \defeq \begin{bmatrix}
			\id_{i} &  \\
						& 2 \globalcond \id_{m-i}.
		\end{bmatrix}	
		\label{eq:agd:K:def}
	\end{equation}
	Then,
	\begin{enumerate}[(a), leftmargin=*]
		\item For any $i \in [m+1]$, $\mat{F}_i$ are non-negative diagonal matrices, $\mat{E}_i$ and $\mat{Z}_i$ are non-negative matrices.
		\item For any $i \in [m]$, the following bound holds
		\begin{align}
			\bphi_i^{\err}(\widehat{\acbsls_{i}}(\x, \v)) & \leq (\mat{F}_i + \mat{E}_i) \bphi_i^{\err}(\x, \v) + 4 \delta L_m \|\x^{\star}\|_2^2 \mat{Z}_i \ones,
			\\
			\bphi_i^r(\widehat{\acbsls_{i}}(\x, \v)) & \leq (\mat{F}_i + \mat{E}_i) \bphi_i^r(\x, \v) + 4 \delta L_m \|\x^{\star}\|_2^2 \mat{Z}_i \ones.
		\end{align}  
	\end{enumerate}
\end{lemma}
\begin{proof}[Proof of \Cref{lem:acbsls:stab:abs}]
	The proof of (a) is the same as the proof for \Cref{lem:bsls:stab:abs}(a). 
	We will prove the first inequalities in  (b) by induction in reverse order (from $m$ back to $1$). The second inequality holds for the same reason. 
	
	The induction base apparently holds. 
	Now assume for any $\x, \v$, 
	\begin{align}
		\bphi_{i+1}^{\err} (\widehat{\acbsls}_{i+1} (\x, \v)) & \leq (\mat{F}_{i+1} + \mat{E}_{i+1}) \bphi_{i+1}^{\err}(\x, \v) +  4 \delta L_m \|\x^{\star}\|_2^2 \mat{Z}_{i+1} \ones, 
	\end{align}
	then we will show that
	\begin{align}
		\bphi_i^{\err}  (\widehat{\acbsls}_{i} (\x, \v))
		& \leq 
		(	\mat{K}_i \widehat{\mat{F}_{i+1}}  \widehat{\mat{D}_{i}})^{T_i} \bphi_i^{\err}(\x, \v),
		+ 
		4 \delta L_m \|\x^{\star}\|_2^2 \cdot \sum_{t_i=0}^{T_i-1} (\mat{K}_i 	\widehat{\mat{F}_{i+1}}  \widehat{\mat{D}_{i}})^{t_i}  (\widehat{\mat{F}_{i+1}} +  2 \mat{Z}_{i+1} ) \ones,
	\end{align}

	To this end, let $ \begin{bmatrix} \x^{(0)} \\ \v^{(0)} \end{bmatrix}, 
	\begin{bmatrix} \tilde{\x}^{(0)} \\ \tilde{\v}^{(0)} \end{bmatrix}, 
	\cdots
	\begin{bmatrix} {\x}^{(T_i)} \\ {\v}^{(T_i)} \end{bmatrix}$ be the trajectory generated by running $\widehat{\acbsls}_{i}$. 
	
	Since $(\tilde{\x}^{(t)}, \tilde{\v}^{(t)}) = \widehat{\agd}(\x^{(t)}, \v^{(t)}; L_i, \mu_i)$, we have by \Cref{lem:implication:req:agd:4}, 	
	\begin{align}
		{\bphi}^\err_{i}(\tilde{\x}^{(t)}, \tilde{\v}^{(t)}) 
		& \leq (\id + 10 \delta \globalcond^2 \ones \ones^\top)  \mat{D}_i 
		{\bphi}^\err_{i}(\x^{(t)}, \v^{(t)}) + 4 \delta L_m \|\x^{\star}\|_2^2 \ones,
	\end{align}
	
	We study the $\x^{(t+1)}, \v^{(t+1)}$ by considering 3 cases.
For $j < i$,
\begin{align}
		& \unit_j^\top \bphi_i^{\err} (\x^{(t+1)}, \v^{(t+1)}) = \err_j^{\max}(\x^{(t+1)}, \v^{(t+1)}) = \max \{ \err_j(\x^{(t+1)}), \err_j(\v^{(t+1)})\}
	\\
	\leq & \max \{ \err_j^{\max} (\widehat{\acbsls}_{i+1} (\tilde{\x}^{(t)}, \tilde{\x}^{(t)})), \err_j^{\max} (\widehat{\acbsls}_{i+1} (\tilde{\v}^{(t)}, \tilde{\v}^{(t)})) \}
	\\
	= & \unit_j^\top \max \{ \bphi_{i+1}^{\err} (\widehat{\acbsls}_{i+1} (\tilde{\x}^{(t)}, \tilde{\x}^{(t)})), \bphi_{i+1}^{\err} (\widehat{\acbsls}_{i+1} (\tilde{\v}^{(t)}, \tilde{\v}^{(t)})) \}
	\tag{by definition of $\bphi_{i+1}^{\err}$}
	\\
	\leq & \unit_j^\top \max \{ (\mat{F}_{i+1} + \mat{E}_{i+1}) \bphi_{i+1}^{\err}(\tilde{\x}^{(t)}, \tilde{\x}^{(t)}), (\mat{F}_{i+1} + \mat{E}_{i+1}) \bphi_{i+1}^{\err}(\tilde{\v}^{(t)},  \tilde{\v}^{(t)}) \}
	+ 4 \delta L_m \|\x^{\star}\|_2^2 \unit_j^\top \mat{Z}_{i+1} \ones
	\\
	\leq & \unit_j^\top (\mat{F}_{i+1} + \mat{E}_{i+1}) \max \{\bphi_{i+1}^{\err}(\tilde{\x}^{(t)}, \tilde{\x}^{(t)}),   \bphi_{i+1}^{\err}(\tilde{\v}^{(t)},  \tilde{\v}^{(t)})\}
	+ 4 \delta L_m \|\x^{\star}\|_2^2 \unit_j^\top \mat{Z}_{i+1} \ones
	\\
	\leq & 	
 	\unit_j^\top (\mat{F}_{i+1} + \mat{E}_{i+1}) 	\begin{bmatrix}
		\id_{i-1} & 0 \\
		0 & 2 \globalcond \cdot \id_{m-i+1}
	\end{bmatrix}	\bphi_i^{\err}(\tilde{\x}^{(t)}, \tilde{\v}^{(t)}) 
	+ 4 \delta L_m \|\x^{\star}\|_2^2 \unit_j^\top \mat{Z}_{i+1} \ones.
	\tag{by \Cref{lem:agd:max:of:potentials}}
\end{align}

For $j = i$,
\begin{align}
	& \unit_i^{\top} \bphi_i^{\err}  (\x^{(t+1)}, \v^{(t+1)}) = \frac{1}{2}	\psi_i (\x^{(t+1)}, \v^{(t+1)}) = \frac{1}{2} \err_i(\x^{(t+1)}) + \frac{1}{2} \res_i(\v^{(t+1)}) 
	\\
	\leq & \frac{1}{2} \err_i^{\max} (\widehat{\acbsls}_{i+1} (\tilde{\x}^{(t)}, \tilde{\x}^{(t)})) + \frac{1}{2}  \res_i^{\max} (\widehat{\acbsls}_{i+1} (\tilde{\v}^{(t)}, \tilde{\v}^{(t)}))
	\\
	= & \frac{1}{2} \unit_i^\top \left( \bphi_{i+1}^{\err} (\widehat{\acbsls}_{i+1} (\tilde{\x}^{(t)}, \tilde{\x}^{(t)})) +  \bphi_{i+1}^r (\widehat{\acbsls}_{i+1} (\tilde{\v}^{(t)}, \tilde{\v}^{(t)})) \right)
	\\
	\leq & \frac{1}{2}  \unit_i^\top \left( (\mat{F}_{i+1} + \mat{E}_{i+1}) \bphi_{i+1}^{\err}(\tilde{\x}^{(t)}, \tilde{\x}^{(t)}) +   (\mat{F}_{i+1} + \mat{E}_{i+1}) \bphi_{i+1}^r(\tilde{\v}^{(t)},  \tilde{\v}^{(t)})  \right) 	+ 4 \delta L_m \|\x^{\star}\|_2^2 \unit_i^\top \mat{Z}_{i+1} \ones
	\\
	\leq &  \unit_i^\top (\mat{F}_{i+1} + \mat{E}_{i+1}) 				\begin{bmatrix}
		\globalcond \id_{i} &  & \\
		 & 1  & \\
		& & \globalcond \id_{m-i-1}
	\end{bmatrix} \bphi_i^{\err}(\tilde{\x}^{(t)}, \tilde{\v}^{(t)}) 
	+ 4 \delta L_m \|\x^{\star}\|_2^2 \unit_i^\top \mat{Z}_{i+1} \ones.
	\tag{by \Cref{lem:agd:sum:of:potentials:er}}
\end{align}

For $j > i$
\begin{align}
	& \unit_j^{\top} \bphi_i^{\err}  (\x^{(t+1)}, \v^{(t+1)}) = \frac{1}{2}	\psi_j (\x^{(t+1)}, \v^{(t+1)}) = \frac{1}{2} \err_j(\x^{(t+1)}) + \frac{1}{2 }r_j(\v^{(t+1)}) 
	\\
	= & \frac{1}{2 }\psi_j(\widehat{\acbsls}_{i+1} (\tilde{\x}^{(t)}, \tilde{\x}^{(t)})) + \frac{1}{2} \psi_j (\widehat{\acbsls}_{i+1} (\tilde{\v}^{(t)}, \tilde{\v}^{(t)}))
	\\
	= & \unit_j^\top \left( \bphi_{i+1}^{\err} (\widehat{\acbsls}_{i+1} (\tilde{\x}^{(t)}, \tilde{\x}^{(t)})) +  \bphi_{i+1}^r (\widehat{\acbsls}_{i+1} (\tilde{\v}^{(t)}, \tilde{\v}^{(t)})) \right)
	\\
	\leq &  \unit_j^\top \left( (\mat{F}_{i+1} + \mat{E}_{i+1}) \bphi_{i+1}^{\err}(\tilde{\x}^{(t)}, \tilde{\x}^{(t)}) +   (\mat{F}_{i+1} + \mat{E}_{i+1}) \bphi_{i+1}^r(\tilde{\v}^{(t)},  \tilde{\v}^{(t)})  \right) 	+ 8 \delta L_m \|\x^{\star}\|_2^2 \unit_j^\top \mat{Z}_{i+1} \ones 
	\\
	\leq &  \unit_j^\top (\mat{F}_{i+1} + \mat{E}_{i+1}) 		\begin{bmatrix}
	  2	\globalcond \id_{i} &  & \\
		 & 2 & \\
		& & 2 \globalcond \id_{m-i-1}
	\end{bmatrix} \bphi_i^{\err}(\tilde{\x}^{(t)}, \tilde{\v}^{(t)}) 
	+ 8 \delta L_m \|\x^{\star}\|_2^2 \unit_j^\top \mat{Z}_{i+1} \ones 
	\\
	\leq & 2 \globalcond \unit_j^\top (\mat{F}_{i+1} + \mat{E}_{i+1}) 	\bphi_i^{\err}(\tilde{\x}^{(t)}, \tilde{\v}^{(t)}) 	+ 8 \delta L_m \|\x^{\star}\|_2^2 \unit_j^\top \mat{Z}_{i+1} \ones.
\end{align}

In matrix form we obtain
\begin{equation}
	\bphi_i^{\err}  (\x^{(t+1)}, \v^{(t+1)})
	\leq 
	\mat{K}_i \widehat{\mat{F}_{i+1}} 	\bphi_i^{\err}  (\tilde{\x}^{(t)}, \tilde{\v}^{(t)}) + 8 \delta L_m \|\x^{\star}\|_2^2 \mat{Z}_{i+1} \ones.
\end{equation}	
Hence
\begin{equation}
	\bphi_i^{\err}  (\x^{(t+1)}, \v^{(t+1)})
	\leq 
	\mat{K}_i 	\widehat{\mat{F}_{i+1}}  \widehat{\mat{D}_{i}} 	\bphi_i^{\err}  (\x^{(t)}, \v^{(t)})
	+
	4 \delta L_m \|\x^{\star}\|_2^2  (	\mat{K}_i \widehat{\mat{F}_{i+1}} +  2 \mat{Z}_{i+1} ) \ones.
\end{equation}
Telescoping 
\begin{equation}
	\bphi_i^{\err}  (\x^{(T_i)}, \v^{(T_i)}) 
	\leq 
	(	\mat{K}_i 	\widehat{\mat{F}_{i+1}}  \widehat{\mat{D}_{i}})^{T_i} \bphi_i^{\err}(\x^{(0)}, \v^{(0)}) 
	+ 
	4 \delta L_m \|\x^{\star}\|_2^2 \cdot \sum_{t_i=0}^{T_i-1} (		\mat{K}_i \widehat{\mat{F}_{i+1}}  \widehat{\mat{D}_{i}})^{t_i}  (
		\mat{K}_i \widehat{\mat{F}_{i+1}} +  2 \mat{Z}_{i+1} ) \ones.
\end{equation}
\end{proof}

Next, we estimate the upper bounds of $\mat{F}_i$ (in \Cref{lem:acbsls:stab:F:bound}), $\mat{E}_i$ (in \Cref{lem:acbsls:stab:E:bound}) and $\mat{Z}_i$'s (in \Cref{lem:acbsls:stab:Z:bound}).
\subsubsection{Upper bound of $\mat{F}$}
We first bound $\|\mat{F}_i\|_1$ with the following lemma.
\begin{lemma}[Upper bound of $\|\mat{F}_i\|_1$]
	\label{lem:acbsls:stab:F:bound}
	Using the same notation as in \Cref{lem:acbsls:stab:abs}, and in addition assume $T_1, \ldots, T_m$ satisfies \eqref{eq:acbsls:t:lb}, then the following statements hold,
	\begin{enumerate}[(a),leftmargin=*]
		\item For any $i \in [m+1]$, $\mat{F}_i$ is a diagonal matrix of the form $\prod_{j=i}^m \left(	(\mat{K}_j {\mat{D}}_j)^{\prod_{k=i}^j T_k} \right)$.
		\item For any $i \in [m]$, $\|\mat{K}_i \mat{F}_{i+1} \mat{D}_i\|_1 \leq 1$.
		\item For any $i \in [m+1]$, $\|\mat{F}_i\|_1 \leq 1$.
		\item $\|\mat{F}_1\|_1 \leq \frac{\epsilon}{\psi(\x^{(0)}, \v^{(0)})}$.
	\end{enumerate}
\end{lemma}
\begin{proof}[Proof of \Cref{lem:acbsls:stab:F:bound}]
	\begin{enumerate}[(a), leftmargin=0pt, itemindent=15pt]
		\item The first statement (a) follows immediately by definition of $\mat{F}_i$'s. We prove by induction in reverse order from $m+1$ doewn to $1$. 
	For $i = m+1$ we have $\mat{F}_{m+1} = \id$ by definition, which is consistent. Now assume the statement holds for the case of $i+1$, then the case of $i$ also holds in that 
		\begin{equation}
			\mat{F}_i = (\mat{K}_i \mat{F}_{i+1} \mat{D}_i)^{T_i} = \left( \mat{K}_i \prod_{j=i+1}^m \left(	(\mat{K}_j {\mat{D}}_j)^{\prod_{k=i+1}^j T_k}\right) \mat{D}_i  \right)^{T_i}
			= \prod_{j=i}^m \left(	(\mat{K}_j {\mat{D}}_j)^{\prod_{k=i}^j T_k} \right),
		\end{equation}
		where the last equality holds because of the commutability among diagonal matrices.

		\item By (a) we have 
		\begin{equation}
			\mat{K}_i \mat{F}_{i+1} \mat{D}_i = \mat{K}_i \prod_{j=i+1}^m \left(	(\mat{K}_j {\mat{D}}_j)^{\prod_{k=i+1}^j T_k}\right) \mat{D}_i 
			=  \prod_{j=i}^m \left(	(\mat{K}_j {\mat{D}}_j)^{\prod_{k=i+1}^j T_k} \right),
		\end{equation}
		By definition of $\mat{D}_i$ we can write the $l$-th diagonal element of $\mat{K}_i \mat{F}_{i+1} \mat{D}_i$ as 
		\begin{equation}
			(\mat{K}_i \mat{F}_{i+1} \mat{D}_i)_{ll} = 
			\begin{cases}
				1 & l < i, \\
				(1 - \kappa_l^{-\frac{1}{2}})^{\prod_{k=i+1}^l T_k} \cdot (4  \globalcond^4)^{\sum_{j=i}^{l-1} \prod_{k=i+1}^{j} T_k } & l \geq i.
			\end{cases}
		\end{equation}
		Note that 
		\begin{align}
			& (1 - \kappa_l^{-\frac{1}{2}})^{\prod_{k=i+1}^l T_k} \cdot (4  \globalcond^4)^{\sum_{j=i}^{l-1} \prod_{k=i+1}^{j} T_k}
			\leq 
			\exp \left(  \underbrace{ - \kappa_l^{-\frac{1}{2}} \prod_{k=i+1}^l T_k  + \log (4  \globalcond^4) \cdot  \sum_{j=i}^{l-1} \prod_{k=i+1}^{j} T_k }_{\text{denoted as $\gamma_l$}}\right).
		\end{align}
		Observe that $\gamma_{i} = -\kappa_i^{-1} < 0$. For $l \geq i$, it is the case that
		\begin{align}
			& \gamma_{l+1} - \gamma_l 
			\\
			= & - \kappa_{l+1}^{-\frac{1}{2}} \prod_{k=i+1}^{l+1} T_k  + \log (4 \globalcond^4) \cdot  \sum_{j=i}^{l} \prod_{k=i+1}^{j} T_k + 
			\kappa_{l}^{-\frac{1}{2}} \prod_{k=i+1}^{l} T_k  - \log (4 \globalcond^4) \cdot  \sum_{j=i}^{l-1} \prod_{k=i+1}^{j} T_k
			\\
			= &  - \kappa_{l+1}^{-\frac{1}{2}} \prod_{k=i+1}^{l+1} T_k + 	\kappa_{l}^{-\frac{1}{2}} \prod_{k=i+1}^{l} T_k + \log (4 \globalcond^4) \cdot   \prod_{k=i+1}^{l} T_k
			\\
			= & \prod_{k=i+1}^{l} T_k \left( - \kappa_{l+1}^{-\frac{1}{2}} T_{l+1} + \kappa_l^{-\frac{1}{2}} + \log (4 \globalcond^4)\right)
			\leq 0 \tag{since $T_{l+1} \geq \kappa_{l+1}^{\frac{1}{2}} (1 + \log (4 \globalcond^4)) $ by \eqref{eq:acbsls:t:lb}}
		\end{align}
		Hence $\gamma_{m} \leq \gamma_{m-1} \leq \cdots \leq \gamma_i \leq 0$. Consequently we have $(\mat{K}_i \mat{F}_{i+1} \mat{D}_i)_{ll} \leq 1$ for all $l$, and thus $\|\mat{K}_i \mat{F}_{i+1} \mat{D}_i\|_1 \leq 1$.  
		\item By (b), $\|\mat{F}_{i}\|_1 = \|(\mat{K}_i \mat{F}_{i+1} \mat{D}_i)^{T_i}\|_1 \leq \|(\mat{K}_i \mat{F}_{i+1} \mat{D}_i)\|_1^{T_i} \leq 1$.
		\item Follows by the same argument as in the exact arithmetic proof \Cref{thm:acbsls}, which we sketch here for completeness. By (a),
		\begin{align}
			& (\mat{F}_1)_{ll} = \left[  \prod_{j=1}^m \left(	(\mat{K}_j {\mat{D}}_j)^{\prod_{k=1}^j T_k} \right) \right]_{ll}
			= 
			(1 - \kappa_l^{-\frac{1}{2}})^{\prod_{k=1}^l T_k} \cdot (4  \globalcond^4)^{\sum_{j=i}^{l-1} \prod_{k=1}^{j} T_k }
			\\
			\leq & \exp \left(  \underbrace{ - \kappa_l^{-\frac{1}{2}} \prod_{k=1}^l T_k  + \log (4  \globalcond^4) \cdot  \sum_{j=1}^{l-1} \prod_{k=1}^{j} T_k }_{\text{denoted as $\gamma_l$}}\right).
		\end{align}
		Since $T_1 \geq \kappa_1^{\frac{1}{2}} \log (\frac{\psi(\x^{(0)}, \v^{(0)})}{\epsilon})$, we have $\gamma_1 = -\kappa_1^{-\frac{1}{2}} T_1 \leq - \log (\frac{\psi(\x^{(0)}, \v^{(0)})}{\epsilon})$.
		Following the same argument as in (b), we have $\gamma_m \leq \gamma_{m-1} \leq \cdots \leq \gamma_1$. 
		Consequently, for any $l \in [m]$
		\begin{equation}
			(\mat{F}_1)_{ll} \leq \exp \left( - \log (\frac{\psi(\x^{(0)}, \v^{(0)})}{\epsilon}) \right) \leq \frac{\epsilon}{\psi(\x^{(0)}, \v^{(0)})},
		\end{equation}
		which implies $\|\mat{F}_1\|_1 \leq \frac{\epsilon}{\psi(\x^{(0)}, \v^{(0)})}$ since $\mat{F}_1$ is diagonal (by (a)).
	\end{enumerate}
\end{proof}

\subsubsection{Upper bound of $\mat{E}$}
Before we state the upper bound of $\mat{E}_i$'s, we first establish the following \Cref{lem:acbsls:stab:diff:bound}, which is essential towards the bound for $\mat{E}_i$'s and $\mat{Z}_i$'s. 
\begin{lemma}
	\label{lem:acbsls:stab:diff:bound}
	Using the same notation of \Cref{lem:acbsls:stab:abs}, and assume $T_1, \ldots, T_m$ satisfies \eqref{eq:acbsls:t:lb}, and assume $\delta  \leq \frac{1}{20 m  \globalcond^2}$, 
	then for any $t \geq 0$, for any $i \in [m]$, the following inequality holds
	\begin{align}
		\left\|(	\mat{K}_i \widehat{\mat{F}_{i+1}} \widehat{\mat{D}_i})^{t} - (	\mat{K}_i \mat{F}_{i+1} {\mat{D}}_i)^{t} \right\|_1 
		\leq
		\begin{cases}
			\varphi(10 t \delta m \globalcond^2 ) & i = m
			\\
			\varphi( 40 t \delta m \globalcond^6 + 12  t \globalcond^4 \|\mat{E}_{i+1}\|_1 )  & i < m,
		\end{cases}
	\end{align}	
	where $\varphi(x) \defeq xe^x$.
\end{lemma}
\begin{proof}[Proof of \Cref{lem:acbsls:stab:diff:bound}]
	Let $\bXi_i \defeq 	\mat{K}_i \widehat{\mat{F}_{i+1}} \widehat{\mat{D}_i} - 	\mat{K}_i \mat{F}_{i+1} \mat{D}_i$ (which is non-negative), then 
	\begin{align}
		& \left\|(	\mat{K}_i \widehat{\mat{F}_{i+1}} \widehat{\mat{D}_i})^{t} - (	\mat{K}_i  \mat{F}_{i+1} \mat{D}_i )^{t} \right\|_1 
		= \left\|( 	\mat{K}_i  \mat{F}_{i+1} \mat{D}_i + \bXi_i)^t -  (	\mat{K}_i  \mat{F}_{i+1} \mat{D}_i)^t \right\|_1 
		\\
		\leq & 
		\sum_{s=1}^{t} {t \choose s} \|\bXi_i\|^s		\left\| 	\mat{K}_i  \mat{F}_{i+1} \mat{D}_i \right\|_1^{t-s} 
		\leq \sum_{s=1}^{t} {t \choose s} \|\bXi_i\|_1^s
		\tag{since $\|	\mat{K}_i  \mat{F}_{i+1} \mat{D}_i\|_1 \leq 1$}
		\\
		\leq &  \|\bXi_i\|_1 t \sum_{s=0}^{t-1} {t-1\choose s} \|\bXi_i\|_1^s
		= \|\bXi_i\|_1 t \left( 1 + \|\bXi_i\|_1 \right)^{t-1} 
		\tag{by \Cref{helper:choose}}
		\\
		\leq &  \|\bXi_i\|_1 t \exp (\|\bXi_i\|_1 t).
		\label{eq:lem:acbsls:stab:diff:bound}
	\end{align}
	It remains to bound $\|\bXi_i\|_1$. 
	For $i = m$ we have $\mat{E}_{m+1} = 0$, $\mat{F}_{m+1} = \id$, $\widehat{\mat{F}_{m+1}} = \id$, $\mat{K}_{m} = \id$,  which implies $	\| \bXi_m \|_1 = \| \widehat{\mat{D}_m} - \id \|_1 \leq 10 \delta m   \globalcond^2$. 

	For other $i < m$, first note that (since $\mat{F}_i$ is diagonal)
	\begin{equation}
		\widehat{\mat{F}_{i+1}} 
		= 
		\mat{F}_{i+1}
		+
		\underbrace{\begin{bmatrix}
			[\id_{i-1} | \mat{0}_{(i-1) \times (m-i+1)}] \mat{E}_{i+1}
				\begin{bmatrix}
				\id_{i-1} & 0 \\
				0 & 2 \globalcond \cdot \id_{m-i+1}
			\end{bmatrix} 
			\\
			\unit_i^\top  \mat{E}_{i+1}		
				\begin{bmatrix}
				 \globalcond \id _{i} &  & \\
				 & 1  & \\
				& & \globalcond \id_{m-i-1}
			\end{bmatrix}
			\\
			 [\mat{0}_{(m-i) \times i} | \id_{m-i}]  \mat{E}_{i+1}
		\end{bmatrix}}_{\text{denoted as \ding{168}}}
		\label{eq:tmp:bound:hatF}
	\end{equation}
	thus 
	\begin{align}
		& 	\widehat{\mat{F}_{i+1}} \widehat{\mat{D}_i} - \mat{F}_{i+1} \mat{D}_i =  (\mat{F}_{i+1} + \text{\ding{168}}) (\id + 10 \delta \globalcond^2 \ones \ones^\top) \mat{D}_i - \mat{F}_{i+1} \mat{D}_i
		\\
		= &  10 \delta \globalcond^2  \mat{F}_{i+1} \ones \ones^\top \mat{D}_i + 
		\text{\ding{168}} (\id + 10 \delta \globalcond^2 \ones \ones^\top) \mat{D}_i  
	\end{align}
	Since $\|{\text{\ding{168}}}\|_1 \leq 2  \globalcond  \|\mat{E}_{i+1}\|_1$ and $\|\mat{D}_i\| \leq 2\globalcond^3 $ we have bound 
	\begin{equation}
		\left\|	\widehat{\mat{F}_{i+1}} \widehat{\mat{D}_i} - \mat{F}_{i+1} \mat{D}_i \right\|_1 
		\leq 
		20 \delta m \globalcond^5 + 4  \globalcond^4 \|\mat{E}_{i+1}\|_1 (1 + 10 \delta m  \globalcond^2).
	\end{equation}
	Thus (for $i < m$)
	\begin{align}
			&	\|\bXi_i \|_1  \leq 2  \globalcond   \left( 20 \delta m \globalcond^5 + 4  \globalcond^4 \|\mat{E}_{i+1}\|_1 (1 + 10 \delta m  \globalcond^2) \right)
			\\
			\leq & 40 \delta m  \globalcond^6 + 8  \globalcond^4 \|\mat{E}_{i+1}\|_1 (1 + 10 \delta m  \globalcond^2) 
			\leq 40 \delta m  \globalcond^6 + 12  \globalcond^4 \|\mat{E}_{i+1}\|_1
			\tag{since $\delta \leq \frac{1}{20 m  \globalcond^2}$}.
	\end{align}
	Substituting back to \cref{eq:lem:acbsls:stab:diff:bound} completes the proof.
\end{proof}

Now we state the bound for $\mat{E}_i$'s.
\begin{lemma}[Upper bound of $\|\mat{E}_i\|_1$]
	\label{lem:acbsls:stab:E:bound}
	Using the same notation of \Cref{lem:acbsls:stab:abs}, and assume $T_1, \ldots, T_m$ satisfies \eqref{eq:acbsls:t:lb}, and assume
	\begin{equation}
		\delta \leq \frac{1}{2 \cdot (10  \globalcond^2)^{2m-1} m \prod_{j=1}^m T_j},
		\label{eq:acbsls:stab:E:delta:bound}
	\end{equation}
	then for any $t \geq 0$, for any $i \in [m]$, the following inequality holds
	then the following inequality holds 
	\begin{equation}
		\|\mat{E}_i\|_1 \leq  2 \cdot (10 \globalcond^2 )^{2(m-i)+1} \delta m \prod_{j=i}^m T_j.
		\label{eq:acbsls:stab:E:bound}
	\end{equation}
\end{lemma}
\begin{proof}[Proof of \Cref{lem:acbsls:stab:E:bound}]
	We prove by induction in reverse order from $m$ down to $1$.

	For $i=m$, by definition of $\mat{E}_m$,
	\begin{align}
			&		\|\mat{E}_m\|_1 = \left\| (\mat{K}_m \widehat{\mat{F}_{m+1}} \widehat{\mat{D}_m})^{T_m} - (\mat{K}_m	\mat{F}_{m+1} \mat{D}_m )^{T_m} \right\|_1
			\\
			\leq & 10 \delta m \globalcond^2 T_m \exp (10 \delta m \globalcond^2 T_m) 
			\tag{by \Cref{lem:acbsls:stab:diff:bound}}
			\\
			\leq & 10\sqrt{e} \delta m \globalcond^2 T_m  \leq 17 \delta m \globalcond^2 T_m, 			\tag{since $\delta \leq \frac{1}{20 m  \globalcond^2}$}
	\end{align}
	which satisfies the bound \eqref{eq:acbsls:stab:E:bound}.

	Now suppose the statement holds for the case of $i + 1$, we then study the case of $i$. 
	By definition of $\mat{E}_i$,
	\begin{align}
		& \|\mat{E}_i\|_1 = 	\left\|(	\mat{K}_i \widehat{\mat{F}_{i+1}} \widehat{\mat{D}_i})^{T_i} - (	\mat{K}_i \mat{F}_{i+1} {\mat{D}}_i)^{T_i} \right\|_1 
		\\
	\leq & \left( 40 \delta m \globalcond^6 + 12  \globalcond^4 \|\mat{E}_{i+1}\|_1    \right) T_i \cdot 
	\exp \left[ \left( 40 \delta m \globalcond^6 + 12  \globalcond^4 \|\mat{E}_{i+1}\|_1    \right) T_i \right]
	\tag{by \Cref{lem:acbsls:stab:diff:bound}}
\end{align}
	Observe that
	\begin{align}
		& 		\left( 40 \delta m  \globalcond^6 + 12  \globalcond^4 \|\mat{E}_{i+1}\|_1    \right) T_i
		\\
		\leq & 40 \delta m \globalcond^6 T_i + 24  \globalcond^4 (10 \globalcond^2)^{2(m-i-1)+1} \delta m \prod_{j=i}^m T_j
		\tag{by induction hypothesis \eqref{eq:acbsls:stab:E:bound}}
		\\
		\leq & 28  \globalcond^4 (10  \globalcond^2 )^{2(m-i-1)+1} \delta m \prod_{j=i}^m T_j
		\label{eq:lem:acbsls:stab:E:bound:1}
	\end{align}
	and thus 
	\begin{align}
			& \exp \left[ \left( 40 \delta m  \globalcond^6 + 12  \globalcond^4 \|\mat{E}_{i+1}\|_1    \right) T_i  \right]
		\leq \exp \left[  28  \globalcond^4 (10  \globalcond^2 )^{2(m-i-1)+1} \delta m \prod_{j=i}^m T_j \right]
		\tag{by \eqref{eq:lem:acbsls:stab:E:bound:1}}
		\\
		= & \exp \left[ 0.28 (10 \globalcond^2)^{2m-1 } \delta m \prod_{j=1}^m T_j \right]
		\leq \exp (0.14) < 1.16
		\tag{by $\delta$ bound \eqref{eq:acbsls:stab:E:delta:bound}}
	\end{align}
	Consequently
	\begin{align}
		&	\|\mat{E}_i\|_1 \leq 28  \globalcond^4 (10  \globalcond^2 )^{2(m-i)+1} \delta m \prod_{j=i}^m T_j \times 1.16
		\\
		\leq & 
		45  \globalcond^4 (10  \globalcond^2 )^{2(m-i-1)+1} \delta m \prod_{j=i}^m T_j
		=  0.45 \cdot (10  \globalcond^2 )^{2(m-i)+1} \delta m \prod_{j=i}^m T_j
		\\
		\leq & 2 \cdot (10  \globalcond^2 )^{2(m-i)+1} \delta m \prod_{j=i}^m T_j.
	\end{align}
\end{proof}

\subsubsection{Upper bound of $\mat{Z}$}
\begin{lemma}[Upper bound of $\|\mat{Z}_i\|_1$]
	\label{lem:acbsls:stab:Z:bound}
	Using the same notation of \Cref{lem:acbsls:stab:abs} and assuming the same assumptions of \Cref{lem:acbsls:stab:E:bound}, then the following inequality holds for any $i \in [m]$,
	\begin{equation}
		\|\mat{Z}_i\|_1 \leq 3^{m-i+2} \globalcond \prod_{j=i}^m T_j.
	\end{equation}
\end{lemma}
\begin{proof}[Proof of \Cref{lem:acbsls:stab:Z:bound}]
	Recall the definition of $\mat{Z}_i$ from \Cref{lem:acbsls:stab:abs},
	\begin{equation}
		\mat{Z}_i \defeq \sum_{t_{i}=0}^{T_i-1} ( \mat{K}_i \widehat{\mat{F}_{i+1}} \widehat{\mat{D}_i})^{t_i} \left(  \mat{K}_i  \widehat{\mat{F}_{i+1}} + 2 \mat{Z}_{i+1} \right) 
	\end{equation}
	We will bound $\|\sum_{t_{i}=0}^{T_i-1} ( \mat{K}_i \widehat{\mat{F}_{i+1}} \widehat{\mat{D}_i})^{t_i}\|_1$ and $\|  \mat{K}_i  \widehat{\mat{F}_{i+1}} + 2 \mat{Z}_{i+1} \|_1$ separately. The former is bounded as 
	\begin{align}
		&  \left\|  \sum_{t_{i}=0}^{T_i-1} (  \mat{K}_i  \widehat{\mat{F}_{i+1}} \widehat{\mat{D}_i})^{t_i} \right\|_1 \leq 
		\sum_{t_{i}=0}^{T_i-1} \left\|  (  \mat{K}_i \widehat{\mat{F}_{i+1}} \widehat{\mat{D}_i})^{t_i} \right\|_1
		\tag{by triangle inequality}
		\\
		\leq & \sum_{t_{i}=0}^{T_i-1} \left( \|  ( \mat{K}_i  \mat{F}_{i+1} {\mat{D}}_i)^{t_i} \|_1  +  \left\| ( \mat{K}_i \widehat{\mat{F}_{i+1}} \widehat{\mat{D}_i})^{t_i} - ( \mat{K}_i \mat{F}_{i+1} {\mat{D}}_i)^{t_i} \right\| \right)
		\tag{by triangle inequality}
		\\
		\leq & T_i + \sum_{t_i=0}^{T_i-1}  \left\| ( \mat{K}_i \widehat{\mat{F}_{i+1}} \widehat{\mat{D}_i})^{t_i} - ( \mat{K}_i  \mat{F}_{i+1} {\mat{D}}_i)^{t_i} \right\| 
		\tag{since $\|\mat{K}_i \widehat{\mat{F}_{i+1}} \widehat{\mat{D}_i}\|_1 \leq 1$ by \Cref{lem:acbsls:stab:F:bound}}
		\\
		\leq & T_i \left( 1 + 0.45 \cdot (10  \globalcond^2 )^{2(m-i)+1} \delta m \prod_{j=i}^m T_j \right)
		\tag{by the proof of \Cref{lem:acbsls:stab:diff:bound,lem:acbsls:stab:E:bound}}
		\\
		\leq & 1.23 T_i 
		\tag{by $\delta$ bound \eqref{eq:acbsls:stab:E:delta:bound}}
	\end{align}

	The second term is bounded as 
	\begin{align}
		& \left\| \mat{K}_i \widehat{\mat{F}_{i+1}} + 2 \mat{Z}_{i+1} \right\|_1 
		\leq 
		\|  \mat{K}_i \|_1 \cdot \|\widehat{\mat{F}_{i+1}} \|_1 + 2 \|\mat{Z}_{i+1}\|_1
		\\
		\leq & 
		2 \globalcond \cdot (1 + 2 \globalcond \|\mat{E}_{i+1}\|_1) + 2 \|\mat{Z}_{i+1}\|_1
		\tag{by  \eqref{eq:tmp:bound:hatF}}
		\\
		\leq & 2.04 \globalcond + 2 \|\mat{Z}_{i+1}\|_1
		\tag{since $		\|\mat{E}_{i+1}\|_1 \leq \frac{1}{100 m \globalcond^2}		$ by \Cref{lem:acbsls:stab:E:bound}}
	\end{align}

	In summary $\|\mat{Z}_{i}\|_1 \leq 1.23 T_i ( 2.04 \globalcond  + 2\|\mat{Z}_{i+1}\|_1)$. 
	By induction we can show that $\|\mat{Z}_i\|_1 \leq 3^{m-i+2} \globalcond \prod_{j=1}^m T_j$.
\end{proof}

\subsection{Finishing the proof of \Cref{thm:acbsls:inexact:general}}
\label{sec:thm:acbsls:inexact:5}
We are ready to finish the proof of \Cref{thm:acbsls:inexact:general}.
\begin{proof}[Proof of \Cref{thm:acbsls:inexact:general}]
	By \Cref{lem:acbsls:stab:abs} we have 
	\begin{equation}
		\bphi_1^{\err}(\widehat{\acbsls_{1}}(\x^{(0)}, \v^{(0)})) 
		\leq 
		(\mat{F}_1 + \mat{E}_1) \bphi_1^{\err}(\x^{(0)}, \v^{(0)}) + 4 \delta L_m \|\x^{\star}\|_2^2 \mat{Z}_1 \ones.
	\end{equation}
	By definition of vector potential $\bphi_1^{\err}$ we have $\|\bphi_1^{\err}(\x, \v)\|_1 = \psi(\x, \v)$ for any $\x, \v$. 
	Consequently
	\begin{align}
		& \psi(\widehat{\acbsls_{1}}(\x^{(0)}, \v^{(0)})) 
		\\
		\leq & \left\| (\mat{F}_1 + \mat{E}_1) \bphi_1^{\err}(\x, \v) + 4 \delta L_m \|\x^{\star}\|_2^2 \mat{Z}_1 \ones  \right\|_1
		\tag{by  \Cref{lem:acbsls:stab:abs}}
		\\
		\leq & 
		(\|\mat{F}_1\|_1 + \|\mat{E}_1\|_1) \psi(\x, \v) + 4 \delta m L_m \|\x^{\star}\|_2^2 \|\mat{Z}_1\|_1 
		\tag{triangle inequality}
	\end{align}
	Plugging in the bound of $\|\mat{F}_1\|$ (from \Cref{lem:acbsls:stab:F:bound}), $\|\mat{E}_1\|$ (from \Cref{lem:acbsls:stab:E:bound}), and $\|\mat{Z}_1\|_1$ (from \Cref{lem:acbsls:stab:Z:bound}), we arrive at
	\begin{align}
		& \psi(\widehat{\acbsls_{1}}(\x^{(0)}, \v^{(0)})) 
		\\
		\leq &  
		\left(  \frac{\epsilon}{\psi(\x^{(0)}, \v^{(0)})} + 2 \cdot (10  \globalcond^2 )^{2(m-i)+1} \delta m \prod_{j=i}^m T_j \right) \psi(\x^{(0)}, \v^{(0)}) +  4 \cdot 3^{m+1} \delta m \globalcond L_m \|\x^{\star}\|_2^2  \prod_{j=1}^m T_j.
	\end{align}
	By $\delta$ bound 	
	\begin{small}
    	\begin{equation}
    		\delta \leq 
    		\left( \prod_{i \in [m]} T_i^{-1} \right)
    		\min
    		\left\{
    		\frac{1}{2 \cdot (10  \globalcond^2)^{2m-1} m }, 
    		\frac{\epsilon}{2 \cdot (10 \globalcond^2)^{2m-1} \cdot \psi(\x^{(0)}, \v^{(0)}) }, 
    		\frac{\epsilon}{4 \cdot 3^{m+1} \cdot m L_m \globalcond \|\x^{\star}\|_2^2 }
    		\right\},
    	\end{equation}
	\end{small}
	we immediately obtain $\psi(\widehat{\acbsls_{1}}(\x^{(0)}, \v^{(0)})) \leq 3 \epsilon$.
\end{proof}

\section{Deferred proof of supporting lemmas of \cref{lem:lb:green:ub}}
\label{apx:lb:deferred}
In this appendix section we provide the proof of several supporting lemmas toward \cref{lem:lb:green:ub}.

\subsubsection{Deferred proof of \cref{lem:reduction:to:roots}}
\label{sec:pf:lem:reduction:to:roots}
\begin{proof}[Proof of \cref{lem:reduction:to:roots}]
    Apply \cref{lem:lb:widom}, since for any $k \in [m-1]$,
    \begin{equation}
        \int_{L_k}^{\mu_{k+1}} \frac{h(\zeta)}{\sqrt{q(\zeta)}} = 0,
    \end{equation}
    we conclude that $h(z)$ has $m-1$ real roots $r_1, r_2, \ldots, r_{m-1}$ such that $r_k \in [L_k, \mu_{k+1}]$. 
    Therefore
    \begin{equation}
        g_S(0)
        =
        \int_0^{\mu_1} \frac{ \prod_{k \in [m-1]} (r_k - \zeta)} { \sqrt{\prod_{k \in [m]}(\mu_k - \zeta) (L_k - \zeta)}} \diff \zeta
    \end{equation}
     By monotonocity, 
    \begin{align}
        & \int_0^{\mu_1} \frac{ \prod_{k \in [m-1]} (r_k - \zeta)} { \sqrt{\prod_{k \in [m]}(\mu_k - \zeta) (L_k - \zeta)}}   \diff \zeta
        \leq 
        \frac{\prod_{k \in [m-1]} r_k}{\sqrt{\prod_{k \in [m]} (L_k - \mu_1) \cdot \prod_{k=2}^{m} (\mu_k - \mu_1)}} \cdot \int_0^{\mu_1} \frac{\diff \zeta}{\sqrt{\mu_1 - \zeta}}  \diff \zeta
        \\
        = & \frac{ \prod_{k \in [m-1]} r_k }{\sqrt{\prod_{k \in [m]} (L_k - \mu_1) \cdot \prod_{k=2}^{m} (\mu_k - \mu_1)}} \cdot 2 \sqrt{\mu_1}
        = \frac{2 \prod_{k \in [m-1]} \frac{r_k}{\mu_k} }{\sqrt{\prod_{k \in [m]} (\frac{L_k}{\mu_k} - \frac{\mu_1}{\mu_k}) \cdot \prod_{k=2}^{m} (1- \frac{\mu_1}{\mu_k})} }.
        \label{eq:pf:lem:reduction:to:roots}
    \end{align}
     By assumption $\frac{L_k}{\mu_k} \geq 2$ we have $\frac{\mu_1}{\mu_k} \leq  \frac{1}{2^{k-1}}$, and thus $\frac{L_k}{\mu_k} - \frac{\mu_1}{\mu_k} \geq (1 - \frac{1}{2^k}) \frac{L_k}{\mu_k}$. Hence 
    \begin{equation}
        \prod_{k \in [m]} \left(\frac{L_k}{\mu_k} - \frac{\mu_1}{\mu_k}\right) \cdot \prod_{k=2}^{m} \left(1- \frac{\mu_1}{\mu_k} \right)
        \geq 
        \left(\prod_{k \in [m]} \frac{L_k}{\mu_k} \right) \cdot 
        \left(\prod_{k \in [m]} (1 - {2^{-k}}) \right)
        \left(\prod_{k \in [m-1]} (1 - {2^{-k}}) \right)
    \end{equation}
    Since $\prod_{k=1}^{\infty} (1 - 2^{-k}) > 0.288$ (see \cref{helper:Pochhammer}) we arrive at 
    \begin{equation}
        g_S(0) 
        \leq 
        \frac{ 2 \prod_{k \in [m-1]} {\frac{r_k}{\mu_{k+1}}} }{\sqrt{0.288^2 \prod_{k \in [m]} \frac{L_k}{\mu_k}}}
        \leq 
        \frac{ 7 \prod_{k \in [m-1]} {\frac{r_k}{\mu_{k+1}}} }{\sqrt{\prod_{k \in [m]} \frac{L_k}{\mu_k}}}.
    \end{equation}
    completing the proof.
\end{proof}

\subsubsection{Deferred proof of \cref{lem:lb:roots:ub}}
\label{sec:pf:lem:roots:ub}
\begin{proof}[Proof of \cref{lem:lb:roots:ub}]
    Since $h(z) = \prod_{j=1}^{m-1} (z - r_j)$ satisfies 
    \begin{equation}
        0 = 
        \int_{L_k}^{\mu_{k+1}} \frac{h(\zeta)}{\sqrt{q(\zeta)}}  \diff \zeta
        =
        \int_{L_k}^{\mu_{k+1}} \frac{ \prod_{j \in [m]} (\zeta - r_j)}{\sqrt{q(\zeta)}} \diff \zeta,
    \end{equation}
    we have for any $k \in [m-1]$
    \begin{equation}
        \int_{L_k}^{\mu_{k+1}} \frac{ r_k \prod_{j \neq k} (\zeta - r_j)}{\sqrt{q(\zeta)}} \diff \zeta
        = 
        \int_{L_k}^{\mu_{k+1}} \frac{ \zeta\prod_{j \neq k} (\zeta - r_j)}{\sqrt{q(\zeta)}} \diff \zeta.
    \end{equation}
    Thus 
    \begin{equation}
        r_k = 
        \frac{
            \int_{L_{k}}^{\mu_{k+1}} 
            \frac{
                \zeta \prod_{j \in [k-1]} (r_j - \zeta) \cdot \prod_{j=k+1}^{m-1} (\zeta - r_j) 
                }
                {
                \sqrt{ \prod_{j \in [m]} (\zeta - \mu_j) \prod_{j \in [m]} (\zeta - L_j) } 
                }
             \diff \zeta
        }
        {
            \int_{L_{k}}^{\mu_{k+1}} 
            \frac{
                \prod_{j \in [k-1]} (r_j - \zeta) \cdot \prod_{j=k+1}^{m-1} (\zeta - r_j)
                } 
                { 
                \sqrt{ \prod_{j \in [m]} (\zeta - \mu_j) \prod_{j \in [m]} (\zeta - L_j) } 
                } 
            \diff \zeta
        }.
    \end{equation}
    Rearranging
    \begin{footnotesize}
    \begin{equation}
        r_k = 
        \frac{
            \int_{L_{k}}^{\mu_{k+1}} 
            \frac{\zeta}{\sqrt{(\zeta - \mu_k)(\zeta - L_k)(\mu_{k+1} - \zeta)(L_{k+1} - \zeta)}}
                 \left( \prod_{j \in [k-1]} \frac{\zeta - r_j}{\sqrt{(\zeta - \mu_j)(\zeta - L_j)}} \right)
                 \left( \prod_{j=k+1}^{m-1} \frac{r_j-\zeta}{\sqrt{(\mu_{j+1} - \zeta)(L_{j+1} - \zeta)}} \right)
             \diff \zeta
        }
        {
            \int_{L_{k}}^{\mu_{k+1}} 
            \frac{1}{\sqrt{(\zeta - \mu_k)(\zeta - L_k)(\mu_{k+1} - \zeta)(L_{k+1} - \zeta)}} 
                 \left( \prod_{j \in [k-1]} \frac{\zeta - r_j}{\sqrt{(\zeta - \mu_j)(\zeta - L_j)}} \right)
                 \left( \prod_{j=k+1}^{m-1} \frac{r_j-\zeta}{\sqrt{(\mu_{j+1} - \zeta)(L_{j+1} - \zeta)}} \right)
             \diff \zeta
        }.
    \end{equation}
    \end{footnotesize}
    Observe that 
    \begin{itemize}[leftmargin=*]
        \item For any $j < k$, $\frac{\zeta - r_j}{\sqrt{(\zeta - \mu_j) (\zeta - L_j)}}$ is non-negative and monotonically \textbf{increasing} in $\zeta \in [L_k, \mu_{k+1}]$ since $\mu_j < L_j < r_j$.
        \item For any $j > k$, $\frac{r_j - \zeta}{\sqrt{(\zeta - \mu_{j+1}) (\zeta - L_{j+1})}}$ is non-negative and monotonically \textbf{decreasing} in $\zeta \in [L_k, \mu_{k+1}]$ since $r_j < \mu_{j+1} < L_{j+1}$.
    \end{itemize}
    Consequently
        \begin{align}
        r_k
        \leq 
        \frac{\int_{L_{k}}^{\mu_{k+1}} \frac{ \zeta \diff \zeta} { \sqrt{ (\zeta - \mu_{k})(\zeta - L_{k}) (\mu_{k+1} - \zeta) ( L_{k+1}-\zeta)}} }
        {\int_{L_{k}}^{\mu_{k+1}} \frac{ \diff \zeta} { \sqrt{ (\zeta - \mu_{k})(\zeta - L_{k}) (\mu_{k+1} - \zeta) ( L_{k+1}-\zeta)}} } 
        \cdot 
        \prod_{j \in [k-1]} 
        \left(
            \underbrace{
                        \frac{
                    \frac{\mu_{k+1} - r_j}{\sqrt{(\mu_{k+1} - \mu_j)(\mu_{k+1} - L_j)}}
                }
                {
                    \frac{L_k - r_j}{\sqrt{(L_{k} - \mu_j)(L_{k} - L_j)}}
                }
            }_{\text{denoted as $\gamma_j$}}
        \right)
        \label{eq:lem:lb:roots:ub:2}
    \end{align}
    Note that 
    \begin{equation}
        \gamma_j :=\frac{
            \frac{\mu_{k+1} - r_j}{\sqrt{(\mu_{k+1} - \mu_j)(\mu_{k+1} - L_j)}}
        }
        {
            \frac{L_k - r_j}{\sqrt{(L_{k} - \mu_j)(L_{k} - L_j)}}
        }
        =
        \frac{1 - \frac{r_j}{\mu_{k+1}}}{1 - \frac{r_j}{L_k}} \cdot \frac{\sqrt{(1 - \frac{\mu_j}{L_{k}}) (1 - \frac{L_j}{L_{k}}) }}{ \sqrt{(1 - \frac{\mu_j}{\mu_{k+1}}) (1 - \frac{L_j}{\mu_{k+1}}) }} 
        \leq 
        \frac{1}{1 - \frac{r_j}{L_k}}.
    \end{equation}
    and by $\min_{j \in [m]} \frac{L_j}{\mu_j} \geq 2$ one has
    \begin{equation}
        \frac{r_j}{L_k} \leq \frac{\mu_{j+1}}{L_k} \leq 2^{-(k-j)}.
    \end{equation}
    We arrive at (by \cref{helper:Pochhammer}, $\prod_{j=1}^{\infty} {1 - 2^{-i}} \geq 0.288$)
    \begin{equation}
        \prod_{j \in [k-1]} \gamma_j
        \leq 
        \prod_{j \in [k-1]} \frac{1}{1 - 2^{-(k-j)}} 
        \leq 
        \frac{1}{0.288}
        \leq 4.
    \end{equation}
    Substitute back to \cref{eq:lem:lb:roots:ub:2},
    \begin{equation}
        r_k \leq  4 \cdot \frac{\int_{L_{k}}^{\mu_{k+1}} \frac{ \zeta \diff \zeta} { \sqrt{ (\zeta - \mu_{k})(\zeta - L_{k}) (\mu_{k+1} - \zeta) ( L_{k+1}-\zeta)}} }
        {\int_{L_{k}}^{\mu_{k+1}} \frac{ \diff \zeta} { \sqrt{ (\zeta - \mu_{k})(\zeta - L_{k}) (\mu_{k+1} - \zeta) ( L_{k+1}-\zeta)}} } .
    \end{equation}
\end{proof}

\subsubsection{Deferred proof of \cref{lem:lb:ratio}}
\label{sec:pf:lem:lb:ratio}
We will prove \cref{lem:lb:ratio} by analyzing the integrals in the numerator and denominator.
In particular, we will apply the tools from elliptic integral theory. %
We adopt the following Legendre forms of elliptic integrals, defined as follows:
\begin{itemize}[leftmargin=*]
    \item $\EllipticF(\phi | p) := \int_0^{\phi} \frac{\diff \theta}{\sqrt{1 - p \sin^2 \theta}} $ denotes the (incomplete) elliptic integral of the \textbf{first} kind.
    \item $\EllipticK(p) := \EllipticF(\frac{\pi}{2} | p)$ denotes the complete elliptic integral of the \textbf{first} kind.
    \item $\EllipticE(\phi | p) := \int_0^{\phi} \sqrt{1 - p \sin^2\theta} \diff \theta$ denotes the (incomplete) elliptic integral of the \textbf{second} kind.
    \item $\EllipticE(p) := \EllipticE(\frac{\pi}{2} | p)$ denotes the complete elliptic integral of the \textbf{second} kind.
    \item $\EllipticPi(n; \phi | p) := \int_0^{\phi} \frac{\diff \theta}{(1 - n \sin^2\theta) \sqrt{(1 - p\sin^2\theta)}}$ denotes the incomplete elliptic integral of the \textbf{third} kind.
    \item $\EllipticPi(n | p) := \EllipticPi(n; \frac{\pi}{2}| p)$ denotes the complete elliptic integral of the \textbf{third} kind. 
\end{itemize}

To simplify the notation, throughout this subsubsection we assume without loss of generality that $k=1$. The result apparently holds for any $k \in [m-1]$.

Denote (throughout this subsubsection) that
\begin{equation}
    \varphi(z) := (z - \mu_1) (z- L_1) (\mu_2 - z) (L_2 - z).
\end{equation}

The following \Cref{lem:lb:0-order} analyzes $ \int_{L_1}^{\mu_2} \frac{\diff \zeta}{\sqrt{\varphi (\zeta)}}$.
\begin{lemma}
    \label{lem:lb:0-order}
    For any $0 < \mu_1 < L_1 < \mu_2 < L_2$, the following equality holds
    \begin{equation}
        \int_{L_1}^{\mu_2} \frac{\diff \zeta}{\sqrt{\varphi (\zeta)}} = 
        \frac{2 \EllipticK \left(  \frac{(L_2 - \mu_1)(\mu_2 - L_1)}{(L_2 - L_1)(\mu_2 - \mu_1)} \right)}{\sqrt{(L_2 - L_1)(\mu_2 - \mu_1)}}.
    \end{equation}
\end{lemma}
\begin{proof}[Proof of \Cref{lem:lb:0-order}]
    The antiderivative of $\frac{1}{\sqrt{\varphi(\zeta)}}$
    in $[L_1, \mu_2]$ is 
    \begin{equation}
        Q(\zeta) := 
        -2 \sqrt{\frac{-1}{(L_1-\mu_1) \left(L_2 - \mu_2 \right)}} \EllipticF \left(\Arcsin
   \left(\sqrt{-\frac{\left(\zeta-L_1\right) (L_2-\mu_2)}{(L_2-L_1)
   \left(\mu_2-\zeta\right)}}\right) \middle| \frac{(L_2-L_1) \left(\mu_2-\mu
   _1\right)}{(L_1-\mu_1) (L_2-\mu_2)}\right).
    \end{equation}
    The lemma then follows by the fact that $\lim_{\zeta \to L_1} Q(\zeta) = 0$ and 
    \begin{align}
        \lim_{\zeta \to \mu_2} Q(\zeta) = 
        2 \sqrt{\frac{1}{(L_1-\mu_1) \left(L_2 - \mu_2 \right)}} 
        \EllipticK \left( - \frac{(L_2 - \mu_1)(\mu_2 - L_1)}{(L_1 - \mu_1) (L_2 - \mu_2)} \right)
        =
        \frac{2 \EllipticK \left(  \frac{(L_2 - \mu_1)(\mu_2 - L_1)}{(L_2 - L_1)(\mu_2 - \mu_1)} \right)}{\sqrt{(L_2 - L_1)(\mu_2 - \mu_1)}}.
    \end{align}
\end{proof}

The following \Cref{lem:lb:1-order} analyzes $\int_{L_1}^{\mu_2} \frac{\zeta \diff \zeta}{\sqrt{\varphi (\zeta)}}$.
\begin{lemma}
    For any $0 < \mu_1 < L_1 < \mu_2 < L_2$, the following equality holds
    \label{lem:lb:1-order}
    \begin{align}
        \int_{L_1}^{\mu_2} \frac{\zeta \diff \zeta}{\sqrt{\varphi (\zeta)}} = & 
        2 \sqrt{\frac{1}{(\mu_2-\mu_1) (L_2-L_1)}} 
        \cdot \\
        & 
        \left(L_2
        \EllipticK \left(\frac{(L_2-\mu_1) (\mu_2-L_1)}{(L_2-L_1)
        (\mu_2-\mu_1)}\right)-(L_2-\mu_2) \EllipticPi \left(\frac{\mu
        _2-L_1}{L_2-L_1} \middle| \frac{(L_2-\mu_1) \left(\mu
        _2-L_1\right)}{(L_2-L_1) (\mu_2-\mu_1)}\right)\right).
    \end{align}
\end{lemma}

\begin{proof}[Proof of \Cref{lem:lb:1-order}]
    The antiderivative of $\frac{\zeta}{\sqrt{\varphi(\zeta)}}$ in $[L_1, \mu_2]$ is 
    \begin{align}
        P(\zeta) := & 
        2 \sqrt{\frac{1}{(\mu_2-\mu_1) (L_2-L_1)}} \cdot \\
        & \left(i L_1 \EllipticF\left(\Arcsin \left(\sqrt{-\frac{(L_2-L_1) \left(\mu_2-\zeta\right)}{\left(\zeta-L_1\right)
        (L_2-\mu_2)}}\right) \middle|\frac{(L_1-\mu_1) \left(L_2-\mu
        _2\right)}{(L_2-L_1) (\mu_2-\mu_1)}\right)\right. \\
        & \left. +i \left(\mu
        _2-L_1\right) \EllipticPi \left(\frac{L_2-\mu_2}{L_2-L_1}; \Arcsin\left(\sqrt{-\frac{(L_2-L_1) \left(\mu_2-\zeta\right)}{\left(\zeta-L_1\right)
        (L_2-\mu_2)}}\right) \middle|\frac{(L_1-\mu_1) \left(L_2-\mu
        _2\right)}{(L_2-L_1) (\mu_2-\mu_1)}\right)\right).
    \end{align}
The lemma then follows by the fact that $\lim_{\zeta \to \mu_2} P(\zeta) = 0$ and 
\begin{align}
    \lim_{\zeta \to L_1} P(\zeta) 
    = & 
    2 \sqrt{\frac{1}{(\mu_2-\mu_1) (L_2-L_1)}} 
    \cdot \\
    & 
    \left(- L_2
    \EllipticK \left(\frac{(L_2-\mu_1) (\mu_2-L_1)}{(L_2-L_1)
    (\mu_2-\mu_1)}\right)
    + (L_2-\mu_2) \EllipticPi \left(\frac{\mu
    _2-L_1}{L_2-L_1} \middle| \frac{(L_2-\mu_1) \left(\mu
    _2-L_1\right)}{(L_2-L_1) (\mu_2-\mu_1)}\right)\right).
\end{align}
\end{proof}

Next, we establish the following inequality regarding elliptic integrals.
\begin{lemma}
    \label{helper:elliptic:ratio}
    For any $x \in [0, \frac{1}{2}]$ and $y \in (0, 1]$, the following inequality holds
    \begin{equation}
        \frac{\EllipticPi(x|1-y)}{\EllipticK(1-y)} \geq \frac{1}{1-x} \left(1 - \frac{4x}{\log ( \frac{16}{y}) } \right).
    \end{equation}
\end{lemma}
\begin{proof}[Proof of \cref{helper:elliptic:ratio}]
    We first prove two claims:
    \begin{claim}
        \label{helper:pi:concave}
        The function $(1-x) \EllipticPi(x|1-y)$ is concave in $x$ for $x \in [0,1)$ and $y \in (0,1]$.
    \end{claim}
    \begin{proof}[Proof of \cref{helper:pi:concave}]
        By standard convex analysis we can show that $\frac{1-x}{1 - x \sin^2{\theta}}$ is concave in $x$ for any $x \in [0, 1)$ and $\theta \in [0, \frac{\pi}{2}]$. 
        Thus by linearity of integrals we have that $(1-x) \EllipticPi(x | 1-y) = \int_0^{\frac{\pi}{2}} \frac{(1-x) \diff \theta}{(1 - x \sin^2\theta) \sqrt{1 - (1-y) \sin^2\theta}}$ is also concave in $x \in [0,1)$ provided that $y \in (0,1]$.
    \end{proof}
    \begin{claim}
        \label{helper:deri:lb}
        Let $\varphi(x, y) := (1-x) \cdot \frac{\EllipticPi(x| 1-y)}{\EllipticK(1-y)}$. Then for any $y \in (0,1]$, the following inequality holds 
        \begin{equation}
            \frac{\partial \varphi(\frac{1}{2}, y)}{\partial x} \geq \frac{-4}{\log \frac{16}{y}}.
        \end{equation}
    \end{claim}
    \begin{proof}[Proof of \cref{helper:deri:lb}]
        By standard elliptic integral analysis 
        \begin{equation}
             \frac{\partial \varphi(\frac{1}{2}, y)}{\partial x} = 
             -1 + \frac{2 \EllipticE(1-y) - \EllipticPi(\frac{1}{2}, 1-y)}{(2y-1) \EllipticK(1-y)}.
        \end{equation}
        Expanding the above quantity around $y = 0$ shows the lower bound $\frac{-4}{\log{\frac{16}{y}}}$.
    \end{proof}
    The proof of \cref{helper:elliptic:ratio} then follows by the above two claims. Since $(1-x) \EllipticPi(x | 1-y)$ is concave in $x \in [0, \frac{1}{2}]$, so is $\varphi(x,y)$ defined in \cref{helper:deri:lb}. Thus for any $x \in [0, \frac{1}{2}]$,
    \begin{equation}
        \varphi(x, y) \geq \varphi(0, y) + x \frac{\partial \varphi(\frac{1}{2}, y)}{\partial x} \geq 1 - \frac{4 x}{\log \frac{16}{y}}.
    \end{equation}
    Hence 
    \begin{equation}
          \frac{\EllipticPi(x|1-y)}{\EllipticK(1-y)} = \frac{\varphi(x,y)}{1-x} 
          \geq \frac{1}{1-x} \left(1 - \frac{4x}{\log ( \frac{16}{y}) } \right).
    \end{equation}
\end{proof}

Finally, the proof of \Cref{lem:lb:ratio} then follows by applying \Cref{lem:lb:0-order,lem:lb:1-order,helper:elliptic:ratio}.
\begin{proof}[Proof of \Cref{lem:lb:ratio}]
    By \Cref{lem:lb:0-order,lem:lb:1-order,helper:elliptic:ratio},
    \begin{align}
        & \left( \int_{L_1}^{\mu_2} \frac{\diff \zeta}{\sqrt{\varphi (\zeta)}} \right)^{-1}
        \left( \int_{L_1}^{\mu_2} \frac{\zeta \diff \zeta}{\sqrt{\varphi (\zeta)}} \right)
        = L_2 - (L_2 - \mu_2) 
        \frac{ \EllipticPi \left(\frac{\mu_2-L_1}{L_2-L_1} \middle| \frac{(L_2-\mu_1) (\mu_2-L_1)}{(L_2-L_1) (\mu_2-\mu_1)}\right)}
        {\EllipticK \left(\frac{(L_2-\mu_1) (\mu_2-L_1)}{(L_2-L_1)(\mu_2-\mu_1)}\right)}
        \\
    \leq & L_2 - (L_2 - \mu_2) \frac{1}{1 -  \frac{\mu_2-L_1}{L_2-L_1} } \left( 1 - \frac{4 \frac{\mu_2-L_1}{L_2-L_1} }{\log \left(16 \frac{(\mu_2-\mu_1) (L_2-L_1)}{(L_1-\mu_1) (L_2-\mu_2)} \right)} \right)
    \tag{by \cref{helper:elliptic:ratio} and assumption $\frac{\mu_2}{L_2} \leq \frac{1}{2}$}
        \\
    = & L_1 + \frac{4 (\mu_2 - L_1)}{\log \left(16 \frac{(\mu_2-\mu_1) (L_2-L_1)}{(L_1-\mu_1) (L_2-\mu_2)} \right)} 
    \leq 
    L_1 + \frac{4 \mu_2 }{\log \left(16 \frac{(\mu_2-\mu_1) (L_2-L_1)}{(L_1-\mu_1) (L_2-\mu_2)} \right)}.
    \end{align}
    Since 
    \begin{equation}
        \frac{(\mu_2-\mu_1) (L_2-L_1)}{(L_1-\mu_1) (L_2-\mu_2)} 
        = \left( 1 + \frac{ \frac{\mu_2}{L_1} - 1}{1 - \frac{\mu_1}{L_1}} \right) \left( 1 + \frac{\mu_2 - L_1}{L_2 - \mu_2} \right)
        \geq
        1 + \frac{\mu_2}{L_1} - 1 
        =
        \frac{\mu_2}{L_1},
    \end{equation}
    and 
    \begin{equation}
        \frac{L_1}{\mu_2} \leq \frac{3}{\log (16 \frac{\mu_2}{L_1})},
    \end{equation}
    we obtain
    \begin{align}
         \left( \int_{L_1}^{\mu_2} \frac{\diff \zeta}{\sqrt{\varphi (\zeta)}} \right)^{-1}
        \left( \int_{L_1}^{\mu_2} \frac{\zeta \diff \zeta}{\sqrt{\varphi (\zeta)}} \right)
         \leq
         L_1 + \frac{4 \mu_2 }{\log (16 \frac{\mu_2}{L_1})}
         \leq
        \frac{7 \mu_2}{\log (16 \frac{\mu_2}{L_1})}.
    \end{align}
\end{proof}

\section{Technical details for \cref{sec:stochastic} }
\begin{remark}
\label{example:distassumption}
We can obtain an example of a non-product, non-Gaussian distribution that satisfies the second-order independence condition by considering any collection of pairwise independent random variables that take values in $\{\pm 1\}^d$ and have expected value 0. As an example: let $\mat{P}$ be the $d\times d$ identity matrix. Let $\vec{a} \in \{\pm 1\}^d$ be uniform on the subset of the hypercube $ \left \{ \pm 1\right \}^d$ with even parity: $\Pi_{i=1}^d \vec{a}_i = 1$. Now, note that the distribution of $\vec{a}$ does not have a product structure but the second-order independence condition is satisfied: $\E[(\mat{P} \vec{a})_i (\mat{P} \vec{a})_k^2 (\mat{P} \vec{a})_j]=0$ since each pair of coordinates of $\vec{a}$ is pairwise independent (for $d>2$). Note that we could also obtain a candidate distribution which satisfies the property if $\mat{P}$ is not the identity: We can choose the distribution over $\vec{a}$ such that $\mat{P}\vec{a}$ is still uniform on the subset of the hypercube with even parity as before. Another different example is any distribution over $\left \{ \vec{a} \in \mathbb{R}^d \textrm{ s.t. } \norm{\mat{P} \vec{a}}_0 \leq 1 \right \} $.

\end{remark}
\subsection{Missing details for proof of \cref{lemma:bslsres} \cref{eqn:messybetacondappend}} 
\label{append:messydetailsstochastic}
We want to show that if 
\begin{align}
    \delta \leq \min \bigg \{ & \frac{1}{3m} \frac{1}{T_{\textrm{max}}(\nsteps_1 + 2)^2 + 1}, \\
    & \frac{1}{6m(4^m)} \left( 48 m T_{\textrm{max}}^2 (\nsteps_1 + 2)^3\right)^{-1/2},\\
    & \frac{1}{6m(4^m)} \left( \max_{\ell \in [m]} \left \{ \cond_{\ell} \right \} \cdot m T_{\textrm{max}}^2 (\nsteps_1 + 2)^3 \right)^{-1/2}  \bigg \},
\end{align}
then 
\begin{equation}
    \label{eqn:messybetacondappend}
     4^m (1 + 3 \delta m)^{m \left( T_{\textrm{max}}( \nsteps_1 + 2)^2 + 1 \right) }  \leq \min \left \{ \frac{\globalcond}{(1 + 3 \delta m)^{\nsteps_1}}, \frac{1}{144 \nsteps_1 T_{\textrm{max}} \delta m}, \frac{1}{2 \max_{\ell \in [m]} \left \{  \cond_{\ell} \right \} } \frac{1}{6(\nsteps_1 + 1) \delta m} \right \}. 
\end{equation} 
Using that for any $x \geq 0$, $(1 + \frac{1}{x})^x \leq e$ and that $\globalcond \geq (4e)^m$ we have,
\begin{align}
     4^m (1 + 3 \delta m)^{m \left( T_{\textrm{max}}( \nsteps_1 + 2)^2 + 1 \right) } \leq (4e)^m \leq \globalcond,
\end{align}
therefore 
\begin{equation}
    4^m (1 + 3 \delta m)^{m \left( T_{\textrm{max}}( \nsteps_1 + 2)^2 + 1 \right) } \leq \frac{\globalcond}{(1 + 3 \delta m)^{\nsteps_1}}.
\end{equation}
Next suppose for some $x, p, C \in \R$ we have $px \leq 1/4$, $x \leq \sqrt{C/4p}$, and $C, p \geq 1$. Then
\begin{align}
    (1 + x)^p x \leq (1 + 2 px) x \leq \sqrt{\frac{C}{4p}} + \frac{C}{4} \leq C. 
\end{align}
Then letting $x = 3 \delta m$, $p = m (T_{\textrm{max}}(\nsteps_1 + 2)^2 + 1)$, and $C = (144 \nsteps_1 T_{\textrm{max}} 4^m)^{-1}$ we have that when $\delta \leq  \frac{1}{6m(4^m)} \left( 48 m T_{\textrm{max}}^2 (\nsteps_1 + 2)^3\right)^{-1/2}$ then 
\begin{equation}
     4^m (1 + 3 \delta m)^{m \left( T_{\textrm{max}}( \nsteps_1 + 2)^2 + 1 \right) }  \leq \frac{1}{144 \nsteps_1 T_{\textrm{max}} \delta m}. 
\end{equation}
Similarly, letting $x = 3 \delta m$, $p = m (T_{\textrm{max}}(\nsteps_1 + 2)^2 + 1)$, and $C = \left( 2 \max_{\ell \in [m]} \left \{ \cond_{\ell} \right \} 6 (\nsteps_1 + 1) 4^m \right)^{-1}$ we have that when $\delta \leq \frac{1}{6m(4^m)} \left( \max_{\ell \in [m]} \left \{ \cond_{\ell} \right \} \cdot m T_{\textrm{max}}^2 (\nsteps_1 + 2)^3 \right)^{-1/2}  $ then
\begin{equation}
     4^m (1 + 3 \delta m)^{m \left( T_{\textrm{max}}( \nsteps_1 + 2)^2 + 1 \right) }  \leq  \frac{1}{2 \max_{\ell \in [m]} \left \{  \cond_{\ell} \right \} } \frac{1}{6(\nsteps_1 + 1) \delta m}. 
\end{equation}
Therefore Eq.~\ref{eqn:messybetacondappend} holds. 

\subsection{Proof of Lemma~\ref{lemma:maininduction}}
\label{appendarxiv:maininductionproof}
\begin{proof}[Proof of \cref{lemma:maininduction}]
We prove Lemma~\ref{lemma:maininduction} by induction on $i$. We start with the base case of $\bslsres_{\nbands}$ and then we will show that if Lemma~\ref{lemma:maininduction} is true for $\bslsres_{i+1}$ then it is true for $\bslsres_i$.  Let $\resv^{(0)}$ denote our input vector $\resv$ to $\bslsres_i$ and $\resv^{(t)}$ denotes the $\nth{t}$ iteration of $\bslsres_i$. Before beginning the proof by induction we establish several useful claims that will be used throughout the proof. First note that if
\begin{equation}
    \navg \geq \distconst \nbands^2 \maxmult \left(\prod_{i=1}^m \cond_i \right) \log^{\nbands}(\globalcond),
\end{equation}
then since $\delta \defeq \distconst \maxmult /\navg$ we have
\begin{equation}
\label{eqn:deltasize}
    \delta \leq \left( \nbands^2 \left(\prod_{i=1}^m \cond_i \right) \log^{\nbands}(\globalcond) \right)^{-1} .
\end{equation}
Note that if $\resv_i^{(t+1)} = \mat{U}_i \resv_i^{(t)}$ we have
\begin{align}
\label{eqn:updates}
    & \resv_i^{(t+1)} \leq \gamma_i \resv_i^{(t)} + \left( \delta \sum_{k = i+1}^m \frac{L_k}{L_i} \resv_k^{(t)} \right) +\left( \delta \sum_{k = 1}^{i-1} \frac{L_k}{L_i} \resv_k^{(t)} \right),\\
    & \resv_j^{(t+1)} \leq  \left(\frac{L_j}{L_i} - 1\right)^2 \resv_j^{(t)} + \delta \sum_{k = 1}^m  \frac{L_k L_j}{L_i^2}\resv_k^{(t)}, \tag{For all $j \geq i+1$} \\
    & \resv_j^{(t+1)} \leq \left(1 - \frac{\mu_j}{C_i L_i}\right)^2 \resv_j^{(t)} + \delta \sum_{k = 1}^m  \frac{L_k L_j}{L_i^2}\resv_k^{(t)} . \tag{For all $j \leq i -1$}
\end{align}
Finally,
\begin{equation}
    \label{eqn:Tisize}
    T_i \geq \threshold{i} .
\end{equation}
To see this note,
\begin{align}
   \threshold{i} & = \frac{\log \left(  \frac{1}{\smallgrowth[\nsteps_{i+1}]{\resv^{(0)}}}  \frac{L_{i-1}}{L_i} \frac{\resv_{i-1}^{(0)}}{\resv_i^{(0)}}\right)}{\log(\tilde{\gamma}_i)} \\
    & = \frac{\log \left( \smallgrowth[\nsteps_{i+1}]{\resv^{(0)}} \frac{L_{i}}{L_{i-1}} \frac{\resv_{i}^{(0)}}{\resv_{i-1}^{(0)}}\right)}{\log(1/\tilde{\gamma}_i)} \\
    & \leq 2 \cond_i \log \left( \smallgrowth[\nsteps_{i+1}]{\resv^{(0)}} \frac{L_i}{L_{i-1}} \frac{\resv_{i}^{(0)}}{\resv_{i-1}^{(0)}} \right) \tag{$\tilde{\gamma}_i = 1 - \frac{1}{2 \cond_i}$ and $\log(1/(1-x)) \geq x$} \\
    & \leq 2 \cond_i \log \left(\smallgrowth[\nsteps_{i+1}]{\resv^{(0)}}  \basegrowthlazy \frac{L_{i}^2}{L_{i-1}^2} \right) \tag{$\resv_i^{(0)}/\resv_{i-1}^{(0)} \leq \basegrowthlazy L_i/L_{i-1}$}\\
    & \leq 6 \cond_i \log \left( \globalcond \right) \tag{$\basegrowthlazy \smallgrowth[\nsteps_{i+1}]{\resv^{(0)}} \leq \globalcond$ and $L_i/L_{i-1} \leq \globalcond$} \\
    & \leq T_i.
\end{align}
\paragraph{Base Case.} Consider $\bslsres_{\nbands}$. Suppose that $\basegrowth{\resv^{(0)}}$ satisfies Eq.~\ref{eqn:betacond}. Since $\resv^{(0)}$ is not ambiguous we will shorten $\basegrowth{\resv^{(0)}}, \smallgrowth{\resv^{(0)}}$, and $\growth{m}{\resv^{(0)}}$ all to $\basegrowthlazy, \smallgrowthlazy$, and $\growthlazy{m}$ respectively.
First we will prove the following claim:
\begin{claim} [Fast Convergence Phase]
    \label{claim:base}
    Recall $\tilde{\gamma}_m$. For any $t \leq \threshold{m} +1 $ we have
    \begin{equation}
        \resv^{(t)} = \mat{V}_m^{t} \resv^{(0)}.
    \end{equation}
    Moreover, defining
\begin{equation}
    \growth{m}{\resv^{(0)}} \defeq  \smallgrowth{\resv^{(0)}}^{\nsteps_{m}+2},
\end{equation}
we have for any $j,k \in \left[ m \right]$,
    \begin{equation}
    \label{eqn:invariant}
    \resv_k^{(t)} \leq  \basegrowthlazy \growthlazy[t]{m} \frac{1}{\tilde{\gamma}_m}\max \left \{ \frac{L_j}{L_k}, \frac{L_k}{L_j} \right \} \resv_j^{(t)}. 
\end{equation}

\end{claim}
We will prove Claim~\ref{claim:base} by induction. The case $t = 0$ holds immediately. Now suppose Claim~\ref{claim:base} is true for $t$. First we prove that for any $t \leq T_m$, 
\begin{equation}
    \label{eqn:bles3}
    \growthlazy[t]{m} \leq 3.
\end{equation}
We first bound $\growthlazy{m}$ by $1 + (1/T_m)$,
\begin{align}
    \growthlazy{m} & = \left( 1 + 3 \delta m \basegrowthlazy \right)^{\nsteps_m +2} \\
    & \leq  1 + 12 (\nsteps_m + 2) \delta m \basegrowthlazy \\
    & \leq 1 + \frac{1}{T_m} \tag{$\basegrowthlazy < 1/(24 \nsteps_m T_m \delta m) $ by Eq.~\ref{eqn:betacond} }.
\end{align}
Therefore since $\sup_{x \geq 0} (1 + (1/x))^x \leq 3$ we have that for any $t \leq T_m$
\begin{equation}
    \growthlazy[t]{m} \leq \growthlazy[T_m]{m} \leq \left(1 + \frac{1}{T_m} \right)^{T_m} \leq 3. 
\end{equation}
This concludes the proof of Eq.~\ref{eqn:bles3}. Now we can continue with our proof by induction of Claim~\ref{claim:base}. Suppose the inductive hypothesis (I.H.) that Claim~\ref{claim:base} holds for $t$. Let $\overline{\resv}^{(t+1)} \defeq \mat{U}_{m} \resv^{(t)}$. By Eq.~\ref{eqn:updates} we have,
\begin{align}
    \overline{\resv}_{\nbands}^{(t+1)} & \leq \gamma_{\nbands} \resv_{\nbands}^{(t)} + \delta \sum_{k = 1}^{\nbands -1} \frac{L_k}{L_{\nbands}} \resv_k^{(t)} \\
    & \leq \gamma_{\nbands} \resv_{\nbands}^{(t)} + \delta \basegrowthlazy \growthlazy[t]{m} \sum_{k = 1}^{\nbands -1} \frac{L_k}{L_{\nbands}} \frac{L_m}{L_k} \resv_m^{(t)} \tag{I.H. of Claim~\ref{claim:base} Eq.~\ref{eqn:invariant}}\\
    & \leq \left( \gamma_{\nbands} + \basegrowthlazy \growthlazy[t]{m} \delta (m-1) \right) \resv_{\nbands}^{(t)} \\
    & \leq \tilde{\gamma}_{\nbands} \resv_{\nbands}^{(t)}.\tag{I.H. of Claim~\ref{claim:base}, Eq.~\ref{eqn:bles3}, and Eq.~\ref{eqn:deltasize}}
\end{align}
Therefore
\begin{equation}
    \label{eqn:part1base}
    \resv^{(t+1)}_m = \max \left \{ \overline{\resv}^{(t+1)}_m, \left( \mat{V}_i \resv^{(t)} \right)_m \right \} = \max \left \{ \overline{\resv}^{(t+1)}, \tilde{\gamma}_m \resv_m^{(t)} \right \} = \tilde{\gamma}_m \resv_m^{(t)}.
\end{equation}
Next for $j \leq m-1$,
\begin{align}
    \overline{\resv}_j^{(t+1)} & \leq \resv_j^{(t)} + \left( \delta \sum_{k = 1}^{j} \frac{L_k L_j}{L_m^2} \resv_k^{(t)} \right) + \left( \delta \sum_{k = j+1}^{m} \frac{L_k L_j}{L_m^2} \resv_k^{(t)} \right) \\
   & \leq \resv_j^{(t)} + \left( \delta \basegrowthlazy \growthlazy[t]{m} \sum_{k = 1}^{j} \frac{L_k L_j}{L_m^2} \frac{L_j}{L_k} \resv_j^{(t)} \right) + \left( \delta \basegrowthlazy \growthlazy[t]{m} \sum_{k = j+1}^{m} \frac{L_k L_j}{L_m^2} \frac{L_k}{L_j} \resv_j^{(t)} \right) \tag{I.H. of Claim~\ref{claim:base} Eq.~\ref{eqn:invariant}} \\
   & = \left(1 + \basegrowthlazy \growthlazy[t]{m} \delta m \right) \resv_j^{(t)} \\
   & \leq \smallgrowthlazy \resv_j^{(t)}. \tag{ Eq.~\ref{eqn:bles3}}
\end{align}
Therefore for $j \leq m-1$,
\begin{equation}
    \label{eqn:part2base}
    \resv^{(t+1)}_j = \max \left \{ \overline{\resv}^{(t+1)}_j, \left( \mat{V}_i \resv^{(t)} \right)_j \right \} = \max \left \{ \overline{\resv}^{(t+1)},\smallgrowthlazy \resv_j^{(t)} \right \} =\smallgrowthlazy \resv_j^{(t)}.
\end{equation}
Combining Eq.~\ref{eqn:part1base} and Eq.~\ref{eqn:part2base} with the I.H. of Claim~\ref{claim:base} for $t$ shows that $\resv^{(t+1)} = \mat{V}_m^{t+1} \resv^{(0)}.$
Finally to establish Claim~\ref{claim:base} we first show that if $t \leq \threshold{i} -1 $ then for all $j,k$,
\begin{equation}
    \resv_k^{(t)} \leq \basegrowthlazy \growthlazy[t]{m} \max \left \{ \frac{L_j}{L_k}, \frac{L_k}{L_j} \right \} \resv_j^{(t)}.
\end{equation}
Then we show that this implies that for $t = \left \lceil \threshold{i} \right \rceil$, for all $j,k$,
\begin{equation}
    \resv_k^{(t)} \leq \basegrowthlazy \growthlazy[t]{m} \frac{1}{\tilde{\gamma}_m}\max \left \{ \frac{L_j}{L_k}, \frac{L_k}{L_j} \right \} \resv_j^{(t)}.
\end{equation}
If both $j,k \leq \nbands - 1$ then
\begin{align}
    \frac{\resv_k^{(t+1)}}{\resv_j^{(t+1)}} & = \frac{\smallgrowthlazy \resv_k^{(t)}}{\smallgrowthlazy \resv_j^{(t)}} \tag{Eq.~\ref{eqn:part2base}} \\
    & \leq \basegrowthlazy\growthlazy[t]{m} \max \left \{ \frac{L_j}{L_k}, \frac{L_k}{L_j} \right \} \tag{I.H. of Claim~\ref{claim:base}} \\
    & \leq \basegrowthlazy \growthlazy[t+1]{m} \max \left \{ \frac{L_j}{L_k}, \frac{L_k}{L_j} \right \}. \tag{$\growthlazy{m} \geq 1$}
\end{align}
Next suppose $k = m$. By Eq.~\ref{eqn:part1base}, $\resv_m^{(t+1)} \leq \resv_m^{(t)}$ and by Eq.~\ref{eqn:part2base} for $j \leq m-1$, $\resv_j^{(t+1)} \geq \resv_j^{(t)}$. Therefore 
\begin{equation}
    \frac{\resv_m^{(t+1)}}{\resv_j^{(t+1)}} \leq \frac{\resv_m^{(t)}}{\resv_j^{(t)}}.
\end{equation}
Then combining this with the I.H. of Claim~\ref{claim:base},
\begin{equation}
    \resv_m^{(t)} \leq \basegrowthlazy \growthlazy[t]{m} \max \left \{\frac{L_j}{L_m}, \frac{L_m}{L_j} \right \} \resv_j^{(t)} \qquad \implies \qquad  \resv_m^{(t+1)} \leq \basegrowthlazy \growthlazy[t+1]{m} \max \left \{\frac{L_j}{L_m}, \frac{L_m}{L_j} \right \} \resv_j^{(t+1)}.
\end{equation}
Finally we consider the case where $j = m$. First suppose $t +1 \leq \threshold{m}$. By I.H. of Claim~\ref{claim:base},
\begin{equation}
\label{eqn:thresholdbase}
\resv_m^{(t+1)} = \tilde{\gamma}_m^{t+1} \resv_m^{(0)} \geq \frac{1}{\smallgrowthlazy}\frac{L_{m-1}}{L_m} \resv_{m-1}^{(0)} \geq  \frac{L_{m-1}}{L_m} \frac{1}{\smallgrowthlazy[t+1]} \resv_{m-1}^{(t)}.
\end{equation} 
Therefore for any $k  \leq m-1$, 
\begin{align}
\resv_k^{(t+1)} & = \smallgrowthlazy \resv_k^{(t)} \\
& \leq  \smallgrowthlazy \basegrowthlazy \growthlazy[t]{m} \max \left \{ \frac{L_{m-1}}{L_k}, \frac{L_k}{L_{m-1}} \right \} \resv_{m-1}^{(t)} \\
& \leq  \smallgrowthlazy \basegrowthlazy \growthlazy[t]{m} \max \left \{ \frac{L_{m-1}}{L_k}, \frac{L_k}{L_{m-1}} \right \} \smallgrowthlazy[t+1] \frac{L_m}{L_{m-1} } \resv_{m}^{(t+1)}    \tag{Eq.~\ref{eqn:thresholdbase}} \\
& = \smallgrowthlazy[t+2] \basegrowthlazy \growthlazy[t]{m} \frac{L_m}{L_k} \resv_m^{(t+1)} \\
& \leq  \basegrowthlazy \growthlazy[t+1]{m} \frac{L_m}{L_k} \resv_m^{(t+1)}. \tag{$\smallgrowthlazy[t+2] \leq \growthlazy{m}$}
\end{align}
Next if $t +1 = \threshold{m} + 1$ then we have
\begin{equation}
\resv_m^{(t+1)}  \geq  \tilde{\gamma}_m \frac{L_{m-1}}{L_m} \frac{1}{\smallgrowthlazy[t+1]} \resv_{m-1}^{(t)},
\end{equation}
and so using the same logic as before,
\begin{equation}
\resv_k^{(t+1)}  \leq \frac{1}{\tilde{\gamma}_m} \basegrowthlazy \growthlazy[t+1]{m} \frac{L_m}{L_k} \resv_m^{(t+1)}.
\end{equation}
This concludes the proof of Claim~\ref{claim:base}. Now we consider the steps for which $t > \left \lceil \threshold{i} \right \rceil$. We have the following claim.
\begin{claim} [Extraneous Steps Phase]
\label{claim:extrabase}
Suppose $t > \left \lceil \threshold{i} \right \rceil$. Then
\begin{equation}
\resv^{(t+1)} = \mat{W}_m \resv^{(t)}.
\end{equation}
Moreover Eq.~\ref{eqn:invariant} holds from Claim~\ref{claim:base}. 
\end{claim}
To prove Claim~\ref{claim:extrabase} we recall that for $\overline{\resv}^{(t+1)} \defeq \mat{U}_m \resv^{(t)}$ we have $\overline{\resv}^{(t+1)} \leq \mat{V}_m \resv^{(t)}$. Next since $\mat{W}_m$ is entry-wise larger than $\mat{V}_m$ we conclude 
\begin{equation}
\resv^{(t+1)} = \max \left \{ \overline{\resv}^{(t+1)}, \mat{W}_m \resv^{(t)} \right \} = \mat{W}_m \resv^{(t)}. 
\end{equation}
Finally we need to show that Eq.~\ref{eqn:invariant} holds. The same reasoning that Eq.~\ref{eqn:invariant} holds in the case $t \leq  \left \lceil \threshold{i} \right \rceil$ can be used to show that Eq.~\ref{eqn:invariant} holds in the case where $t > \left \lceil \threshold{i} \right \rceil$ and $j,k \leq m-1$. We simply need to consider the case where $j = m$ or $k = m$. First suppose $j = m$. We have for any $k \leq m-1$, 
\begin{align}
\resv_k^{(t+1)} & = \smallgrowthlazy \resv_k^{(t)} \tag{I.H. of Claim~\ref{claim:extrabase}} \\
& \leq  \smallgrowthlazy \basegrowthlazy \growthlazy[t]{m} \frac{L_m}{L_k} \resv_m^{(t)} \tag{I.H. of Claim~\ref{claim:extrabase}} \\
& = \smallgrowthlazy \basegrowthlazy \growthlazy[t]{m} \frac{L_m}{L_k} \resv_m^{(t+1)} \tag{$\left(\mat{W}_m\right)_{mm} = 1$} \\
& \leq  \basegrowthlazy \growthlazy[t+1]{m} \frac{L_m}{L_k} \resv_m^{(t+1)}. \tag{$\smallgrowthlazy \leq \growthlazy{m}$}
\end{align}
Next suppose $k = m$ and $j \leq m-1$. This case is trivial since $\resv_j^{(t+1)} \geq \resv_j^{(t)}$ and $\resv_i^{(t+1)} = \resv_i^{(t)}$. Thus since $\growthlazy{i} \geq 1$,
\begin{equation}
    \resv_i^{(t)} \leq \basegrowthlazy \growthlazy[t]{m} \resv_j^{(t)} \implies \resv_i^{(t+1)} \leq \basegrowthlazy \growthlazy[t+1]{m} \resv_j^{(t+1)}.
\end{equation}
This concludes the proof of Claim~\ref{claim:extrabase}.
\paragraph{Conclusion of Base Case} To conclude we use Claim~\ref{claim:base}, Claim~\ref{claim:extrabase}, and Eq.~\ref{eqn:Tisize} which guarantee that
\begin{align}
    \tilde{\resv}_m  & = \resv_m^{(T_m)} \\
    & = \tilde{\gamma}_m^{\left \lceil \threshold{m} \right \rceil} \resv_m^{(0)} \\
    & \leq \tilde{\gamma}_m^{ \threshold{m} } \resv_{m}^{(0)}  \\
    & =  \frac{L_{m-1}}{L_m} \resv_{m-1}^{(0)}.
\end{align}
Next since $\nsteps_m = T_m$ we have for any $j \leq \nbands - 1$,
    \begin{equation}
        \label{eqn:jbound}
        \tilde{\resv}_j = \resv_j^{(\nsteps_m)} = \smallgrowthlazy[\nsteps_m] \resv_j^{(0)}.
    \end{equation}
This concludes the base case.
\paragraph{Inductive Case} 
Suppose $\basegrowth{\resv^{(0)}}$ satisfies Eq.~\ref{eqn:betacond}. 
\begin{claim} [Fast Convergence Phase]
    \label{claim:induct}
    Suppose we are in the $\nth{t}$ iteration of $\bslsres_i$ and assume 
    \begin{equation}
        t \leq \left \lceil \threshold{i} \right \rceil.
    \end{equation}
Then 
\begin{equation}
    \resv^{(t)} = \mat{V}_i^t \resv^{(0)},
\end{equation}
and letting
   \begin{equation}
    \growth{i}{\resv} \defeq \smallgrowth{\resv}^{(\nsteps_i +2 )(N_{i+1} +1)+1}
\end{equation}
we have for any $j,k \in \left[ m \right]$,
\begin{equation}
    \label{eqn:needfortop}
    \resv_k^{(t)} \leq \basegrowth{\resv^{(0)}} \max \left \{ \frac{1}{\tilde{\gamma}_i^2  } \growth[t]{i}{\resv^{(0)}}, \frac{\smallgrowth[\nsteps_{i+1}]{\resv^{(0)}}}{\tilde{\gamma}_i}\right \}\max \left \{ \frac{L_j}{L_k}, \frac{L_k}{L_j}\right \} \resv_j^{(t)}.
\end{equation}

\end{claim}
We will prove Claim~\ref{claim:induct} by induction (be careful not to confuse this with the overarching proof by induction of Lemma~\ref{lemma:maininduction}). Again since $\resv^{(0)}$ unambiguously denotes the initial vector passed to $\bslsres_i$ for the majority of the proof we will shorten $\basegrowth{\resv^{(0)}}, \smallgrowth{\resv^{(0)}}$, and $\growth{i}{\resv^{(0)}}$ all to $\basegrowthlazy, \smallgrowthlazy$, and $\growthlazy{i}$ respectively. Before proceeding with our proof of Claim~\ref{claim:induct} we show that for any $t \leq T_i$, 
\begin{equation}
\label{eqn:maxbless3ind}
    \max \left \{\frac{1}{\tilde{\gamma}_i^2}\growthlazy[t]{i}, \frac{\smallgrowthlazy[\nsteps_{i+1}]}{\tilde{\gamma_i}} \right \} \leq 3,
\end{equation}
so that
\begin{equation}
    \resv_k^{(t)} \leq 3 \basegrowthlazy \max \left \{ \frac{L_j}{L_k}, \frac{L_k}{L_j}\right \} \resv_j^{(t)}.
\end{equation}
First we bound $\growth{i}{\resv^{(0)}}$ by $1 + 1/(3T_i)$. Using that $(1 + x)^p \leq 1 + 2px$ if $px \leq 1/4$ and that $(3 \basegrowthlazy \delta m) (2 \nsteps_{i+1}) \leq 1/4$ we have,
\begin{align}
    \growthlazy{i} & = \left(1 + 3 \basegrowthlazy \delta m \right)^{2 \nsteps_{i+1}} \\
    & \leq \left(1 + 12 \nsteps_{i+1} \basegrowthlazy \delta m \right) \\
    & \leq 1 + \frac{1}{3 T_i}.
    \tag{$\basegrowthlazy \leq 1/(144 \nsteps_{i+1} \delta m T_i)$ by Eq.~\ref{eqn:betacond} }
\end{align}
Therefore for any $t \leq T_i$,
\begin{equation}
\label{eqn:bless3ind}
    \frac{1}{\tilde{\gamma}_i^2} \growthlazy[t]{i} \leq 2 \left(1 + \frac{1}{3 T_i} \right)^t \leq 2 \left( \frac{3}{2} \right) \leq 3. 
\end{equation}
Next we have
\begin{align}
    \frac{\smallgrowthlazy[\nsteps_{i+1}]}{\tilde{\gamma_i}} &  \leq 2\smallgrowthlazy[\nsteps_{i+1}] \tag{$\tilde{\gamma}_i \geq 1/2$ since $\cond_i \geq 1$} \\
     & \leq 2(1 + 6 \nsteps_{i+1} \basegrowthlazy \delta m) \tag{$(1 + x)^p \leq 1 + 2px$ if $px \leq 1/4$ and $(3 \basegrowthlazy \delta m) (\nsteps_{i+1}) \leq 1/4$} \\
     & \leq 3. \tag{$12 \nsteps_{i+1} \basegrowthlazy \delta m \leq 1$ }
\end{align}
Therefore Eq.~\ref{eqn:maxbless3ind} holds. Next note that
\begin{equation}
\label{eqn:gammaired} 
    \left( \gamma_i + 3 \basegrowthlazy \delta m \right) \smallgrowthlazy[\nsteps_{i+1}] \leq 1 - \frac{1}{2 \cond_i}.
\end{equation}
Indeed,
\begin{align}
       \left( \gamma_i + 3 \basegrowthlazy \delta m \right) \smallgrowthlazy[\nsteps_{i+1}]
      & =  \left( 1 - \frac{2}{\cond_i}  + \frac{1}{\cond_i^2} + 3 \basegrowthlazy \delta m \right) \smallgrowthlazy[\nsteps_{i+1}]\\
    & \leq  \left( 1 - \frac{2}{\cond_i}  + \frac{1}{\cond_i^2} + 3 \basegrowthlazy \delta m \right) \left( 1 + 6 \nsteps_{i+1} \basegrowthlazy \delta m \right) \tag{$(1 + x)^p \leq 1 + 2 px$ if $px \leq 1/4$ and $(3 \basegrowthlazy \delta m)(\nsteps_{i+1}) \leq 1/4$} \\
     & \leq  1 - \frac{2}{\cond_i}  + \frac{2}{\cond_i^2} + 6 \basegrowthlazy \delta m + 6 \nsteps_{i+1} \basegrowthlazy \delta m \tag{$6 \nsteps_{i+1} \basegrowthlazy \delta m \leq 1$} \\
     & \leq  1 - \frac{1}{2 \cond_i}. \tag{$6 (\nsteps_{i+1}+1) \basegrowthlazy \delta m \leq 1/(2\cond_i)$}
\end{align}

With Eq.~\ref{eqn:maxbless3ind} and Eq.~\ref{eqn:gammaired} in hand we are ready to  prove Claim~\ref{claim:induct} by induction. The base case $t=0$ holds immediately by assumption. Now using the I.H. of Claim~\ref{claim:induct} at $t$ we will show Claim~\ref{claim:induct} holds at $t+1$. Let $\overline{\resv}^{(t+1)} \defeq \mat{U}_{(1/L_i)} \resv^{(t)}$. If Claim~\ref{claim:induct} holds for $\resv^{(t)}$ then the necessary conditions to apply Lemma~\ref{lemma:maininduction} to $\bslsres_{i+1}$ hold. Indeed, when we call $\bslsres_{i+1}$ to $\resv^{(t)}$ we now have initial vector $\resv_{\textrm{recursive}}^{(0)} = \resv^{(t)}$ and so we simply need to show that Eq.~\ref{eqn:betacond} holds for $\resv_{\textrm{recursive}}^{(0)} $. To this end we see by Claim~\ref{claim:induct}
\begin{equation}
    \basegrowth{\resv^{(t)}} \leq \basegrowth{\resv^{(0)}}  \frac{1}{\tilde{\gamma}_i^2} \growth[t]{i}{\resv^{(0)}} \leq \basegrowth{\resv^{(0)}}   \totalgrowth{\resv^{(0)}}. 
\end{equation}
Therefore
\begin{align}
    \basegrowth{\resv_{\textrm{recursive}}^{(0)}} \totalgrowth[m-(i+1)+1]{\resv_{\textrm{recursive}}^{(0)}} & = \basegrowth{\resv^{(t)}} \totalgrowth[m-(i+1)+1]{\resv^{(t)}} \\
    & \leq \left( \basegrowth{\resv^{(0)}}   \totalgrowth{\resv^{(0)}} \right)  \totalgrowth[m-(i+1)+1]{\resv^{(t)}} \\
    & = \left( \basegrowth{\resv^{(0)}}   \totalgrowth{\resv^{(0)}} \right) \left( 1 + 3 \delta m \basegrowth{\resv^{(t)}} \right)^{T_{\textrm{max}}(\nsteps_1 + 1) + 1} \cdot \max_{\ell \in \left[ m \right]} \left \{  \frac{1}{\tilde{\gamma}_{\ell}^2 }\right \} \\
    & \leq  \left( \basegrowth{\resv^{(0)}}   \totalgrowth{\resv^{(0)}} \right) \left( 1 + 3 \delta m \basegrowth{\resv^{(0)}} \frac{1}{\tilde{\gamma}_i^2} \growth[t]{i}{\resv^{(0)}} \right)^{T_{\textrm{max}}(\nsteps_1 + 1) + 1} \\
    & \qquad \cdot \max_{\ell \in \left[ m \right]} \left \{  \frac{1}{\tilde{\gamma}_{\ell}^2 }\right \}.
\end{align}
Now that we have expressed the above only in terms of $\resv^{(0)}$ we drop its notation and use our typical shorthand. Thus we have,
\begin{align}
    \basegrowth{\resv_{\textrm{recursive}}^{(0)}} \totalgrowth[m-(i+1)+1]{\resv_{\textrm{recursive}}^{(0)}} & \leq \left(  \basegrowthlazy \totalgrowthlazy^{m-(i+1)+1} \right) \left( 1 + 3 \delta m \basegrowthlazy \frac{1}{\tilde{\gamma}_i^2} \growthlazy[t]{i} \right)^{T_{\textrm{max}}(\nsteps_1 + 1) + 1}  \cdot \max_{\ell \in \left[ m \right]} \left \{  \frac{1}{\tilde{\gamma}_{\ell}^2 }\right \} \\
    & \leq \left(  \basegrowthlazy \totalgrowthlazy^{m-(i+1)+1} \right) \left( 1 + 9 \delta m \basegrowthlazy \frac{1}{\tilde{\gamma}_i^2} \right)^{T_{\textrm{max}}(\nsteps_1 + 1) + 1}  \cdot \max_{\ell \in \left[ m \right]} \left \{  \frac{1}{\tilde{\gamma}_{\ell}^2 }\right \} \\
    & \leq \left(  \basegrowthlazy \totalgrowthlazy^{m-(i+1)+1} \right) \totalgrowthlazy \\
    & =  \basegrowthlazy  \totalgrowthlazy^{m-i+1}
\end{align}
Then since Eq.~\ref{eqn:betacond} holds for $\bslsres_i$ we can conclude that it holds for $\bslsres_{i+1}$ with initialization $\resv^{(t)}$. Recalling the notation that $\tilde{\resv}^{(t)} = \bslsres_{i+1}(\resv^{(t)})$, by the (original) inductive hypothesis that Lemma~\ref{lemma:maininduction} is true for $\bslsres_{i+1}$ we have, 
\begin{align}
    \overline{\resv}_i^{(t+1)} & = \gamma_i \tilde{\resv}_i^{(t)} + \left( \delta \sum_{k = i+1}^{m} \frac{L_k}{L_i} \tilde{\resv}_k^{(t)} \right) + \left(\delta \sum_{k = 1}^{i-1} \frac{L_k}{L_i} \tilde{\resv}_k^{(t)} \right) \\
     & \leq \gamma_i \smallgrowthlazy[\nsteps_{i+1}] \resv_i^{(t)} + \left(  \delta \sum_{k = i+1}^{m}    \resv_i^{(t)} \right) +  \smallgrowthlazy[\nsteps_{i+1}] \left(\delta \sum_{k = 1}^{i-1} \frac{L_k}{L_i} \resv_k^{(t)} \right) \\
     & \leq  \gamma_i \smallgrowthlazy[\nsteps_{i+1}] \resv_i^{(t)} +  \delta (m-i-1)  \resv_i^{(t)}+  \smallgrowthlazy[\nsteps_{i+1}] \left(\delta  \basegrowthlazy \growthlazy[t]{i} \sum_{k = 1}^{i-1} \frac{L_k}{L_i} \frac{L_i}{L_k} \resv_i^{(t)} \right) \tag{I.H. of Claim~\ref{claim:induct}} \\
     & \leq \left( \gamma_i \smallgrowthlazy[\nsteps_{i+1}] +   \delta (m-i-1) + \smallgrowthlazy[\nsteps_{i+1}] \delta i \basegrowthlazy \growthlazy[t]{i} \right) \resv_i^{(t)} \\
     & \leq \left( \gamma_i + 3 \basegrowthlazy \delta m \right) \smallgrowthlazy[\nsteps_{i+1}] \resv_i^{(t)} \tag{Eq.~\ref{eqn:bless3ind}} \\
     & \leq \tilde{\gamma}_i \resv_i^{(t)} \tag{Eq.~\ref{eqn:gammaired}}.
\end{align}
Then
\begin{equation}
    \label{eqn:part1induct}
    \resv_i^{(t+1)} = \max \left \{ \overline{\resv}_i^{(t+1)}, \left( \mat{V}_i \resv^{(t)} \right)_i \right \} = \tilde{\gamma}_i \resv_i^{(t)}.
\end{equation}
Now we turn our attention to $\resv_{j}^{(t+1)}$ for $j \leq i-1$. We have
\begin{align}
     \overline{\resv}_{j}^{(t+1)} & = \left(1 - \frac{\mu_{j}}{L_i} \right)^2 \tilde{\resv}_{j}^{(t)} + \left( \delta \sum_{k = 1}^{j} \frac{L_k L_{j}}{L_i^2} \tilde{\resv}_k^{(t)} \right)  +
    \left( \delta \sum_{k = j+1}^{i-1} \frac{L_k L_{j}}{L_i^2} \tilde{\resv}_k^{(t)} \right) + \left( 2 \delta \sum_{k = i}^m \frac{L_k L_{j}}{L_i^2} \tilde{\resv}_k^{(t)} \right) \\
    & \leq  \smallgrowthlazy[\nsteps_{i+1}] \left( \resv_j^{(t)} + \delta \sum_{k = 1}^{j} \frac{L_k L_{j}}{L_i^2} \resv_k^{(t)}  + \delta \sum_{k = j+1}^{i-1} \frac{L_k L_{j}}{L_i^2} \resv_k^{(t)} \right)  + \left( \delta \sum_{k = i}^m \frac{L_k L_{j}}{L_i^2}   \frac{L_i}{L_k}   \resv_{i}^{(t)}  \right) \tag{I.H. of Lemma~\ref{lemma:maininduction}} \\
    & \leq  \smallgrowthlazy[\nsteps_{i+1}] \left( \resv_j^{(t)} + \delta  \growthlazy[t]{i} \sum_{k = 1}^{j} \frac{L_k L_{j}}{L_i^2} \frac{L_j}{L_k} \resv_j^{(t)}  + \delta \growthlazy[t]{i} \sum_{k = j+1}^{i-1} \frac{L_k L_{j}}{L_i^2} \frac{L_k}{L_j} \resv_j^{(t)} \right) \\
        & \qquad + \left( 2 \delta \growthlazy[t]{i} \sum_{k = i}^m \frac{L_k L_{j}}{L_i^2}   \frac{L_i}{L_k}   \frac{L_i}{L_j}\resv_{j}^{(t)}  \right) \tag{I.H. of Claim~\ref{claim:induct}} \\
     & = \smallgrowthlazy[\nsteps_{i+1}] \left( 1 + \delta \growthlazy[t]{i} \sum_{k = 1}^{j} \frac{L_{j}^2}{L_i^2}  + \delta \growthlazy[t]{i} \sum_{k = j+1}^{i-1} \frac{L_k^2}{L_i^2}   \right) \resv_j^{(t)} + \left( 2 \delta \growthlazy[t]{i} \sum_{k = i}^m \resv_{j}^{(t)}  \right)  \\
    & \leq  \smallgrowthlazy[\nsteps_{i+1} + 1] \resv_j^{(t)}. \tag{Eq.~\ref{eqn:bless3ind}}
\end{align}
Therefore,
\begin{equation}
    \label{eqn:part2induct}
    \resv_j^{(t+1)} = \max \left \{ \overline{\resv}_j^{(t+1)}, \left( \mat{V}_i \resv \right)_j \right \} = \smallgrowthlazy[ \nsteps_{i+1} + 1] \resv_j^{(t)}.
\end{equation}
Next we consider $\resv_{\ell}^{(t+1)}$ for $\ell \geq i+1$. We have,
\begin{align}
    \label{eqn:ellibound}
    \overline{\resv}_{\ell}^{(t+1)} & \leq  \left( \frac{L_{\ell}}{L_i} \right)^2 \tilde{\resv}_{\ell}^{(t)} +  \left( \delta \sum_{k = i+1}^m \frac{L_k L_{\ell}}{L_i^2} \tilde{\resv}_k^{(t)} \right) + \left( \delta \sum_{k = 1}^i \frac{L_k L_{\ell}}{L_i^2} \tilde{\resv}_k^{(t)} \right) \\
   &  \leq  \left( \frac{L_{\ell}}{L_i} \right)^2 \frac{L_i}{L_{\ell}}\resv_i^{(t)} +  \left( 2 \delta \sum_{k = i+1}^m \frac{L_k L_{\ell}}{L_i^2} \frac{L_i}{L_k}  \resv_i^{(t)} \right)  + \left( \delta \sum_{k = 1}^i \frac{L_k L_{\ell}}{L_i^2} \smallgrowthlazy[\nsteps_{i+1}] \resv_k^{(t)} \right)  \tag{I.H. of Lemma~\ref{lemma:maininduction}} \\
   &  \leq  \frac{L_{\ell}}{L_i} \left(  \left( 1 +  2 \delta (m-i-1) \right)  \resv_i^{(t)}  + \left( \delta \sum_{k = 1}^i \frac{L_k}{L_i}  \smallgrowthlazy[\nsteps_{i+1}] \basegrowthlazy \growthlazy[t]{i} \frac{L_i}{L_k} \resv_i^{(t)} \right)  \right) \tag{I.H. of Claim~\ref{claim:induct}} \\
    &  \leq  \frac{L_{\ell}}{L_i} \left( 1 +  2 \delta (m-i-1) + 3 \basegrowthlazy \delta  i \smallgrowthlazy[\nsteps_{i+1}] \right) \resv_i^{(t)} \tag{$\growthlazy[t]{i} \leq 3$ by Eq.~\ref{eqn:bless3ind}} \\
    & \leq \frac{L_{\ell}}{L_i} \smallgrowthlazy[\nsteps_{i+1} + 1]  \resv_i^{(t)} .
\end{align}
Therefore,
\begin{equation}
    \label{eqn:part3induct}
    \resv_{\ell}^{(t+1)} = \max \left \{ \overline{\resv}_{\ell}^{(t+1)}, \left( \mat{V}_i \resv^{(t)} \right)_{\ell} \right \} = \frac{L_{\ell}}{L_i} \smallgrowthlazy[\nsteps_{i+1} + 1]  \resv_i^{(t)}. 
\end{equation}
Combining Eq.~\ref{eqn:part1induct}, Eq.~\ref{eqn:part2induct}, and Eq.~\ref{eqn:part3induct} we conclude
\begin{equation}
    \resv^{(t+1)} = \mat{V}_i \resv^{(t)}.
\end{equation}
Finally we prove Eq.~\ref{eqn:needfortop} of Claim~\ref{claim:induct}. First we will prove by induction that if the pair $(j,k)$ is such that
Case One, Case Two, or Case Three holds in Table~\ref{table:invariantcases} \emph{and} $t \leq \threshold{i} - 1$, then we have
\begin{equation}
    \label{eqn:finaltoshow}
    \resv_k^{(t)} \leq \basegrowthlazy \growthlazy[t]{i} \max \left \{ \frac{L_k}{L_j}, \frac{L_j}{L_k} \right \} \resv_j^{(t)}.
\end{equation}
Instead if $t \in \left[ \threshold{i}-1, \threshold{i} + 1 \right]$
\begin{equation}
    \label{eqn:finaltoshowedit}
    \resv_k^{(t)} \leq \basegrowthlazy \growthlazy[t]{i} \left(\frac{1}{\tilde{\gamma}_i^2} \right) \max \left \{ \frac{L_k}{L_j}, \frac{L_j}{L_k} \right \} \resv_j^{(t)}.
\end{equation}
After proving this we will show that if the pair $(j,k)$ is such that $j \leq i$ and $k \geq i+1$ then
\begin{equation}
    \label{eqn:finaltoshowspecialcase}
    \resv_k^{(t)} \leq \basegrowthlazy \frac{\smallgrowthlazy[\nsteps_{i+1}]}{\tilde{\gamma}_i} \frac{L_k}{L_j} \resv_j^{(t)}.
\end{equation}
\begin{table}[H]
\caption{Cases for which Eq.~\ref{eqn:finaltoshow} holds}
\centering
\label{table:invariantcases}
\begin{tabular}{@{}lllll@{}}
\toprule
Case One   & $j \leq i$ &  $k \leq i$  \\ 
Case Two & $j \geq i+1$   & $k \leq i$  \\
Case Three  & $j \geq i+1$    & $ k \geq i + 1$ \\ \bottomrule
\end{tabular}
\end{table}
To that end, we begin by considering Case One where $j \leq i$ and $k \leq i  $. First assume $j \leq i-1$. Then we have
\begin{align}
\frac{\resv_k^{(t+1)}}{\resv_j^{(t+1)}} & \leq \frac{\smallgrowthlazy[\nsteps_{i+1} +1] \resv_k^{(t)}}{\smallgrowthlazy[\nsteps_{i+1} +1] \resv_j^{(t)}} = \frac{\resv_k^{(t)}}{\resv_j^{(t)}}.
\end{align}
Therefore since $\growthlazy{i} \geq 1$, 
\begin{equation}
\resv_k^{(t)} \leq \basegrowthlazy \growthlazy[t]{i} \max \left \{ \frac{L_k}{L_j}, \frac{L_j}{L_k} \right \} \resv_j^{(t)} \implies\resv_k^{(t+1)} \leq \basegrowthlazy \growthlazy[t+1]{i} \max \left \{ \frac{L_k}{L_j}, \frac{L_j}{L_k} \right \} \resv_j^{(t+1)}.
\end{equation}
Next we set $j = i$. Suppose $t + 1 \leq \threshold{i} $. Then we have
\begin{equation}
\resv_i^{(t+1)} = \tilde{\gamma}_i^{t+1} \resv_i^{(0)} \geq  \tilde{\gamma}_i^{\threshold{i}} \resv_i^{(0)}  = \frac{1}{\smallgrowthlazy[\nsteps_{i+1}]} \frac{L_{i-1}}{L_i} \resv_{i-1}^{(0)} \geq \frac{L_{i-1}}{L_i} \frac{1}{\smallgrowthlazy[(t+1)( \nsteps_{i+1}+1)]} \resv_{i-1}^{(t)}.
\end{equation}
Therefore if $j = i$ and $k \leq i$, 
\begin{align}
\resv_k^{(t+1)} & \leq \smallgrowthlazy[\nsteps_{i+1} + 1] \resv_k^{(t)} \\
& \leq \smallgrowthlazy[\nsteps_{i+1} + 1]  \basegrowthlazy \growthlazy[t]{i} \max \left \{ \frac{L_{i-1}}{L_k}, \frac{L_k}{L_{i-1}} \right \} \resv_{i-1}^{(t)} \\
& \leq \smallgrowthlazy[\nsteps_{i+1} + 1]  \basegrowthlazy \growthlazy[t]{i} \max \left \{ \frac{L_{i-1}}{L_k}, \frac{L_k}{L_{i-1}} \right \}  \frac{L_i}{L_{i-1}} \smallgrowthlazy[(t+1)( \nsteps_{i+1}+1)] \resv_i^{(t+1)} \\
& =  \smallgrowthlazy[(t+2)( \nsteps_{i+1}+1)+1] \basegrowthlazy \growthlazy[t]{i} \frac{L_i}{L_k} \resv_i^{(t+1)} \\
& \leq \basegrowthlazy \growthlazy[t+1]{i} \frac{L_i}{L_k} \resv_i^{(t+1)}. \tag{$\smallgrowthlazy[(t+2)(\nsteps_{i+1} + 1) + 1] \leq \growthlazy{i}$}
\end{align}
Next we consider Case Two where $j \geq i+1$ and $k \leq i$. Using similar reasoning we find,
\begin{align}
    \resv_k^{(t+1)} & \leq \smallgrowthlazy[\nsteps_{i+1}] \resv_k^{(t)} \tag{I.H. of Claim~\ref{claim:induct}} \\
    & \leq \smallgrowthlazy[\nsteps_{i+1}] \basegrowthlazy \growthlazy[t]{i} \frac{L_i}{L_k} \resv_i^{(t)} \tag{I.H. of Claim~\ref{claim:induct}}  \\
    & = \smallgrowthlazy[\nsteps_{i+1}]  \basegrowthlazy \growthlazy[t]{i} \frac{L_i}{L_k} \left( \frac{L_i}{L_j} \frac{1}{\smallgrowthlazy[\nsteps_{i+1}]} \resv_j^{(t+1)} \right) \tag{Eq.~\ref{eqn:part3induct}} \\
    & \leq  \basegrowthlazy \growthlazy[t]{i} \frac{L_i}{L_k} \frac{L_j}{L_i} \resv_j^{(t+1)} \tag{$L_i < L_j$} \\
    & \leq \basegrowthlazy \growthlazy[t+1]{i}  \frac{L_j}{L_k} \resv_j^{(t+1)}. \tag{$\growthlazy{i} \geq 1$}
\end{align}
Finally we consider Case Three where $j \geq i+1$ and $k \geq i +1$. We have
\begin{equation}
    \frac{\resv_k^{(t+1)}}{\resv_j^{(t+1)}} = \frac{\frac{L_k}{L_i}\smallgrowthlazy[\nsteps_{i+1}+1] \resv_i^{(t)}}{\frac{L_j}{L_i}\smallgrowthlazy[\nsteps_{i+1}+1] \resv_i^{(t)}} = \frac{L_k}{L_j}.
\end{equation}
Now we address the case where $t \in \left[ \threshold{i}-1, \threshold{i}+1 \right]$. For any $t$ in this range
\begin{equation}
    \resv_i^{(t)} = \tilde{\gamma}_i^{t} \resv_i^{(0)} \geq \tilde{\gamma}_i^{\threshold{i} + 1} \resv_i^{(0)} \geq \tilde{\gamma}_i  \frac{1}{\smallgrowthlazy} \frac{L_{i-1}}{L_i} \resv_{i-1}^{(0)} = \tilde{\gamma}_i  \frac{L_{i-1}}{L_i} \frac{1}{\smallgrowthlazy[(t+1)(\nsteps_{i+1}+1)]} \resv_{i-1}^{(t)}.
\end{equation}
Therefore, using the same logic as before we have that if $j=i$ and $k \leq i$,
\begin{align}
    \resv_k^{(t)} \leq \basegrowthlazy \growthlazy[t]{i} \frac{1}{\tilde{\gamma}_i}^2 \frac{L_i}{L_k} \resv_i^{(t)}.
\end{align}
All other cases remain the same. \\
As promised, we now consider the case where $j \leq i$ and $k \geq i + 1$ separately. By I.H. of Claim~\ref{claim:induct} for $t$ we have 
\begin{equation}
\resv_k^{(t+1)} = \frac{L_k}{L_i} \smallgrowthlazy[\nsteps_{i+1}] \resv_i^{(t)}  = \frac{L_k}{L_i} \smallgrowthlazy[\nsteps_{i+1}] \resv_i^{(t)} = \left(\frac{\smallgrowthlazy[\nsteps_{i+1}]}{\tilde{\gamma}_i} \right) \frac{L_k}{L_i} \resv_i^{(t+1)}.
\end{equation}
Also note that for any $j \leq i$, $\resv_i^{(t+1)} \leq \resv_i^{(t)}$ and $\resv_j^{(t+1)} \geq \resv_j^{(t)}$ and so for any $t \leq \threshold{i} -1$,
\begin{equation}
    \resv_i^{(0)} \leq \basegrowthlazy \resv_j^{(0)} \implies \resv_i^{(t)} \leq \basegrowthlazy \resv_j^{(t)}.
\end{equation}
Therefore for any $j \leq i$,
\begin{align}
    \resv_k^{(t+1)} & = \left(\frac{\smallgrowthlazy[\nsteps_{i+1}]}{\tilde{\gamma}_i} \right) \frac{L_k}{L_i} \resv_i^{(t+1)} \\
    & \leq \left(\frac{ \smallgrowthlazy[\nsteps_{i+1}]}{\tilde{\gamma}_i} \right) \frac{L_k}{L_i} \left( \basegrowthlazy \resv_j^{(t+1)}  \right)\\
    & \leq \left(\frac{\basegrowthlazy \smallgrowthlazy[\nsteps_{i+1}]}{\tilde{\gamma}_i} \right) \frac{L_k}{L_j} \resv_j^{(t+1)}. \tag{$L_j \leq L_i$}
\end{align}
We conclude that for $t \leq \threshold{i} + 1$ and for any $j,k$,
\begin{equation}
    \resv_k^{(t)} \leq \basegrowthlazy \max \left \{\growthlazy[t]{i}, \frac{\smallgrowthlazy[\nsteps_{i+1}]}{\tilde{\gamma}_i} \right \} \max \left \{ \frac{L_j}{L_k}, \frac{L_k}{L_j}\right \} \resv_j^{(t)}.
\end{equation}
This concludes the proof of Claim~\ref{claim:induct}. Now we consider the steps for which $t >  \left \lceil \threshold{i} \right \rceil$. We have the following claim. 
\begin{claim} [Extraneous Steps Phase.]
    \label{claim:extraind}
    Suppose $t >  \left \lceil \threshold{i} \right \rceil$. Then
    \begin{equation}
        \resv^{(t+1)} = \mat{W}_i \resv^{(t)}.
    \end{equation}
    Moreover, Eq.~\ref{eqn:needfortop} holds from Claim~\ref{claim:induct}.
\end{claim}
To prove Claim~\ref{claim:extraind} we first need to show that
\begin{equation}
    \max \left \{ \mat{U}_i \tilde{\resv}_i^{(t)}, \mat{W}_i \tilde{\resv}_i^{(t)}  \right \} = \mat{W}_i \tilde{\resv}_i^{(t)}.
\end{equation}
However we have already shown this while proving earlier that
\begin{equation}
    \max \left \{ \mat{U}_i \tilde{\resv}_i^{(t)}, \mat{V}_i \tilde{\resv}_i^{(t)}  \right \} = \mat{V}_i \tilde{\resv}_i^{(t)}.
\end{equation}
All that is left is to note that $\mat{V}_i$ and $\mat{W}_i$ are identical except $\left(\mat{V}_i \right)_{ii} < \left( \mat{W}_i \right)_{ii}$. Next we show that Eq.~\ref{eqn:needfortop} holds when $t > \left \lceil  \threshold{i} \right \rceil$. Note that for any $j \leq i-1$
\begin{equation}
    \resv_j^{(t+1)} = \left( \mat{W}_i \resv^{(t)} \right)_j =  \left( \mat{V}_i \resv^{(t)} \right)_j
\end{equation}
and for any $j \geq i$, 
\begin{equation}
     \resv_j^{(t+1)} = \left( \mat{W}_i \resv^{(t)} \right)_j \geq \left( \mat{V}_i \resv^{(t)} \right)_j.
\end{equation}
Therefore whenever the pair $(j,k)$ is such that $k \leq i-1$ we can recycle the proof showing Eq.~\ref{eqn:needfortop} holds in Claim~\ref{claim:induct} (i.e. $t \leq \left \lceil \threshold{i} \right \rceil$). We simply need to consider the case where the pair $(j,k)$ is such that $k \geq i$. Examining the proof of Eq.~\ref{eqn:needfortop} for Claim~\ref{claim:induct} we see that it suffices to simply consider $k=i$. First suppose that $j \leq i -1$. Then we are done since $\resv_i^{(t+1)} = \resv_i^{(t)}$ and $\resv_j^{(t+1)} \geq \resv_j^{(j)}$. Therefore since $\growthlazy{i} \geq 1$, 
\begin{equation}
\resv_i^{(t)} \leq \basegrowthlazy \growthlazy[t]{i} \max \left \{ \frac{L_i}{L_j}, \frac{L_j}{L_i} \right \} \resv_j^{(t)} \implies \resv_i^{(t+1)} \leq \basegrowthlazy \growthlazy[t+1]{i} \max \left \{ \frac{L_i}{L_j}, \frac{L_j}{L_i} \right \} \resv_j^{(t+1)}.
\end{equation}
Next suppose that $j \geq i + 1$. Then we have for any $t$,
\begin{equation}
    \resv_i^{(t)} = \frac{1}{\smallgrowthlazy[\nsteps_{i+1}+1]} \frac{L_i}{L_j} \resv_j^{(t)} \leq \frac{L_i}{L_j} \resv_j^{(t)}.
\end{equation}
Therefore we conclude that Eq.~\ref{eqn:needfortop} holds when $t > \left \lceil \threshold{i} \right \rceil$. 
\paragraph{Conclusion of Inductive Case} To conclude we recall Eq.~\ref{eqn:Tisize},
\begin{equation}
    T_i \geq  \left \lceil  \threshold{i} \right \rceil.
\end{equation}
Using this we see,
\begin{align}
    \resv_i^{(T_i)} & = \tilde{\gamma}_i^{\left \lceil \threshold{i} \right \rceil} \resv_i^{(0)} \\
    & \leq \tilde{\gamma}_i^{ \threshold{i}} \resv_i^{(0)} \\
    & = \frac{1}{\smallgrowthlazy[\nsteps_{i+1}]} \frac{L_{i-1}}{L_i} \resv_{i-1}^{(0)}
\end{align}
Next, since we return $\bslsres_{i+1}(\resv^{(T_i)})$ we have by the Inductive Hypothesis that Lemma~\ref{lemma:maininduction} holds for $\bslsres_{i+1}$ that if $\tilde{\resv} = \bslsres_{i+1}(\resv^{(T_i)}) $ then
\begin{align}
    \tilde{\resv}_i & \leq \smallgrowthlazy[\nsteps_{i+1}] \resv_i^{(T_i)}  \leq   \frac{L_{i-1}}{L_i} \resv_{i-1}^{(0)}.
\end{align}
Then for all $j \geq i + 1$  we have 
\begin{equation}
    \tilde{\resv}_j \leq \frac{L_i}{L_j} \resv_i^{(T_i)} \leq \frac{L_i}{L_j} \left(  \frac{L_{i-1}}{L_i} \resv_{i-1}^{(0)}\right)  \leq  \frac{L_{i-1}}{L_j} \resv_{i-1}^{(0)}.
\end{equation}
Finally since $\nsteps_i = \nsteps_{i+1} (2 T_i + 1)$, for all $j \leq i-1$
\begin{equation}
    \resv_{j}^{(T_m)} \leq \smallgrowthlazy[T_i \nsteps_{i+1}] \resv_{j}^{(0)} \leq \smallgrowthlazy[\nsteps_i] \resv_{j}^{(0)}.
\end{equation}
\end{proof}

\section{Miscellaneous helper lemmas}
\label{apx:helper}
In this section we provide some helper lemmas and that are used throughout the other sections. 
\begin{lemma}
	\label{helper:choose}
	For any integer pairs $(p, q)$ such that $p \geq q \geq 1$, we have ${p \choose q} \leq p \cdot {p - 1 \choose q-1}$.
\end{lemma}
\begin{proof}[Proof of \Cref{helper:choose}]
	By definition
	\begin{equation}
		{p \choose q} = \frac{p!}{q! (p-q)!} = \frac{p}{q} \frac{(p-1)!}{(q-1)! (p-q)!} = \frac{p}{q} \cdot {p - 1 \choose q- 1} \leq p \cdot {p - 1 \choose q- 1}.
	\end{equation}
\end{proof}

\begin{lemma}
    \label{helper:Pochhammer}
    \begin{equation}
        \prod_{j=1}^{\infty} \left( 1 - 2^{-j} \right) > 0.288.
    \end{equation}
\end{lemma}
The lemma is standard from combinatoric analysis.
We provide an elementary proof here for completeness.
\begin{proof}[Proof of \cref{helper:Pochhammer}]
    Decompose the infinity product into two parts
    \begin{equation}
        \prod_{j=1}^{\infty} (1-2^{-j}) = \left( \prod_{j=1}^{10} (1-2^{-j})  \right) \cdot \exp \left( \sum_{j=11}^{\infty} \log ( 1- 2^{-j})\right).
    \end{equation}
    Since $\log(1-2^{-j}) \geq -2^{-j+1}$ for $j \geq 1$, we obtain
    \begin{equation}
        \exp \left( \sum_{j=11}^{\infty} \log ( 1- 2^{-j})\right)
        \geq
        \exp \left( - \sum_{j=11}^{\infty} 2^{-j+1} \right) 
        =
        \exp \left( - 2^{-9} \right)
        > 
        0.998.
    \end{equation}
    On the other hand we know that $\prod_{j=1}^{10} (1-2^{-j})  > 0.289$. Therefore $\prod_{j=1}^{\infty} (1-2^{-j}) > 0.289 \times 0.998 > 0.288$.
\end{proof}

\begin{lemma}%
\label{lemma:helper}
Suppose for any $i$, $a_i, b_i, x_i > 0$. Then for any $n$,
\begin{equation}
    \frac{\sum_{i = 1}^n a_i x_i}{\sum_{i = 1}^n b_i x_i} \leq \max \limits_{i \in [n]} \left \{ \frac{a_i}{b_i} \right \}.
\end{equation}
\end{lemma}
\begin{proof}[Proof of \cref{lemma:helper}]
This follows immediately from the fact that we can write the expression on left-hand side of the inequality as a convex combination of the fractions $\frac{a_i}{b_i}$ (with coefficients $c_i=\frac{b_i x_i}{\sum_{i = 1}^n b_i x_i}$):
\[
\frac{\sum_{i = 1}^n a_i x_i}{\sum_{i = 1}^n b_i x_i} 
    =\sum_{i=1}^n \frac{a_i}{b_i} \frac{b_i x_i}{\sum_{i = 1}^n b_i x_i}
    \leq  
    \max \limits_{i \in [n]} \left \{ \frac{a_i}{b_i} \right \}
    \sum_{i=1}^n \frac{b_i x_i}{\sum_{i = 1}^n b_i x_i}
    =
    \max \limits_{i \in [n]} \left \{ \frac{a_i}{b_i} \right \}.
\]
\end{proof}

\end{document}